\documentclass[fleqn,10pt]{article}

\usepackage{titlesec}

\titleformat{\paragraph}
{\normalfont\normalsize\bfseries}{\theparagraph}{1em}{}
\titlespacing*{\paragraph}
{0pt}{3.25ex plus 1ex minus .2ex}{1.5ex plus .2ex}

\usepackage{latexsym, graphicx, epsfig, amsmath, amssymb,amsfonts}
\usepackage{natbib,amsthm,version}
\usepackage{amsbsy,bm,multirow}
\usepackage[titletoc,page]{appendix}
\usepackage[mathscr]{eucal}
\usepackage{mathtools}
\usepackage{color}
\usepackage[utf8]{inputenc}
\usepackage[english]{babel}
\usepackage{amsthm}
\usepackage[hidelinks]{hyperref}
\usepackage{url}
\usepackage[lined,boxed,linesnumbered,ruled]{algorithm2e}
\usepackage{float}
\usepackage{arydshln}
\usepackage{dsfont,bbm}
\usepackage{courier}
\usepackage{caption,subcaption}
\usepackage{epstopdf}
\usepackage{mathabx}

\usepackage[shortlabels]{enumitem}
\setlist{nolistsep}

\usepackage{natbib}
\usepackage{pgfplots}
\pgfplotsset{compat=1.17}
\usetikzlibrary{arrows.meta}

\newcommand{\mbs}[1]{\mathbf{#1}}
\newcommand{\mbb}[1]{\mathbb{#1}}

\newtheorem{theorem}{Theorem}[section]

\newtheorem{lemma}[theorem]{Lemma}
\newtheorem{property}{Property}[section]
\newtheorem{remark}{Remark}[section]

\theoremstyle{definition}

\title{
  A Method for
  Exact Enforcement of Boundary Conditions on General Curved
  Triangular Domains: Algorithm and Application to Scientific Machine Learning
  
} 
\author{
  Suchuan Dong\thanks{Author of correspondence. Email: sdong@purdue.edu}, \
  Yuchuan Zhang
  \\
  Center for Computational and Applied Mathematics \\
  Department of Mathematics \\
  Purdue University, USA 
 } 

\date{(October 12, 2026)}

\begin{document}
\maketitle


\begin{abstract}

  We present a method for enforcing exactly Dirichlet, Neumann, and Robin
  conditions on general triangular domains with arbitrarily curved boundaries.
  The algorithmic constructions developed here complements those in
  a previous work for general quadrilateral domains.
  Our method combines (i) the exact mapping between a general curved
  triangular domain and the standard triangle, (ii) TFC (theory of functional
  connections) constrained expressions, and (iii) Barnhill transfinite
  interpolations, to provide explicit closed-form trial ansatze that mathematically
  satisfy the imposed boundary conditions exactly.
  Corner compatibility induced by the boundary conditions is
  pivotal, and enforcing exactly the imposed boundary conditions hinges on
  the exact enforcement of the induced compatibility constraints.
  The corner compatibility  associated with
  the acute vertices of the standard triangle and the coupling between 
  gradient components when the hypotenuse corresponds to a Neumann or Robin
  boundary pose particular challenges. 
  Two  strategies are critical to
  resolving these issues, by reformulating
  conditional corner constraints into equivalent unconditional forms and by introducing a
  separate TFC-style auxiliary edge construction together with an associated linear solve
  on the hypotenuse. We analyze the corner compatibility constraints
  and present formulations for the general form of functions that exactly
  satisfy the imposed Dirichlet, Neumann, or Robin conditions
  on general curved triangular domains.
  The proposed method has been implemented together with extreme learning machine (ELM) type
  randomized neural networks for solving PDE problems.
  We present ample numerical experiments with linear/nonlinear, stationary/dynamic
  boundary-value or initial/boundary-value problems on a variety of domains
  with complex boundary geometries. Numerical results show that
  the current method has enforced the Dirichlet, Neumann, and Robin conditions
  on curved domain boundaries to the machine precision.

\end{abstract}


\vspace{0.05cm}
Keywords: {\em
  exact boundary condition enforcement,
  artificial neural networks,
  scientific machine learning,
  transfinite interpolation,
  extreme learning machine
}



\section{Introduction}


Physics informed machine learning has recently emerged as a viable alternative and complement to
traditional scientific computing techniques. Physical
laws and thermodynamic principles, often expressed in the form of partial or ordinary
differential equations (PDEs/ODEs),
together with the associated boundary or initial conditions (BC/IC), are incorporated
into the training of artificial neural networks (NNs), either through
loss functions or by other means, to approximate the solution fields and
other quantities of interest in the underlying physical problems.
The universal approximation capabilities and related theoretical properties
of neural networks suggest enormous potential for scientific computing. 
However, this theoretical expressive power does not necessarily translate into
efficient, accurate, or robust numerical algorithms. Bridging this gap remains
a fundamental challenge in scientific machine learning.

Treatment of boundary conditions (including initial conditions) is a central
issue  in physics informed machine learning.
Boundary conditions represent how a physical system interacts with its environment and,
together with the governing equations, determine the admissible solution.
Their treatment intimately influences the accuracy and fidelity of scientific machine learning:
a neural approximation can exhibit a small interior PDE residual and yet remain inconsistent
with the prescribed boundary data. This issue is addressed in the literature  through
several approaches,
including residual penalties, constrained optimization, admissible trial functions,
and variational boundary terms.
The distinction between these approaches concerns not only the implementation and
training cost, but also whether boundary satisfaction relies on the
outcome of optimization or is an intrinsic property of the neural representation.

In the canonical physics-informed neural network (PINN) formulation, boundary and initial conditions contribute weighted residual terms to the training loss alongside the
differential equation \citep{Raissi2019}. This soft treatment readily accommodates
mixed boundary operators
and permits fitting incomplete or noisy observations.
However, the competing objectives may not decrease at comparable rates,
and satisfying the boundary loss at sampled points
does not ensure uniform satisfaction on the boundary.
These difficulties are linked 
to stiffness and unbalanced gradients in composite PINN objectives
by~\citep{WangTengPerdikaris2021},
which introduced adaptive balancing based on gradient statistics.
Subsequent developments include trainable pointwise loss weights in self-adaptive PINNs \citep{McClennyBragaNeto2023}, relative loss balancing with random lookback \citep{BischofKraus2025},
and Lagrangian or augmented-Lagrangian approaches that update multipliers and
penalties \citep{LuEtAl2021,Son2023,MakridakisEtAl2024}.
These techniques improve the treatment of competing constraints,
but numerical boundary satisfaction still depends on the approximation, sampling, and convergence.
In particular, constraints designated ``hard'' in an optimization formulation
may not hold identically for every network output.

Representation-based enforcement instead restricts the approximation to functions satisfying the prescribed data. This strategy predates contemporary physics informed neural network methods.
Lagaris et al.\ \citep{Lagaris1998} introduced the neural trial-solution form with a boundary-satisfying component and a correction term that leaves the boundary data unchanged. Subsequent work
explored radial-basis corrections for irregular boundaries \citep{Lagaris2000} and length-factor constructions for mixed Dirichlet--Neumann problems \citep{McFallMahan2009}.
For Dirichlet conditions, this familiar structure is a lifting of the prescribed data
plus a neural correction multiplied by a factor that vanishes on the constrained boundary.
This approach was extended to complex geometries in
\citep{BergNystrom2018}  with auxiliary neural functions.
In the penalty-free neural network method~\citep{ShengYang2021}
the boundary and interior components were constructed separately and combined into
a weak formulation.
When the lifting and correction have the required trace properties,
this structure removes the corresponding boundary constraint from the NN training.

The geometric components of admissible trial functions are crucial to both exactness and numerical performance. Sukumar and Srivastava \citep{SukumarSrivastava2022} developed a systematic framework
using approximate distance functions and the theory of R-functions
for several boundary operators. An approximate distance function can differ from the
Euclidean distance while retaining the exact boundary as its zero set; its use therefore need not introduce an error in homogeneous Dirichlet data. In contrast, unconstrained regression of a lifting or length factor generally leaves an auxiliary fitting error in the boundary trace.
Automatic boundary fitting is investigated in
\citep{XieMaWang2023}   with radial-basis and neural-network functions, reducing manual construction.
Dirichlet treatments in PINNs and variational PINNs are compared in
\citep{Berrone2023}, which finds a substantial advantage in exact output modification
from the numerical tests.

Derivative conditions introduce additional issues. A vanishing factor preserves
a Dirichlet trace but generally changes the normal derivative, and so Neumann and Robin
data require constructions that account for the boundary normal and derivatives of the
correction. Exact boundary and initial condition enforcement in
the deep mixed residual method is investigated in
\citep{LyuEtAl2021}  using distance functions and explicit constructions.
Auxiliary variables are used in \citep{LiuEtAl2022} to
reformulate several types of boundary conditions
as linear constraints and construct admissible outputs on complex geometries.
FO-PINN similarly reduces the derivative order of the PDE residual,
alleviating difficulties associated with differentiating distance-based trial
functions \citep{GladstoneEtAl2025}. For periodic problems,
neural-network layers are introduced in \citep{DongNi2021}
that enforce periodicity of the represented solution and its derivatives to a specified
finite order or to infinite orders. 
Fourier-feature embeddings for hard Neumann constraints
are subsequently developed in \citep{Straub2025} for diffusion problems,
including multiscale settings.
Such architectures encode particular boundary operators directly, although
extending them to other geometries and operators requires additional analysis.

Functional interpolation provides an alternative to the multiplicative boundary factor.
In the theory of functional connections (TFC), linear value and derivative constraints are embedded analytically in constrained expressions containing an otherwise
free function \citep{Mortari2017,MortariLeake2019}.
Deep TFC chooses a neural network as this free function and optimizes the PDE
residual \citep{LeakeMortari2020}, and extreme TFC combines constrained expressions with
randomized hidden-layer features in the neural network \citep{SchiassiEtAl2021}.
Using constrained expressions beyond regular rectangular domains requires
appropriate geometric constructions. Several bijective mappings for this purpose are investigated
in  \citep{MortariArnas2020}. In the more recent work \citep{DongZ2026}
exact domain mappings, TFC constrained expressions, and transfinite interpolation are combined
to impose Dirichlet, Neumann, and Robin conditions exactly on general curved quadrilateral domains.
The treatment therein explicitly incorporates the compatibility constraints
induced at intersecting boundaries, and their numerical tests achieve boundary
residuals at machine accuracy.
A Wachspress-based transfinite construction for exact Dirichlet conditions is developed in
\citep{SukumarRoy2026} for convex polygonal domains, addressing the derivative regularity
limitation of earlier distance-based trial functions.
These developments highlight that geometric flexibility, boundary regularity, and
compatibility of the prescribed data are integral to an exact construction.
At a corner, normal derivatives must be accounted for on the incident sides,
and a smooth trial function must satisfy the resulting compatibility requirements.

Variational formulations offer a different treatment of the boundary data.
In weak or energy formulations, Neumann conditions enter naturally through boundary
integrals, whereas essential Dirichlet conditions require an admissible trial space
or additional boundary terms. The deep Ritz method \citep{EYu2018} and hp-variational
PINNs \citep{KharazmiEtAl2021} integrate neural approximations into these classical formulations.
The deep Nitsche method \citep{LiaoMing2021} imposes essential
conditions using consistent boundary terms and stabilization, avoiding a global lifting.
Natural or Nitsche enforcement does not render the prescribed boundary operator a
pointwise identity of the network. Its accuracy depends on the variational approximation,
quadrature, and stability of the formulation. Accordingly, variational weak enforcement
should be distinguished from simply adding a boundary residual penalty to a strong-form loss.

Boundary treatment has also entered operator-learning investigations.
Generalized solution structures and orthogonal projections are introduced in \citep{Goschel2025}
for physics-informed neural operators to address instabilities in distance-function
constructions on boundaries that are only piecewise smooth. In \citep{FangWangPerdikaris2024}
boundary-integral representations are used to learn parametric PDE solutions from
boundary-only residuals, exploiting the PDE operator's integral structure.

In the present work we develop a method for exact enforcement of Dirichlet, Neumann, and
Robin conditions on general triangular domains with arbitrarily curved boundaries.
This triangular domain-based approach complements
the quadrilateral-based method of \cite{DongZ2026},
and provides an effective means of exactly enforcing boundary conditions
on a class of computational domains with complex boundary geometries.


The current method leverages a combination of three strategies --
exact mapping between a general curved triangular domain and the standard triangle,
TFC constrained expressions, and Barnhill-type transfinite interpolations --
to produce an explicit closed-form trial ansatz that mathematically satisfies
the imposed boundary conditions exactly.
The Barnhill strategy for transfinite interpolation on the right triangle~\cite{BarnhillG1975}
is key to the current constructions.
Unlike Gordon-type transfinite interpolations~\cite{Gordon1971},
which involve commutative operators,
Barnhill transfinite interpolation employs the boolean sum of two operators
that are non-commutative, one for the hypotenuse and the other for the two legs
of the right triangle. The hypotenuse operator blends the
Taylor approximations of appropriate orders on the hypotenuse
employing two-dimensional (2D) Barnhill
polynomials, while the second operator results from a boolean sum of the Taylor approximation
operators of appropriate orders on the two legs. 

The exact mapping between a general triangular domain and the standard triangular
domain is achieved by employing the Barnhill transfinite interpolation and TFC constrained
expressions.
For severely distorted domains, we further impose the Winslow
conditions~\cite{Winslow1967} to maintain univalency while preserving the exactness of
the map. Enforcing Dirichlet conditions on general triangular domains is conceptually
analogous to the formulation of the exact domain map.
Enforcing derivative-type conditions (Neumann or Robin) on curved triangular
domains, on the other hand, pose significant challenges, due to coupling
between the gradient components after transforming
into the standard triangle, especially on the hypotenuse.
Along with the Barnhill transfinite interpolation and the
four-step procedure of~\cite{DongZ2026},  two strategies
are critical to the success of the formulation for  Neumann/Robin conditions:
(i) transforming all conditional compatibility constraints associated with
an acute corner of the standard triangle
into unconditional forms, whenever that corner has an incident
Neumann or Robin side; and (ii) introducing a separate TFC-style auxiliary
edge function on the hypotenuse and an associated linear solve for
the gradient components thereupon,
whenever the hypotenuse corresponds to a Neumann or Robin boundary.
Leveraging these strategies, we analyze the corner compatibility constraints induced 
by the Dirichlet/Neumann/Robin conditions on
the incident boundaries, and devise algorithmic constructions to treat the imposed boundary
conditions and their induced compatibility constraints under two configurations:
(i) Neumann (or Robin) BC imposed on a single boundary, and (ii) Neumann (or Robin)
BC imposed on two triangle boundaries.
We present formulations for the general form of trial functions that exactly satisfy
the imposed boundary conditions and their compatibility constraints.
For all other boundary-condition configurations on the triangular domain
(e.g.~involving combinations of
Dirichlet, Neumann, and Robin conditions), the induced corner compatibility conditions
can be reduced to those associated with the configurations discussed above,
and the general form
of the trial ansatz can be constructed analogously.


The method developed here for exact boundary-condition enforcement has been implemented
together with extreme learning machine (ELM) type randomized neural
networks~\cite{DongL2021,DongL2021bip,NiD2023,DongW2023,WangD2024} for solving boundary-value or
initial/boundary-value problems (BVPs/IBVPs). ELM employs randomized feed-forward neural networks,
in which the hidden-layer coefficients are assigned to random values
and fixed (non-trainable)~\cite{HuangZS2006,IgelnikP1995} and only the output-layer
parameters are trained. Moreover, training of the neural network is by the linear
least squares method for linear problems and by the nonlinear least squares method
(Gauss-Newton method) for nonlinear
problems~\cite{DongL2021,Bjorck1996,DongY2022,DongN2025,YangD2026}, not by gradient
descent type algorithms. We refer
to~\cite{CalabroFS2021,SchiassiEtAl2021,ChenCEY2022,SunDF2024,ZhangBJZ2024,FabianiKSY2025,FalcoSC2026}
(among others)
for several related techniques in the literature.

We present extensive numerical experiments involving a range of
linear and nonlinear, stationary and time-dependent PDEs to assess
the effectiveness and performance of the proposed method.
These experiments encompass a 
variety of computational domains with complex boundary geometries,
including time-evolving domains.
The numerical results demonstrate that the ELM technique,
combined with the proposed boundary-condition enforcement method,
produces highly accurate solutions. Furthermore, the proposed method enforces
Dirichlet, Neumann, and Robin boundary conditions exactly by construction,
with the numerical boundary-condition errors consistently at
the machine precision along curved domain boundaries.


The primary contribution of this work is the development of a systematic method
for constructing trial ansatz functions that exactly satisfy Dirichlet, Neumann, and Robin
boundary conditions on general triangular domains with arbitrarily curved boundaries.
A central technical challenge arises from the corner compatibility constraints
when either acute vertex of the standard triangle lies on a Neumann or Robin boundary,
or when the hypotenuse corresponds to a Neumann or Robin boundary.
The characterization of these constraints, induced by different combinations
of Dirichlet, Neumann, and Robin conditions on adjacent boundaries,
and the development of algorithmic procedures for systematically incorporating and
enforcing them constitute a key aspect of the proposed method.
A further contribution is the numerical demonstration of the method's effectiveness
on computational domains with complex boundary geometries,
achieving boundary-condition errors at the machine precision
for Dirichlet, Neumann, and Robin conditions.

The remainder of this paper is organized as follows. Section \ref{sec_method}
develops the exact mapping between general curved triangular domains and the standard
triangle, together with formulations and algorithms for exact enforcement of Dirichlet,
Neumann, and Robin boundary conditions. It also discusses the ELM implementation for
PDE problems and the use of Winslow conditions to maintain the univalency of domain
mapping. Section \ref{sec_tests} presents numerical experiments involving linear and
nonlinear, stationary and time-dependent PDEs on domains with complex boundary geometries,
demonstrating boundary-condition enforcement to machine precision.
Section \ref{sec_summary} summarizes the method and offers concluding remarks.
The Appendix (Section \ref{sec_append}) provides the formulations for Neumann or
Robin conditions on the hypotenuse of the standard triangle, along with geometric
parameters for all computational domains used in Section \ref{sec_tests}.


\section{Exact Boundary-Condition Enforcement on Curved Triangular Domains}
\label{sec_method}

\subsection{Blending Functions on $[-1,1]$ and 2D Barnhill Polynomials}

We recall the properties of several Lagrange or Hermite type interpolation polynomials
on $\xi\in[-1,1]$, which are used extensively in 
 subsequent sections. 
Following the terminology in \cite{Gordon1971},
we refer to these polynomials  as  blending functions in one dimension (1D).
We further recall several 2D blending functions (i.e.~Barnhill polynomials), whose defining
conditions are given in~\cite{BarnhillG1975}. These polynomials
are critical to the exact enforcement of boundary conditions over the triangle.

The linear blending functions (Lagrange polynomials) are given by,
\begin{align}\label{eq_5}
  \phi_0(\xi) = \frac12(1-\xi), \quad
  \phi_1(\xi) = \frac12(1+\xi), \quad \xi\in[-1,1],
\end{align}
with the properties $\phi_0(\mp 1)=1/0$ and $\phi_1(\mp 1)=0/1$.
There is a slight abuse of notation here, and ``/'' has been used to separate
the two values. For example, $\phi_0(\mp 1)=1/0$ means that $\phi_0(-1)=1$ and $\phi_0(1)=0$.
Similar notations are adopted subsequently.

The $C^1$ blending functions
(Hermite interpolation polynomials) on $\xi\in[-1,1]$ are given by (see~\cite{DongZ2026}),
\begin{align}\label{eq_29}
  \begin{array}{ll}
    \varphi_0(\xi) = \phi_0^2(\xi)\left[1+2\phi_1(\xi) \right], &
    \varphi_1(\xi) = \phi_1^2(\xi)\left[1+2\phi_0(\xi) \right], \\
    \psi_0(\xi) = 2\phi_0^2(\xi)\phi_1(\xi), &
    \psi_1(\xi) = -2\phi_0(\xi)\phi_1^2(\xi),
  \end{array}
\end{align}
where $\phi_0$ and $\phi_1$ are the linear blending functions.
These functions have the properties,
\begin{align}\label{eq_30}
  \left\{\begin{array}{llll}
    \varphi_0(\mp 1)=1/0, & \varphi_1(\mp 1)=0/1, & \psi_0(\mp 1)=0, & \psi_1(\mp 1)=0, \\
    \varphi_0'(\mp 1)=0, & \varphi_1'(\mp 1)=0, & \psi_0'(\mp 1) = 1/0,
    & \psi_1'(\mp 1) = 0/1; \\
    \varphi_0(\xi)+\varphi_1(\xi)=1,  & \forall\xi\in[-1,1].
  \end{array}\right.
\end{align}

The $C^2$ blending
functions (Hermite interpolation polynomials) on $\xi\in[-1,1]$
are given by (see~\cite{DongZ2026}),
\begin{equation}\label{eq_n4}
  \begin{array}{lll}
    \rho_0(\xi) = \phi_0^3(\xi)\left[1+3\phi_1(\xi)+6\phi_1^2(\xi) \right],
    & \upsilon_0(\xi) = 2\phi_0^3(\xi)\phi_1(\xi)\left[1+3\phi_1(\xi) \right],
    & \omega_0(\xi) = 2\phi_0^3(\xi)\phi_1^2(\xi), \\[2pt]
    \rho_1(\xi) = \phi_1^3(\xi)\left[1+3\phi_0(\xi)+6\phi_0^2(\xi) \right],
    & \upsilon_1(\xi) = -2\phi_1^3(\xi)\phi_0(\xi)\left[1+3\phi_0(\xi) \right],
    & \omega_1(\xi) = 2\phi_1^3(\xi)\phi_0^2(\xi).
  \end{array}
\end{equation}
They have the following properties,
\begin{equation}\label{eq_93}
  \left\{\begin{array}{llllll}
    \rho_0(\mp 1)=1/0, & \rho_1(\mp 1)=0/1, & \upsilon_0(\mp 1)=0, & \upsilon_1(\mp 1)=0
    & \omega_0(\mp 1)=0, & \omega_1(\mp 1)=0 \\
    \rho_0'(\mp 1)=0, & \rho_1'(\mp 1)=0, & \upsilon_0'(\mp 1)=1/0, & \upsilon_1'(\mp 1)=0/1,
    & \omega_0'(\mp 1)=0, & \omega_1'(\mp 1)=0, \\
    \rho_0''(\mp 1)=0, & \rho_1''(\mp 1)=0, & \upsilon_0''(\mp 1)=0, & \upsilon_1''(\mp 1)=0,
    & \omega_0''(\mp 1)=1/0, & \omega_1''(\mp 1)=0/1,
  \end{array}\right.
\end{equation}
and
\begin{align}\label{eq_94}
  &\rho_0(\xi) + \rho_1(\xi)=1, \quad \forall\xi\in[-1,1].
\end{align}

We next recall the 2D blending functions, referred to as Barnhill polynomials
in this paper.
Their defining conditions are specified for the unit square $[0,1]^2$
in~\cite{BarnhillG1975}, which also provides 
the explicit forms for the $C^0$ and $C^1$ Barnhill polynomials.
When adapted to the domain $[-1,1]^2$, the $C^0$ Barnhill polynomials
consist of $\phi_1(\xi)$ and $\phi_1(\eta)$ (linear blending functions).
The $C^1$ Barnhill polynomials are given by,
\begin{align}\label{eq_31}
  \begin{array}{l}
    \alpha(\xi,\eta) = \varphi_1(\eta) + 6\phi_1(\xi)\phi_1^2(\eta)\left[
      1-\phi_1(\xi)-\phi_1(\eta) \right], \\
    \beta(\xi,\eta) = \varphi_1(\xi) + 6\phi_1^2(\xi)\phi_1(\eta)\left[
      1-\phi_1(\xi)-\phi_1(\eta) \right], \quad \xi,\eta\in[-1,1],
  \end{array}
\end{align}
where $\varphi_1$ is defined in \eqref{eq_29}. These functions have the following properties,
\begin{align}\label{eq_32}
  \small
  \left\{
  \begin{array}{llll}
    \alpha(-1,\eta) = \varphi_1(\eta), & \alpha(\xi,-1)=0,
    & \alpha(-\eta,\eta)=\varphi_1(\eta), & \alpha(\xi,-\xi)=\varphi_0(\xi);
    \\
    \beta(-1,\eta)=0, & \beta(\xi,-1) = \varphi_1(\xi), 
    & \beta(-\eta,\eta) = \varphi_0(\eta), & \beta(\xi,-\xi)=\varphi_1(\xi); \\
    \alpha_{\xi}(\xi,-1)=0, & \alpha_{\eta}(\xi,-1)=0,
    & \beta_{\xi}(-1,\eta)=0, & \beta_{\eta}(-1,\eta)=0, \\
    \alpha(-\eta,\eta)+\beta(-\eta,\eta)=1, & \alpha(\xi,-\xi)+\beta(\xi,-\xi)=1, \\
    \alpha_{\xi}(-\eta,\eta) + \beta_{\xi}(-\eta,\eta)=0,
    & \alpha_{\eta}(-\eta,\eta) + \beta_{\eta}(-\eta,\eta)=0.
  \end{array}
  \right.
\end{align}

The explicit form of $C^2$ Barnhill polynomials is not available
in~\cite{BarnhillG1975}. 
But based on the defining conditions therein,
the $C^2$ Barnhill polynomials can be explicitly expressed as,
\begin{align}\label{eq_n9}
  \begin{array}{l}
    \kappa(\xi,\eta) = \rho_1(\eta) + 30\phi_1(\xi)\phi_0(\eta)\phi_1^3(\eta)s(\xi,\eta)[1+3s(\xi,\eta)], \\
    \delta(\xi,\eta) = \rho_1(\xi) + 30\phi_1(\eta)\phi_0(\xi)\phi_1^3(\xi)s(\xi,\eta)[1+3s(\xi,\eta)],
    \quad \xi,\eta\in[-1,1],
  \end{array}
\end{align}
where $s(\xi,\eta) = 1-\phi_1(\xi) - \phi_1(\eta)=\phi_0(\xi)-\phi_1(\eta)$ and
$\rho_1$ is defined in \eqref{eq_n4}.
They have the following properties,
\begin{align}\label{eq_n10}
  \left\{\begin{array}{lll}
  \kappa(-1,\eta) = \rho_1(\eta), &  \kappa(-\eta,\eta) = \rho_1(\eta), & \kappa(\xi,-\xi)=\rho_0(\xi), \\
  \delta(\xi,-1) = \rho_1(\xi), & \delta(-\eta,\eta) = \rho_0(\eta), & \delta(\xi,-\xi)=\rho_1(\xi), \\
  \kappa(\xi,-1)=0, & \kappa_{\eta}(\xi,-1)=0, & \kappa_{\eta\eta}(\xi,-1)=0, \\
  \delta(-1,\eta) = 0, & \delta_{\xi}(-1,\eta)=0, & \delta_{\xi\xi}(-1,\eta)=0, \\
  \kappa(-\eta,\eta) + \delta(-\eta,\eta)=1, & \kappa_{\xi}(-\eta,\eta)+\delta_{\xi}(-\eta,\eta)=0,
  & \kappa_{\eta}(-\eta,\eta) + \delta_{\eta}(-\eta,\eta)=0, \\
  \kappa_{\xi\xi}(-\eta,\eta)+\delta_{\xi\xi}(-\eta,\eta)=0,
  & \kappa_{\eta\eta}(-\eta,\eta)+\delta_{\eta\eta}(-\eta,\eta)=0,
  & \kappa_{\xi\eta}(-\eta,\eta)+\delta_{\xi\eta}(-\eta,\eta)=0.
  \end{array}\right.
\end{align}
The $C^1$ and $C^2$ Barnhill polynomials are critical to the exact enforcement
of Neumann and Robin 
conditions on those boundaries that correspond to the hypotenuse of the standard triangle.

\subsection{Mapping General Triangular Domains with Curved Boundaries}
\label{sec_21}

\begin{figure}[tb]
  \centerline{
    \begin{tabular}{ccc}
    \includegraphics[width=2in]{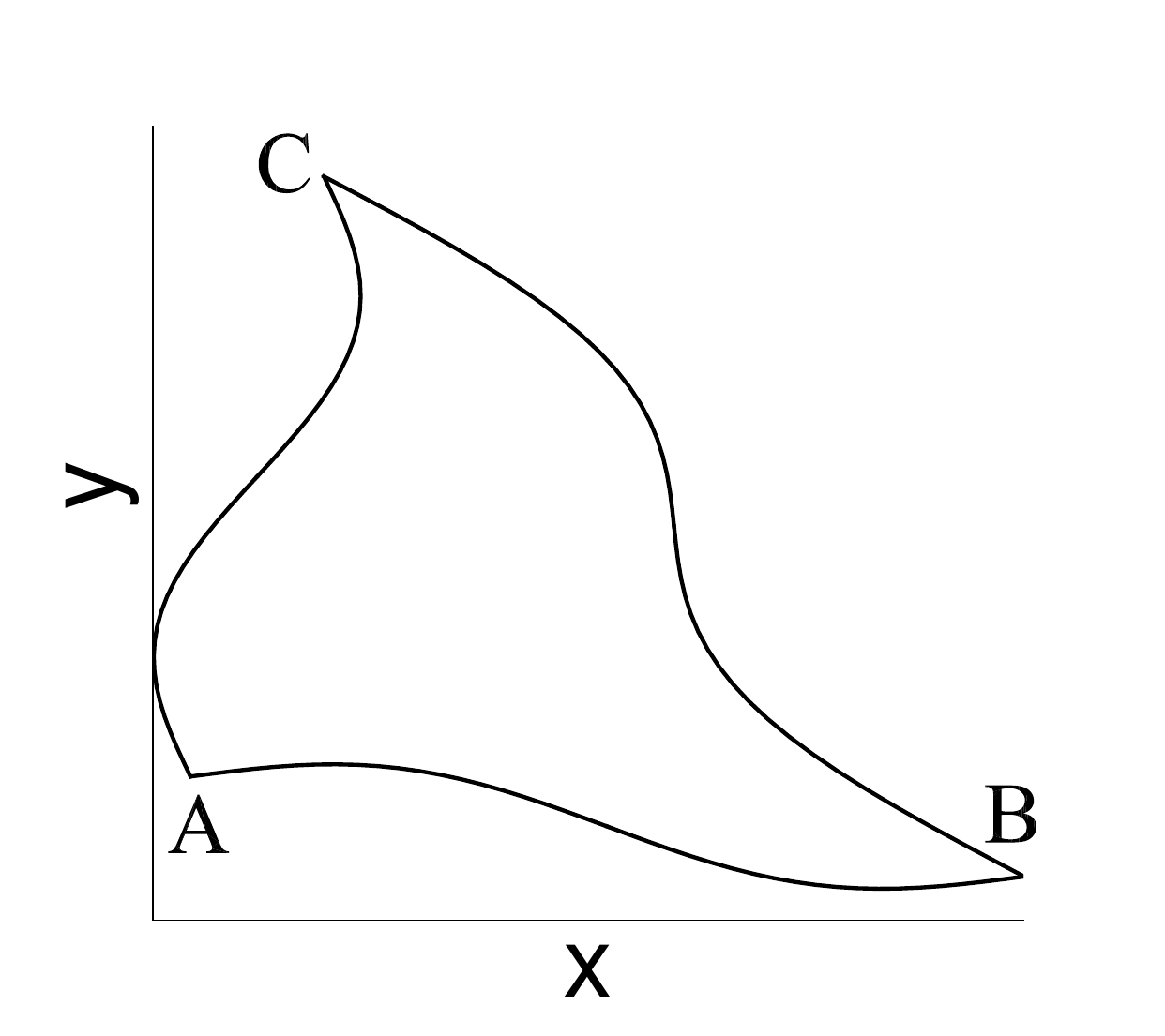}(a) 
    &
    \raisebox{0.8in}{\Large $\xLeftrightarrow{\  \text{mapping}\ }$} $\quad$
    &
    \includegraphics[width=2in]{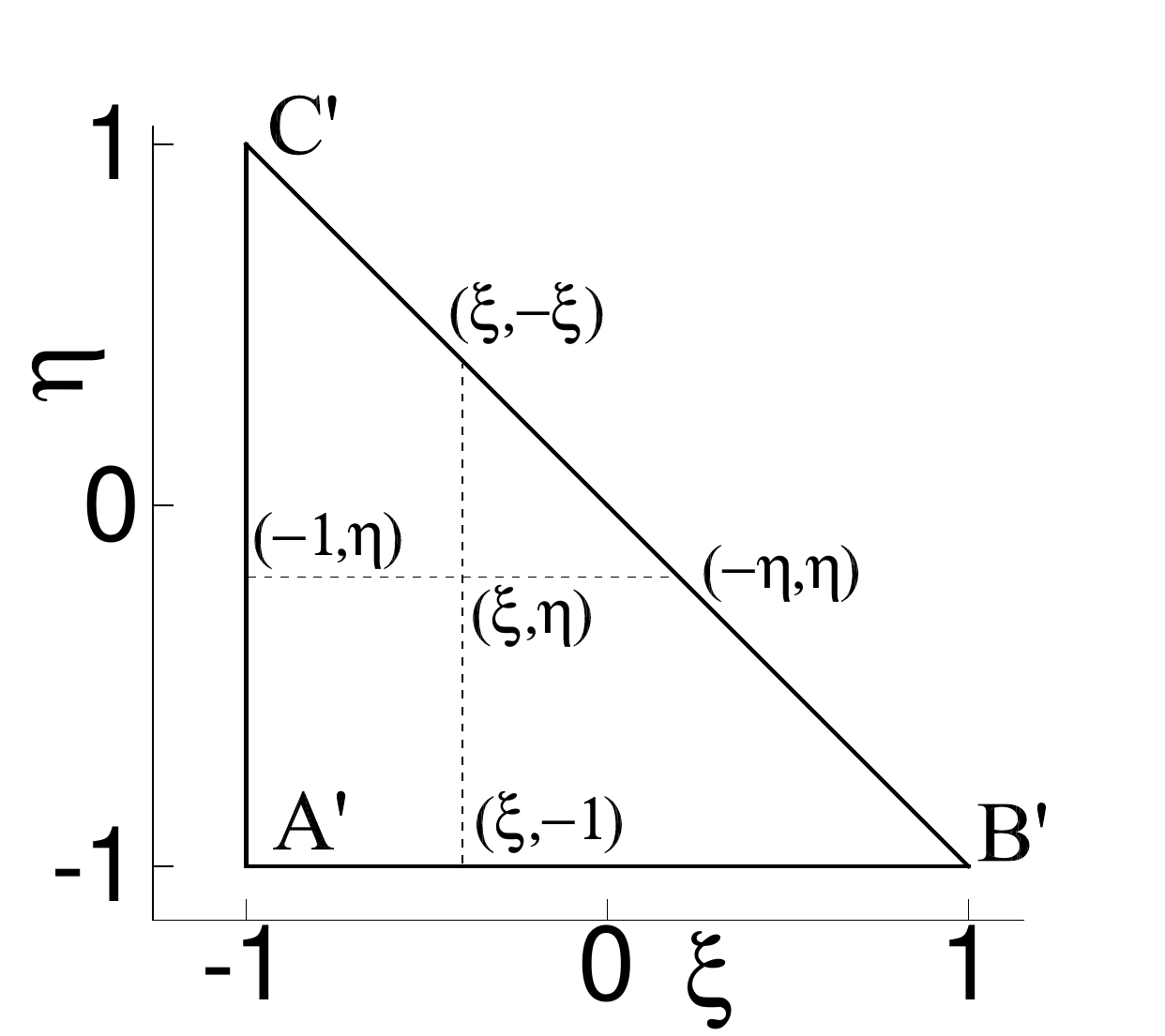}(b)
    \end{tabular}
  }
  \caption{Sketch of (a) a general triangular domain $\Omega$ with curved boundaries,
    and (b) the standard triangular domain
    $\Omega_{st}=\{(\xi,\eta)\ :\ -1\leqslant \xi, \eta \leqslant 1,\ \xi+\eta\leqslant 0 \}$.
  }
  \label{fg_1}
\end{figure}

Consider a general triangular domain $\Omega=\overline{ABC}\subset\mbb R^2$
with arbitrarily curved boundaries, as illustrated in Figure~\ref{fg_1}(a).
The boundary curves are prescribed in parametric forms, $\mbs x_{AB}(\xi)$,
$\mbs x_{BC}(\xi)$ and $\mbs x_{AC}(\xi)$ for $\xi\in[-1,1]$,
which are compatible at the shared vertices,
\begin{align}
  \mbs x_{AB}(-1)=\mbs x_{AC}(-1)=\mbs x_A, \quad
  \mbs x_{AB}(1)=\mbs x_{BC}(-1)=\mbs x_B, \quad
  \mbs x_{AC}(1)=\mbs x_{BC}(1)=\mbs x_C,
\end{align}
where $\mbs x_A$, $\mbs x_B$ and $\mbs x_C$ denote the coordinates of the
vertices $A$, $B$ and $C$, respectively.

To facilitate subsequent developments on the boundary-condition enforcement 
over $\Omega$, let $\mbs x=(x,y)=\mbs x(\xi,\eta)$ denote a regular
map~\cite{Fleming1977} (of class at least $C^1$, univalent, non-singular Jacobian matrix)
between $\Omega$ and the standard triangular domain
$\Omega_{st}=\overline{A'B'C'}=\{(\xi,\eta) : -1\leqslant \xi,\eta\leqslant 1, 
\ \xi+\eta\leqslant 0 \}$ (see Figure~\ref{fg_1}(b)).
The map should satisfy the following conditions,
\begin{subequations}\label{eq_2}
\begin{align}
  & \mbs x(-1,\eta) = \mbs x_{AC}(\eta) =: X(-1,\eta), \quad \eta\in[-1,1]; \\
  & \mbs x(\xi,-1) = \mbs x_{AB}(\xi) =: X(\xi,-1), \quad \xi\in[-1,1]; \\
  & \mbs x(-\eta,\eta) = \mbs x_{BC}(\eta) =: X(-\eta,\eta), \quad
  \mbs x(\xi,-\xi) = \mbs x_{BC}(-\xi)=X(\xi,-\xi), \quad \xi,\eta\in[-1,1].
\end{align}
\end{subequations}

The general form of functions that exactly satisfy the conditions in~\eqref{eq_2}
is given by the following TFC constrained expression,
\begin{equation}\label{eq_3}
  \mbs x(\xi,\eta) = Z(\xi,\eta) - P Z(\xi,\eta) + P X(\xi,\eta),
\end{equation}
where $Z(\xi,\eta)\in\mbb R^2$ is a free (arbitrary) function.
$P$ is a 2D transfinite interpolation operator, given by
\begingroup\small
\begin{align}\label{eq_4}
  PX(\xi,\eta) =&\ (P_1\oplus P_2)X(\xi,\eta) = P_1X(\xi,\eta) + P_2X(\xi,\eta)
  - P_1P_2X(\xi,\eta) \notag \\
  =&\ \big[X(\xi,-1)+X(-1,\eta)-X(-1,-1) \big] + 
  \phi_1(\eta)\big[X(\xi,-\xi)-X(\xi,-1) - X(-1,-\xi) \notag \\
    &\ + X(-1,-1) \big] +
  \phi_1(\xi)\big[X(-\eta,\eta) - X(-\eta,-1) - X(-1,\eta) + X(-1,-1) \big],
\end{align}\endgroup
where $\oplus$ denotes the boolean sum, and $\phi_0(\xi)$ and $\phi_1(\xi)$ are the
linear blending functions defined in \eqref{eq_5}.
$P_1$ and $P_2$ in~\eqref{eq_4} are two transfinite interpolation
operators constructed by following the
idea of Barnhill~\cite{BarnhillG1975} for right triangles, given by
\begin{subequations}\label{eq_6}\small
  \begin{align}
    P_1X(\xi,\eta) =&\ \phi_1(\eta)X(\xi,-\xi) + \phi_1(\xi)X(-\eta,\eta), \label{eq_6a}  \\
    P_2X(\xi,\eta) =&\ (T_{\xi}\oplus T_{\eta})X(\xi,\eta)
    = T_{\xi}X(\xi,\eta) + T_{\eta}X(\xi,\eta) - T_{\xi}T_{\eta}X(\xi,\eta) \notag \\
    =& X(\xi,-1) + X(-1,\eta) - X(-1,-1). \label{eq_6b}
  \end{align}
\end{subequations}
Here $T_{\xi}$ and $T_{\eta}$ are  Taylor approximation operators (0-th order) along the
$\xi$ and $\eta$ directions (about $\xi=-1$ and $\eta=-1$, respectively),
given by $T_{\xi}X(\xi,\eta)=X(-1,\eta)$ and $T_{\eta}X(\xi,\eta)=X(\xi,-1)$;
see Figure \ref{fg_1}(b).
$P_2X$ is a boolean sum of these Taylor approximators.
Note that in~\eqref{eq_6a}  $X(-\eta,\eta)$ and $X(\xi,-\xi)$
are respectively 0-th order Taylor approximations of $X(\xi,\eta)$
along the $\xi$ and $\eta$ directions (about the points  $(-\eta,\eta)$ and $(\xi,-\xi)$
on $\overline{B'C'}$; see Figure \ref{fg_1}(b)).
So $P_1X$ blends these Taylor approximations on the hypotenuse
using the $C^0$ Barnhill polynomials. 
It should also be noted that $P_1$ and $P_2$ do not commute ($P_1P_2\neq P_2P_1$), which is
the crucial difference that distinguishes the transfinite interpolations
of Barnhill~\cite{BarnhillG1975} (non-commutative)
and Gordon~\cite{Gordon1971} (commutative).

The Barnhill strategy for constructing the transfinite interpolation operator $P$
from Taylor approximation operators, by boolean sum on the legs or by blending
based on Barnhill polynomials on the hypotenuse, will be repeatedly used in this paper.


It is straightforward to verify that $PX(\xi,\eta)$ given by~\eqref{eq_4}
has the properties,
\begingroup\small\begin{equation}
  PX(-1,\eta)=X(-1,\eta), \
  PX(\xi,-1) = X(\xi,-1), \
  PX(-\eta,\eta) = X(-\eta, \eta), \
  PX(\xi,-\xi) = X(\xi,-\xi),
\end{equation}\endgroup
and that the $\mbs x(\xi,\eta)$ given by~\eqref{eq_3}
satisfies exactly the conditions in~\eqref{eq_2} for any given $Z(\xi,\eta)$.


In the present work we employ the simple map (by setting $Z=0$ in \eqref{eq_3})
\begin{align}\label{eq_8}
  \mbs x(\xi,\eta) = PX(\xi,\eta) =&\
  \phi_0(\xi)\mbs x_{AC}(\eta) + \phi_1(\xi)\left[\mbs x_{BC}(\eta) - \mbs x_{AB}(-\eta) \right]
  \notag \\
  &+ \phi_0(\eta)\mbs x_{AB}(\xi) 
  + \phi_1(\eta)\left[\mbs x_{BC}(-\xi) - \mbs x_{AC}(-\xi) \right]
  + \left[\phi_1(\xi) - \phi_0(\eta) \right]\mbs x_A,
\end{align}
in most numerical simulations. For certain cases with more complex or more severely
distorted domains, the map~\eqref{eq_8} can become inadequate (losing univalency)
and we employ the more general map~\eqref{eq_3} with a non-zero $Z(\xi,\eta)$
(see Remark~\ref{rem_1} below).

It is in general very difficult to ensure analytically the univalency of the map~\eqref{eq_8}
or~\eqref{eq_3}, as already noticed in~\cite{DongZ2026} previously 
for curved quadrilateral domains.
To achieve univalency in practice, we follow a strategy
similar to~\cite{GordonH1973,DongZ2026} in this work, by combining numerical simulations with
visualizations and by numerically monitoring the sign of the map's Jacobian.
Crossover of the $\xi$ (or $\eta$) grid lines in the visualization or
the occurance of sign change in the Jacobian suggests a compromise
in univalency, leading to a potential ``folding'' or ``spillover'' in the map. 
These abnormalities can oftentimes be remedied by reparameterization
or by introducing additional constraints in the map.


\begin{remark}\label{rem_1}
  If the domain is strongly distorted and the simple map~\eqref{eq_8} fails
  in univalency,
  we compute a more robust map~\eqref{eq_3} by imposing
  the Winslow conditions~\cite{Winslow1967,ThompsonTM1974},
   \begingroup\small \begin{align}\label{eq_9}
      & \frac{\partial^2\xi}{\partial x^2} + \frac{\partial^2\xi}{\partial y^2} = 0, \quad
      \frac{\partial^2\eta}{\partial x^2} + \frac{\partial^2\eta}{\partial y^2} = 0,
    \end{align}\endgroup
  where $\xi=\xi(x,y)$ and $\eta=\eta(x,y)$ are the inverse of the map~\eqref{eq_3}.
  These conditions can ensure that the map is folding free thanks to the maximum
  principle of harmonic functions and the convexity of the standard triangle~\cite{MastinT1978}. 
  One can reformulate~\eqref{eq_9} 
  in terms of the original map
  $\mbs x(\xi,\eta)=(x(\xi,\eta), y(\xi,\eta))$~\cite{ThompsonTM1974},
  \begin{subequations}\label{eq_10}\small
    \begin{align}
      & a_{22}\frac{\partial^2x}{\partial\xi^2}
      - 2a_{12}\frac{\partial^2x}{\partial\xi\partial\eta}
      + a_{11}\frac{\partial^2x}{\partial\eta^2}=0, \\
      & a_{22}\frac{\partial^2y}{\partial\xi^2}
      - 2a_{12}\frac{\partial^2y}{\partial\xi\partial\eta}
      + a_{11}\frac{\partial^2y}{\partial\eta^2}=0,
    \end{align}
  \end{subequations}
  where $a_{11}=\|\mbs x_{\xi}(\xi,\eta) \|^2$,
  $a_{12}=\mbs x_{\xi}(\xi,\eta)\cdot\mbs x_{\eta}(\xi,\eta)$,
  and $a_{22}=\|\mbs x_{\eta}(\xi,\eta) \|^2$.
  We represent the free function $Z(\xi,\eta)$ in~\eqref{eq_3}
  in terms of the barycentric coordinates of $\Omega_{st}$ (in Bernstein-Bezier form),
  \begingroup\small\begin{align}\label{eq_11}
    Z(\xi,\eta) = \sum_{\substack{i,j,k\geqslant 0\\i+j+k=p_w}}
    \mathcal C_{ijk}\frac{(p_w+3)!}{(i+1)!(j+1)!(k+1)!}\lambda_1^{i+1}\lambda_2^{j+1}\lambda_3^{k+1},
  \end{align}\endgroup
  where $p_w\geqslant 0$ is a user-specified integer parameter, $\mathcal C_{ijk}\in\mbb R^2$ are
  the coefficients to be determined, and the barycentric
  coordinates are given by $\lambda_1=\phi_1(\xi)$,
  $\lambda_2=\phi_1(\eta)$ and $\lambda_3=\phi_0(\xi)-\phi_1(\eta)$.
  Noting that $Z(\xi,\eta)$ given by~\eqref{eq_11} satisfies
  $Z|_{\partial\Omega_{st}}= 0$ and $PZ(\xi,\eta)\equiv 0$, the
  map~\eqref{eq_3} is reduced to
  \begingroup\small\begin{align}\label{eq_12}
    & \mbs x(\xi,\eta) = Z(\xi,\eta) + PX(\xi,\eta)
    = \sum_{\substack{i,j,k\geqslant 0\\ i+j+k=p_w}}
    \mathcal C_{ijk}\frac{(p_w+3)!}{(i+1)!(j+1)!(k+1)!}\lambda_1^{i+1}\lambda_2^{j+1}\lambda_3^{k+1}
    + PX(\xi,\eta),
  \end{align}\endgroup
  which satisfies the boundary conditions in~\eqref{eq_2} unconditionally.
  We determine $\mathcal C_{ijk}$ by enforcing the equations in~\eqref{eq_10} on
  a set of $2Q_w$ interior (random) collocation points of $\Omega_{st}$ and solving the
  resultant nonlinear algebraic system by the nonlinear least squares (Gauss-Newton)
  method~\cite{Bjorck1996}. This produces a more robust map, 
  given by~\eqref{eq_12} and hereafter referred to as a Winslow map,
  between $\Omega_{st}$ and the general triangular domain.
  When generating the interior collocation points over $\Omega_{st}$,
  we observe that enforcing a symmetry of these points with respect
  to the line $\eta=\xi$ produces better Winslow maps. In the implementation
  we generate $Q_w$ random interior points first, $(\xi_i,\eta_i)$ for $1\leqslant i\leqslant Q_w$.
  The set of $2Q_w$ collocation points for computing the Winslow map then
  consists of $\{(\xi_i,\eta_i), (\eta_i,\xi_i), 1\leqslant i\leqslant Q_w \}$.
  
\end{remark}

\subsection{Enforcing Exactly Dirichlet Boundary Conditions (DBC)}
\label{sec_dbc}

Consider the general triangular domain $\Omega$ in Figure~\ref{fg_1}(a)
and a scalar function $u(\mbs x)$ ($\mbs x=(x,y)\in\Omega$) that
satisfies the Dirichlet boundary conditions on $\partial\Omega$,
  \begin{align}\label{eq_13}
    & u|_{\mbs x\in\overline{AB}} = u_{AB}(\mbs x), \quad 
    u|_{\mbs x\in\overline{BC}} = u_{BC}(\mbs x), \quad 
    u|_{\mbs x\in\overline{AC}} = u_{AC}(\mbs x).
  \end{align}
Here $u_{AB}(\mbs x)$, $u_{BC}(\mbs x)$ and $u_{AC}(\mbs x)$ are the
prescribed Dirichlet data that must be mutually compatible,
\begin{align}
  & u_{AB}(\mbs x_A) = u_{AC}(\mbs x_A) =: u_{A}, \quad
  u_{AB}(\mbs x_B) = u_{BC}(\mbs x_B) =: u_{B}, \quad
  u_{AC}(\mbs x_C) = u_{BC}(\mbs x_C) =: u_{C}.
\end{align}
Our objective is to formulate the general form of $u(\mbs x)$  that satisfies
the Dirichlet conditions~\eqref{eq_13} exactly.

We leverage the map $\mbs x(\xi,\eta)$ between $\Omega$ and $\Omega_{st}$
from Section~\ref{sec_21} and formulate $u(\mbs x)$ with respect to
the standard domain $\Omega_{st}$, i.e.~in parametric forms.
In light of this map, $u(\mbs x)$ and the boundary data are transformed into,
\begin{subequations}\label{eq_15}\small
  \begin{align}
    & u(\mbs x) = u(\mbs x(\xi,\eta)) =: V(\xi,\eta), \quad (\xi,\eta)\in\Omega_{st};
    \label{eq_15a} \\
    & u_{AB}(\mbs x) = u_{AB}(\mbs x(\xi,-1)) =: F(\xi,-1), \quad \xi\in[-1,1];
    \label{eq_15b} \\
    & u_{AC}(\mbs x) = u_{AC}(\mbs x(-1,\eta)) =: F(-1,\eta), \quad \eta\in[-1,1];
    \label{eq_15c}  \\
    & u_{BC}(\mbs x) = u_{BC}(\mbs x(-\eta,\eta)) =: U_{BC}(\eta),
    \quad \eta\in[-1,1]; \label{eq_15d} \\
    & u_{BC}(\mbs x) = u_{BC}(\mbs x(\xi,-\xi)) = U_{BC}(-\xi),
    \quad \xi\in[-1,1]. \label{eq_15e}
  \end{align}
\end{subequations}
Therefore, the formulation for $u(\mbs x)$ is reduced to:
find $V(\xi,\eta)$, $(\xi,\eta)\in\Omega_{st}$, such that
\begin{subequations}\label{eq_16}\small
  \begin{align}
     V(\xi,-1) &= F(\xi,-1), \label{eq_16a} \\
     V(-1,\eta) &= F(-1,\eta), \label{eq_16b} \\
     V(-\eta,\eta) &= U_{BC}(\eta) =: F(-\eta,\eta), \label{eq_16c} \quad
     \text{or}\ V(\xi,-\xi) = U_{BC}(-\xi) = F(\xi,-\xi).
  \end{align}
\end{subequations}

Recognizing the similarity between~\eqref{eq_16}
and~\eqref{eq_2}, we can construct $V(\xi,\eta)$ analogously. Let
\begingroup\small\begin{align}\label{eq_17}
  V(\xi,\eta) = g(\xi,\eta) - Pg(\xi,\eta) + PF(\xi,\eta),
\end{align}\endgroup
where $g(\xi,\eta)$ is a free (arbitrary) function.
The transfinite interpolations 
$Pg$ and $PF$ here are defined accordingly based on $PX$ in~\eqref{eq_4},
specifically given by,
\begin{subequations}\small
  \begin{align}
    Pg(\xi,\eta) =&\
    \big[g(\xi,-1)+g(-1,\eta)-g(-1,-1) \big] + 
    \phi_1(\eta)\big[g(\xi,-\xi)-g(\xi,-1) - g(-1,-\xi) \notag \\
      &\ + g(-1,-1) \big] +
    \phi_1(\xi)\big[g(-\eta,\eta) - g(-\eta,-1) - g(-1,\eta) + g(-1,-1) \big],
    \label{eq_e27a} \\
    PF(\xi,\eta) =&\
    \big[F(\xi,-1)+F(-1,\eta)-F(-1,-1) \big] + 
    \phi_1(\eta)\big[F(\xi,-\xi)-F(\xi,-1) - F(-1,-\xi) \notag \\
      &\ + F(-1,-1) \big] +
    \phi_1(\xi)\big[F(-\eta,\eta) - F(-\eta,-1) - F(-1,\eta) + F(-1,-1) \big].
    \label{eq_e27b}
  \end{align}
\end{subequations}
A straightforward verification establishes the following property
for $V(\xi,\eta)$.
\begin{theorem}\label{thm_21}
  $V(\xi,\eta)$ given by~\eqref{eq_17} satisfies the Dirichlet conditions in~\eqref{eq_16},
  for any $g(\xi,\eta)$ therein.
\end{theorem}

Suppose additional conditions about $u(\mbs x)$
are supplied (apart from the Dirichlet boundary conditions),
for instance, the function data or the equations that $u(\mbs x)$ must satisfy.
These conditions can be used to determine the free function $g(\xi,\eta)$ in~\eqref{eq_17}.
In a later section (Section~\ref{sec_25}) we consider $u(\mbs x)$
as PDE solutions and use artificial neural networks to represent and approximate
the free function $g$ and in turn the PDE solutions.

\subsection{Enforcing Exactly Neumann Boundary Conditions (NBC)}
\label{sec_nbc}

Exact enforcement of
derivative-type conditions (Neumann or Robin) 
on general triangular boundaries is significantly more challenging.
Four strategies are crucial in our constructions to achieve this goal:
\begin{itemize}
\item
  We follow the four-step procedure
  of~\cite{DongZ2026} when developing the formulations for general triangular domains.

\item
  We adopt the Barnhill strategy to
  construct transfinite interpolations:
  (i) formulate the transfinite operator for legs of the standard triangle
  by the boolean sum of Taylor approximation operators; (ii) construct
  the hypotenuse operator  by blending Taylor approximations
  based on Barnhill polynomials.
  
\item
  Whenever an acute vertex of the standard triangle belongs to a Neumann or Robin
  boundary, all conditional compatibility constraints associated with that vertex
  should be transformed into unconditional forms.

\item
  Whenever the hypotenuse of the standard triangle corresponds to a Neumann
  or Robin boundary, a separate TFC-style auxiliary edge function and an
  associated linear solve will be required on the hypotenuse.

\end{itemize}
The last two strategies are critical tricks that make the formulations successful.

The developments in the subsequent sections follow
three well-defined stages: (i) Transformation of the imposed
boundary conditions from the general triangular domain
to the standard triangular domain; (ii) Analysis of the compatibility
constraints induced by the imposed boundary conditions;
(iii) Formulation of the general form of functions that
satisfies the imposed boundary conditions and the
induced compatibility constraints. The formulation stage in turn
consists of four steps (i.e.~four-step procedure):
(a) Identify the type of constraints involved in
and transform the point constraints into constraints on the
edge profiles; (b) Construct the transfinite interpolation
for the identified constraints; (c) Use the constructed
transfinite interpolation to form TFC
expression as a preliminary general form for the function;
(d) Update the placeholder terms in the transfinite interpolant
by their compatible data, thus
leading to the final formulation.


In this section we focus on a combination of Neumann and
Dirichlet boundary conditions  on the triangle boundaries.
Our objective is to construct the trial ansatz
to satisfy these conditions exactly.
We assume that the triangular domain contains at least one Dirichlet boundary
and one Neumann boundary. 
We distinguish two configurations: (i) NBC is imposed on a single triangle
boundary, with the rest being of Dirichlet type; (ii) NBCs are
imposed on two triangle boundaries, with the third one being of
Dirichlet type.
These are treated individually in the following subsections.

\subsubsection{NBC Imposed on a Single Triangle Boundary}
\label{sec_2311}

When the triangular domain involves a single Neumann boundary, with the rest being of
Dirichlet type,
this boundary can be mapped to
any of the three edges of the standard triangle $\Omega_{st}$ by
a proper re-ordering of the vertices. Therefore, devising a
formulation to enforce the NBC on one of the legs
of $\Omega_{st}$ ($\overline{AB}$ or $\overline{AC}$),
which is conceptually simpler, would suffice.
Without loss of generality,
we only present the formulation for the NBC imposed on
 $\overline{AC}$ in what follows,  with  the reciprocal 
formulation for the NBC on $\overline{AB}$ relegated to a remark
at the end of this section.
When the Neumann boundary corresponds to  the hypotenuse of $\Omega_{st}$,
the formulation is structurally more involved and its presentation will be deferred to
an Appendix (Section \ref{sec_2312}).
The notations below follow those of the previous sections.

Consider the general triangular domain $\Omega=\overline{ABC}$
from Figure~\ref{fg_1}(a) and a scalar function $u(\mbs x)$
($\mbs x\in\Omega$) satisfying the following boundary
conditions,
  \begingroup\small\begin{align}\label{eq_20}
    & u|_{\mbs x\in\overline{AB}} = u_{AB}(\mbs x),  \quad
     u|_{\mbs x\in\overline{BC}} = u_{BC}(\mbs x),  \quad
     \mbs n\cdot\nabla u|_{\mbs x\in\overline{AC}} = u_{nAC}(\mbs x).
  \end{align}\endgroup
Here $\mbs n$ is the outward-pointing unit normal vector
and $u_{nAC}$ denotes the prescribed Neumann boundary data on $\overline{AC}$.
The prescribed Dirichlet data ($u_{AB}$, $u_{BC}$) and
the Neumann data must be compatible on the shared vertices,
which will be discussed below.

Leveraging the map $\mbs x(\xi,\eta)$ from
Section~\ref{sec_21}, we transform $u(\mbs x)$ into $V(\xi,\eta)$
based on~\eqref{eq_15a}. Accordingly,
the DBCs in~\eqref{eq_20} are transformed into~\eqref{eq_16a}
and~\eqref{eq_16c}, in light of~\eqref{eq_15b}
and~\eqref{eq_15d}--\eqref{eq_15e}.
The Neumann condition on $\overline{AC}$ is transformed into
\begin{subequations}\label{eq_21}\small
\begin{align}
  & V_{\xi}(-1,\eta) + S_{AC}(\eta)V_{\eta}(-1,\eta) = T_{AC}(\eta), \label{eq_21a}  \\
  \text{or}\ & V_{\xi}(-1,\eta) = T_{AC}(\eta) - S_{AC}(\eta)V_{\eta}(-1,\eta)
  =: F_{\xi}(-1,\eta).
  \label{eq_21b}
\end{align}
\end{subequations}
Here the symbols  are defined by: ($\mbs J$ denoting the Jacobian matrix)
\begin{align}\label{eq_22}\small
  \left\{
    \begin{array}{l}
     S_{AC}(\eta) = K_{yAC}(\eta)W_{AC}(\eta)
    = -\frac{\mbs x_{\xi}(-1,\eta)\cdot\mbs x_{\eta}(-1,\eta)}{\|\mbs x_{\eta}(-1,\eta) \|^2},
    \quad W_{AC}(\eta) = \frac{1}{K_{xAC}(\eta)}
    = -\frac{\text{det}\mbs J(-1,\eta)}{\|\mbs x_{\eta}(-1,\eta) \|}, \\
     T_{AC}(\eta) = u_{nAC}(\mbs x(-1,\eta))W_{AC}(\eta), \quad
    \mbs J(\xi,\eta) = \frac{\partial(x,y)}{\partial (\xi,\eta)}
    =\begin{bmatrix}x_{\xi} & x_{\eta} \\ y_{\xi} & y_{\eta}  \end{bmatrix}, \quad
    \mbs n_{AC}=
    \begin{bmatrix}-\frac{y_{\eta}(-1,\eta)}{\|\mbs x_{\eta}(-1,\eta) \|}
      \\ \frac{x_{\eta}(-1,\eta)}{\|\mbs x_{\eta}(-1,\eta) \|}  \end{bmatrix},
    \\
     \begin{bmatrix}K_{xAC} \\ K_{yAC} \end{bmatrix}
    = \mbs J^{-1}(-1,\eta)\mbs n_{AC}(\eta)
    = -\frac{1}{\text{det}\mbs J(-1,\eta)}
    \begin{bmatrix}\|\mbs x_{\eta}(-1,\eta) \| \\
      -\frac{\mbs x_{\xi}(-1,\eta)\cdot\mbs x_{\eta}(-1,\eta)}{\|\mbs x_{\eta}(-1,\eta) \|}
    \end{bmatrix}.
    \end{array}
    \right.
\end{align}
In light of the above definition for $S_{AC}(\eta)$, and noting that
$\mbs x_{\xi}(-1,1)=[\mbs x_{AC}'(1)-\mbs x_{BC}'(1)]$ and $\mbs x_{\eta}(-1,1)=\mbs x_{AC}'(1)$
based on equations~\eqref{eq_3}--\eqref{eq_4}, we conclude that
\begin{align}\label{eq_27}\small
\left\{
  \begin{array}{l}
  S_{AC}(1) = -1 \quad \Longleftrightarrow \quad
  \mbs x_{AC}'(1)\cdot\mbs x_{BC}'(1)=0, \ \text{i.e.}\
  \overline{AC}\perp\overline{BC}\ \text{at}\ C; \\
  S_{AC}(-1)=0 \quad \Longleftrightarrow \quad
  \overline{AB}\perp\overline{AC}\ \text{at}\ A.
\end{array}\right.
\end{align}

Let us next analyze the compatibility constraints induced by
the Dirichlet/Neumann conditions~\eqref{eq_16a}, \eqref{eq_16c}, and~\eqref{eq_21a}.
These induced constraints must be honored when formulating $V(\xi,\eta)$.
At the three vertices, the continuity of $V(\xi,\eta)$ leads to
\begin{subequations}\small
  \begin{align}
    & V(-1,-1) = F(-1,-1)= \lim_{\xi\rightarrow -1}F(\xi,-1) = u_A, \label{eq_23a} \\
    & V(1,-1) = F(1,-1) = U_{BC}(-1)=\lim_{\xi\rightarrow 1}F(\xi,-1)
    = \lim_{\eta\rightarrow -1}U_{BC}(\eta) = u_B, \label{eq_23b} \\
    & V(-1,1) = F(-1,1) = U_{BC}(1) = \lim_{\eta\rightarrow 1}U_{BC}(\eta)=u_C. \label{eq_23c}
  \end{align}
\end{subequations}

\vspace{-8pt}
\noindent Differentiating \eqref{eq_16a} with respect to (w.r.t.)
$\xi$ and evaluating it
at $\xi=-1$, combined with~\eqref{eq_21a} evaluated at $\eta=-1$,
lead to
\begin{subequations}\label{eq_24}
  \begin{align}
    & T_{AC}(-1) - F_{\xi}(-1,-1) = 0, \quad
    \text{if}\ \overline{AB}\perp\overline{AC}\ \text{at}\ A, \label{eq_24a} \\
    & V_{\eta}(-1,-1)= F_{\eta}^a(-1,-1) = \frac{T_{AC}(-1)-F_{\xi}(-1,-1)}{S_{AC}(-1)}
    := F_{\eta}(-1,-1), \quad
    \text{if}\ \overline{AB}\not\perp\overline{AC}\ \text{at}\ A. \label{eq_24b}
  \end{align}
\end{subequations}
Equation~\eqref{eq_24a} is a compatibility condition
on the prescribed boundary data, which we
assume will always be satisfied, and~\eqref{eq_24b} is a constraint involving
the unknown $V_{\eta}(-1,-1)$ at vertex $A$.

Differentiating equation~\eqref{eq_16c} w.r.t.~$\eta$ and evaluating it at $\eta=1$
lead to
\begin{align}\label{eq_25}
  & -V_{\xi}(-1,1) + V_{\eta}(-1,1) = U_{BC}'(1).
\end{align}
Evaluating~\eqref{eq_21a} at $\eta=1$ leads to
\begin{align}\label{eq_26}
  & V_{\xi}(-1,1) + S_{AC}(1)V_{\eta}(-1,1) = T_{AC}(1).
\end{align}
A combination of~\eqref{eq_25} and~\eqref{eq_26} gives rise to
\begin{subequations}\label{eq_28}\small
  \begin{align}
    & T_{AC}(1) + U_{BC}'(1) = 0, \quad
    \text{if}\ \overline{AC}\perp\overline{BC}\ \text{at}\ C; \label{eq_28a} \\
    & V_{\eta}(-1,1)=F_{\eta}^a(-1,1)=\frac{T_{AC}(1)+U_{BC}'(1)}{1+S_{AC}(1)},
     \quad
    \text{if}\ \overline{AC}\not\perp\overline{BC}\ \text{at}\ C. \label{eq_28b}
  \end{align}
\end{subequations}
Equation~\eqref{eq_28a} is a compatibility condition on the prescribed boundary
data, which is assumed to be always true,
and~\eqref{eq_28b} is a constraint involving the unknown $V_{\eta}(-1,1)$
at vertex $C$.

In light of these compatibility constraints,
the formulation for $u(\mbs x)$ satisfying~\eqref{eq_20}
is reduced to: Find $V(\xi,\eta)$ such that the
conditions in~\eqref{eq_16a}, \eqref{eq_16c}, \eqref{eq_21b},
\eqref{eq_23a}, \eqref{eq_23c}, \eqref{eq_24b}, and~\eqref{eq_28b} are satisfied.
To facilitate subsequent discussions,
we introduce three constants to flag whether the boundaries are orthogonal to each other
at the three vertices of the triangular domain,
\begingroup\small\begin{align}\label{eq_33}
  \lambda_A=\left\{\begin{array}{ll} 0, & \overline{AB}\perp\overline{AC}\ \text{at}\ A, \\
  1, & \text{otherwise};
  \end{array}  \right.
  \quad
  \lambda_B=\left\{\begin{array}{ll} 0, & \overline{AB}\perp\overline{BC}\ \text{at}\ B, \\
  1, & \text{otherwise};
  \end{array}  \right.
  \quad
  \lambda_C=\left\{\begin{array}{ll} 0, & \overline{AC}\perp\overline{BC}\ \text{at}\ C, \\
  1, & \text{otherwise}.
  \end{array}  \right.
\end{align}\endgroup

At this point we would like to further process the compatibility constraints
associated with vertex $C$, which corresponds to an acute corner of $\Omega_{st}$
and also lies on the Neumann boundary $\overline{AC}$.
Equation~\eqref{eq_28b} is a conditional constraint on the unknown $V_{\eta}(-1,1)$,
dictating that $V_{\eta}(-1,1)$ is constrained to be $F_{\eta}^a(-1,1)$
if $\overline{AC}\not\perp\overline{BC}$ at $C$ and is free otherwise.
We re-write this conditional constraint into an equivalent form,
\begingroup\small\begin{align}\label{eq_34}
  & V_{\eta}(-1,1) = (1-\lambda_C)\Gamma + \lambda_C F_{\eta}^a(-1,1)
  = \left\{\begin{array}{ll} \Gamma, & \overline{AC}\perp\overline{BC}\ \text{at}\ C, \\
  F_{\eta}^a(-1,1), & \overline{AC}\not\perp\overline{BC}\ \text{at}\ C
  \end{array}\right\}
  =: F_{\eta}(-1,1),
\end{align}\endgroup
where $\Gamma$ is a free scalar parameter (undetermined) and we have used
the constant flag $\lambda_C$.
Note that equation~\eqref{eq_34} represents an unconditional constraint on
$V_{\eta}(-1,1)$, with its  data being an undetermined value $\Gamma$
if $\overline{AC}\perp\overline{BC}$
at $C$, or given by $F_{\eta}^a(-1,1)$ otherwise.
In contrast, equation~\eqref{eq_28b} is a conditional constraint
on $V_{\eta}(-1,1)$, imposed only if $\overline{AC}\not\perp\overline{BC}$ at $C$.
We will use the unconditional form~\eqref{eq_34} subsequently
when formulating $V(\xi,\eta)$.
We emphasize that this treatment, by transforming a conditional
constraint into an unconditional form at the acute vertex (i.e.~$C$ or $B$)
when it belongs to a Neumann (or Robin) boundary,
is critical for the construction of $V(\xi,\eta)$ to
satisfy the imposed Neumann (or Robin) conditions.

We follow the four-step procedure of~\cite{DongZ2026}
to construct the general form of $V(\xi,\eta)$
that satisfies the imposed boundary conditions and the induced compatibility constraints.
\\[-5pt]

\noindent{\bf Identify the type of constraints involved (step \#1).}\ \
The constraints for the Dirichlet conditions~\eqref{eq_16a}
and~\eqref{eq_16c}
are self-evident. For the Neumann condition we will employ~\eqref{eq_21b},
treating this as a constraint on $V_{\xi}(-1,\eta)$ with the data $F_{\xi}(-1,\eta)$
being dependent on the unknown $V_{\eta}(-1,\eta)$.
The conditions~\eqref{eq_23a}, \eqref{eq_23c}, \eqref{eq_24b} and~\eqref{eq_34}
are point constraints on the vertices $A$ and $C$.
The strategy is to treat them as conditions on the edge profile $V(-1,\eta)$
of the Neumann boundary.
In light of the $C^1$ blending functions, let
\begingroup\small\begin{align}\label{eq_35}
  V(-1,\eta) =&\ F(-1,-1)\varphi_0(\eta) + F(-1,1)\varphi_1(\eta)
  + \lambda_A F_{\eta}(-1,-1)\psi_0(\eta) + F_{\eta}(-1,1)\psi_1(\eta) 
  =: F(-1,\eta),
\end{align}\endgroup
where $F(-1,1)$, $F_{\eta}(-1,-1)$ and $F_{\eta}(-1,1)$ are defined in~\eqref{eq_23c},
\eqref{eq_24b} and~\eqref{eq_34}, respectively, and $F(-1,\eta)$ denotes
the synthesized  profile for $V(-1,\eta)$. 
It can be verified that the edge profile~\eqref{eq_35} satisfies
the point constraints~\eqref{eq_23a}, \eqref{eq_23c}, \eqref{eq_24b} and~\eqref{eq_34}.
Note that~\eqref{eq_34} is imposed unconditionally at $C$ and that
$\lambda_A$ in~\eqref{eq_35} ensures that the condition~\eqref{eq_24b}
is imposed only when $\overline{AB}\not\perp\overline{AC}$ at $A$.

As a result, the formulation problem is reduced to the following:
find the general form of $V(\xi,\eta)$ on $\Omega_{st}$
such that the conditions~\eqref{eq_16a},
\eqref{eq_16c}, \eqref{eq_21b}, and~\eqref{eq_35} are satisfied.
\\[-5pt]

\noindent{\bf Construct transfinite interpolation operator for identified
   constraints (step \#2).}\ \
We adopt the Barnhill strategy to construct the transfinite interpolation operator  for
the constraints in~\eqref{eq_16a}, \eqref{eq_16c},
\eqref{eq_35}, and~\eqref{eq_21b}.

Let $T_{\xi}$ and $T_{\eta}$ denote the Taylor approximation operators along the $\xi$
and $\eta$ directions, respectively, with $T_{\xi}$ being of first order about $\xi=-1$
and $T_{\eta}$ being of zeroth order about $\eta=-1$, specifically defined by
\begingroup\small\begin{align}
    & T_{\xi}F(\xi,\eta) = F(-1,\eta) + (\xi+1)F_{\xi}(-1,\eta), \quad
    T_{\eta}F(\xi,\eta) = F(\xi,-1),
\end{align}\endgroup
where $F(-1,\eta)$ is defined in~\eqref{eq_35} and
$F_{\xi}(-1,\eta)$ is defined in~\eqref{eq_21b}. Let
\begin{subequations}\small
\begin{align}
  P_2F(\xi,\eta) &= (T_{\xi}\oplus T_{\eta})F(\xi,\eta)
  = T_{\xi}F(\xi,\eta) + T_{\eta}F(\xi,\eta) - T_{\xi}T_{\eta}F(\xi,\eta) \nonumber \\
  &= F(-1,\eta) + F(\xi,-1) - F(-1,-1) + (\xi+1)[F_{\xi}(-1,\eta)-F_{\xi}(-1,-1)], \\
  P_1F(\xi,\eta) &= \alpha(\xi,\eta)F(\xi,-\xi) + \beta(\xi,\eta)F(-\eta,\eta),
\end{align}
\end{subequations}
where $\alpha(\xi,\eta)$ and $\beta(\xi,\eta)$ denote the $C^1$ Barnhill polynomials
defined in~\eqref{eq_31}.
Note that $F(-\eta,\eta)$ and $F(\xi,-\xi)$ are zeroth order Taylor approximations
of $F(\xi,\eta)$ along the $\xi$ and $\eta$ directions about
points $(-\eta,\eta)$ and $(\xi,-\xi)$ on $\overline{BC}$ (see Figure \ref{fg_1}(b)). So
$P_1F$ blends these Taylor approximations using Barnhill polynomials.

We define the transfinite interpolation operator $P$ by
\begingroup\small\begin{align}\label{eq_38}
  PF(\xi,\eta) &= (P_1\oplus P_2)F(\xi,\eta)
  = P_1F(\xi,\eta) + P_2F(\xi,\eta) - P_1P_2F(\xi,\eta) \notag \\
  &= A_F(\xi,\eta) + \alpha(\xi,\eta)\left[F(\xi,-\xi) - A_F(\xi,-\xi) \right]
  + \beta(\xi,\eta)\left[F(-\eta,\eta) - A_F(-\eta,\eta) \right].
\end{align}\endgroup
In this equation $A_{F}$ is given by
\begin{equation}\small\label{eq_39}
  \begin{split}
    A_F(\xi,\eta)= P_2F(\xi,\eta)
    &= F(\xi,-1)+(\xi+1)\left[F_{\xi}(-1,\eta)-F_{\xi}(-1,-1) \right] - F(-1,-1) \\
    &+ \left[F(-1,-1)\varphi_0(\eta)+F(-1,1)\varphi_1(\eta)+\lambda_AF_{\eta}(-1,-1)\psi_0(\eta)
      +F_{\eta}(-1,1)\psi_1(\eta)\right],
  \end{split}
\end{equation}
where the expression~\eqref{eq_35} for $F(-1,\eta)$ has been used.
We would like to point out that in~\eqref{eq_38} the $F$ terms
should be treated as placeholders at
this stage, which
in the final step of formulation will be replaced by the Dirichlet data
or boundary-/corner-compatible data.
$PF$ has the following properties.
\begin{property}\label{lem_23}
  $PF(\xi,\eta)$ given by~\eqref{eq_38} satisfies
  \begin{equation}\small\left\{
    \begin{split}
    &PF(\xi,-1)=F(\xi,-1), \quad
    PF(-\eta,\eta) = F(-\eta,\eta), \quad PF(\xi,-\xi)=F(\xi,-\xi), \\
    &PF_{\xi}(-1,\eta) = F_{\xi}(-1,\eta), \quad PF(-1,-1)=F(-1,-1), \quad
    PF(-1,1)=F(-1,1), \\
    & PF_{\eta}(-1,-1)=\lambda_AF_{\eta}(-1,-1), \quad PF_{\eta}(-1,1)=F_{\eta}(-1,1), \\
    & PF(-1,\eta) = F(-1,-1)\varphi_0(\eta) + F(-1,1)\varphi_1(\eta)
    + \lambda_AF_{\eta}(-1,-1)\psi_0(\eta) + F_{\eta}(-1,1)\psi_1(\eta).
    \end{split}\right.
  \end{equation}
\vspace{-10pt}
\end{property}
\begin{proof} By verification and by using the properties in~\eqref{eq_32} and \eqref{eq_30}.
\end{proof}

\vspace{5pt}
\noindent{\bf Formulate preliminary general form for $V(\xi,\eta)$ (step \#3).}\ \
Given  $PF$ in~\eqref{eq_38},
we formulate a preliminary general
form for $V(\xi,\eta)$ by the TFC expression,
\begingroup\small\begin{equation}\label{eq_41}
  V(\xi,\eta) = g(\xi,\eta) - Pg(\xi,\eta) + PF(\xi,\eta),
\end{equation}\endgroup
where $g(\xi,\eta)$ is a free (arbitrary) function that is sufficiently differentiable.
$Pg$ is defined according to \eqref{eq_38}, specifically given by
\begin{subequations}\label{eq_42}\small
  \begin{align}
    Pg(\xi,\eta) =&\ A_g(\xi,\eta) + \alpha(\xi,\eta)\left[g(\xi,-\xi) - A_g(\xi,-\xi) \right]
    + \beta(\xi,\eta)\left[g(-\eta,\eta) - A_g(-\eta,\eta) \right], \label{eq_42a}
    \\
     A_g(\xi,\eta)
     =&\ g(\xi,-1)+(\xi+1)\left[g_{\xi}(-1,\eta)-g_{\xi}(-1,-1) \right] - g(-1,-1) \notag \\
     &+ \left[g(-1,-1)\varphi_0(\eta)+g(-1,1)\varphi_1(\eta)+\lambda_Ag_{\eta}(-1,-1)\psi_0(\eta)
       +g_{\eta}(-1,1)\psi_1(\eta)\right]. \label{eq_42b}
  \end{align}
\end{subequations}

\vspace{5pt}
\noindent{\bf Substitute placeholder terms in $PF$ with
  compatible data (step \#4).}\ \
In $PF(\xi,\eta)$ (equation~\eqref{eq_38}), the terms
$F(\xi,-1)$, $F(-1,-1)$, $F_{\xi}(-1,-1)$, $F(\xi,-\xi)$, $F(-\eta,\eta)$, and
$F(-1,1)$ are given by or computed from the prescribed Dirichlet data,
while $F_{\xi}(-1,\eta)$, $F_{\eta}(-1,-1)$ and $F_{\eta}(-1,1)$ are
placeholders for the boundary data or corner-compatible data in
equations~\eqref{eq_21b}, \eqref{eq_24b} and \eqref{eq_34}.
In the final step we substitute these compatible data into $PF$ to attain
an updated transfinite interpolant $PF^g(\xi,\eta)$, where the superscript ``$g$''
accentuates the fact that it explicitly depends on the free function $g(\xi,\eta)$
due to the unknown $V(\xi,\eta)$ involved in the data.
$PF^g$ is specifically given by
\begingroup\small\begin{align}\label{eq_44}
  & PF^g(\xi,\eta) = A_F^g(\xi,\eta)
    + \alpha(\xi,\eta)\left[U_{BC}(-\xi) - A_F^g(\xi,-\xi) \right]
    + \beta(\xi,\eta)\left[U_{BC}(\eta)  - A_F^g(-\eta,\eta) \right].
\end{align}\endgroup
The symbols in this equation are defined by
\begin{subequations}\small
  \begin{align}
    A_F^g(\xi,\eta) =&\ F(\xi,-1)+(\xi+1)\left[F_{\xi}^g(-1,\eta)-F_{\xi}(-1,-1) \right]
    - F(-1,-1) \notag \\
    &+ \left[F(-1,-1)\varphi_0(\eta)+U_{BC}(1)\varphi_1(\eta)
      +\lambda_AF_{\eta}^a(-1,-1)\psi_0(\eta)
      +F_{\eta}^g(-1,1)\psi_1(\eta)\right]; \label{eq_45a} \\
    F_{\eta}^g(-1,1) =&\ (1-\lambda_C)\Gamma + \lambda_CF_{\eta}^a(-1,1)
    = (1-\lambda_C)g_{\eta}(-1,1) + \lambda_CF_{\eta}^a(-1,1); \label{eq_45b} \\
    F_{\xi}^g(-1,\eta) =&\ T_{AC}(\eta) - S_{AC}(\eta)V_{\eta}^g(-1,\eta); \label{eq_45c} \\
    V_{\eta}^g(-1,\eta) =&\ g_{\eta}(-1,\eta)- [g(-1,-1)-F(-1,-1)]\varphi_0'(\eta)
     - [g(-1,1)-U_{BC}(1)]\varphi_1'(\eta) \notag \\
     & - \lambda_A[g_{\eta}(-1,-1)-F_{\eta}^a(-1,-1)]\psi_0'(\eta)
     - [g_{\eta}(-1,1)-F_{\eta}^g(-1,1)]\psi_1'(\eta).
  \end{align}
\end{subequations}

\begin{remark}\label{rem_2}
  $F_{\eta}(-1,1)$ in $PF$ is a placeholder for the expression~\eqref{eq_34},
  which contains an unknown (free) parameter $\Gamma$,
  representing the value $V_{\eta}(-1,1)$ when
  $\overline{AC}\perp\overline{BC}$ at vertex $C$.
  When attaining the updated $F_{\eta}^g(-1,1)$ in equation~\eqref{eq_45b},
  we have chosen $\Gamma=g_{\eta}(-1,1)$,
  since both are unknown and to be determined. This choice eliminates $\Gamma$
  from the formulation and effectively imposes
  the following constraint at vertex $C$, 
  \begingroup\small\begin{equation}
    V_{\eta}(-1,1) = (1-\lambda_C)g_{\eta}(-1,1) + \lambda_C F_{\eta}^a(-1,1)
    = \left\{\begin{array}{ll} g_{\eta}(-1,1), &
    \overline{AC}\perp\overline{BC}\ \text{at}\ C, \\
    F_{\eta}^a(-1,1), & \overline{AC}\not\perp\overline{BC}\ \text{at}\ C.
    \end{array}\right.
  \end{equation}\endgroup
\end{remark}

\begin{property}\label{lem_25}
  $F_{\xi}^g(-1,\eta)$ and $F_{\eta}^g(-1,1)$
  given in~\eqref{eq_45b}--\eqref{eq_45c}, 
  for any $g(\xi,\eta)$ therein
  that is sufficiently differentiable and for
  any combination of $(\lambda_A,\lambda_C)=(0/1,0/1)$, satisfy
  \begin{equation}
    \begin{split}
    & F_{\xi}^g(-1,-1) = F_{\xi}(-1,-1), \quad
    -F_{\xi}^g(-1,1) + F_{\eta}^g(-1,1) = U_{BC}'(1).
    \end{split}
  \end{equation}
\end{property}
\begin{proof} By verification and by using equations~\eqref{eq_24a}--\eqref{eq_24b},
  \eqref{eq_27}, and
  \eqref{eq_28a}--\eqref{eq_28b}.
\end{proof}

\begin{property}\label{lem_26}
  $PF^{g}(\xi,\eta)$ given by~\eqref{eq_44}, for any $g(\xi,\eta)$ therein that
  is sufficiently differentiable, satisfies
  \begingroup\small\begin{equation}\left\{
    \begin{split}
    &PF^g(\xi,-1)=F(\xi,-1), \
    PF^g(-\eta,\eta) = U_{BC}(\eta), \ PF^g(\xi,-\xi)=U_{BC}(-\xi), \ 
    (PF^g)_{\xi}(-1,\eta) = F_{\xi}^g(-1,\eta), \\
    & PF^g(-1,\eta) = F(-1,-1)\varphi_0(\eta) + U_{BC}(1)\varphi_1(\eta)
    + \lambda_AF_{\eta}^a(-1,-1)\psi_0(\eta) + F_{\eta}^g(-1,1)\psi_1(\eta).
    \end{split}\right.
  \end{equation}\endgroup
\end{property}
\begin{proof}By verification and by
  using Property~\ref{lem_25} and the results in~\eqref{eq_32} and \eqref{eq_30}.
\end{proof}

Putting the above components together, we attain the final general form
for $V$,
\begingroup\small\begin{equation}\label{eq_49}
  V(\xi,\eta) = g(\xi,\eta) - Pg(\xi,\eta) + PF^g(\xi,\eta),
\end{equation}\endgroup
where $g(\xi,\eta)$ is a free (arbitrary) function that is sufficiently differentiable,
and $Pg$ and $PF^g$ are defined in~\eqref{eq_42} and \eqref{eq_44}, respectively.
It has the following properties.
\begin{theorem}
  $V(\xi,\eta)$ given by~\eqref{eq_49} satisfies the Dirichlet conditions~\eqref{eq_16a} and
  \eqref{eq_16c} and the Neumann condition~\eqref{eq_21a},
  for any $g(\xi,\eta)$ therein that is
  sufficiently differentiable.
\end{theorem}

\begin{proof} The Dirichlet conditions can be verified by using Properties~\ref{lem_23}
  and \ref{lem_26}. The Neumann condition can be verified by noting that
  \begingroup\small\begin{equation}
    \left\{
    \begin{split}
    &V_{\xi}(-1,\eta) = F_{\xi}^g(-1,\eta), \\
    & V_{\eta}(-1,\eta)=V_{\eta}^g(-1,\eta) = g_{\eta}(-1,\eta)- [g(-1,-1)-F(-1,-1)]\varphi_0'(\eta)
     - [g(-1,1)-U_{BC}(1)]\varphi_1'(\eta) \\
     &\qquad\qquad\qquad\qquad\qquad\ \ - \lambda_A[g_{\eta}(-1,-1)-F_{\eta}(-1,-1)]\psi_0'(\eta)
     - [g_{\eta}(-1,1)-F_{\eta}^g(-1,1)]\psi_1'(\eta),
    \end{split}\right.
  \end{equation}\endgroup
  and by using equation~\eqref{eq_45c}.
\end{proof}


\begin{remark}
  Vertices $A$ and $C$ are the junctions between a Neumann boundary and
  a Dirichlet boundary. If the two boundary curves meet smoothly at
  either vertex (i.e.~common tangent, tangent-continuous), they
  will be parallel there and the Jacobian
  matrix will become singular. The formulation presented in this section,
  with no change, equally applies to such configurations. Suppose both $A$
  and $C$ are smooth vertices, leading to $\lambda_A=1$ and $\lambda_C=1$. Using
  equations~\eqref{eq_22}, \eqref{eq_24b} and~\eqref{eq_28b}, we have
  \begingroup\small\begin{equation}
    \left\{\begin{split}
    & \textnormal{det}\mbs J(-1,\pm 1)=0, \quad W_{AC}(\pm 1)=0, \quad T_{AC}(\pm 1)=0, \quad
    S_{AC}(-1)   \neq 0, \quad S_{AC}(1) \neq -1, \\
    &F_{\eta}^a(-1,-1) = -\frac{F_{\xi}(-1,-1)}{S_{AC}(-1)}, \quad
    F_{\eta}^a(-1,1) = \frac{U_{BC}'(1)}{1+S_{AC}(-1)}.
    \quad 
    \end{split}\right.
  \end{equation}\endgroup
Therefore all the quantities in the formulation are well defined, and no
special treatment for smooth vertices is needed. If the solution gradient at
the smooth vertices (i.e.~$\nabla u(\mbs x_A)$, $\nabla u(\mbs x_C)$) are required,
they cannot be computed through the Jacobian matrix (due to singularity).
They can be determined, however, using the idea as
discussed in~\cite{DongZ2026} (Remark 2.3 therein),
by a combination of the chain rule, the Neumann boundary condition, and
the existence of a common tangent at the smooth vertex. The results are given by
\begin{subequations}\small
\begin{align}
  & \nabla u(\mbs x_A)= \begin{bmatrix} -\frac{F_{\xi}(-1,-1)}{\|\mbs x_{AB}'(-1) \|^2} &
    u_{nAC}(\mbs x_A) \end{bmatrix}
  \begin{bmatrix}\tau_{xAC}(-1) & \tau_{yAC}(-1) \\
    -\tau_{yAC}(-1) & \tau_{xAC}(-1) \end{bmatrix}, \\
  & \nabla u(\mbs x_C)= \begin{bmatrix} -\frac{U_{BC}'(1)}{\|\mbs x_{BC}'(1) \|^2} &
    u_{nAC}(\mbs x_C) \end{bmatrix}
  \begin{bmatrix}\tau_{xAC}(1) & \tau_{yAC}(1) \\
    -\tau_{yAC}(1) & \tau_{xAC}(1) \end{bmatrix},
\end{align}
\end{subequations}
where $\bm\tau_{AC}(\eta)=(\tau_{xAC},\tau_{yAC})
=\frac{\mbs x_{AC}'(\eta)}{\|\mbs x_{AC}'(\eta) \|}$ is the tangent vector on $\overline{AC}$.

\end{remark}


\begin{remark}
  If the sole Neumann condition is imposed on $\overline{AB}$ (instead of $\overline{AC}$),
  with the other
  boundaries being of Dirichlet type, a reciprocal formulation
  can be attained analogously to exactly satisfy these boundary conditions
  by following the four-step procedure and using
  the Barnhill polynomials $\alpha(\xi,\eta)$ and $\beta(\xi,\eta)$.
  
\end{remark}


\subsubsection{NBC Imposed on Two Triangle Boundaries}
\label{sec_232}

When two Neumann boundaries are present
(with the third being of Dirichlet
type), they can always be mapped to the two legs ($\overline{AB}$, $\overline{AC}$)
of  $\Omega_{st}$ by a reordering of the vertices.
In this section
we focus on the formulation with the two Neumann boundaries
corresponding to $\overline{AB}$ and $\overline{AC}$,
without loss of generality.
If one of the Neumann boundaries corresponds to the hypotenuse,
the formulation is more involved and will be
presented in an Appendix (Section \ref{sec_n_ACBC}).

Consider a scalar field $u(\mbs x)$ defined on
the general triangular domain $\Omega=\overline{ABC}$ from Figure~\ref{fg_1}(a),
which satisfies the following boundary conditions,
\begin{subequations}\label{eq_81}
  \begingroup\small
\begin{align}
  &\mbs n\cdot\nabla u|_{\mbs x\in\overline{AB}} = u_{nAB}(\mbs x), \label{eq_81a} \\
  & u|_{\mbs x\in\overline{BC}} = u_{BC}(\mbs x), \label{eq_81b}  \\
  & \mbs n\cdot\nabla u|_{\mbs x\in\overline{AC}} = u_{nAC}(\mbs x), \label{eq_81c}
\end{align}\endgroup
\end{subequations}
where the right hand sides (RHS) are the prescribed Dirichlet or Neumann data.
In light of the domain map from Section~\ref{sec_21},
equations~\eqref{eq_81b} and \eqref{eq_81c} are transformed into \eqref{eq_16c}
and \eqref{eq_21}, respectively. The Neumann condition~\eqref{eq_81a} is transformed into
\begin{subequations}\label{eq_82}
  \begingroup\small
  \begin{align}
    & V_{\eta}(\xi,-1) + S_{AB}(\xi)V_{\xi}(\xi,-1) = T_{AB}(\xi), \label{eq_82a} \\
    \text{or}\ & V_{\eta}(\xi,-1) = T_{AB}(\xi) - S_{AB}(\xi)V_{\xi}(\xi,-1)
    =: F_{\eta}(\xi,-1). \label{eq_82b}
  \end{align}\endgroup
\end{subequations}
The symbols  in this equation are defined by ($\mbs J$ again denoting the Jacobian matrix),
\begingroup\small\begin{align}\label{eq_83}
    \left\{\begin{array}{l}
    S_{AB}(\xi) = K_{xAB}(\xi)W_{AB}(\xi)
    = -\frac{\mbs x_{\xi}(\xi,-1)\cdot\mbs x_{\eta}(\xi,-1)}{\|\mbs x_{\xi}(\xi,-1) \|^2},
    \quad W_{AB}(\xi) = \frac{1}{K_{yAB}(\xi)}
    = -\frac{\text{det}\mbs J(\xi,-1)}{\|\mbs x_{\xi}(\xi,-1) \|}, \\
    T_{AB}(\xi) = u_{nAB}(\mbs x(\xi,-1))W_{AB}(\xi), \quad
    \mbs n_{AB}=
    \begin{bmatrix} \frac{y_{\xi}(\xi,-1)}{\|\mbs x_{\xi}(\xi,-1) \|}
      \\ -\frac{x_{\xi}(\xi,-1)}{\|\mbs x_{\xi}(\xi,-1) \|}  \end{bmatrix},
    \\
    \begin{bmatrix}K_{xAB}(\xi) \\ K_{yAB}(\xi) \end{bmatrix}
    = \mbs J^{-1}(\xi,-1)\mbs n_{AB}(\xi)
    = \frac{1}{\text{det}\mbs J(\xi,-1)}
    \begin{bmatrix}
      \frac{\mbs x_{\xi}(\xi,-1)\cdot\mbs x_{\eta}(\xi,-1)}{\|\mbs x_{\xi}(\xi,-1) \|} \\
      -\|\mbs x_{\xi}(\xi,-1) \|
    \end{bmatrix}.
    \end{array}\right.
  \end{align}\endgroup
The transformed boundary conditions consist of~\eqref{eq_16c},
\eqref{eq_21}, and \eqref{eq_82}.
Note that
\begingroup\small\begin{align}\label{eq_n61}
  \left\{\begin{array}{l}
  \overline{AB}\perp\overline{BC}\ \text{at}\ B \quad \Longleftrightarrow \quad
  S_{AB}(1)=-1; \\
  \overline{AB}\perp\overline{AC}\ \text{at}\ A \quad \Longleftrightarrow \quad
  S_{AB}(-1)= S_{AC}(-1)=0.
  \end{array}\right.
\end{align}\endgroup


These boundary conditions must be compatible on the shared vertices, and we next
analyze these compatibility constraints.
The continuity of $V(\xi,\eta)$ at vertices $B$ and $C$ results in
\begin{subequations}\label{eq_85}
  \begingroup\small\begin{align}
    & V(1,-1) = U_{BC}(-1) =: F(1,-1), \label{eq_85a} \\
    & V(-1,1) = U_{BC}(1) =: F(-1,1). \label{eq_85b}
  \end{align}\endgroup
\end{subequations}
Differentiating equation~\eqref{eq_16c} w.r.t.~$\eta$ and valuating it at $\eta=-1$,
combined with \eqref{eq_82a} evaluated at $\xi=1$, leads to
  \begingroup\small
\begin{subequations}\label{eq_86}
  \begin{align}
    & T_{AB}(1) - U_{BC}'(-1) = 0, \quad \text{if}\ \overline{AB}\perp\overline{BC}\ \text{at}\ B;
    \label{eq_86a}\\
    & V_{\xi}(1,-1) = \frac{T_{AB}(1)-U_{BC}'(-1)}{1+S_{AB}(1)} =: F_{\xi}^a(1,-1), \quad
    \text{if}\ \overline{AB}\not\perp\overline{BC}\ \text{at}\ B. \label{eq_86b}
  \end{align}
\end{subequations}\endgroup
Differentiating equation~\eqref{eq_16c} w.r.t.~$\eta$ and valuating it at $\eta=1$,
combined with \eqref{eq_21a} evaluated at $\eta=1$, leads to
\begin{subequations}\label{eq_87}\small
  \begin{align}
    & T_{AC}(1) + U_{BC}'(1) = 0, \quad \text{if}\ \overline{AC}\perp\overline{BC}\ \text{at}\ C;
    \label{eq_87a}\\
    & V_{\eta}(-1,1) = \frac{T_{AC}(1)+U_{BC}'(1)}{1+S_{AC}(1)} =: F_{\eta}^a(-1,1), \quad
    \text{if}\ \overline{AC}\not\perp\overline{BC}\ \text{at}\ C. \label{eq_87b}
  \end{align}
\end{subequations}
We then transform the conditional constraints in~\eqref{eq_86b} and~\eqref{eq_87b} into
equivalent unconditional forms,
\begin{subequations}\label{eq_88}\small
  \begin{align}
    & V_{\xi}(1,-1) = (1-\lambda_B)\Gamma_B + \lambda_B F_{\xi}^a(1,-1)
    =\left\{\begin{array}{ll}\Gamma_B, & \overline{AB}\perp\overline{BC}\ \text{at}\ B, \\
    F_{\xi}^a(1,-1), & \text{otherwise}
    \end{array}\right\} =: F_{\xi}(1,-1); \label{eq_88a} \\
    & V_{\eta}(-1,1) = (1-\lambda_C)\Gamma_C + \lambda_C F_{\eta}^a(-1,1)
    = \left\{\begin{array}{ll}\Gamma_C, & \overline{AC}\perp\overline{BC}\ \text{at}\ C, \\
    F_{\eta}^a(-1,1), & \text{otherwise}
    \end{array}\right\} =: F_{\eta}(-1,1), \label{eq_88b}
  \end{align}
\end{subequations}
where $\Gamma_B$ and $\Gamma_C$ are undetermined free  parameters,
and $\lambda_B$ and $\lambda_C$ are defined in \eqref{eq_33}.

Compatibility constraints at vertex $A$ (double-Neumann)
consist of the following.
The combination of equation~\eqref{eq_21a} evaluated at $\eta=-1$ and
equation~\eqref{eq_82a} evaluated at $\xi=-1$ results in,
\begingroup\small\begin{align}\label{eq_89}
  & \begin{bmatrix}V_{\xi}(-1,-1)\\ V_{\eta}(-1,-1)  \end{bmatrix}
  = \begin{bmatrix}F_{\xi}^a(-1,-1)\\ F_{\eta}^a(-1,-1)  \end{bmatrix}
  = \begin{bmatrix}
    \frac{T_{AC}(-1) - S_{AC}(-1)T_{AB}(-1)}{1-S_{AC}(-1)S_{AB}(-1)} \\ 
    \frac{T_{AB}(-1) - S_{AB}(-1)T_{AC}(-1)}{1-S_{AC}(-1)S_{AB}(-1)}
  \end{bmatrix}
  =: \begin{bmatrix}F_{\xi}(-1,-1)\\ F_{\eta}(-1,-1)  \end{bmatrix}.
\end{align}\endgroup
Notice that $S_{AB}(-1)S_{AC}(-1)\leqslant 1$ based on the Cauchy-Schwarz inequality,
where equality holds only if $\overline{AB}\parallel\overline{AC}$ at $A$.
In this paper we require that the incident boundaries not be parallel at
a double-Neumann (or double-Robin) vertex of the triangular domain.
In other words, the double-Neumann
(or double-Robin) vertex should not be degenerate (smooth or a cusp).
The derivative of equation~\eqref{eq_21a} evaluated at $\eta=-1$,
combined with the derivative of equation~\eqref{eq_82a} evaluated at $\xi=-1$,
leads to
\begin{subequations}\small
  \begin{align}
    & [T_{AB}'(-1)-S_{AB}'(-1)F_{\xi}^a(-1,-1)] =
    [T_{AC}'(-1)-S_{AC}'(-1)F_{\eta}^a(-1,-1)], \quad
    \text{if}\ \overline{AB}\perp\overline{AC}\ \text{at}\ A, \label{eq_90a} \\
    & V_{\xi\xi}(-1,-1) = \frac{S_{AC}(-1)}{S_{AB}(-1)}V_{\eta\eta}(-1,-1)
    + \frac{Q^a(-1,-1)}{S_{AB}(-1)} =: F_{\xi\xi}(-1,-1), \quad
    \text{if}\ \overline{AB}\not\perp\overline{AC}\ \text{at}\ A, \label{eq_90b}
  \end{align}
\end{subequations}
where
$Q^a(-1,-1)=[T_{AB}'(-1)-S_{AB}'(-1)F_{\xi}^a(-1,-1)] -
[T_{AC}'(-1)-S_{AC}'(-1)F_{\eta}^a(-1,-1)]$,
and
\begingroup\small\begin{align}\label{eq_91}
  V_{\xi\eta}(-1,-1) = - S_{AC}(-1)V_{\eta\eta}(-1,-1)+ [T_{AC}'(-1)-S_{AC}'(-1)F_{\eta}^a(-1,-1)] 
  =: F_{\xi\eta}(-1,-1).
\end{align}\endgroup

The formulation for $u(\mbs x)$ that satisfies the conditions in~\eqref{eq_81}
is thus reduced to: Find the general form of $V(\xi,\eta)$
such that it satisfies conditions in~\eqref{eq_16c}, \eqref{eq_21},
\eqref{eq_82}, \eqref{eq_85a}--\eqref{eq_85b}, \eqref{eq_88a}--\eqref{eq_88b},
\eqref{eq_89}, \eqref{eq_90b}, and \eqref{eq_91}.


We next transform the point constraints on $V$ into conditions about
its edge profiles. Define
\begingroup\small\begin{align}
  & F(-1,\eta) := F(-1,1)\rho_1(\eta)
  + F_{\eta}(-1,1)\upsilon_1(\eta)+ F_{\eta}(-1,-1)\upsilon_0(\eta), \label{eq_95} \\
  & F(\xi,-1) := F(1,-1)\rho_1(\xi) + F_{\xi}(1,-1)\upsilon_1(\xi)
  + F_{\xi}(-1,-1)\upsilon_0(\xi) + \lambda_AF_{\xi\xi}(-1,-1)\omega_0(\xi), \label{eq_96}
\end{align}\endgroup
where $F(-1,1)$ and $F(1,-1)$ are defined in~\eqref{eq_85},
$F_{\xi}(1,-1)$ and $F_{\eta}(-1,1)$ are defined in \eqref{eq_88},
$F_{\xi}(-1,-1)$ and $F_{\eta}(-1,-1)$ are defined in \eqref{eq_89},
$F_{\xi\xi}(-1,-1)$ is defined in \eqref{eq_90b}, and $\lambda_A$ is defined
in \eqref{eq_33}. $\rho_i$, $\upsilon_i$ and $\omega_i$ ($i=0,1$) are the
$C^2$ blending functions defined in \eqref{eq_n4}.
Note that
\begingroup\small\begin{align} \label{eq_97}
& V(-1,\eta)=F(-1,\eta)
\end{align}\endgroup
satisfies
the point constraints in \eqref{eq_85b} and \eqref{eq_88b},
and the second equation in~\eqref{eq_89}. Similarly,
\begingroup\small\begin{align}\label{eq_98}
&V(\xi,-1)=F(\xi,-1)
\end{align}\endgroup
satisfies \eqref{eq_85a}, \eqref{eq_88a}, the first equation in \eqref{eq_89},
and \eqref{eq_90b}.

The final transformed formulation problem becomes:
Find the general form of $V(\xi,\eta)$ such that
it satisfies~\eqref{eq_16c},
\eqref{eq_97}, \eqref{eq_98}, \eqref{eq_21b}, \eqref{eq_82b}, and \eqref{eq_91}.
Note that $F(-\eta,\eta)$, $F(-1,\eta)$, $F(\xi,-1)$, $F_{\xi}(-1,\eta)$,
$F_{\eta}(\xi,-1)$, $F_{\xi\eta}(-1,-1)$ in these equations are placeholders,
and that the data for these terms contain the
free parameters $\Gamma_B$ and $\Gamma_C$ or
the unknown function $V$ to be formulated.

Next we devise a transfinite interpolation operator for these conditions.
Let
\begingroup\small\begin{align}\label{eq_99}
  PF(\xi,\eta) =&\ (P_1\oplus P_2)F(\xi,\eta)
  = P_1F(\xi,\eta) + P_2F(\xi,\eta) - P_1P_2F(\xi,\eta) \notag \\
  =&\ A_F(\xi,\eta) + \alpha(\xi,\eta)\big[F(\xi,-\xi) - A_F(\xi,-\xi) \big]
  + \beta(\xi,\eta)\big[F(-\eta,\eta) - A_F(-\eta,\eta) \big],
\end{align}\endgroup
where
\begin{subequations}\label{eq_100}
  \small
  \begin{align}
    &P_1F(\xi,\eta) = \alpha(\xi,\eta)F(\xi,-\xi) + \beta(\xi,\eta)F(-\eta,\eta);
    \label{eq_100a}\\
    &A_F(\xi,\eta) = P_2F(\xi,\eta) := (T_{\xi}\oplus T_{\eta})F(\xi,\eta)
    = T_{\xi}F(\xi,\eta) + T_{\eta}F(\xi,\eta) - T_{\xi}T_{\eta}F(\xi,\eta) \notag \\
    &\qquad\quad\qquad\qquad\qquad\
    = F(-1,\eta)+ F(\xi,-1) + (\xi+1)\big[F_{\xi}(-1,\eta)-F_{\xi}(-1,-1) \big]
     \notag \\
     &\qquad\qquad\qquad\qquad\quad\quad\
     + (\eta+1)\big[F_{\eta}(\xi,-1)-F_{\eta}(-1,-1) \big]
     - (\xi+1)(\eta+1) F_{\xi\eta}(-1,-1);
     \label{eq_100b}\\
     &T_{\xi}F(\xi,\eta) = F(-1,\eta) + (\xi+1)F_{\xi}(-1,\eta); \quad
     T_{\eta}F(\xi,\eta) = F(\xi,-1) + (\eta+1)F_{\eta}(\xi,-1). \label{eq_100d}
  \end{align}
\end{subequations}
In the above equations $F(-1,\eta)$ and $F(\xi,-1)$ are defined
by~\eqref{eq_95} and \eqref{eq_96}, respectively.
$T_{\xi}F$ and $T_{\eta}F$ are 1st-order Taylor approximations of
$F(\xi,\eta)$ along the $\xi$ and $\eta$ directions, about $\xi=-1$ and $\eta=-1$,
respectively. By using the relations in~\eqref{eq_93} and \eqref{eq_32},
we can establish the following result.
\begin{property}\label{lem_214}
  $PF(\xi,\eta)$ given by~\eqref{eq_99} satisfies:
  \begingroup\small\begin{align}
    \left\{\begin{array}{l}
    PF(-1,\eta) = F(-1,\eta)
    = F(-1,1)\rho_1(\eta)
    + F_{\eta}(-1,1)\upsilon_1(\eta)+ F_{\eta}(-1,-1)\upsilon_0(\eta),
    \\
    PF(\xi,-1)=F(\xi,-1) 
    = F(1,-1)\rho_1(\xi) + F_{\xi}(1,-1)\upsilon_1(\xi)
    + F_{\xi}(-1,-1)\upsilon_0(\xi) + \lambda_AF_{\xi\xi}(-1,-1)\omega_0(\xi),
    \\
    PF(-\eta,\eta) = F(-\eta,\eta),  \quad
    (PF)_{\xi}(-1,\eta) = F_{\xi}(-1,\eta), \quad
    (PF)_{\eta}(\xi,-1) = F_{\eta}(\xi,-1), \\
    (PF)_{\xi\xi}(-1,-1) = \lambda_A F_{\xi\xi}(-1,-1), \quad
    (PF)_{\xi\eta}(-1,-1) = F_{\xi\eta}(-1,-1).
    \end{array}\right.
  \end{align}\endgroup
\end{property}

Then we employ $PF$ to construct a preliminary general form for $V$,
\begingroup\small\begin{equation}\label{eq_102}
  V(\xi,\eta) = g(\xi,\eta) - Pg(\xi,\eta) + PF(\xi,\eta),
\end{equation}\endgroup
where $g(\xi,\eta)$ is a sufficiently-differentiable free function,
and $Pg$ is defined according to \eqref{eq_99} and given by
\begin{subequations}\label{eq_103}
  \small
  \begin{align}
    & Pg(\xi,\eta) = A_g(\xi,\eta) + \alpha(\xi,\eta)\big[g(\xi,-\xi) - A_g(\xi,-\xi) \big]
    + \beta(\xi,\eta)\big[g(-\eta,\eta) - A_g(-\eta,\eta) \big],
    \label{eq_103a} \\
    &A_g(\xi,\eta) = \big[g(-1,1)\rho_1(\eta) + g_{\eta}(-1,1)\upsilon_1(\eta)
      + g_{\eta}(-1,-1)\upsilon_0(\eta)  \big]
    + \big[g(1,-1)\rho_1(\xi) + g_{\xi}(1,-1)\upsilon_1(\xi) 
    \notag \\
    &\qquad\qquad\ \
    + g_{\xi}(-1,-1)\upsilon_0(\xi)
    + \lambda_Ag_{\xi\xi}(-1,-1)\omega_0(\xi) \big]
    + (\xi+1)\big[g_{\xi}(-1,\eta)-g_{\xi}(-1,-1) \big]
    \notag \\
    &\qquad\qquad\ \
     + (\eta+1)\big[g_{\eta}(\xi,-1)-g_{\eta}(-1,-1) \big]
     - (\xi+1)(\eta+1) g_{\xi\eta}(-1,-1). \label{eq_103b}
  \end{align}
\end{subequations}

Lastly, we replace the placeholder terms in $PF(\xi,\eta)$ by their compatible
data to attain an updated $PF^g(\xi,\eta)$,
which will lead to the final
form for $V$. The updated transfinite interpolation is
\begingroup\small\begin{align}\label{eq_105}
  PF^g(\xi,\eta) =&\ A_F^g(\xi,\eta) + \alpha(\xi,\eta)\big[U_{BC}(-\xi) - A_F^g(\xi,-\xi) \big]
  + \beta(\xi,\eta)\big[U_{BC}(\eta) - A_F^g(-\eta,\eta) \big].
\end{align}\endgroup
Here $A_{F}^g$ is defined by
\begin{subequations}\small
  \begin{align}
    A_F^g(\xi,\eta) =&\ \big[U_{BC}(1)\rho_1(\eta) + F_{\eta}^g(-1,1)\upsilon_1(\eta)
      + F_{\eta}^a(-1,-1)\upsilon_0(\eta)  \big] \notag \\
    &
    + \big[U_{BC}(-1)\rho_1(\xi) + F_{\xi}^g(1,-1)\upsilon_1(\xi) + F_{\xi}^a(-1,-1)\upsilon_0(\xi)
      + \lambda_AF_{\xi\xi}^g(-1,-1)\omega_0(\xi) \big] \notag \\
    &
    + (\xi+1)\big[F_{\xi}^{gg}(-1,\eta)-F_{\xi}^a(-1,-1) \big]
     + (\eta+1)\big[F_{\eta}^{gg}(\xi,-1)-F_{\eta}^a(-1,-1) \big]
     \notag \\
     &
     - (\xi+1)(\eta+1) F_{\xi\eta}^g(-1,-1); \\
     F_{\eta}^g(-1,1) =&\ (1-\lambda_C)\Gamma_C + \lambda_C F_{\eta}^a(-1,1)
     = (1-\lambda_C)g_{\eta}(-1,1)+\lambda_C F_{\eta}^a(-1,1); 
     \label{eq_106b} \\
     F_{\xi}^g(1,-1) =&\ (1-\lambda_B)\Gamma_B + \lambda_B F_{\xi}^a(1,-1)
     = (1-\lambda_B)g_{\xi}(1,-1) + \lambda_B F_{\xi}^a(1,-1);
     \label{eq_106c} \\
     F_{\xi\xi}^g(-1,-1) =&\  \frac{S_{AC}(-1)}{S_{AB}(-1)}g_{\eta\eta}(-1,-1)
     + \frac{Q^a(-1,-1)}{S_{AB}(-1)}; \\
     F_{\xi\eta}^g(-1,-1) =&\ T_{AC}'(-1) - S_{AC}'(-1)F_{\eta}^a(-1,-1)
     - S_{AC}(-1)g_{\eta\eta}(-1,-1); 
     \\
     F_{\xi}^{gg}(-1,\eta) =&\ T_{AC}(\eta) - S_{AC}(\eta)\big[
       g_{\eta}(-1,\eta) - \big(g(-1,1)-U_{BC}(1) \big)\rho_1'(\eta) \notag \\
       &
       - \big(g_{\eta}(-1,1)-F_{\eta}^g(-1,1) \big)\upsilon_1'(\eta)
       - \big(g_{\eta}(-1,-1)-F_{\eta}^a(-1,-1) \big)\upsilon_0'(\eta)
       \big]; \label{eq_106f} \\
     F_{\eta}^{gg}(\xi,-1) =&\ T_{AB}(\xi) - S_{AB}(\xi)\big[
       g_{\xi}(\xi,-1) - \big(g(1,-1)-U_{BC}(-1) \big)\rho_1'(\xi)
       - \big(g_{\xi}(1,-1)-F_{\xi}^g(1,-1) \big)\upsilon_1'(\xi)
       \notag \\
       &
       - \big(g_{\xi}(-1,-1)-F_{\xi}^a(-1,-1) \big)\upsilon_0'(\xi)
       -\lambda_A\big(g_{\xi\xi}(-1,-1) - F_{\xi\xi}^g(-1,-1) \big)\omega_0'(\xi)
       \big]. \label{eq_106g}
  \end{align}
\end{subequations}
In equations \eqref{eq_106b} and \eqref{eq_106c} we have chosen the free parameters
$\Gamma_B$ and $\Gamma_C$ to be
$\Gamma_B=g_{\xi}(1,-1)$ and $\Gamma_C=g_{\eta}(-1,1)$, thus 
leaving $g(\xi,\eta)$ as the only unknown in the formulation.
Using equations~\eqref{eq_93}, \eqref{eq_90a}, \eqref{eq_86a},
\eqref{eq_28a}, and \eqref{eq_32}, one can
establish the following relations and properties.
\begin{property}\label{lem_216}
  The following relations hold, for any combination
  of $(\lambda_A,\lambda_B,\lambda_C)=(0/1,0/1,0/1)$,
  \begingroup\small\begin{equation}
    \left\{\begin{split}
      &F_{\eta}^{gg}(-1,-1)=\lim_{\xi\rightarrow -1}F_{\eta}^{gg}(\xi,-1)=F_{\eta}^a(-1,-1); \quad
      F_{\xi}^{gg}(-1,-1)=\lim_{\eta\rightarrow -1}F_{\xi}^{gg}(-1,\eta)=F_{\xi}^a(-1,-1); \\
    &\big(F_{\eta}^{gg}\big)_{\xi}(-1,-1) = F_{\xi\eta}^g(-1,-1); \qquad
    \big(F_{\xi}^{gg} \big)_{\eta}(-1,-1) = F_{\xi\eta}^g(-1,-1); \\
    & -F_{\xi}^{gg}(-1,1) + F_{\eta}^g(-1,1) = U_{BC}'(1); \qquad
    -F_{\xi}^g(1,-1) + F_{\eta}^{gg}(1,-1) = U_{BC}'(-1).
  \end{split}\right.\end{equation}\endgroup
\end{property}
\begin{property}\label{lem_217}
  $PF^g(\xi,\eta)$ given by \eqref{eq_105} satisfies,
  for any $g(\xi,\eta)$ therein that is
  sufficiently differentiable and for any combination
  of $(\lambda_A,\lambda_B,\lambda_C)=(0/1,0/1,0/1)$,
  \begingroup\small\begin{align}
    \left\{\begin{array}{l}
    PF^g(-1,\eta) = U_{BC}(1) \rho_1(\eta) + F_{\eta}^g(-1,1)\upsilon_1(\eta)
    + F_{\eta}^a(-1,-1)\upsilon_0(\eta); \\
    PF^g(\xi,-1) = U_{BC}(-1)\rho_1(\xi) + F_{\xi}^g(1,-1)\upsilon_1(\xi)
    + F_{\xi}^a(-1,-1)\upsilon_0(\xi) + \lambda_AF_{\xi\xi}^g(-1,-1)\omega_0(\xi); \\
    PF^g(-\eta,\eta) = U_{BC}(\eta); \quad
    \big(PF^g \big)_{\xi}(-1,\eta) = F_{\xi}^{gg}(-1,\eta); \quad
    \big(PF^g \big)_{\eta}(\xi,-1) = F_{\eta}^{gg}(\xi,-1); \\
    \big(PF^g \big)_{\xi\xi}(-1,-1) = \lambda_A F_{\xi\xi}^g(-1,-1); \qquad
    \big(PF^g \big)_{\xi\eta}(-1,-1) = F_{\xi\eta}^g(-1,-1).
    \end{array}\right.
  \end{align}\endgroup
\end{property}

We use the updated $PF^g$ to attain the final form for $V(\xi,\eta)$,
\begingroup\small\begin{equation}\label{eq_109}
  V(\xi,\eta) = g(\xi,\eta) - Pg(\xi,\eta) + PF^g(\xi,\eta),
\end{equation}\endgroup
where $Pg$ is defined in \eqref{eq_103} and $PF^g$ is given by \eqref{eq_105}.
This leads to the following result.
\begin{theorem}
  $V(\xi,\eta)$ given by \eqref{eq_109} satisfies the Dirichlet
  condition \eqref{eq_16c} and the Neumann conditions
  \eqref{eq_21a} and \eqref{eq_82a}, for any $g(\xi,\eta)$ therein
  that is sufficiently differentiable and for any combination
  of $(\lambda_A,\lambda_B,\lambda_C)=(0/1,0/1,0/1)$.
\end{theorem}
\begin{proof}By using Properties \ref{lem_214} and \ref{lem_217}
  and equations \eqref{eq_106f}--\eqref{eq_106g}, and by noting that
  \begingroup\small\begin{align*}
    V_{\xi}(-1,\eta) =&\ F_{\xi}^{gg}(-1,\eta), \quad
    V_{\eta}(\xi,-1) =\ F_{\eta}^{gg}(\xi,-1), \\    
    V_{\eta}(-1,\eta) =&\ g_{\eta}(-1,\eta) - \big[g(-1,1)-U_{BC}(1) \big]\rho_1'(\eta)
    -\big[g_{\eta}(-1,1)-F_{\eta}^g(-1,1) \big]\upsilon_1'(\eta) \\
    & - \big[g_{\eta}(-1,-1)-F_{\eta}^a(-1,-1) \big]\upsilon_0'(\eta), \\
    V_{\xi}(\xi,-1) =&\ g_{\xi}(\xi,-1) - \big[g(1,-1)-U_{BC}(-1) \big]\rho_1'(\xi)
    - \big[g_{\xi}(1,-1) - F_{\xi}^g(1,-1) \big]\upsilon_1'(\xi) \\
    & - \big[g_{\xi}(-1,-1) - F_{\xi}^a(-1,-1) \big]\upsilon_0(\xi)
    - \lambda_A\big[g_{\xi\xi}(-1,-1) - F_{\xi\xi}^g(-1,-1) \big]\omega_0(\xi).
  \end{align*}\endgroup
\end{proof}

\begin{remark}
  With Neumann conditions imposed on $\overline{AB}$ and $\overline{AC}$,
  vertices $B$ and $C$ are junctions between a Dirichlet boundary and
  a Neumann boundary. The above formulation, with no change,
  applies to configurations where the boundary is smooth at $B$ or $C$ or both.
  Note that we require that the double-Neumann vertex $A$
  should not be smooth or a cusp.
\end{remark}


\subsection{Enforcing Exactly Robin Boundary Conditions (RBC)}
\label{sec_rbc}

The Neumann formulations in
Section \ref{sec_nbc}  can largely carry over
to Robin boundaries, with some modification
to account for  the extra Robin  term.
In this section we consider a combination of Robin and Dirichlet
conditions on the general triangular domain, with at least
one Robin boundary and one Dirichlet boundary.

\subsubsection{RBC Imposed on a Single Triangle Boundary}

If the domain involves a single Robin boundary,
without loss of generality, we focus on the configuration with the RBC imposed
on $\overline{AC}$, mirroring Section \ref{sec_2311} for a
single Neumann boundary. We leave the formulation 
for the sole RBC imposed on the hypotenuse
to an Appendix (Section \ref{sec_r_BC}).


Consider the general triangular domain $\Omega$ in Figure~\ref{fg_1}(a)
and a scalar function $u(\mbs x)$ defined on $\Omega$ that satisfies the conditions,
\begingroup\small\begin{equation}
  u|_{\mbs x\in\overline{AB}} = u_{AB}(\mbs x),  \quad
     u|_{\mbs x\in\overline{BC}} = u_{BC}(\mbs x),  \quad
     \mbs n\cdot\nabla u|_{\mbs x\in\overline{AC}} + \alpha_{AC}u|_{\mbs x\in\overline{AC}} = u_{rAC}(\mbs x),
\end{equation}\endgroup
where $\alpha_{AC}$ is the prescribed Robin coefficient and $u_{rAC}$ denotes
the prescribed data on the Robin boundary.
Under the map from Section \ref{sec_21}, the Dirichlet conditions
are transformed into equations~\eqref{eq_16a} and \eqref{eq_16c}.
The Robin condition is transformed into
\begin{subequations}\label{eq_111}
  \small
  \begin{align}
    & V_{\xi}(-1,\eta) + S_{AC}(\eta)V_{\eta}(-1,\eta) + \alpha_{AC}W_{AC}(\eta)V(-1,\eta)
    = T_{rAC}(\eta), \label{eq_111a}  \\
    \text{or}\ & V_{\xi}(-1,\eta) = T_{rAC}(\eta) - S_{AC}(\eta)V_{\eta}(-1,\eta)
    - \alpha_{AC}W_{AC}(\eta)V(-1,\eta)
    =: F_{\xi}(-1,\eta),
  \label{eq_111b}
\end{align}
\end{subequations}
where $S_{AC}(\eta)$ and $W_{AC}(\eta)$ are defined in \eqref{eq_22}, and
$T_{rAC}(\eta)=u_{rAC}(\mbs x(-1,\eta))W_{AC}(\eta)$.

An analysis similar to in Section \ref{sec_2311} produces
the following compatibility constraints. At vertex $A$,
\begin{subequations}\label{eq_112}
  \small
  \begin{align}
    & T_{rAC}(-1) - F_{\xi}(-1,-1) - \alpha_{AC}W_{AC}(-1)F(-1,-1) = 0, \quad
    \text{if}\ \overline{AB}\perp\overline{AC}\ \text{at}\ A; \label{eq_112a} \\
    & V_{\eta}(-1,-1) = F_{\eta}^a(-1,-1)
    = \frac{1}{S_{AC}(-1)}\big[T_{rAC}(-1)-F_{\xi}(-1,-1) - \alpha_{AC}W_{AC}(-1)F(-1,-1) \big]
    \notag \\
    &\qquad\qquad\qquad\qquad\qquad\quad\
    =: F_{\eta}(-1,-1), \quad
    \text{if}\ \overline{AB}\not\perp\overline{AC}\ \text{at}\ A. \label{eq_112b}
  \end{align}
\end{subequations}
At vertex $C$, we have
\begin{subequations}\label{eq_113}
  \small
  \begin{align}
    & T_{rAC}(1) + U_{BC}'(1) - \alpha_{AC}W_{AC}(1)U_{BC}(1) = 0, \quad
    \text{if}\ \overline{AC}\perp\overline{BC}\ \text{at}\ C; \\
    & V_{\eta}(-1,1) 
    = \frac{T_{rAC}(1)+U_{BC}'(1) - \alpha_{AC}W_{AC}(1)U_{BC}(1)}{1+S_{AC}(1)} := F_{\eta}^a(-1,1),
    \quad
    \text{if}\ \overline{AC}\not\perp\overline{BC}\ \text{at}\ C. \label{eq_113b}
  \end{align}
\end{subequations}
We transform the conditional constraint in~\eqref{eq_113b} into an equivalent
unconditional form, as given by~\eqref{eq_34}, in which $F_{\eta}^a(-1,1)$ is now
defined by \eqref{eq_113b}.

The objective here is to find the general form of $V(\xi,\eta)$ that satisfies
\eqref{eq_16a}, \eqref{eq_16c}, \eqref{eq_111b}, \eqref{eq_23a}, \eqref{eq_23c},
\eqref{eq_112b}, and \eqref{eq_34}.
We replace the point constraints in \eqref{eq_23a}, \eqref{eq_23c},
\eqref{eq_112b}, and \eqref{eq_34} by the
$V$-profile condition \eqref{eq_35},
in which $F_{\eta}(-1,-1)$ and $F_\eta(-1,1)$ are now defined by
\eqref{eq_112b} and \eqref{eq_34}, respectively.
Therefore, the final transformed constraints consist of 
\eqref{eq_16a}, \eqref{eq_16c}, \eqref{eq_111b}, and \eqref{eq_35}.

Since the types of constraints here are identical to those of Section \ref{sec_2311},
the transfinite interpolation $PF(\xi,\eta)$  given in
 \eqref{eq_38}--\eqref{eq_39} and
the preliminary $V(\xi,\eta)$ given by \eqref{eq_41}--\eqref{eq_42b}
both carry over here. The updated transfinite interpolation $PF^g(\xi,\eta)$
is similarly given by \eqref{eq_44}, in which $A_F^g$ and $F_{\eta}^g(-1,1)$ are given by
\eqref{eq_45a}--\eqref{eq_45b}. However, $F_{\xi}^g(-1,\eta)$ in the $A_F^g(\xi,\eta)$
expression is now given by
\begin{subequations}\label{eq_114}
  \small
  \begin{align}
    & F_{\xi}^g(-1,\eta) = T_{rAC}(\eta) - S_{AC}(\eta)V_{\eta}^g(-1,\eta)
    -\alpha_{AC}W_{AC}(\eta)V^g(-1,\eta), \label{eq_114a} \\
    & V^g(-1,\eta) = g(-1,\eta)- [g(-1,-1)-F(-1,-1)]\varphi_0(\eta)
     - [g(-1,1)-U_{BC}(1)]\varphi_1(\eta) \notag \\
     &\qquad\qquad\quad\ - \lambda_A[g_{\eta}(-1,-1)-F_{\eta}^a(-1,-1)]\psi_0(\eta)
     - [g_{\eta}(-1,1)-F_{\eta}^g(-1,1)]\psi_1(\eta), \label{eq_114b} \\
    & V_{\eta}^g(-1,\eta) = g_{\eta}(-1,\eta)- [g(-1,-1)-F(-1,-1)]\varphi_0'(\eta)
     - [g(-1,1)-U_{BC}(1)]\varphi_1'(\eta) \notag \\
     &\qquad\qquad\quad\
     - \lambda_A[g_{\eta}(-1,-1)-F_{\eta}^a(-1,-1)]\psi_0'(\eta)
     - [g_{\eta}(-1,1)-F_{\eta}^g(-1,1)]\psi_1'(\eta), \label{eq_114c}
  \end{align}
\end{subequations}
where $F_{\eta}^a(-1,-1)$ and  $F_{\eta}^g(-1,1)$ are
given by \eqref{eq_112b} and \eqref{eq_45b},
respectively.
The final form of $V(\xi,\eta)$ is then given by
\begingroup\small\begin{equation}\label{eq_115}
  V(\xi,\eta) = g(\xi,\eta) - Pg(\xi,\eta) + PF^g(\xi,\eta),
\end{equation}\endgroup
where $g$ is a free function, $Pg$ is defined by \eqref{eq_42}, and
$PF^g$ is defined by \eqref{eq_44}--\eqref{eq_45b} and
\eqref{eq_114a}--\eqref{eq_114c}.
\begin{theorem}
  $V(\xi,\eta)$ given by \eqref{eq_115} satisfies the Dirichlet conditions \eqref{eq_16a}
  and \eqref{eq_16c} and the Robin condition \eqref{eq_111a},
  for any $g(\xi,\eta)$ therein that is sufficiently differentiable.
\end{theorem}
\begin{proof} Note that $V_{\xi}(-1,\eta)=F_{\xi}^g(-1,\eta)$, $V(-1,\eta)=V^g(-1,\eta)$,
  and $V_{\eta}(-1,\eta)=V_{\eta}^g(-1,\eta)$. In light of \eqref{eq_114a}, the Robin
  condition \eqref{eq_111a} holds. The Dirichlet conditions can be verified by
  noting that $\alpha(-\eta,\eta)+\beta(-\eta,\eta)=1$ in \eqref{eq_32} and
  that $F_{\xi}^g(-1,-1)=\lim_{\eta\rightarrow -1}F_{\xi}^g(-1,\eta)=F_{\xi}(-1,-1)$.
\end{proof}

\begin{remark}
  When the sole RBC is imposed on $\overline{AB}$ (instead of $\overline{AC}$),
  a reciprocal formulation
  can be constructed analogously.
\end{remark}


\subsubsection{RBC Imposed on Two Triangle Boundaries}

For two Robin boundaries, we focus here on the configuration in which they correspond to
the two legs of the standard triangle, mirroring Section \ref{sec_232}.
The formulation with one of the Robin boundaries corresponding to
the hypotenuse will be deferred to an Appendix (Section \ref{sec_r_ACBC}).

Consider a scalar function $u(\mbs x)$ defined on
the general triangle $\Omega$ from Figure~\ref{fg_1}(a) satisfying
\begingroup\small\begin{equation}
  \begin{split}
  &\mbs n\cdot\nabla u|_{\mbs x\in\overline{AB}}
  + \alpha_{AB}u|_{\mbs x\in\overline{AB}} = u_{rAB}(\mbs x),\quad
  &u|_{\mbs x\in\overline{BC}} = u_{BC}(\mbs x),  \quad
  &\mbs n\cdot\nabla u|_{\mbs x\in\overline{AC}}
  + \alpha_{AC}u|_{\mbs x\in\overline{AC}} = u_{rAC}(\mbs x),
  \end{split}\end{equation}\endgroup
where $\alpha_{AB}$ and $\alpha_{AC}$ are the Robin coefficients, and
$u_{rAB}$ and $u_{rAC}$ denote the prescribed Robin data.
Leveraging the map $\mbs x(\xi,\eta)$ from Section \ref{sec_21},
these conditions are transformed into \eqref{eq_16c} for the DBC
and \eqref{eq_111} for the RBC on $\overline{AC}$, together with
the following form for the RBC on $\overline{AB}$,
\begin{subequations}\label{eq_124}
  \small
  \begin{align}
    & V_{\eta}(\xi,-1) + S_{AB}(\xi)V_{\xi}(\xi,-1) + \alpha_{AB}W_{AB}(\xi)V(\xi,-1) = T_{rAB}(\xi),
    \label{eq_124a} \\
   \text{or}\ & V_{\eta}(\xi,-1) = T_{rAB}(\xi) - S_{AB}(\xi)V_{\xi}(\xi,-1) - \alpha_{AB}W_{AB}(\xi)V(\xi,-1)
    =: F_{\eta}(\xi,-1), \label{eq_124b}
  \end{align}
\end{subequations}
where $S_{AB}(\xi)$ and $W_{AB}(\xi)$ are defined in \eqref{eq_83}, and
$T_{rAB}(\xi)=u_{rAB}(\mbs x(\xi,-1))W_{AB}(\xi)$.
Our objective is to construct the general form of $V(\xi,\eta)$ that satisfies
the conditions \eqref{eq_16c}, \eqref{eq_111} and \eqref{eq_124}.


An analysis analogous to in Section \ref{sec_232} leads to the following compatibility
constraints. At vertex $B$,
\begin{subequations}\small
  \begin{align}
    & T_{AB}(1) - \alpha_{AB}W_{AB}(1)U_{BC}(-1) - U_{BC}'(-1) = 0, \quad
    \text{if}\ \overline{AB}\perp\overline{BC}\ \text{at}\ B;
    \\
    & V_{\xi}(1,-1) = \frac{T_{AB}(1)-\alpha_{AB}W_{AB}(-1)U_{BC}(-1)-U_{BC}'(-1)}{1+S_{AB}(1)}
    =: F_{\xi}^a(1,-1), \quad
    \text{if}\ \overline{AB}\not\perp\overline{BC}\ \text{at}\ B. \label{eq_125b}
  \end{align}
\end{subequations}
At vertex $C$,
\begin{subequations}\small
  \begin{align}
    & T_{AC}(1)-\alpha_{AC}W_{AC}(1)U_{BC}(1) + U_{BC}'(1) = 0, \quad
    \text{if}\ \overline{AC}\perp\overline{BC}\ \text{at}\ C;
    \\
    & V_{\eta}(-1,1) = \frac{T_{AC}(1)-\alpha_{AC}W_{AC}(1)U_{BC}(1)+U_{BC}'(1)}{1+S_{AC}(1)}
    =: F_{\eta}^a(-1,1), \quad
    \text{if}\ \overline{AC}\not\perp\overline{BC}\ \text{at}\ C.  \label{eq_126b}
  \end{align}
\end{subequations}
At vertex $A$,
\begingroup\small\begin{align}\label{eq_127}
  & \begin{bmatrix}V_{\xi}(-1,-1)\\ V_{\eta}(-1,-1)  \end{bmatrix}
  = \begin{bmatrix}F_{\xi}^a(-1,-1)\\ F_{\eta}^a(-1,-1)  \end{bmatrix}
  - \begin{bmatrix}F_{\xi}^b(-1,-1)\\ F_{\eta}^b(-1,-1)  \end{bmatrix}V(-1,-1)
  =: \begin{bmatrix}F_{\xi}(-1,-1)\\ F_{\eta}(-1,-1)  \end{bmatrix},
\end{align}\endgroup
where 
\begingroup\small\begin{equation}
  \begin{bmatrix}F_{\xi}^a(-1,-1)\\ F_{\eta}^a(-1,-1)  \end{bmatrix}
  = \begin{bmatrix}
    \frac{T_{AC}(-1) - S_{AC}(-1)T_{AB}(-1)}{1-S_{AC}(-1)S_{AB}(-1)} \\ 
    \frac{T_{AB}(-1) - S_{AB}(-1)T_{AC}(-1)}{1-S_{AC}(-1)S_{AB}(-1)}
  \end{bmatrix}, \
  \begin{bmatrix}F_{\xi}^b(-1,-1)\\ F_{\eta}^b(-1,-1)  \end{bmatrix}
  = \begin{bmatrix}
    \frac{\alpha_{AC}W_{AC}(-1) - S_{AC}(-1)\alpha_{AB}W_{AB}(-1)}{1-S_{AC}(-1)S_{AB}(-1)} \\ 
    \frac{\alpha_{AB}W_{AB}(-1) - S_{AB}(-1)\alpha_{AC}W_{AC}(-1)}{1-S_{AC}(-1)S_{AB}(-1)}
  \end{bmatrix}.
\end{equation}\endgroup
Furthermore,
\begingroup\small\begin{align}\label{eq_129}
  & V_{\xi\eta}(-1,-1) = -S_{AC}(-1)V_{\eta\eta}(-1,-1) + Q_{AC}^a(-1) + Q_{AC}^b(-1) V(-1,-1)
  =: F_{\xi\eta}(-1,-1),
\end{align}\endgroup
together with
\begin{subequations}\label{eq_130}
  \small
  \begin{align}
    & V_{\xi\xi}(-1,-1)=
    \frac{S_{AC}(-1)V_{\eta\eta}(-1,-1) + S_1V(-1,-1) + S_2}{S_{AB}(-1)}
    =:F_{\xi\xi}(-1,-1),\quad
    \text{if}\ \overline{AB}\not\perp\overline{AC}\ \text{at}\ A, \\
    & S_2=0,
    \qquad\qquad\qquad\qquad\qquad\qquad\qquad\qquad\qquad\qquad
    \text{if}\ \overline{AB}\perp\overline{AC}\ \text{at}\ A\
    \text{and}\ S_1=0, \\
    & V(-1,-1) = F^a(-1,-1) = -S_2/S_1 =: F(-1,-1), \quad
    \text{if}\ \overline{AB}\perp\overline{AC}\ \text{at}\ A\
    \text{and}\ S_1\neq 0,
  \end{align}
\end{subequations}
where
\begingroup\small\begin{equation}
  \left\{
  \begin{split}
    & Q_{AC}^a(-1) = T_{AC}'(-1)-\big[S_{AC}'(-1)+\alpha_{AC}W_{AC}(-1) \big]F_{\eta}^a(-1,-1), \\
    & Q_{AC}^b(-1) = \big[S_{AC}'(-1)+\alpha_{AC}W_{AC}(-1) \big]F_{\eta}^b(-1,-1) - \alpha_{AC}W_{AC}'(-1), \\
    & S_1 = Q_{AB}^b(-1) - Q_{AC}^b(-1), \quad
    S_2 = Q_{AB}^a(-1) - Q_{AC}^a(-1), \\
    & Q_{AB}^a(-1) = T_{AB}'(-1)-\big[S_{AB}'(-1)+\alpha_{AB}W_{AB}(-1) \big]F_{\xi}^a(-1,-1), \\
    & Q_{AB}^b(-1) = \big[S_{AB}'(-1)+\alpha_{AB}W_{AB}(-1) \big]F_{\xi}^b(-1,-1) - \alpha_{AB}W_{AB}'(-1).
  \end{split}
\right.\end{equation}\endgroup
The conditional constraints in \eqref{eq_125b} and \eqref{eq_126b}
can then be transformed into the equivalent unconditional forms given in \eqref{eq_88a}
and \eqref{eq_88b}, in which $\Gamma_B$ and $\Gamma_C$ are free parameters
and $F_{\xi}^a(1,-1)$ and $F_{\eta}^a(-1,1)$ are now defined in \eqref{eq_125b}
and \eqref{eq_126b}.

We next transform the point constraints in equations \eqref{eq_127},
\eqref{eq_88a}--\eqref{eq_88b},
\eqref{eq_129} and \eqref{eq_130} into conditions about
the edge profiles of $V$ on $\overline{AB}$ and $\overline{AC}$,
\begingroup\small\begin{align} 
   V(\xi,-1) &= (1-\lambda_A)\gamma_A F(-1,-1)\rho_0(\xi) + F(1,-1)\rho_1(\xi)
  + F_{\xi}(-1,-1)\upsilon_0(\xi) + F_{\xi}(1,-1)\upsilon_1(\xi) \notag \\
  &\ \ \ + \lambda_A F_{\xi\xi}(-1,-1)\omega_0(\xi) =: F(\xi,-1), \label{eq_132} \\
  V(-1,\eta) &= (1-\lambda_A)\gamma_A F(-1,-1)\rho_0(\eta) + F(-1,1)\rho_1(\eta)
  + F_{\eta}(-1,-1)\upsilon_0(\eta) + F_{\eta}(-1,1)\upsilon_1(\eta)
  \notag \\
  &=: F(-1,\eta), \label{eq_133}
\end{align}\endgroup
where the constant $\gamma_A$ is defined by
\begingroup\small\begin{equation}
\gamma_A=\text{sign}(|S_1|)=
\left\{\begin{array}{ll}0, & \text{if}\ S_1=0; \\ 1, & \text{otherwise}. \end{array}\right.
\end{equation}\endgroup

So the formulation problem becomes:
Find the general form of $V(\xi,\eta)$ that satisfies \eqref{eq_16c},
\eqref{eq_111b}, \eqref{eq_124b}, \eqref{eq_132} and \eqref{eq_133},
together with \eqref{eq_129}.
The transfinite interpolation $PF(\xi,\eta)$ that
satisfies these constraints can be constructed
in a way analogous to in Section \ref{sec_232} and is given by
\begin{subequations}\label{eq_135}
  \small
\begin{align}
  PF(\xi,\eta) =&\ A_F(\xi,\eta) + \alpha(\xi,\eta)\big[F(\xi,-\xi) - A_F(\xi,-\xi) \big]
  + \beta(\xi,\eta)\big[F(-\eta,\eta) - A_F(-\eta,\eta) \big], \\
  A_F(\xi,\eta) =&\ F(-1,\eta) + F(\xi,-1) - (1-\lambda_A)\gamma_A F(-1,-1)
  + (\xi+1)\big[F_{\xi}(-1,\eta) - F_{\xi}(-1,-1) \big] \notag \\
  &+ (\eta+1)\big[F_{\eta}(\xi,-1) - F_{\eta}(-1,-1) \big]
  - (\xi+1)(\eta+1)F_{\xi\eta}(-1,-1),
\end{align}
\end{subequations}
where $F(-1,\eta)$ and $F(\xi,-1)$ are defined in \eqref{eq_132}--\eqref{eq_133}.
The preliminary general form for $V$ is
\begingroup\small\begin{equation}
  V(\xi,\eta) = g(\xi,\eta) - Pg(\xi,\eta) + PF(\xi,\eta),
\end{equation}\endgroup
where $g$ is a free function, $PF$ is given by \eqref{eq_135}, and
$Pg$ is defined analogously and specifically given by
\begin{subequations}\label{eq_137}
  \small
\begin{align}
  Pg(\xi,\eta) =&\ A_g(\xi,\eta) + \alpha(\xi,\eta)\big[g(\xi,-\xi) - A_g(\xi,-\xi) \big]
  + \beta(\xi,\eta)\big[g(-\eta,\eta) - A_g(-\eta,\eta) \big], \\
  A_g(\xi,\eta) =&\ \big[ (1-\lambda_A)\gamma_A g(-1,-1)\rho_0(\xi) + g(1,-1)\rho_1(\xi)
    + g_{\xi}(-1,-1)\upsilon_0(\xi) + g_{\xi}(1,-1)\upsilon_1(\xi) \notag \\
    &+ \lambda_Ag_{\xi\xi}(-1,-1)\omega_0(\xi) \big] 
  + \big[(1-\lambda_A)\gamma_A g(-1,-1)\rho_0(\eta) + g(-1,1)\rho_1(\eta)
    + g_{\eta}(-1,-1)\upsilon_0(\eta) \notag \\
    &+ g_{\eta}(-1,1)\upsilon_1(\eta)\big] 
  - (1-\lambda_A)\gamma_A g(-1,-1)
  + (\xi+1)\big[g_{\xi}(-1,\eta) - g_{\xi}(-1,-1) \big] \notag \\
  &+ (\eta+1)\big[g_{\eta}(\xi,-1) - g_{\eta}(-1,-1) \big]
  - (\xi+1)(\eta+1)g_{\xi\eta}(-1,-1).
\end{align}
\end{subequations}

Updated with  compatible data, the transfinite interpolation becomes
\begingroup\small\begin{equation}\label{eq_138}
  PF^g(\xi,\eta) = A_F^g(\xi,\eta) + \alpha(\xi,\eta)\big[U_{BC}(-\xi) - A_F^g(\xi,-\xi) \big]
  + \beta(\xi,\eta)\big[U_{BC}(\eta) - A_F^g(-\eta,\eta) \big],
\end{equation}\endgroup
where
\begingroup\small\begin{subequations}\label{eq_139}
  \begin{align}
    A_F^g(\xi,\eta) =&\ F^g(-1,\eta) + F^g(\xi,-1) - (1-\lambda_A)\gamma_A F^a(-1,-1)
    + (\xi+1)\big[F_{\xi}^{gg}(-1,\eta)-F_{\xi}^g(-1,-1) \big] \notag \\
    &+ (\eta+1)\big[F_{\eta}^{gg}(\xi,-1)-F_{\eta}^g(-1,-1) \big]
    - (\xi+1)(\eta+1) F_{\xi\eta}^g(-1,-1); \\
    F^g(-1,\eta) =&\ (1-\lambda_A)\gamma_A F^a(-1,-1)\rho_0(\eta) + U_{BC}(1)\rho_1(\eta)
    + F_{\eta}^g(-1,-1)\upsilon_0(\eta) + F_{\eta}^g(-1,1)\upsilon_1(\eta), \\
    F^g(\xi,-1) =&\ (1-\lambda_A)\gamma_A F^a(-1,-1)\rho_0(\xi) + U_{BC}(-1)\rho_1(\xi)
    + F_{\xi}^g(-1,-1)\upsilon_0(\xi) + F_{\xi}^g(1,-1)\upsilon_1(\xi) \notag \\
    & + \lambda_A F_{\xi\xi}^g(-1,-1)\omega_0(\xi), \\
    F^g_{\xi}(-1,-1) =&\ F_{\xi}^a(-1,-1) - F_{\xi}^b(-1,-1)\left[
      \big(1-(1-\lambda_A)\gamma_A \big)g(-1,-1) + (1-\lambda_A)\gamma_A F^a(-1,-1)
      \right], \\
    F^g_{\eta}(-1,-1) =&\ F_{\eta}^a(-1,-1) - F_{\eta}^b(-1,-1)\left[
      \big(1-(1-\lambda_A)\gamma_A \big)g(-1,-1) + (1-\lambda_A)\gamma_A F^a(-1,-1)
      \right], \\
    F_{\xi}^g(1,-1) =&\ (1-\lambda_B)\Gamma_B + \lambda_B F_{\xi}^a(1,-1)
    = (1-\lambda_B)g_{\xi}(1,-1) + \lambda_B F_{\xi}^a(1,-1),  \label{eq_139f} \\
    F_{\eta}^g(-1,1) =&\ (1-\lambda_C)\Gamma_C + \lambda_C F_{\eta}^a(-1,1)
    = (1-\lambda_C)g_{\eta}(-1,1) + \lambda_C F_{\eta}^a(-1,1), \label{eq_139g} \\
    F_{\xi\eta}^g(-1,-1) =&\ Q_{AC}^a(-1) - S_{AC}(-1)g_{\eta\eta}(-1,-1) \notag \\
    &+ Q_{AC}^b(-1)\left[
      \big(1-(1-\lambda_A)\gamma_A \big)g(-1,-1) + (1-\lambda_A)\gamma_A F^a(-1,-1)
      \right], \\
    F_{\xi\xi}^g(-1,-1) =&\ \frac{S_{AC}(-1)}{S_{AB}(-1)}g_{\eta\eta}(-1,-1)
    + \frac{S_2}{S_{AB}(-1)}
    \notag \\
    &+\frac{1}{S_{AB}(-1)}\left[
      \big(1-(1-\lambda_A)\gamma_A \big)g(-1,-1) + (1-\lambda_A)\gamma_A F^a(-1,-1)
      \right],
    \\
    F_{\xi}^{gg}(-1,\eta) =&\ T_{AC}(\eta) - S_{AC}(\eta) V_{\eta}^g(-1,\eta)
    - \alpha_{AC}W_{AC}(\eta) V^g(-1,\eta), \\
    F_{\eta}^{gg}(\xi,-1) =&\ T_{AB}(\xi) - S_{AB}(\xi) V^g_{\xi}(\xi,-1)
    - \alpha_{AB}W_{AB}(\xi) V^g(\xi,-1), \\
    V^g(-1,\eta) =&\ g(-1,\eta) - (1-\lambda_A)\gamma_A\big[g(-1,-1) - F^a(-1,-1) \big]\rho_0(\eta)
    - \big[g(-1,1) - U_{BC}(1) \big]\rho_1(\eta) \notag \\
    &- \big[g_{\eta}(-1,-1) - F_{\eta}^g(-1,-1) \big]\upsilon_0(\eta)
    - \big[g_{\eta}(-1,1) - F_{\eta}^g(-1,1) \big]\upsilon_1(\eta),  \\
    V_{\eta}^g(-1,\eta) =&\ g_{\eta}(-1,\eta) - (1-\lambda_A)\gamma_A\big[g(-1,-1)
      - F^a(-1,-1) \big]\rho_0'(\eta)
    - \big[g(-1,1) - U_{BC}(1) \big]\rho_1'(\eta) \notag \\
    &- \big[g_{\eta}(-1,-1) - F_{\eta}^g(-1,-1) \big]\upsilon_0'(\eta)
    - \big[g_{\eta}(-1,1) - F_{\eta}^g(-1,1) \big]\upsilon_1'(\eta), \\
    V^g(\xi,-1) =&\ g(\xi,-1) - (1-\lambda_A)\gamma_A\big[g(-1,-1) - F^a(-1,-1) \big]\rho_0(\xi)
    - \big[g(1,-1) - U_{BC}(-1) \big]\rho_1(\xi) \notag \\
    &- \big[g_{\xi}(-1,-1) - F_{\xi}^g(-1,-1) \big]\upsilon_0(\xi)
    - \big[g_{\xi}(1,-1) - F_{\xi}^g(1,-1) \big]\upsilon_1(\xi) \notag \\
    &-\lambda_A\big[g_{\xi\xi}(-1,-1) - F_{\xi\xi}^g(-1,-1) \big]\omega_0(\xi), \\
    V^g_{\xi}(\xi,-1) =&\ g_{\xi}(\xi,-1) - (1-\lambda_A)\gamma_A\big[g(-1,-1) - F^a(-1,-1) \big]\rho_0'(\xi)
    - \big[g(1,-1) - U_{BC}(-1) \big]\rho_1'(\xi) \notag \\
    &- \big[g_{\xi}(-1,-1) - F_{\xi}^g(-1,-1) \big]\upsilon_0'(\xi)
    - \big[g_{\xi}(1,-1) - F_{\xi}^g(1,-1) \big]\upsilon_1'(\xi) \notag \\
    &-\lambda_A\big[g_{\xi\xi}(-1,-1) - F_{\xi\xi}^g(-1,-1) \big]\omega_0'(\xi).
  \end{align}
\end{subequations}\endgroup
Note that in \eqref{eq_139f}--\eqref{eq_139g} we have chosen the free parameters
to be $\Gamma_B=g_{\xi}(1,-1)$ and $\Gamma_C=g_{\eta}(-1,1)$.
The final form of $V(\xi,\eta)$ is  given by,
\begingroup\small\begin{equation}\label{eq_140}
  V(\xi,\eta) = g(\xi,\eta) - Pg(\xi,\eta) + PF^g(\xi,\eta),
\end{equation}\endgroup
where $Pg$ is defined in \eqref{eq_137} and $PF^g$ is given by \eqref{eq_138}.

\begin{theorem}
  $V(\xi,\eta)$ given by \eqref{eq_140} satisfies the Dirichlet
  condition \eqref{eq_16c} and the Robin conditions \eqref{eq_111a}
  and \eqref{eq_124a}, for any $g(\xi,\eta)$ therein that is sufficiently differentiable.
\end{theorem}
\begin{proof}
  By verification and by noting that
  \begingroup\small\begin{equation}
    \left\{\begin{split}
    & F_{\eta}^{gg}(-1,-1) = \lim_{\eta\rightarrow -1} F_{\eta}^{gg}(-1,\eta)=F_{\eta}^g(-1,-1); \quad
    F_{\xi}^{gg}(-1,-1) = \lim_{\xi\rightarrow -1} F_{\xi}^{gg}(\xi,-1) = F_{\xi}^g(-1,-1); \\
    & F_{\eta,\xi}^{gg}(-1,-1) = \left.\frac{d}{d\xi}F_{\eta}^{gg}(\xi,-1) \right|_{\xi=-1}
    = F_{\xi\eta}^g(-1,-1); \quad
    F_{\xi,\eta}^{gg}(-1,-1) = \left.\frac{d}{d\eta}F_{\xi}^{gg}(-1,\eta) \right|_{\xi=-1}
    = F_{\xi\eta}^g(-1,-1); \\
    & -F_{\xi}^{gg}(-1,1) + F_{\eta}^g(-1,1) = U_{BC}'(1); \quad
    -F_{\xi}^{g}(1,-1) + F_{\eta}^{gg}(1,-1) = U_{BC}'(-1); \\
    & (PF^g)_{\xi}(-1,\eta) = F_{\xi}^{gg}(-1,\eta); \quad
    (PF^g)_{\eta}(\xi,-1) = F_{\eta}^{gg}(\xi,-1); \\
    & PF^g(-1,\eta) = F^g(-1,\eta); \quad PF^g(\xi,-1) = F^g(\xi,-1); \quad
    PF^g(-\eta,\eta) = U_{BC}(\eta).
    \end{split}\right.
  \end{equation}\endgroup
\end{proof}


\subsection{Application to Solving PDE Problems}
\label{sec_25}

We combine the constructions from previous sections for enforcing boundary conditions,
together with the extreme learning machine (ELM)
technique, to solve PDE problems.
The approximation process is summarized below, and
we refer to e.g.~\cite{DongL2021,DongY2022rm,NiD2023}
for detailed expositions of the ELM method. Consider the
boundary value problem on the general triangular domain $\Omega$ from Figure~\ref{fg_1}(a),
\begin{subequations}\label{eq_n107}\small
  \begin{align}
    & \mathcal L u(\mbs x) + \mathcal N(u) = f(\mbs x), \label{eq_n107a} \\
    & \left.\mathcal B u(\mbs x)\right|_{\mbs x\in\partial\Omega} = f_b(\mbs x), \label{eq_n107b}
  \end{align}
\end{subequations}
where $u(\mbs x)$ is the unknown field to be solved for, $\mathcal L$ is a 2nd-order
linear differential operator, $\mathcal N$ is a nonlinear operator (of order at most one),
and $f$ and $f_b$ are prescribed source
terms over the domain and on the boundary, respectively.
$\mathcal B$ is a linear boundary operator, with \eqref{eq_n107b}
representing the Dirichlet, Neumann, or Robin conditions on
the triangular domain boundary. We assume that the boundary curves $\mbs x_{AB}(\xi)$,
$\mbs x_{BC}(\xi)$ and $\mbs x_{AC}(\xi)$ are sufficiently smooth to accommodate
the differential operators in the PDE and the boundary conditions.

We leverage the domain map in Section \ref{sec_21} and
the formulations from previous sections that satisfy the
boundary conditions \eqref{eq_n107b} exactly,
\begingroup\small\begin{equation}\label{eq_e108}
  V(\xi,\eta) = g(\xi,\eta) - Pg(\xi,\eta) + PF^g(\xi,\eta)
  = g(\xi,\eta) - Pg(\xi,\eta) + PF^{gb}(\xi,\eta) + PF^{a}(\xi,\eta),
\end{equation}\endgroup
where $PF^g$ has been split into two components: (i) $PF^{gb}$ has a linear dependence on
the free function $g$, and (ii) $PF^a$ does not depend on $g$ (data only).
Note that $PF^{gb}(\xi,\eta)=0$ if the domain involves only Dirichlet boundaries.
Equation \eqref{eq_n107a} is then transformed into
\begingroup\small\begin{align}\label{eq_n109}
  & \mathcal L V^{gb}(\xi,\eta) + \mathcal N(V(\xi,\eta)) = R(\xi,\eta),
\end{align}\endgroup
where 
\begingroup\small\begin{align}
  &V^{gb}(\xi,\eta) = g(\xi,\eta) - Pg(\xi,\eta) + PF^{gb}(\xi,\eta), \label{eq_n110a} \quad
  R(\xi,\eta) = f(\mbs x(\xi,\eta)) - \mathcal L [PF^a(\xi,\eta)].
\end{align}\endgroup

We represent the free function $g(\xi,\eta)$ by an ELM-type randomized feedforward NN
with a single hidden layer, in which the hidden-layer parameters are randomly assigned
and fixed and only the output-layer coefficients are trained.
The NN logic of the output layer leads to the relation,
\begingroup\small\begin{align}\label{eq_n111}
  g(\xi,\eta) = \sum_{j=1}^M \beta_j \varpi_j(\xi,\eta) = \bm\Phi(\xi,\eta)\bm\beta,
\end{align}\endgroup
where $M$ is the hidden-layer width,
$\bm\Phi(\xi,\eta)=(\varpi_1,\dots,\varpi_M)$ denotes the set of output fields of the
hidden layer,
and $\bm\beta=(\beta_1,\dots,\beta_M)^T$ is the vector of the output-layer weights,
which are the trainable parameters of the network.
Note that we have imposed the usual requirement that the output
layer should contain no activation
and have zero bias (see~\cite{DongL2021} for details). 

We determine the trainable parameters $\bm\beta$ by approximating equation~\eqref{eq_n109}
based on the physics informed approach. 
We choose a set of $Q_c$ collocation points from the interior of the standard
triangle, $(\xi_i,\eta_i)\in\Omega_{st}$ for $1\leqslant i\leqslant Q_c$, and
enforce \eqref{eq_n109} on these points. This gives rise to
the following algebraic system,
\begingroup\small\begin{align}\label{eq_e112}
  & \big[\mathcal L\bm\Psi(\xi_i,\eta_i)\big]\bm\beta
  + \mathcal N\big(\bm\Psi(\xi_i,\eta_i)\bm\beta+PF^a(\xi_i,\eta_i)\big)
  = R(\xi_i,\eta_i), \quad 1\leqslant i\leqslant Q_c,
\end{align}\endgroup
where we have used \eqref{eq_n110a} and \eqref{eq_n111}, and
\begingroup\small\begin{align}
& \bm\Psi(\xi,\eta) = \bm\Phi(\xi,\eta) - P\bm\Phi(\xi,\eta) + PF^{\bm\Phi b}(\xi,\eta),
\quad
PF^{gb}(\xi,\eta) = PF^{\bm\Phi b}(\xi,\eta)\bm\beta.
\end{align}\endgroup
Equation \eqref{eq_e112} represents a system of nonlinear algebraic equations about $\bm\beta$,
with $Q_c$ equations and $M$ unknowns. We pursue a least squares solution to
this system, and compute such a solution by the nonlinear least squares (or Gauss-Newton)
method~\cite{Bjorck1996}, specifically using
the NLLSQ-perturb (nonlinear least squares with perturbations) algorithm
from~\cite{DongL2021}. The output-layer coefficients of the ELM network
are then set by this least squares solution, concluding the NN training.
The trained neural network encodes the free function $g(\xi,\eta)$, and
the solution to the PDE system~\eqref{eq_n107} is given by \eqref{eq_e108}.

\begin{remark}
  If the PDE \eqref{eq_n107a} is linear (with $\mathcal N(u)=0$),
  the resultant algebraic system \eqref{eq_e112} will be linear w.r.t.~$\bm\beta$.
  Then we compute the least squares solution to this system by
  the linear least squares method~\cite{Bjorck1996}.
\end{remark}

\begin{algorithm}[tb]\small
  \DontPrintSemicolon
  \SetKwInOut{Input}{input}\SetKwInOut{Output}{output}

  \Input{ $g(\xi,\eta)$, $g(\xi,-\xi)$, $g(-\eta,\eta)$,
    $g(-1,\eta)$, $g(-1,-\xi)$, $g(\xi, -1)$, $g(-\eta,-1)$, $g(- 1, - 1)$;
    $F(\xi,\eta)$, $F(\xi,-\xi)$, $F(-\eta,\eta)$,
    $F(-1,\eta)$, $F(-1,-\xi)$, $F(\xi, -1)$, $F(-\eta,-1)$, $F(- 1, - 1)$;
  }
  \Output{$V(\xi,\eta)$}
  \BlankLine
  calculate $H(\xi,\eta)$, $H(\xi,-\xi)$ and $H(-\eta,\eta)$ based on equation~\eqref{eq_e115}, i.e.
  \vspace{-5pt}\begingroup\small\begin{align*}
    H(\xi,\eta) =& \  \big[g(\xi,\eta)-g(\xi,-1) \big] - \big[g(-1,\eta)-g(-1,-1) \big], \\
    H(\xi,-\xi) =& \ \big[g(\xi,-\xi)-g(\xi,-1) \big] - \big[g(-1,-\xi)-g(-1,-1) \big], \\
    H(-\eta,\eta) =& \ \big[g(-\eta,\eta)-g(-\eta,-1) \big] - \big[g(-1,\eta)-g(-1,-1) \big].
  \end{align*}\endgroup\;
  \vspace{-15pt}calculate $A(\xi,\eta)=g(\xi,\eta)-Pg(\xi,\eta)$ based on equation
  \eqref{eq_e116}, i.e.
  \vspace{-5pt}\begin{align*}
    A(\xi,\eta) =& \ \big[\phi_0(\xi)H(\xi,\eta) - \phi_1(\eta)H(\xi,-\xi) \big]
    + \phi_1(\xi)\big[H(\xi,\eta) - H(-\eta,\eta) \big].
  \end{align*}\;
  \vspace{-15pt}calculate $B(\xi,\eta)=PF(\xi,\eta)$ based on equation~\eqref{eq_e27b}, i.e.
  \vspace{-5pt}\begin{align*}
    B(\xi,\eta) =& \  \big[F(\xi,-1)+F(-1,\eta)-F(-1,-1) \big] + 
    \phi_1(\eta)\big[F(\xi,-\xi)-F(\xi,-1) - F(-1,-\xi) \notag \\
      &\ + F(-1,-1) \big] +
    \phi_1(\xi)\big[F(-\eta,\eta) - F(-\eta,-1) - F(-1,\eta) + F(-1,-1) \big].
  \end{align*}\;
  \vspace{-15pt}calculate $V(\xi,\eta) = A(\xi,\eta) + B(\xi,\eta)$
  
  \caption{Computing the NN solution (DBC on all boundaries)
  }
  \label{alg_1}
\end{algorithm}

\begin{remark}\label{rem_28}
  After the NN is trained, when evaluating the TFC expression
  on boundary test points to compute the solution
  field $V(\xi,\eta)$ (e.g.~in order to assess the boundary-condition errors),
  we observe that ``catastrophic cancellation'' (a.k.a.~cancellation error)
  due to subtraction of nearly equal real numbers in $(g-Pg)$ can be very significant.
  Let's use the TFC expression with all DBCs as an example to illustrate this point;
  see equations \eqref{eq_17}--\eqref{eq_e27b}.
  Numerical experiments show that catastrophic cancellation can lead to prominent numerical
  boundary-condition errors, which, while
  small (e.g.~on the order $10^{-8}$ at isolated points for Dirichlet boundaries),
  are far above the level of
  machine precision. A similar issue has been encountered in~\cite{DongZ2026}.
  The large BC errors are primarily caused by
  those terms involving the free function $g(\xi,\eta)$
  in the TFC expression. For example, the expression $[g(\xi,\eta)-Pg(\xi,\eta)]$
  should be zero mathematically on the test points from Dirichlet boundaries.
  When evaluated directly based on \eqref{eq_17}--\eqref{eq_e27a},
  however, it results in numerical errors on the order
  $\mathcal O(\text{max}_{\substack{(\xi,\eta)\in\Omega_{st}}}|g(\xi,\eta)|\cdot\epsilon)$,
  where $\epsilon$ denotes the unit roundoff error ($\sim 10^{-16}$ in double precision),
  and the numerical magnitude of $g(\xi,\eta)$ can be large when the problem is
  severely ill-conditioned. Fortunately, this cancellation error can be effectively mitigated
  by re-structuring the computations appropriately.
  We observe that computing $(g-Pg)$ using
  the following equivalent forms effectively eliminates the cancellation error,
  leading to numerical BC errors at the machine precision, irrespective
  of the numerical magnitude of $g(\xi,\eta)$.
  Noting that $Pg = (P_1\oplus P_2)g = P_1g + P_2g - P_1P_2g$, where the operators
  $P_1$ and $P_2$ are defined in \eqref{eq_6a}--\eqref{eq_6b}, we have
  \begingroup\small\begin{align}
    & g - Pg = g - P_1g - P_2g + P_1P_2g = (I-P_1)(I-P_2)g,
  \end{align}\endgroup
  where $I$ denotes the identity operator. Let
  \begingroup\small\begin{align}\label{eq_e115}
    H(\xi,\eta) = (I-P_2)g(\xi,\eta) = g(\xi,\eta)-P_2g(\xi,\eta)
    = \big[g(\xi,\eta) - g(\xi,-1) \big] - \big[g(-1,\eta) - g(-1,-1) \big].
  \end{align}\endgroup
  It follows that,
  \begingroup\small\begin{align}
    g(\xi,\eta)-Pg(\xi,\eta) &= (I-P_1)H(\xi,\eta)
    = H(\xi,\eta) - \phi_1(\eta)H(\xi,-\xi) - \phi_1(\xi)H(-\eta,\eta) \notag \\
    &= \big[\phi_0(\xi)+\phi_1(\xi) \big]H(\xi,\eta)
    - \phi_1(\eta)H(\xi,-\xi) - \phi_1(\xi)H(-\eta,\eta) \notag \\
    &= \big[\phi_0(\xi)H(\xi,\eta) - \phi_1(\eta)H(\xi,-\xi) \big]
    + \phi_1(\xi)\big[H(\xi,\eta) - H(-\eta,\eta) \big]. \label{eq_e116}
  \end{align}\endgroup
  The grouped expressions in \eqref{eq_e115} and \eqref{eq_e116} are crucial.
  Algorithm~\ref{alg_1} summarizes the procedure for 
  computing the solution $V(\xi,\eta)$ using the forms \eqref{eq_e115} and \eqref{eq_e116}.
  For other boundary-condition configurations, re-structured expressions
  for computing $V(\xi,\eta)$ can be formulated analogously. 
  
\end{remark}


\section{Numerical Examples}
\label{sec_tests}

We next test the BC enforcement method 
from  the previous section using several
boundary value  or initial/boundary value problems
with complex domain geometries.
The tested domains consist of curved boundaries that are smooth or
piecewise smooth, having corners
or even  cusps. These geometries can be treated as
general triangular domains, by choosing corners or smooth points from the boundaries
as the triangle vertices.
The test problems involve linear and nonlinear stationary PDEs, and also dynamic PDEs
on spatial domains that evolve over time.

We employ ELM-type randomized feed-forward NNs with a single hidden layer in
all the simulations, with an architecture $[2, M, 1]$ and the Gaussian
activation function $\sigma(x)=e^{-x^2}$. The input layer consists of
two nodes representing $(\xi, \eta)$, and the output node represents
the free function $g(\xi,\eta)$.
The weight/bias coefficients in the $M$ hidden nodes are assigned
to random values drawn from a uniform distribution from $[-R_m,R_m]$,
where the scope constant $R_m$ is determined based on the algorithm
from~\cite{DongY2022rm}.
For the NN training we employ $Q_c=\frac12(Q-2)(Q-1)$
($Q$ denoting a prescribed integer) interior barycentric
points from $\Omega_{st}$ specified by
\begingroup\small\begin{equation}
  (\text{training collocation points})\ 
  \left\{\
  \bm\xi_{ijk}=\frac{1}{Q}(i\bm\xi_A+j\bm\xi_B+k\bm\xi_C)\ |\ 
  i+j+k=Q,\ i,j,k\geqslant 1 
  \ \right\},
\end{equation}\endgroup
where $\bm\xi_A=(-1,-1)$, $\bm\xi_B=(1,-1)$ and $\bm\xi_C=(-1,1)$.

We report the maximum and root-mean-squares (rms) errors of the NN solution
over the domain and the maximum/rms boundary-condition errors on
the individual boundaries. The solution errors are defined by
\begingroup\small\begin{align}
  & e_{max}(u) = \max_{\substack{1\leqslant i\leqslant N_v}}\left|u(\mbs x_i)-u_{ex}(\mbs x_i) \right|,
  \quad
  e_{rms}(u) = \sqrt{\frac{1}{N_v}\sum_{i=1}^{N_v}\left| u(\mbs x_i)-u_{ex}(\mbs x_i) \right|^2}
\end{align}\endgroup
where $u(\mbs x)$ and $u_{ex}(\mbs x)$ denote the NN solution and
the exact solution, respectively, and $N_v$ is the number of test points.
The points $\mbs x_i$ ($1\leqslant i\leqslant N_v$)
denote the image of the following barycentric test points from $\Omega_{st}$,
\begingroup\small\begin{equation}
  (\text{domain test points})\ 
  \left\{\
  \bm\xi_{ijk}=\frac{1}{Q_v}(i\bm\xi_A+j\bm\xi_B+k\bm\xi_C)\ |\ 
  i+j+k=Q_v,\ i,j,k\geqslant 0 
  \ \right\},
\end{equation}\endgroup
where $Q_v$ is a prescribed integer parameter and $N_v=\frac12(Q_v+1)(Q_v+2)$.
Unless otherwise specified, we employ $Q_v=100$ for computing the domain solution
errors $e_{max}$ and $e_{rms}$ in the simulations.

The boundary-condition errors are defined by (see equation \eqref{eq_n107b}),
\begingroup\small\begin{equation}
  \varepsilon_{max}(u) = \max_{\substack{1\leqslant i\leqslant Q_b}}
  \left|\mathcal Bu(\mbs x_i)- f_b(\mbs x_i) \right|, \quad
  \varepsilon_{rms}(u) = \sqrt{\frac{1}{Q_b}\sum_{i=1}^{Q_b}\left| \mathcal B u(\mbs x_i)-f_b(\mbs x_i) \right|^2},
\end{equation}\endgroup
where $\mbs x_i$ ($1\leqslant i\leqslant Q_b$) denote the image of $Q_b$ uniform
grid points on the individual edges of $\Omega_{st}$. For Neumann and Robin boundaries,
we exclude the end points of the edge from the set of boundary test points,
due to the fact that in a number of cases the end point corresponds to a smooth vertex
of the triangle, where the Jacobian matrix is singular.
By taking the boundary test points from the individual edges, we can define
$\varepsilon_{max}$ and $\varepsilon_{rms}$ on $\overline{AB}$, $\overline{BC}$
and $\overline{AC}$.
Unless otherwise specified, we employ $Q_b=200$ test points
for computing the boundary-condition errors on each boundary.

Our method is implemented in Python, employing the Tensorflow library
for the NN logic and automatic differentiations.
A fixed random seed value ``0'' has been used for the random number generators
in all the simulations below.
The values for the other simulation parameters, including $R_m$ (scope constant for ELM random
coefficients), $Q$ (training collocation-point parameter), $M$ (number of hidden nodes),
$p_w$ (Winslow parameter) and $Q_w$ (parameter for Winslow collocation points),
will be provided later when presenting individual simulation results.

\subsection{Variable-Coefficient Poisson Equation}
\label{sec_31}

\begin{figure}
  \centerline{
    \includegraphics[width=1.1in]{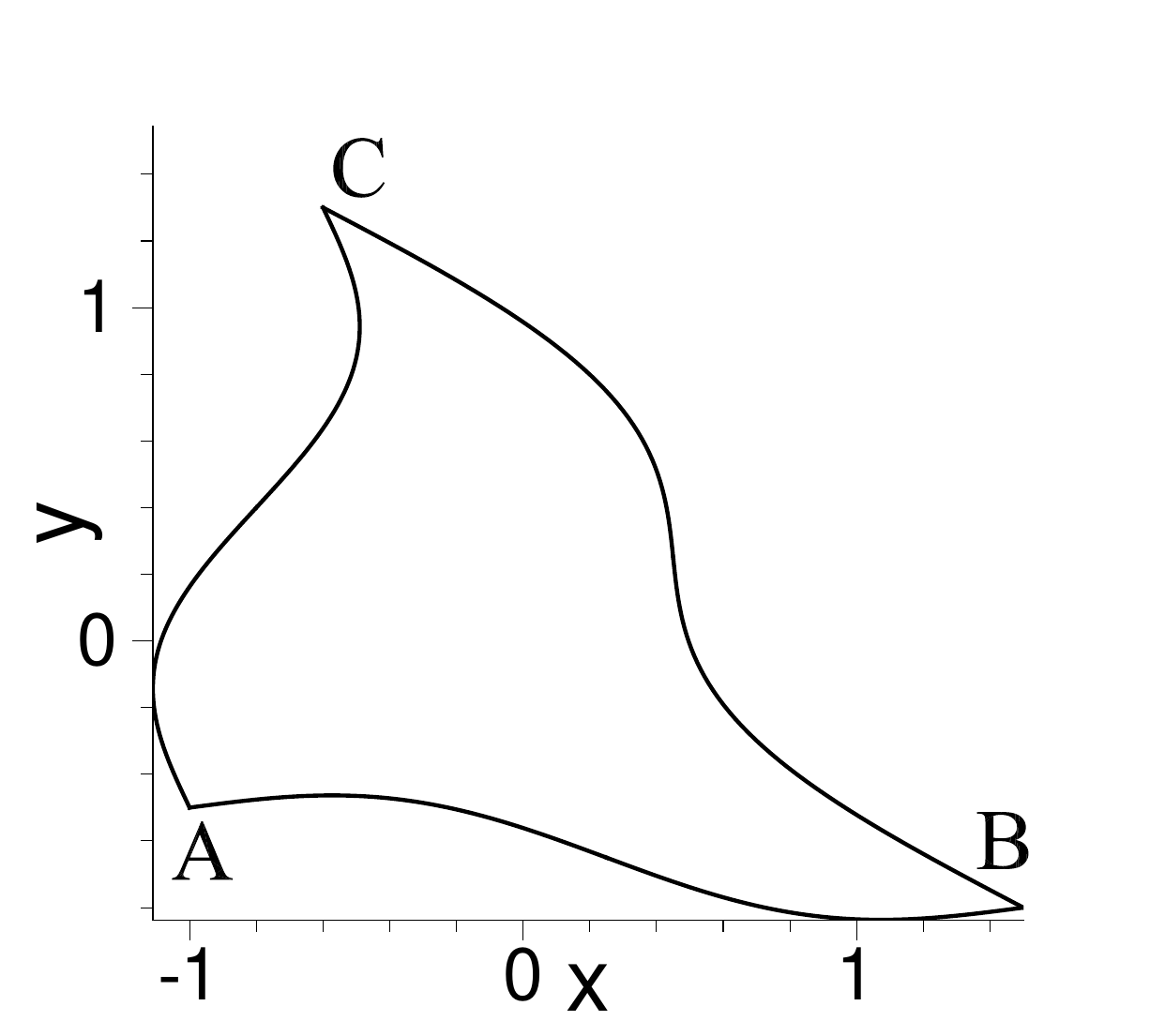}(a)
    \includegraphics[width=1.1in]{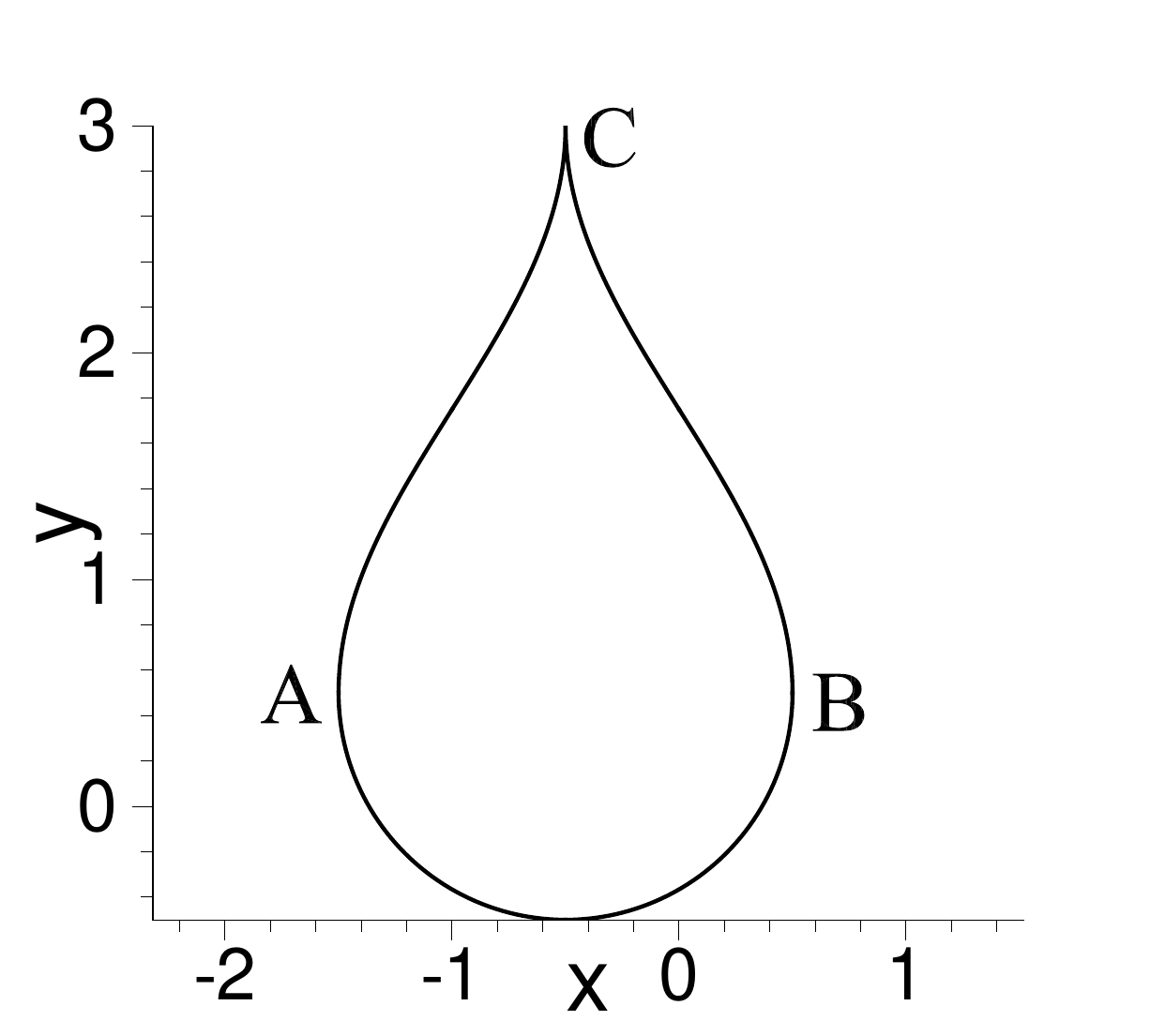}(b)
    \includegraphics[width=1.1in]{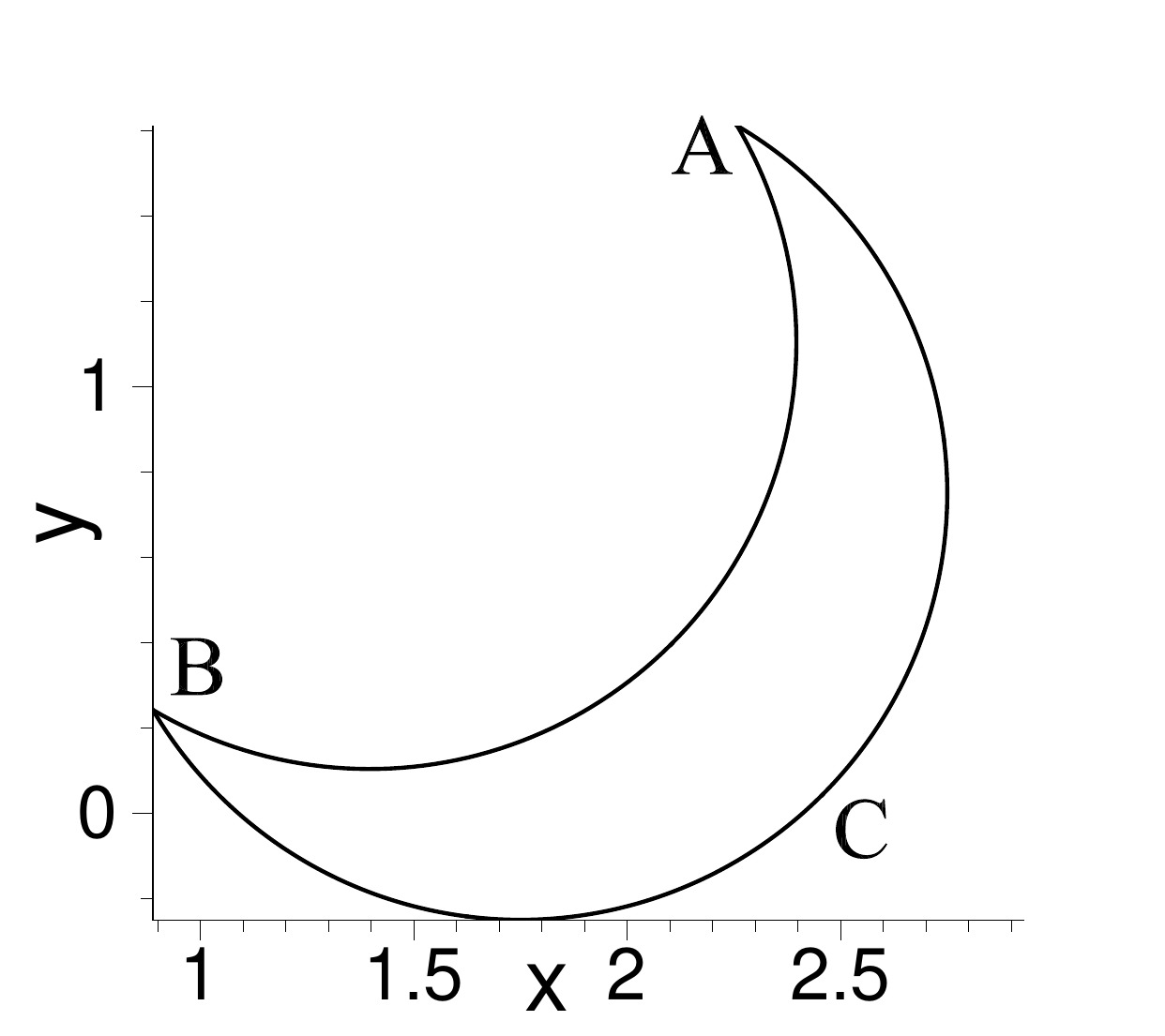}(c)
    \includegraphics[width=1.1in]{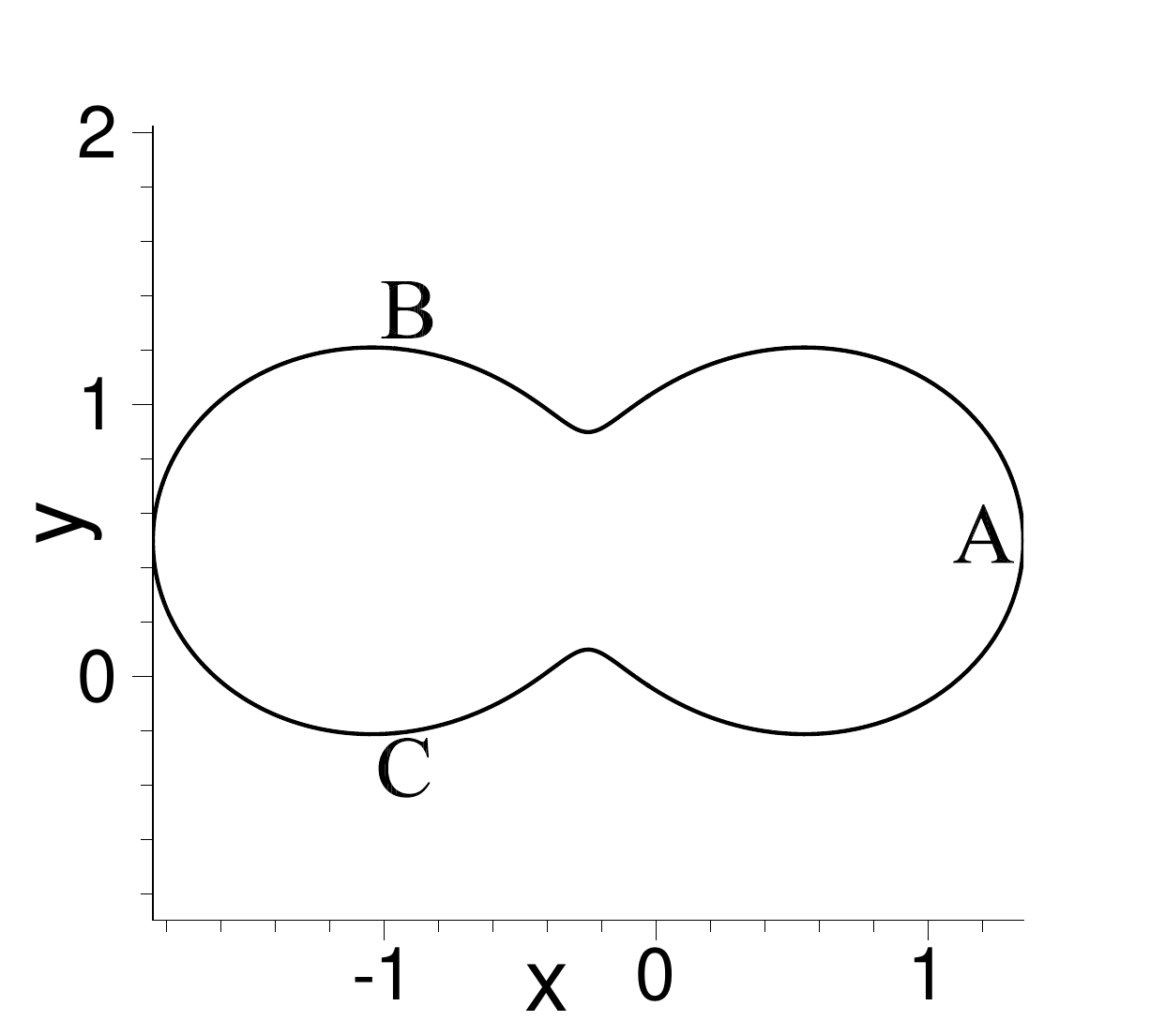}(d)
    \includegraphics[width=1.1in]{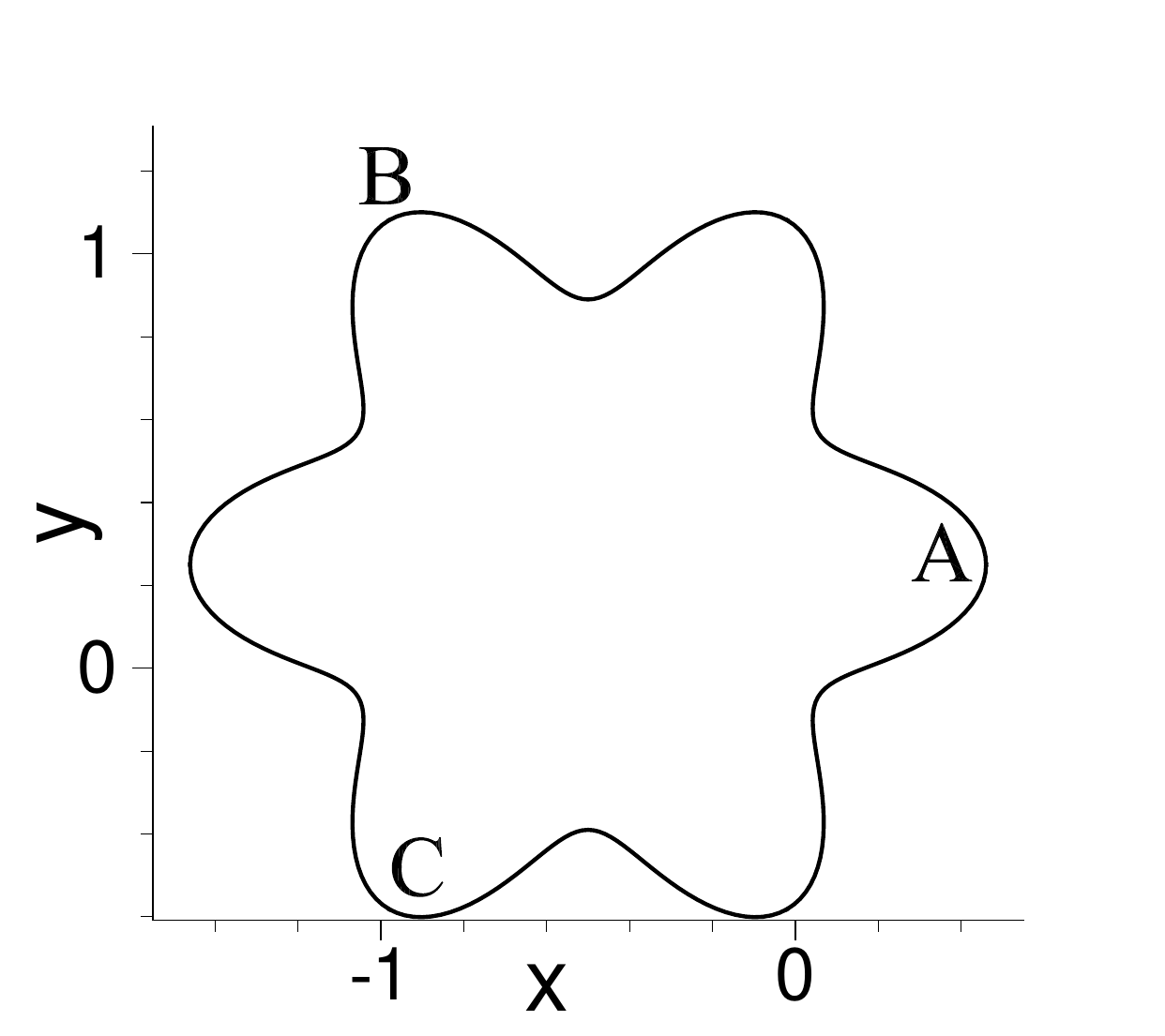}(e)
  }
  \centerline{
    \includegraphics[width=1.1in]{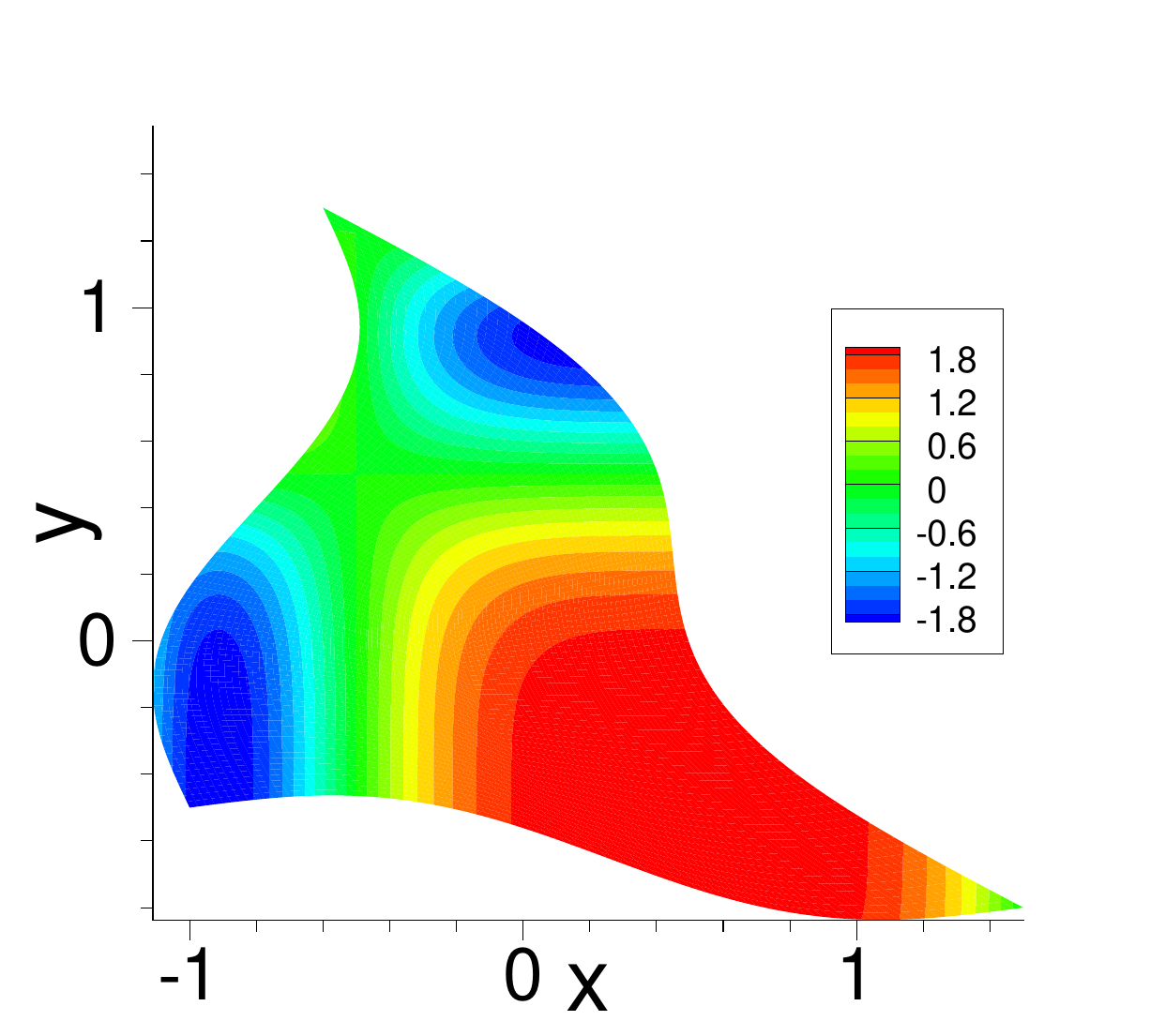}(f)
    \includegraphics[width=1.1in]{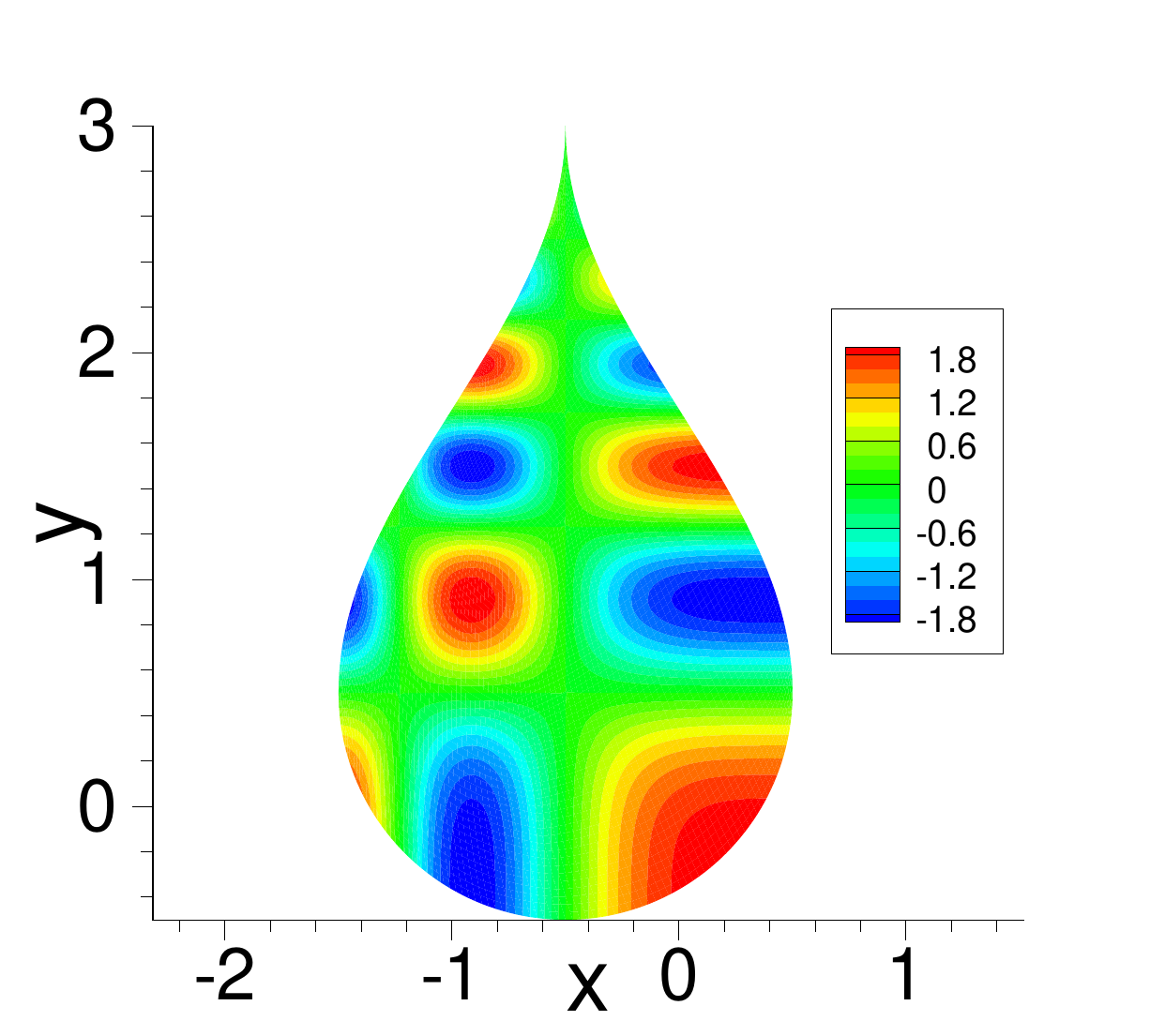}(g)
    \includegraphics[width=1.1in]{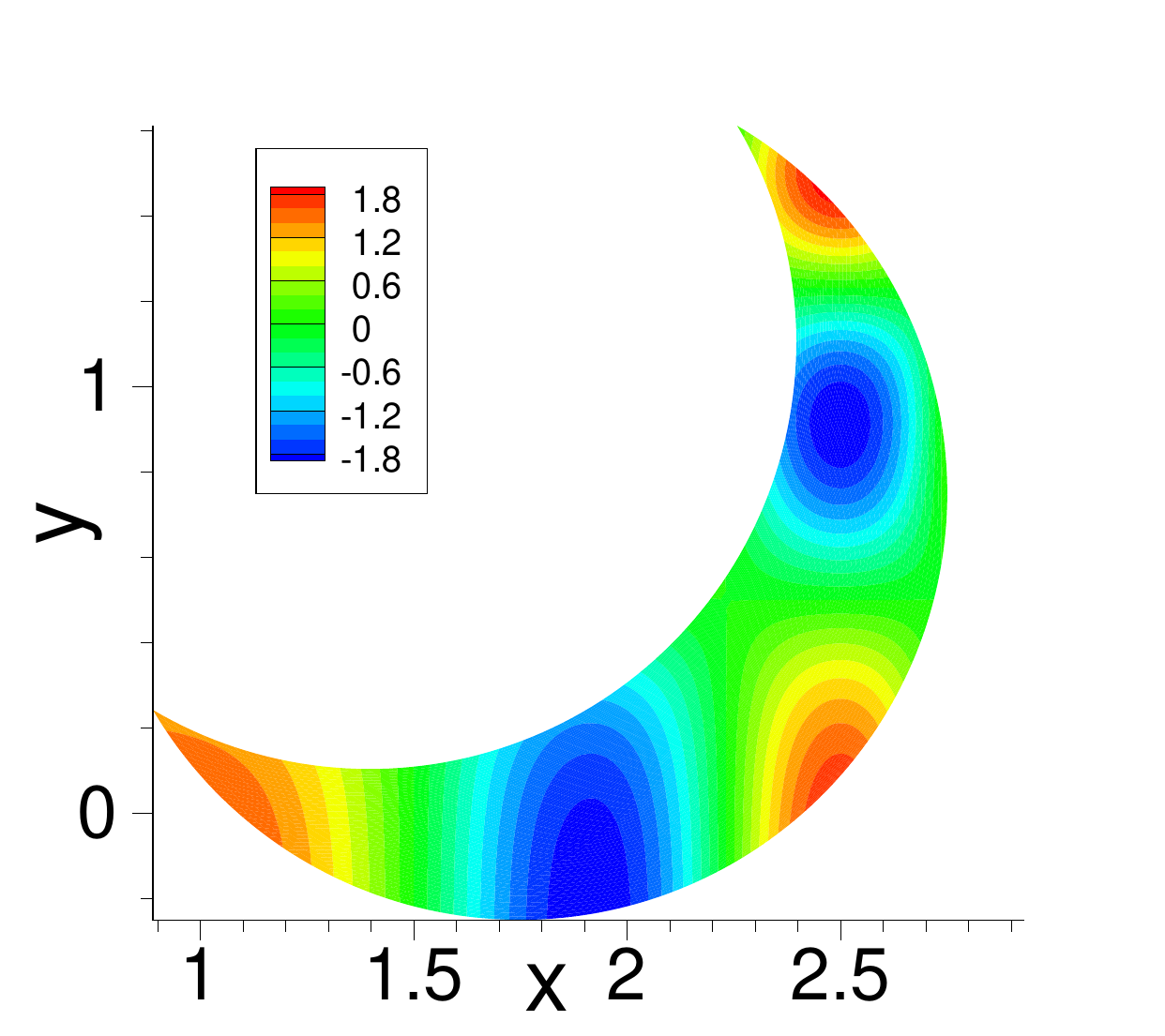}(h)
    \includegraphics[width=1.1in]{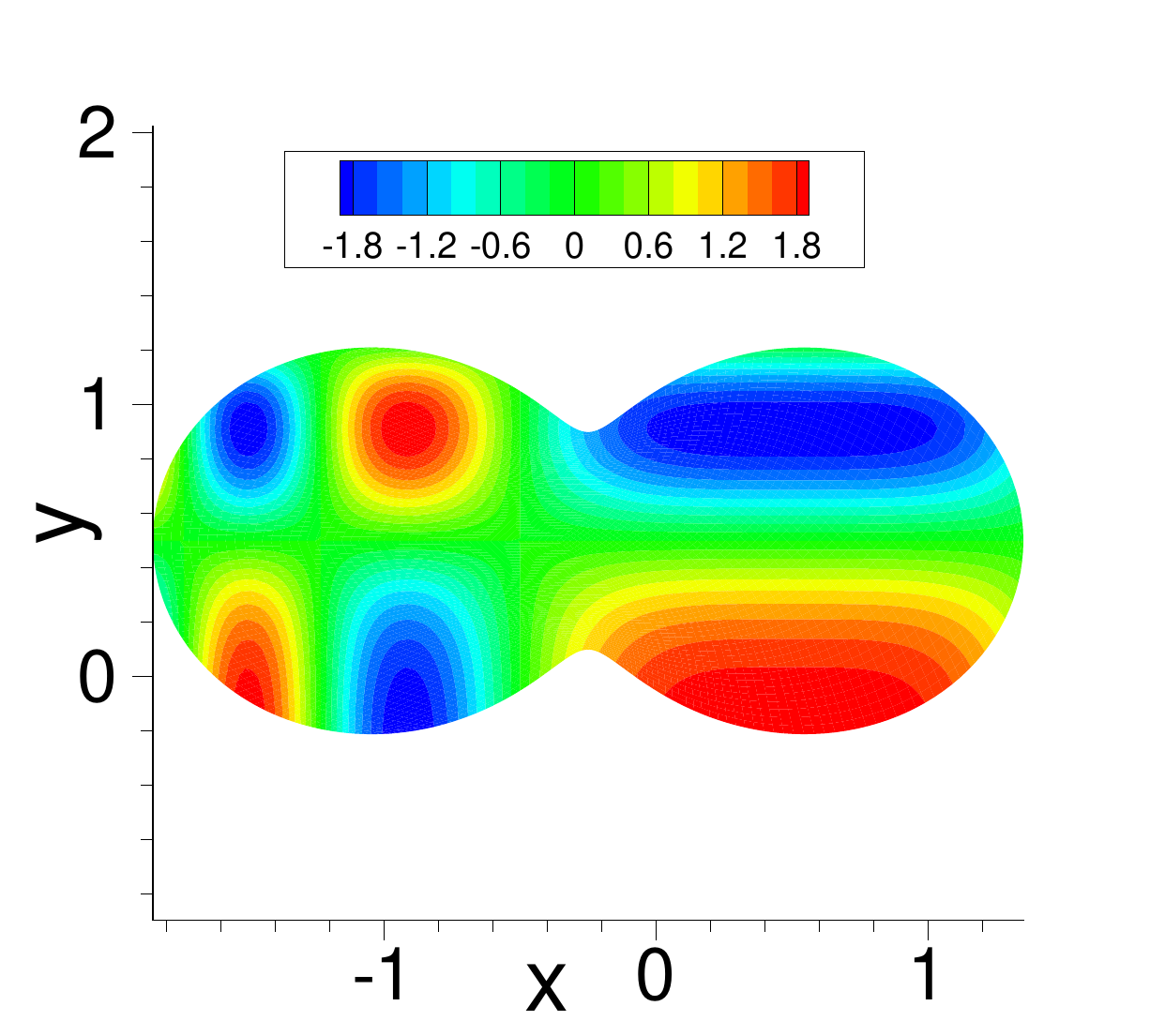}(i)
    \includegraphics[width=1.1in]{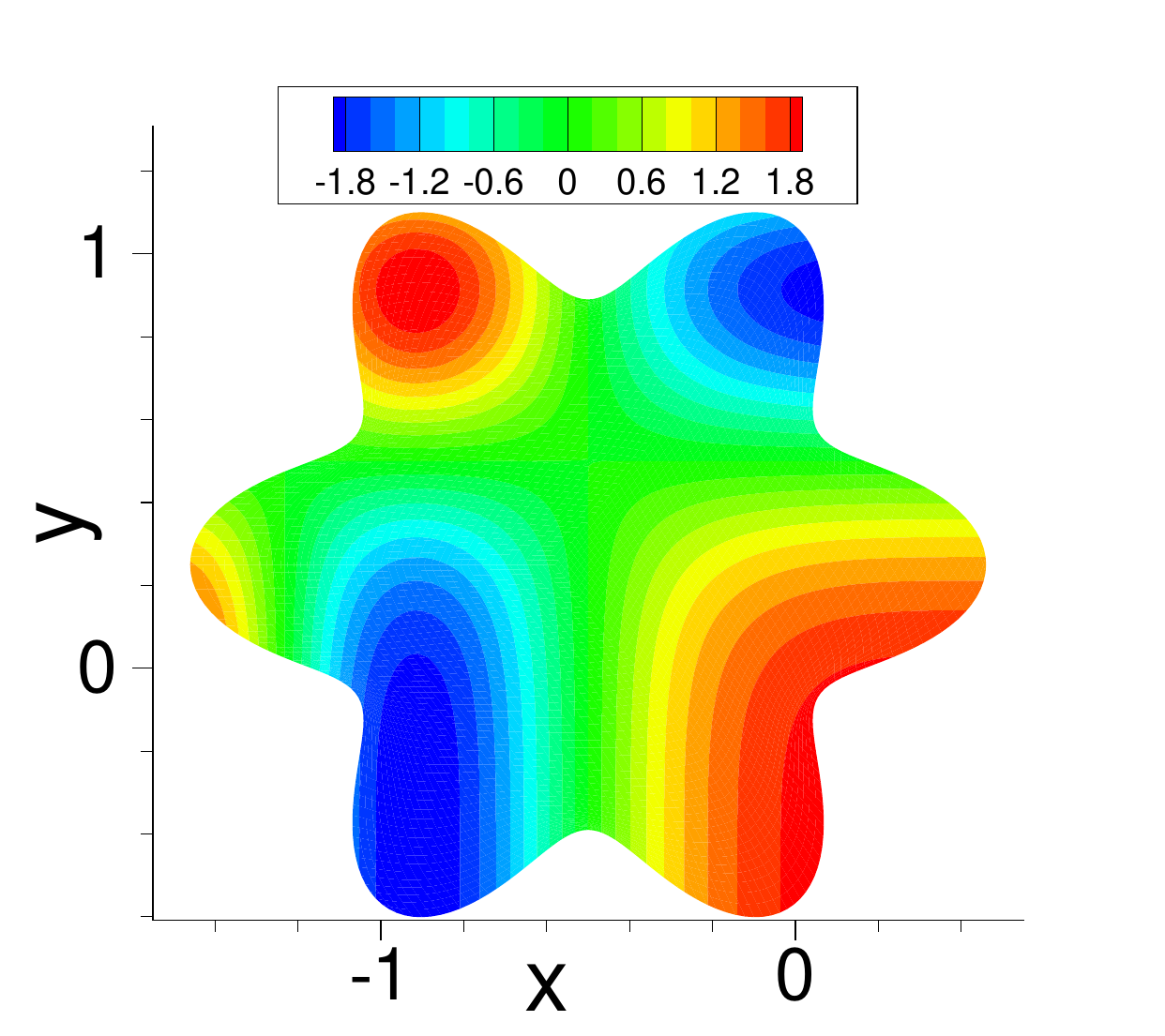}(j)
  }
  \caption{Variable-coefficient Poisson equation: Domain geometries (top row)
    and distributions of the exact solution (bottom row) for five domains:
    (a,f) domain \#1, (b,g) domain \#2, (c,h) domain \#3, (d,i) domain \#4,
    (e,j) domain \#5.
  }
  \label{fg_2}
\end{figure}

We consider the following boundary value problem with the variable-coefficient
Poisson equation on the five domains as shown in Figure~\ref{fg_2} (top row),
\begin{subequations}\small
\begin{align}
  & \nabla\cdot\left[a(x,y)\nabla u \right] = f(x,y), \\
  & \left.\mathcal B u\right|_{(x,y)\in\partial\Omega} = f_b(x,y),
\end{align}
\end{subequations}
where $a(x,y)=2+0.3\sin(x)\cos(y)$, 
$u(x,y)$ is the function to be solved for, $f$ and $f_b$ are prescribed
source terms, and $\mathcal B u$ denotes the boundary conditions of
Dirichlet, Neumann, or Robin types.
The specific geometric parameters of these domains are provided
in an appendix (Section \ref{sec_52}).
We choose the source terms $f$ and $f_b$ appropriately such that
this problem has the following exact solution for different boundary conditions,
\begingroup\small\begin{align}
  & u(x,y) = 2\cos\left[\frac{\pi}{2}\left(x-\frac12 \right)^2 \right]
  \cos\left[\frac{\pi}{2}\left(y+\frac12 \right)^2 \right].
\end{align}\endgroup
Figures \ref{fg_2}(f)--(j) illustrate this solution on the five domains.


\begin{figure}
  \centerline{
    \includegraphics[width=1.1in]{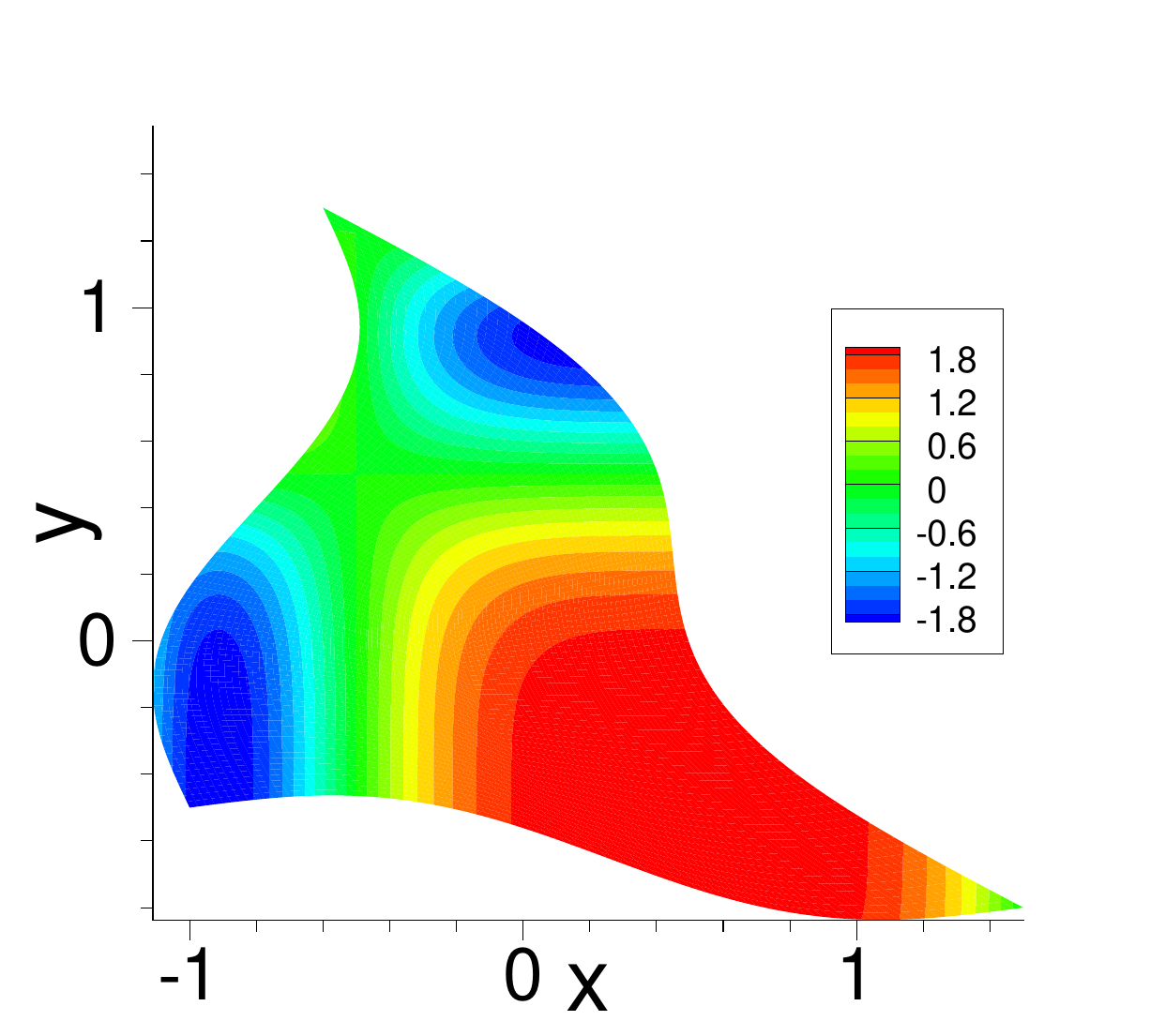}(a)
    \includegraphics[width=1.1in]{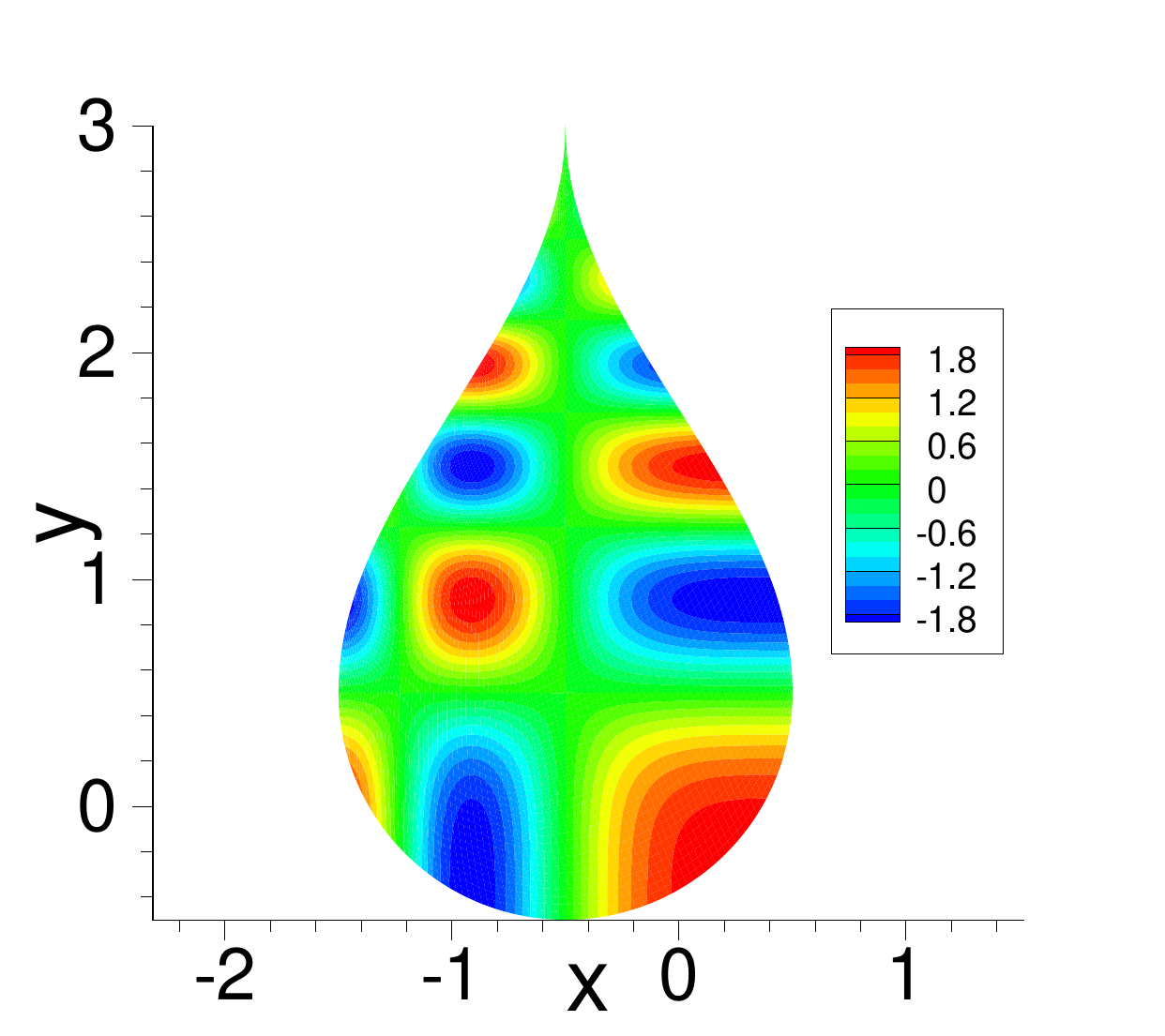}(b)
    \includegraphics[width=1.1in]{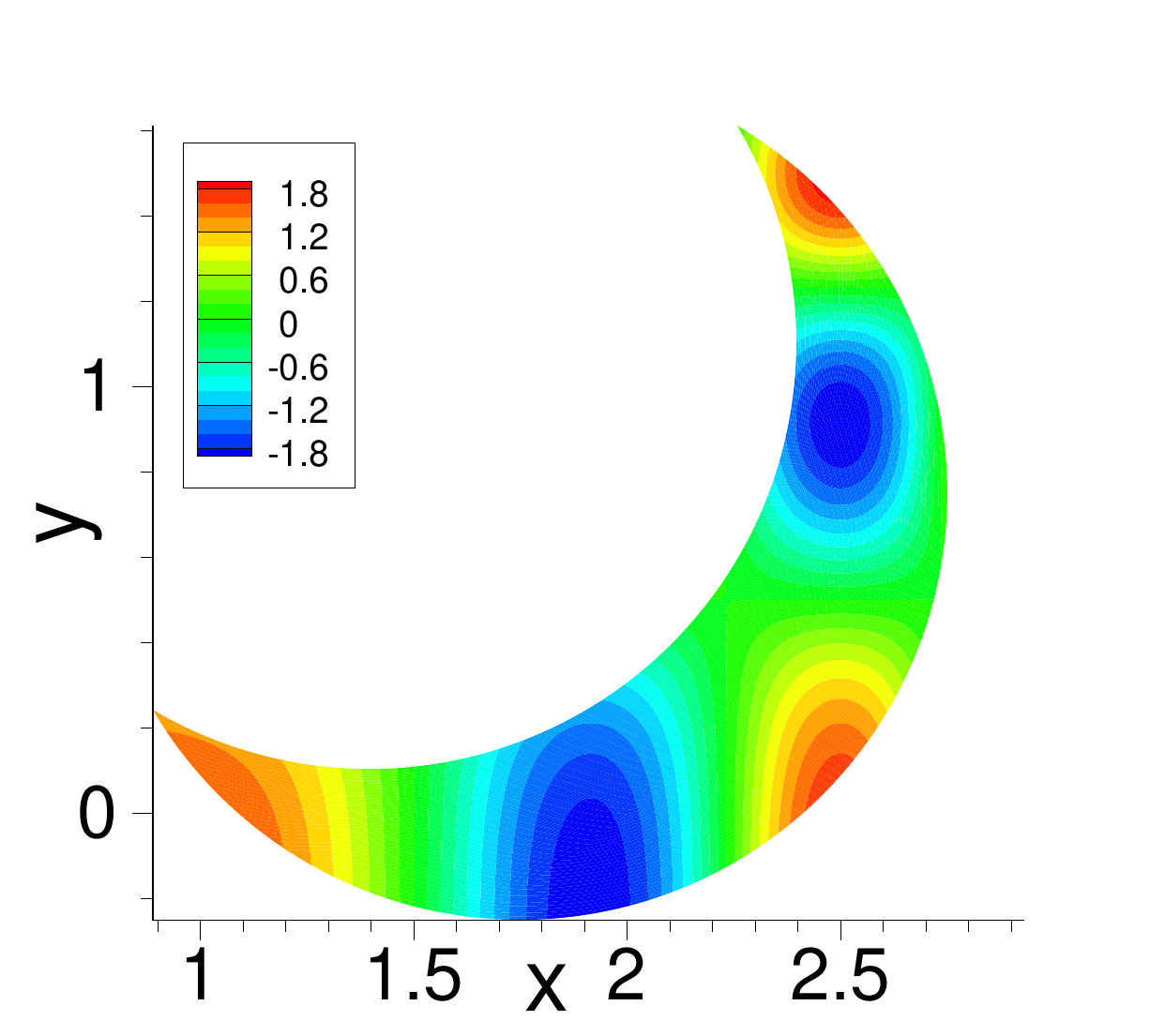}(c)
    \includegraphics[width=1.1in]{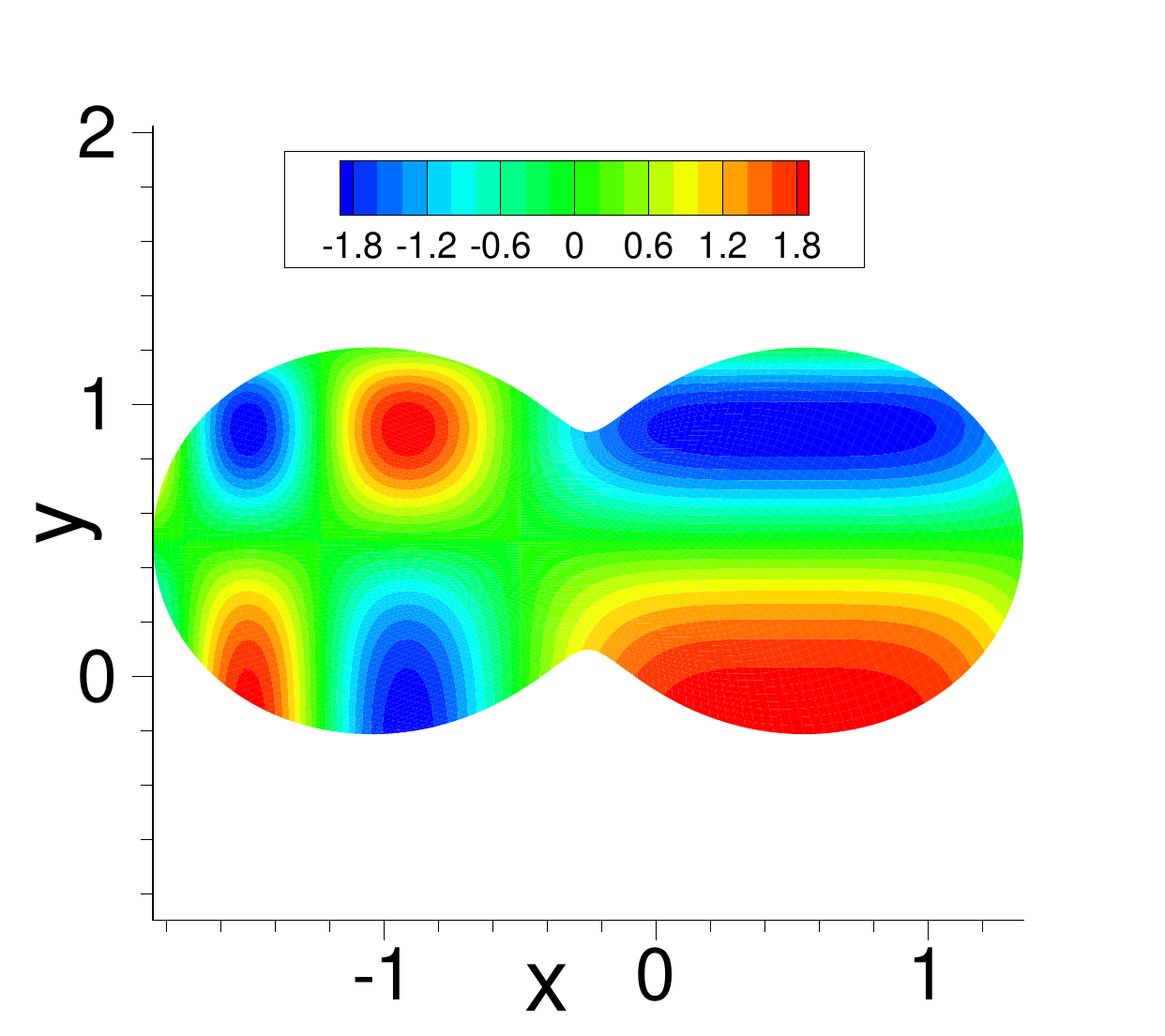}(d)
    \includegraphics[width=1.1in]{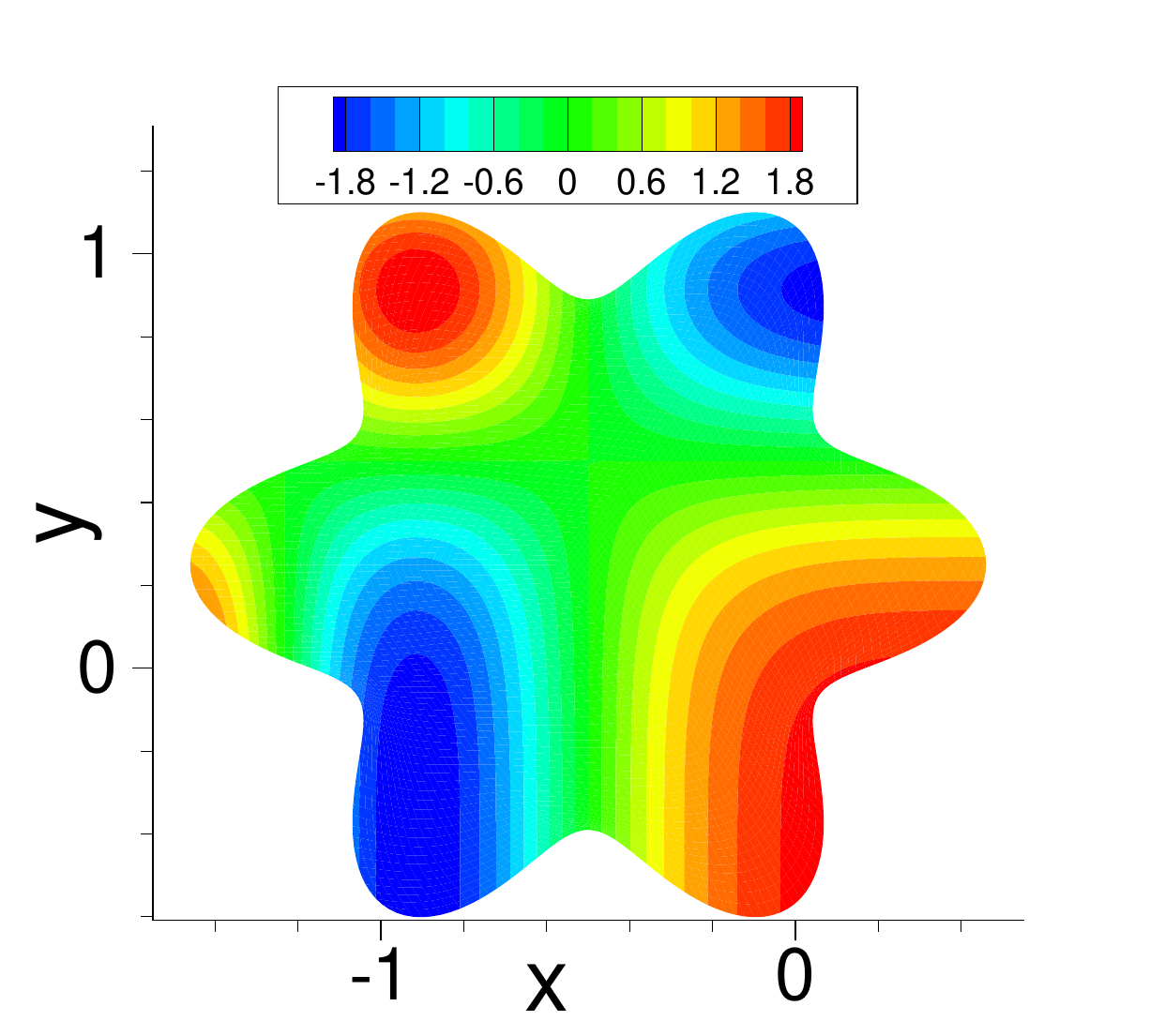}(e)
  }
  \centerline{
    \includegraphics[width=1.1in]{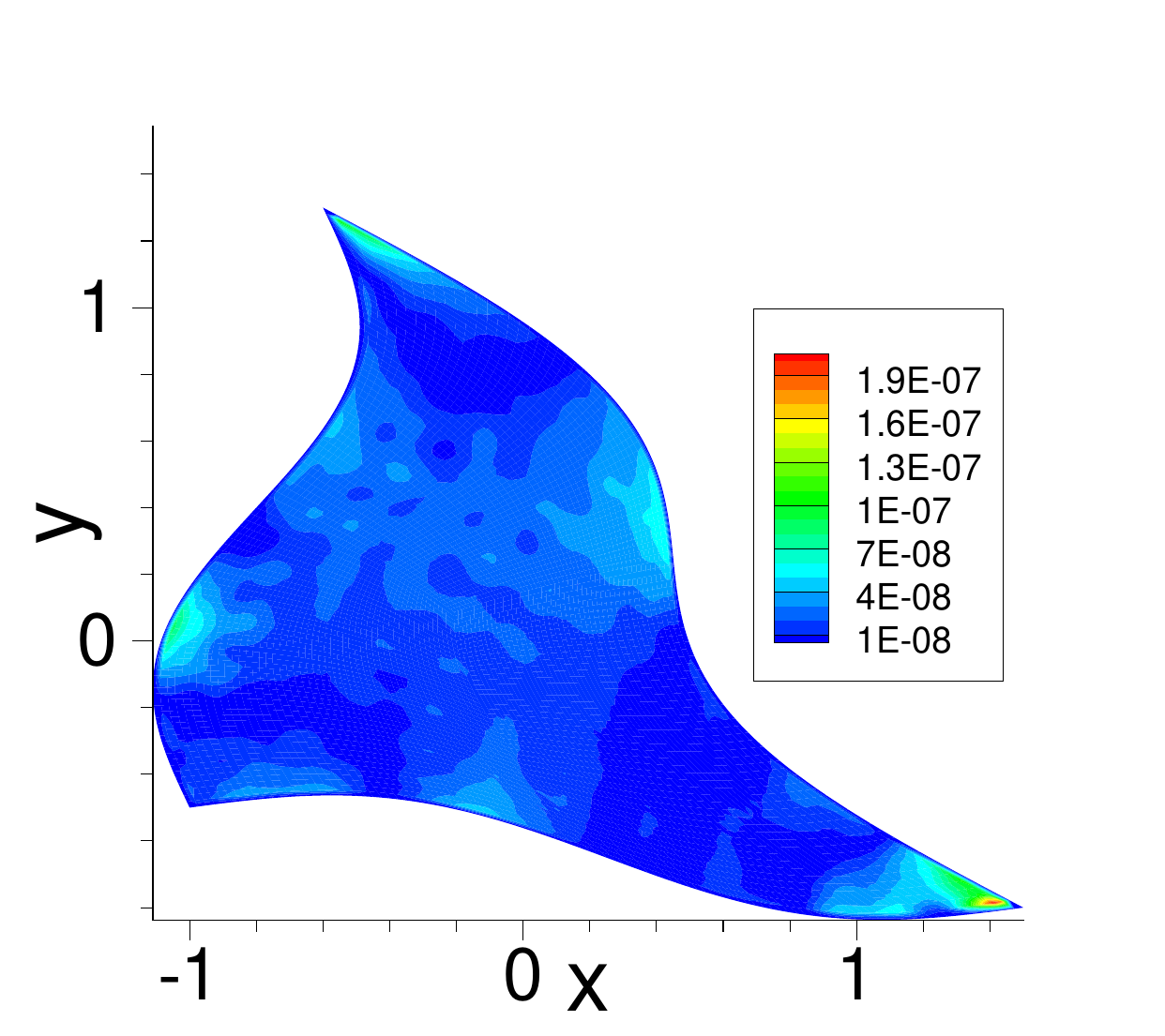}(f)
    \includegraphics[width=1.1in]{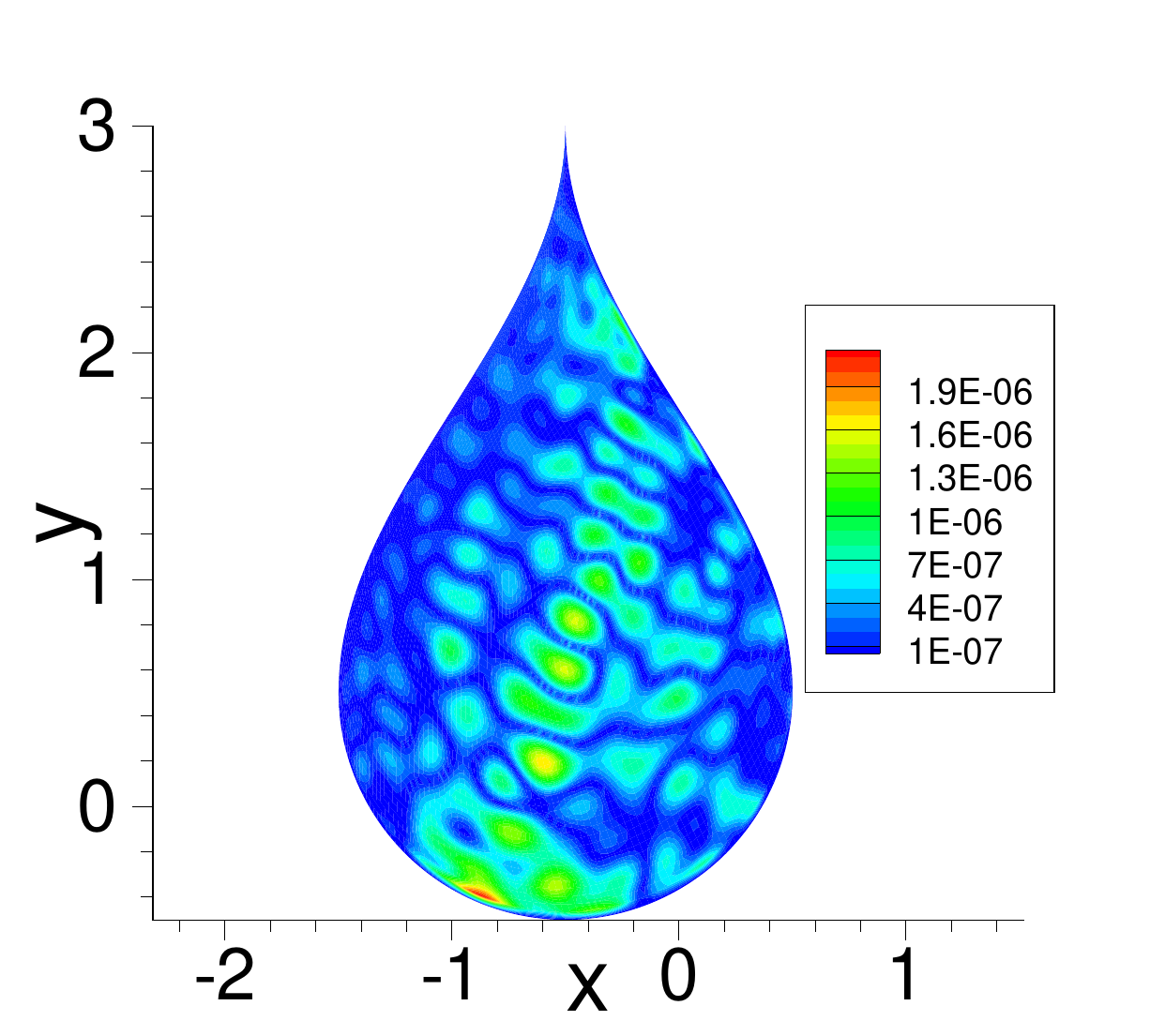}(g)
    \includegraphics[width=1.1in]{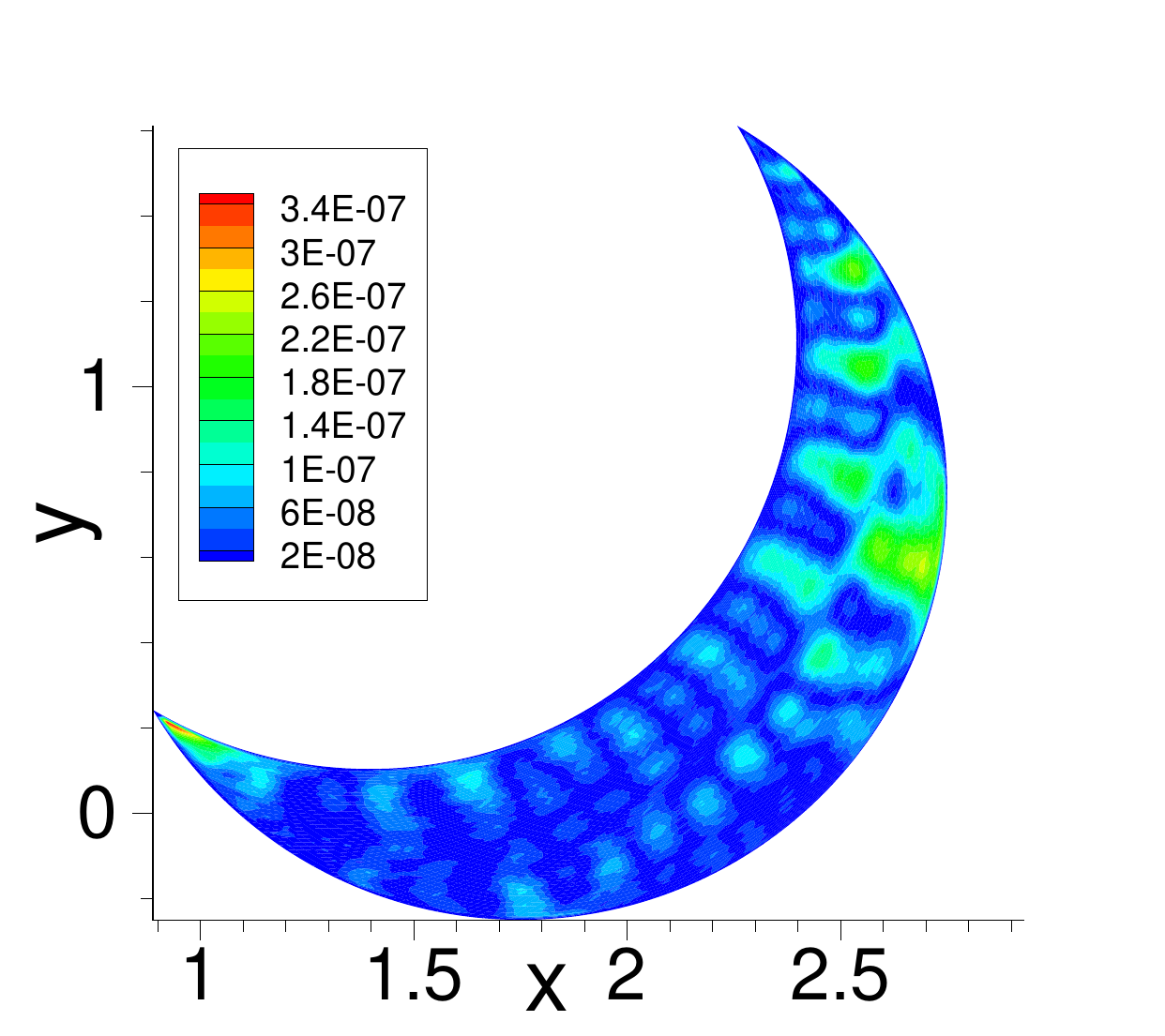}(h)
    \includegraphics[width=1.1in]{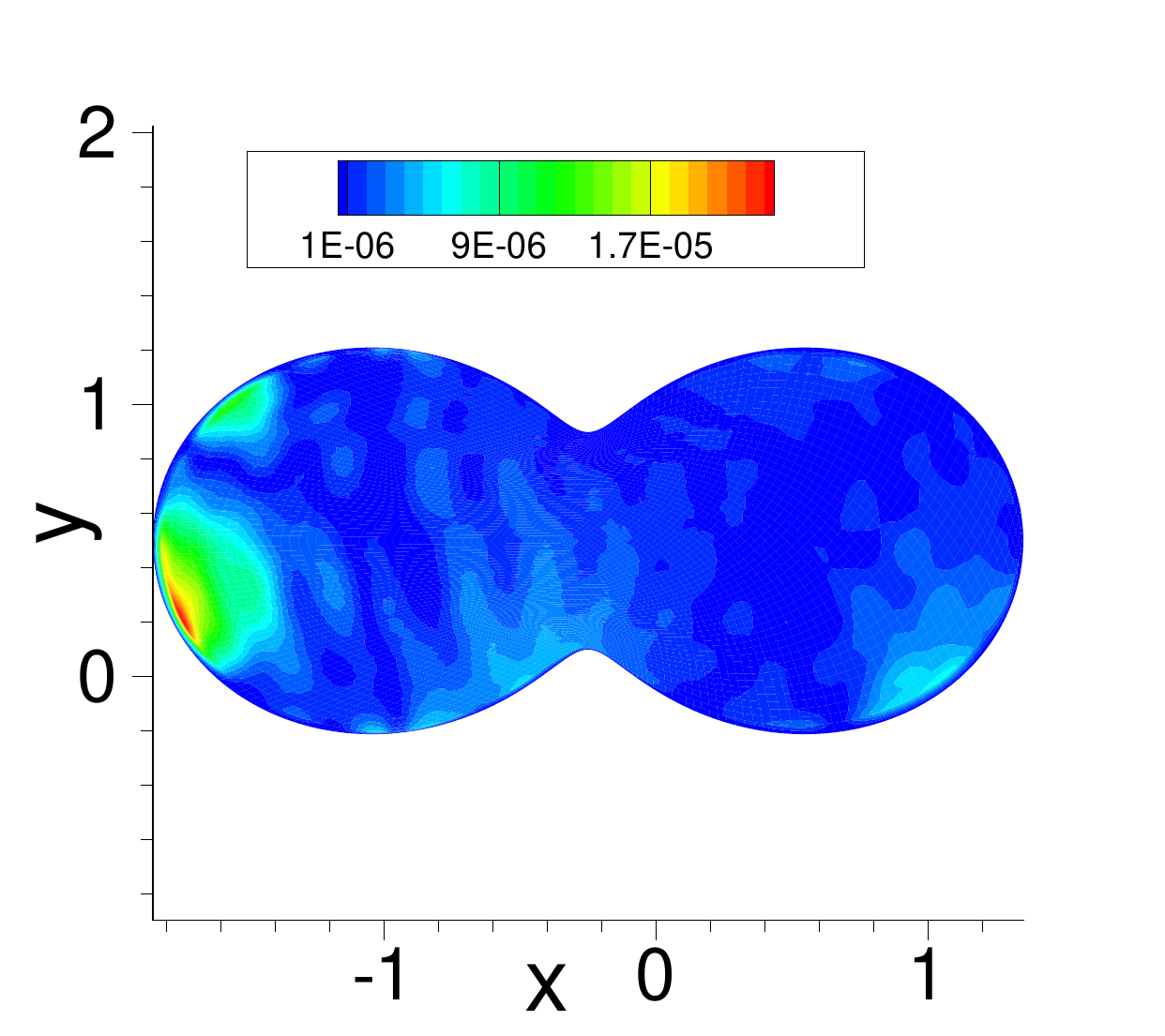}(i)
    \includegraphics[width=1.1in]{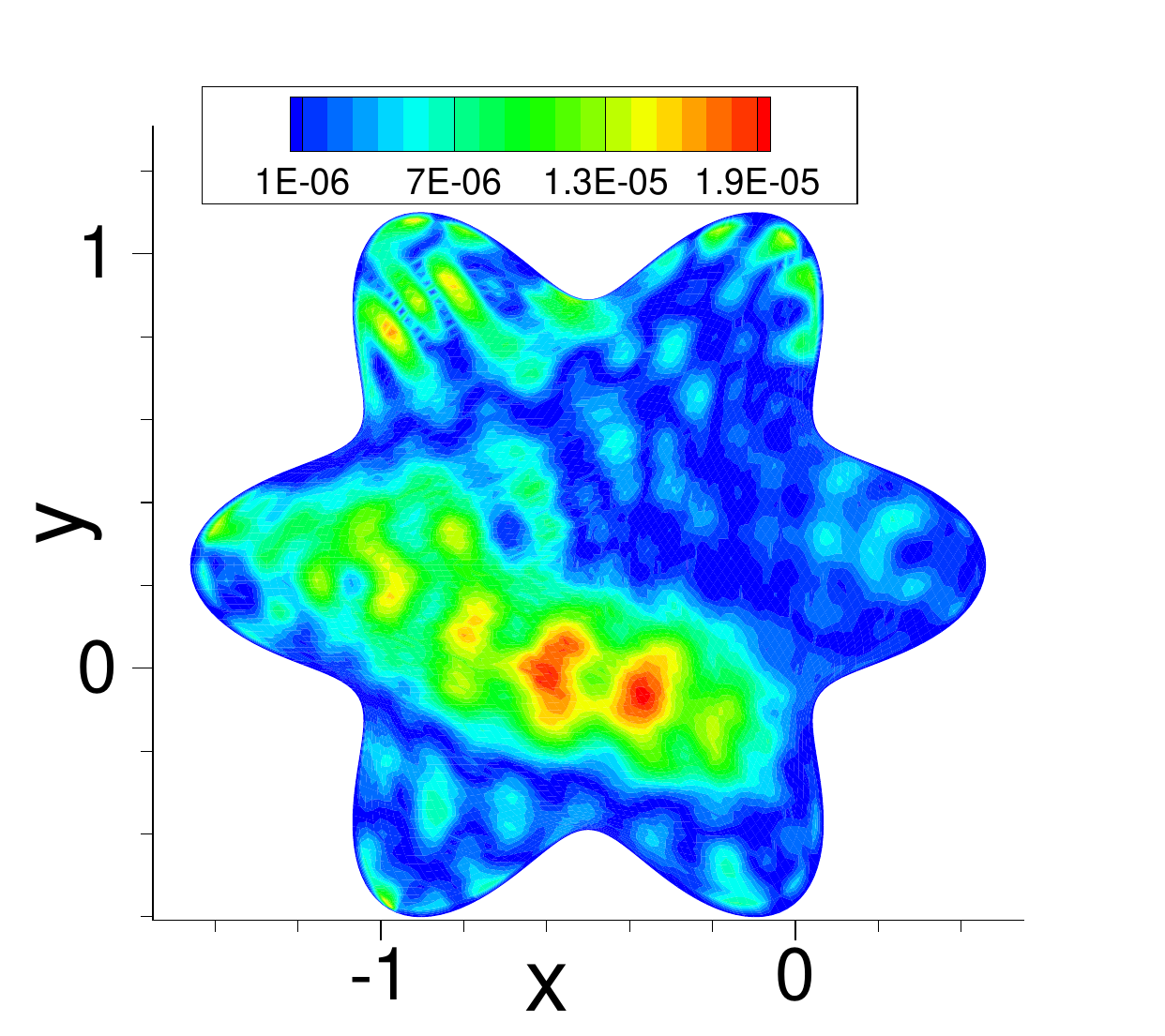}(j)
  }
  \caption{Variable-coefficient Poisson equation (DBC on all boundaries): Distributions of
    the NN solutions (top row) and their point-wise absolute errors (bottom row) on the five domains.
    Simulation parameters: (a,f) $R_m=4.0$, $Q=70$, $M=400$, simple map;
    (b,g) $R_m=3.5$, $Q=100$, $M=500$, simple map;
    (c,h) $R_m=4.0$, $Q=100$, $M=500$, Winslow map ($p_w=0$, $Q_w=200$);
    (d,i) $R_m=4.5$, $Q=100$, $M=700$, Winslow map ($p_w=2$, $Q_w=200$);
    (e,j) $R_m=4.4$, $Q=120$, $M=900$, Winslow map ($p_w=5$, $Q_w=200$).
  }
  \label{fg_3}
\end{figure}


\begin{table}[tb]
  \centering\small
  \begin{tabular}{l|ccccc}
    \hline
    & Domain \#1 & Domain \#2 & Domain \#3 & Domain \#4 & Domain \#5 \\ \hline
    max $u$-error (domain) & $2.027\times 10^{-7}$ & $1.989\times 10^{-6}$ & $3.433\times 10^{-7}$ & $2.185\times 10^{-5}$ & $2.008\times 10^{-5}$  \\
    rms $u$-error (domain) & $2.184\times 10^{-8}$ & $3.830\times 10^{-7}$ & $5.932\times 10^{-8}$ & $2.852\times 10^{-6}$ & $5.593\times 10^{-6}$  \\ \hline
    max DBC-error ($\overline{AB}$) & $2.776\times 10^{-16}$ & $0.0$ & $2.220\times 10^{-16}$ & $2.220\times 10^{-16}$ & $6.661\times 10^{-16}$  \\
    rms DBC-error ($\overline{AB}$) & $1.687\times 10^{-16}$ & $0.0$ & $7.697\times 10^{-17}$ & $3.445\times 10^{-17}$ & $1.750\times 10^{-16}$  \\ \hline
    max DBC-error ($\overline{BC}$) & $8.882\times 10^{-16}$ & $4.441\times 10^{-16}$ & $6.661\times 10^{-16}$ & $6.661\times 10^{-16}$ & $4.441\times 10^{-16}$  \\
    rms DBC-error ($\overline{BC}$) & $2.351\times 10^{-16}$ & $1.243\times 10^{-16}$ & $1.519\times 10^{-16}$ & $1.057\times 10^{-16}$ & $1.472\times 10^{-16}$  \\ \hline
    max DBC-error ($\overline{AC}$) & $4.441\times 10^{-16}$ & $0.0$ & $3.331\times 10^{-16}$ & $2.220\times 10^{-16}$ & $2.220\times 10^{-16}$  \\
    rms DBC-error ($\overline{AC}$) & $1.486\times 10^{-16}$ & $0.0$ & $9.853\times 10^{-17}$ & $8.250\times 10^{-17}$ & $1.420\times 10^{-16}$ \\
    \hline
  \end{tabular}
  \caption{Variable-coefficient Poisson equation (DBC on all boundaries):
    Maximum and rms errors of the NN solution over the domain
    and the maximum and rms DBC errors on the three boundaries. Simulation parameters
    follow those of Figure \ref{fg_3}.
  }
  \label{tab_1}
\end{table}


Figure \ref{fg_3} and Table \ref{tab_1} illustrate the NN solutions
on these five domains, where Dirichlet conditions are imposed on all
 boundaries. Figure~\ref{fg_3} shows distributions (contours) of the NN solution
$u(\mbs x)$ (plots (a)--(e)) and its point-wise absolute error $|u(\mbs x)-u_{ex}(\mbs x)|$
(plots (f)--(j)) on these domains. The NN/simulation parameters, including
$R_m$, $M$ and $Q$ are provided in the figure caption. The simple domain map \eqref{eq_8}
is used for domains \#1 and \#2. For the other three domains, the simple map
is inadequate for univalency, and we have employed the Winslow map \eqref{eq_12}
(see Remark \ref{rem_1}) for these problems.
The figure caption also includes the parameters for the Winslow map
($p_w$ and $Q_w$) with domains \#3, \#4 and \#5.
Evidently the NN method has captured the solution accurately, with
the maximum domain errors on the order of $10^{-5}$ to $10^{-7}$ for different
geometries. It should be noted that domain \#2 has a cusp at vertex $C$ (zero angle),
and the maximum domain error of the NN solution is on the order of $10^{-6}$
for this domain.

Table \ref{tab_1} exemplifies the accuracy of the current method for enforcing
Dirichlet conditions on these domains.
The maximum and rms errors of the NN solution over the domain, together
with the maximum/rms DBC errors on the three boundaries,
are listed in this table. The maximum DBC error has an order of
magnitude of $10^{-16}$ identically for all five domains.
This signifies that our method has enforced
the boundary conditions to the machine precision on these curved domain boundaries.


\begin{figure}
  \centerline{
    \includegraphics[width=1.2in]{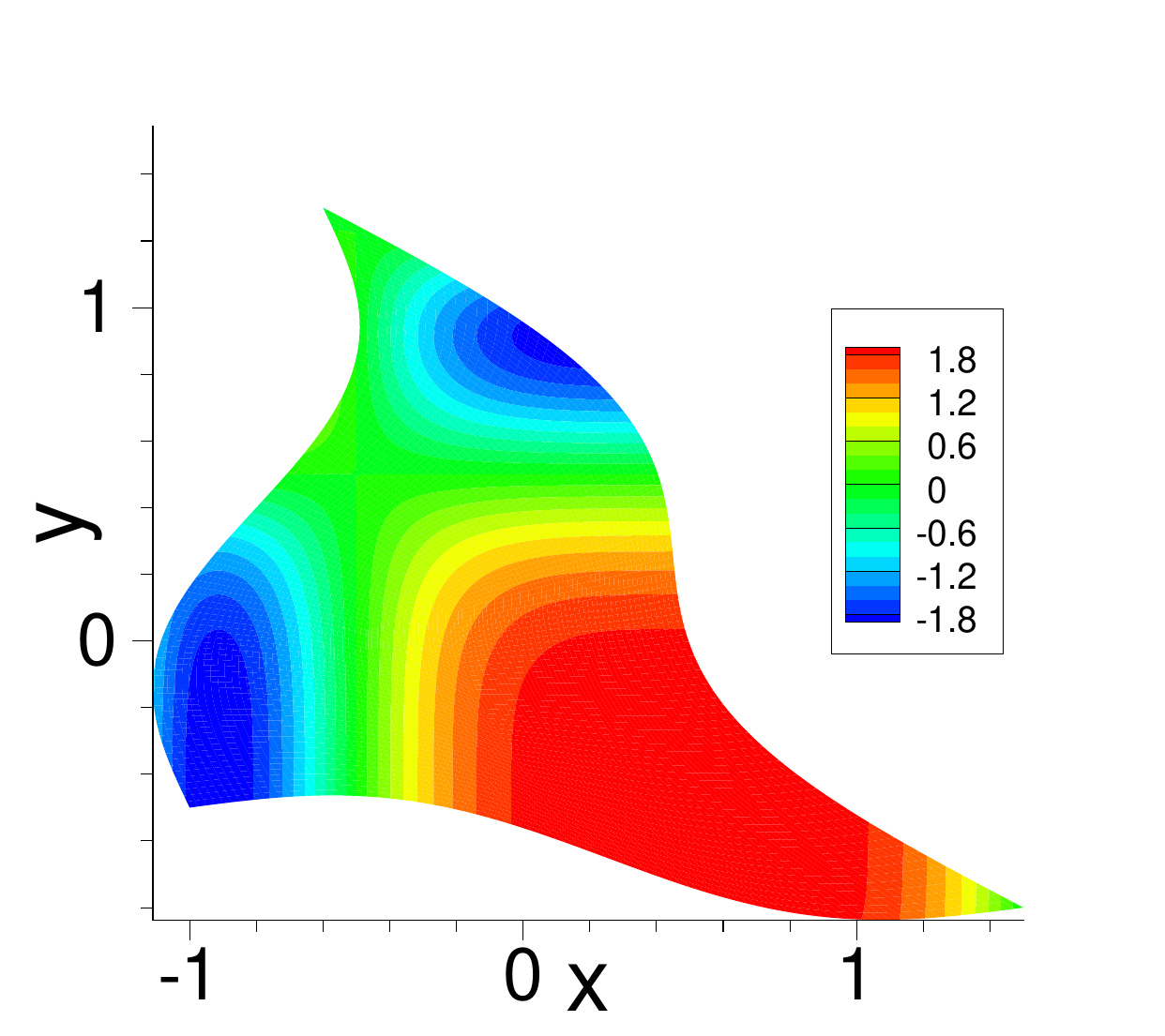}(a)
    \includegraphics[width=1.2in]{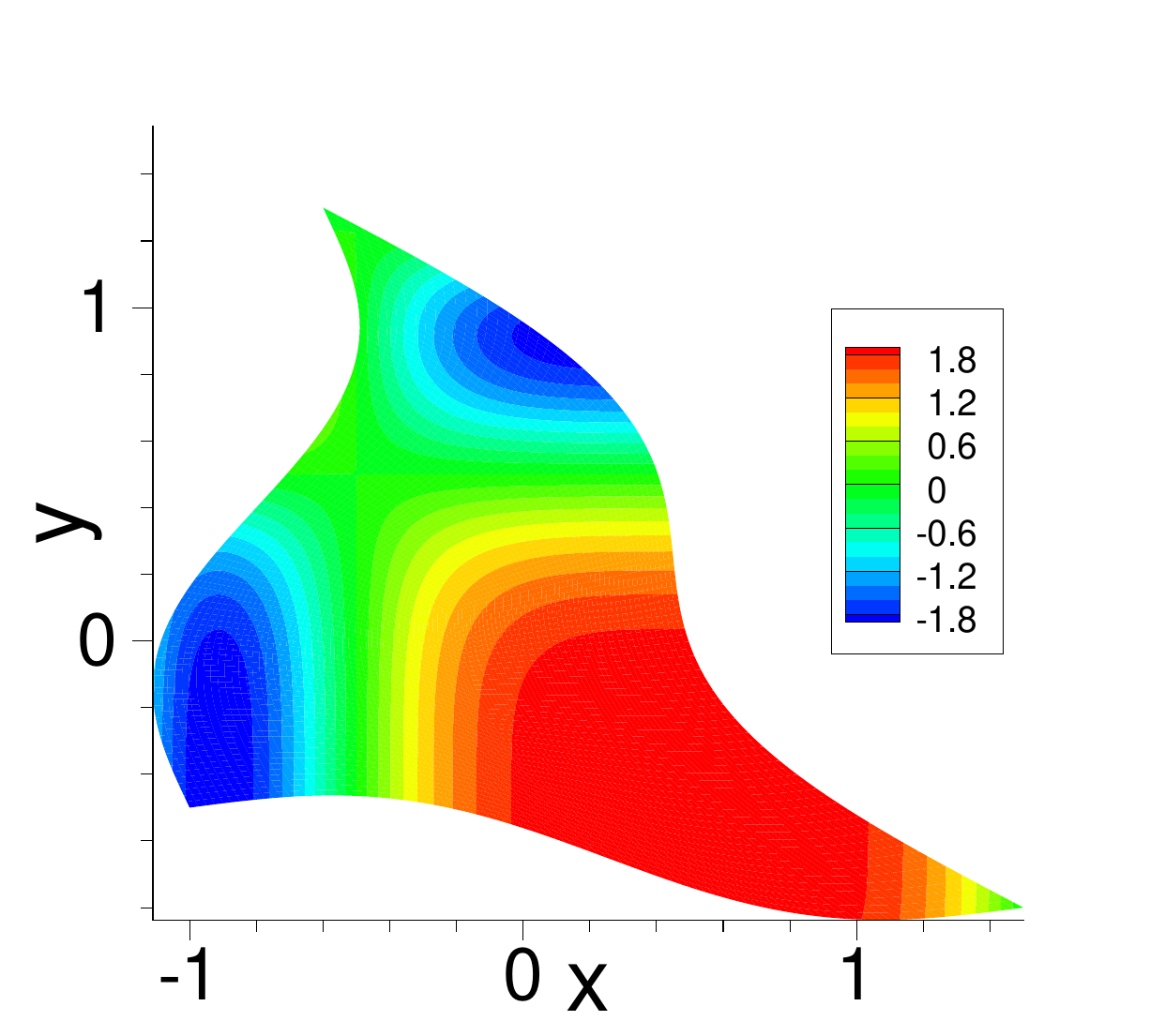}(b)
    \includegraphics[width=1.2in]{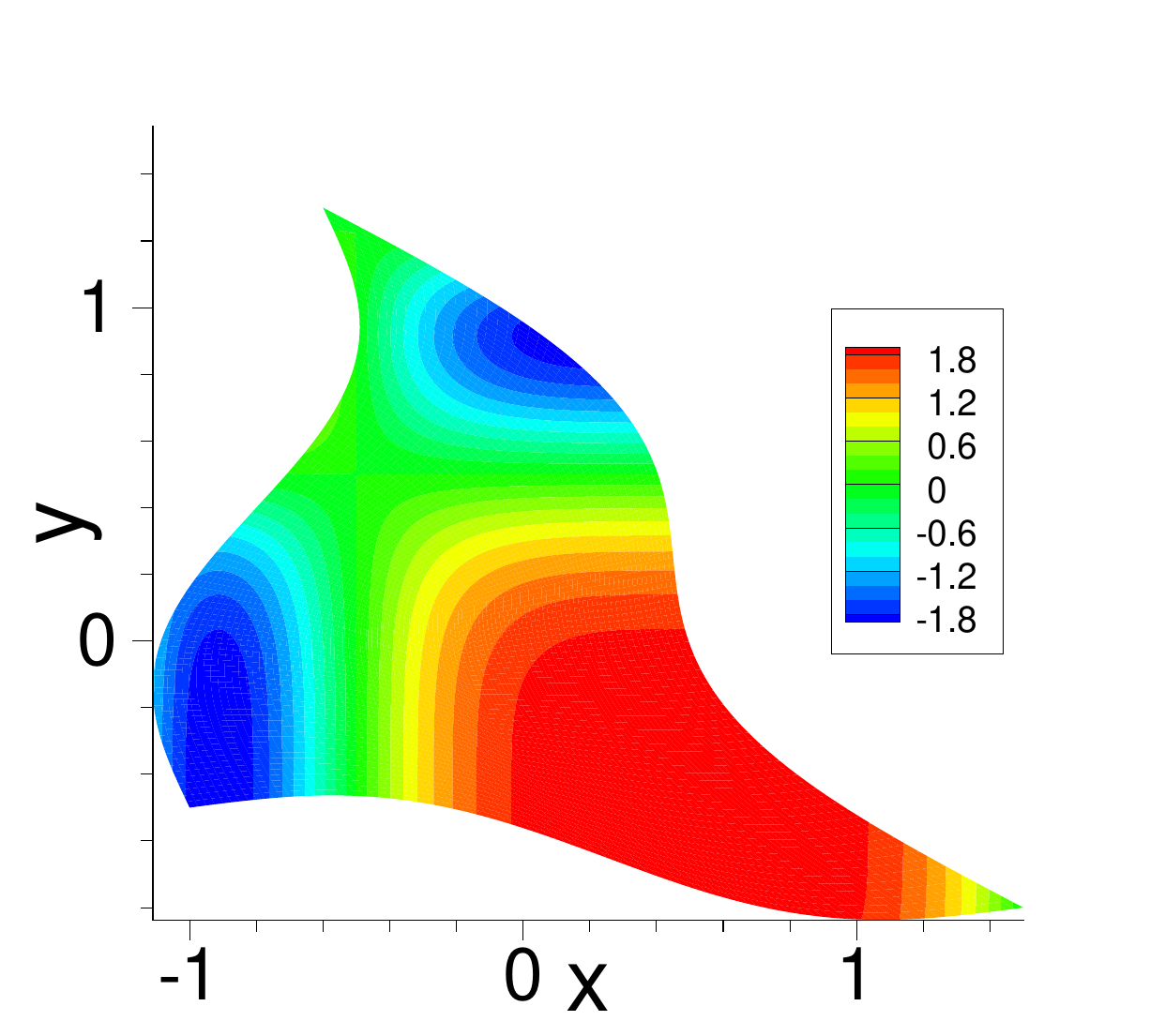}(c)
    \includegraphics[width=1.2in]{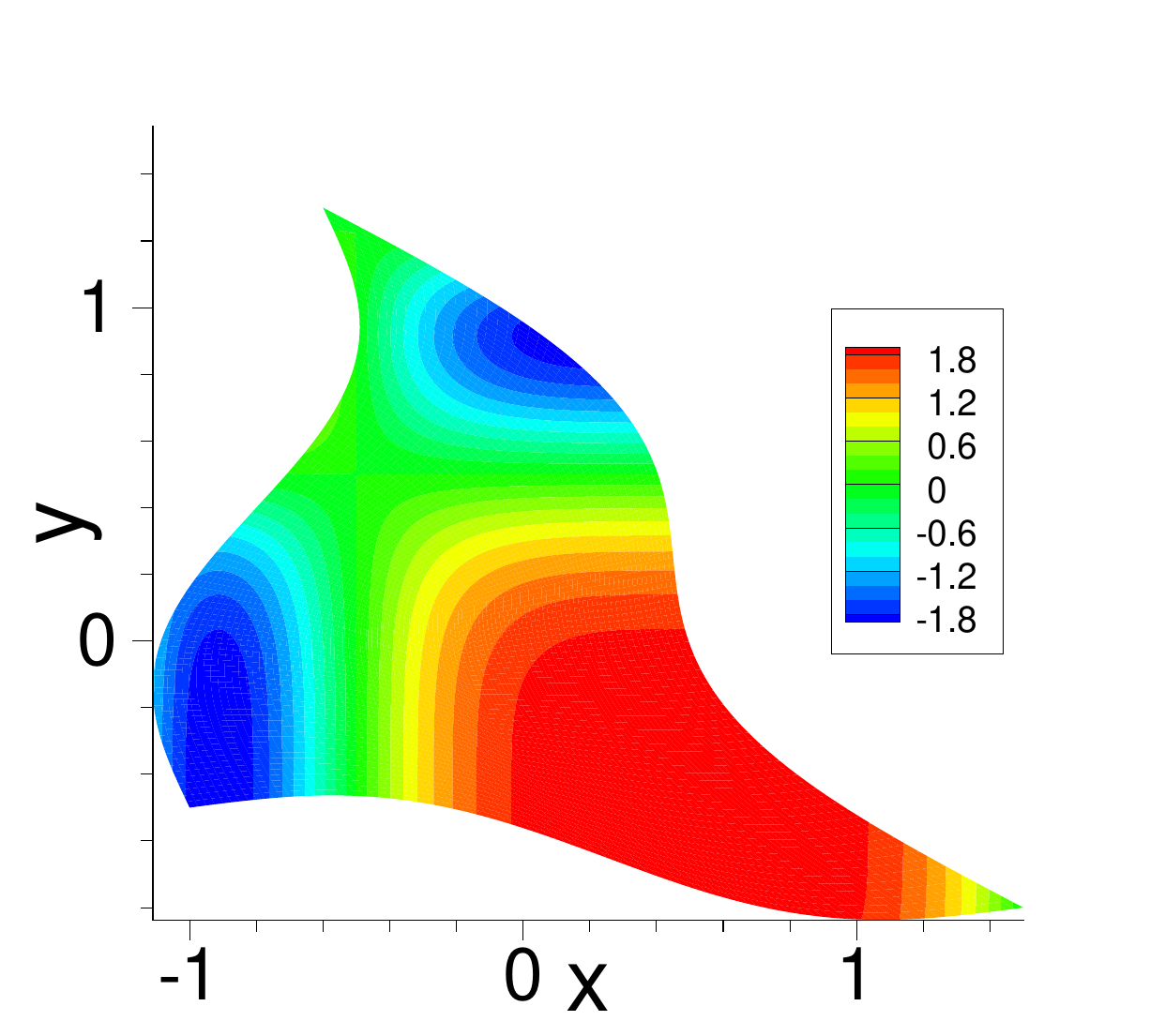}(d)
  }
  \centerline{
    \includegraphics[width=1.2in]{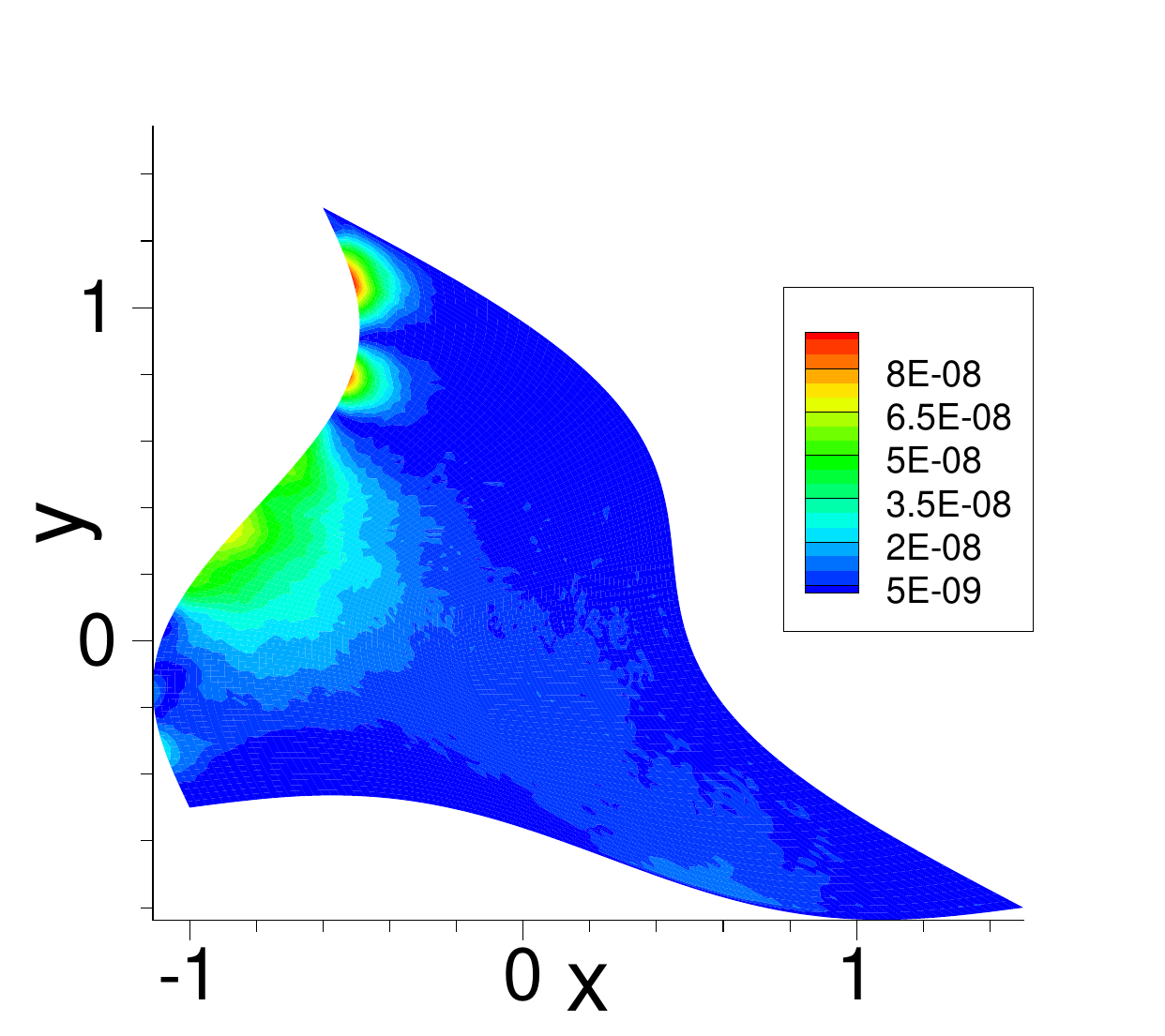}(e)
    \includegraphics[width=1.2in]{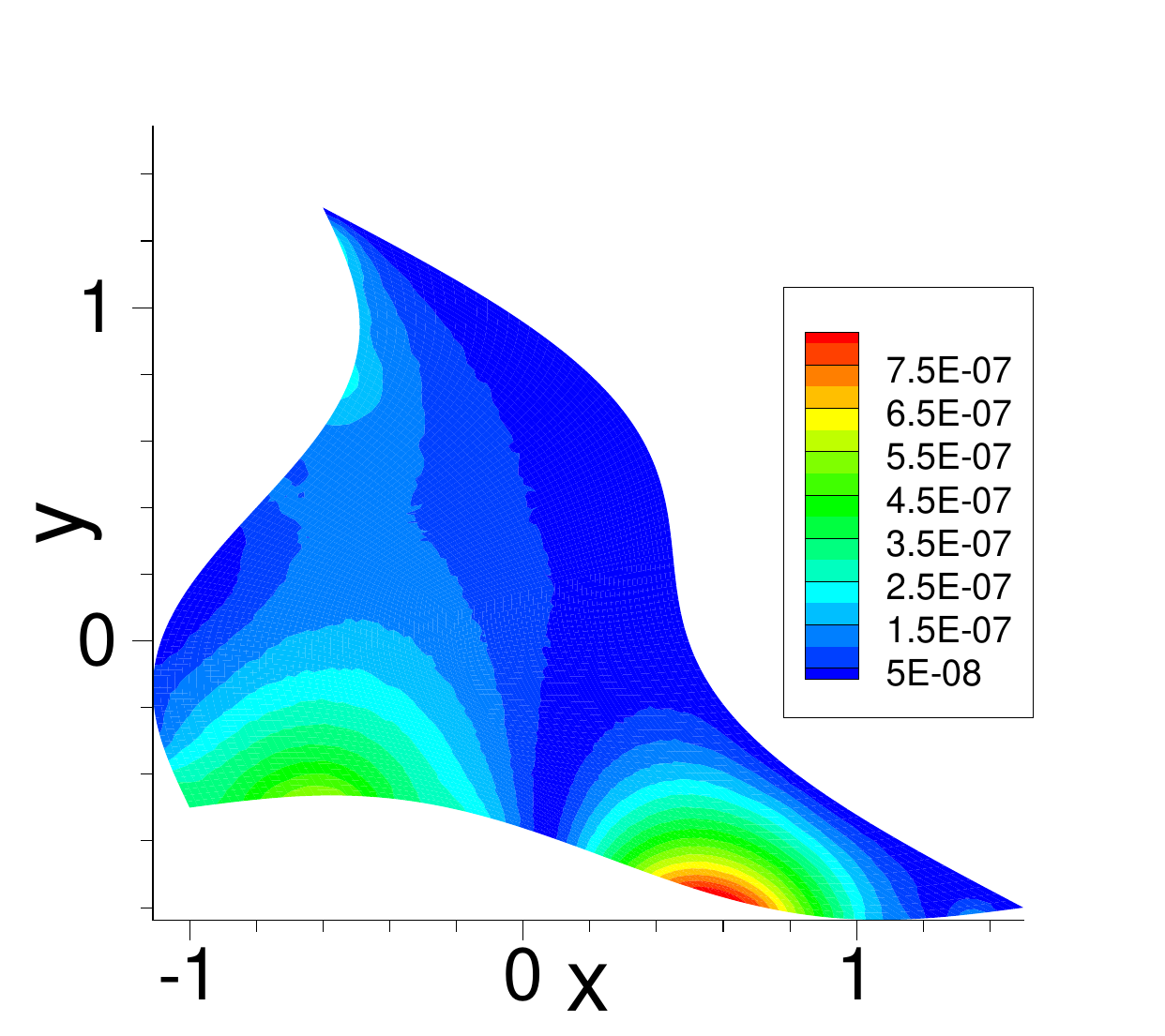}(f)
    \includegraphics[width=1.2in]{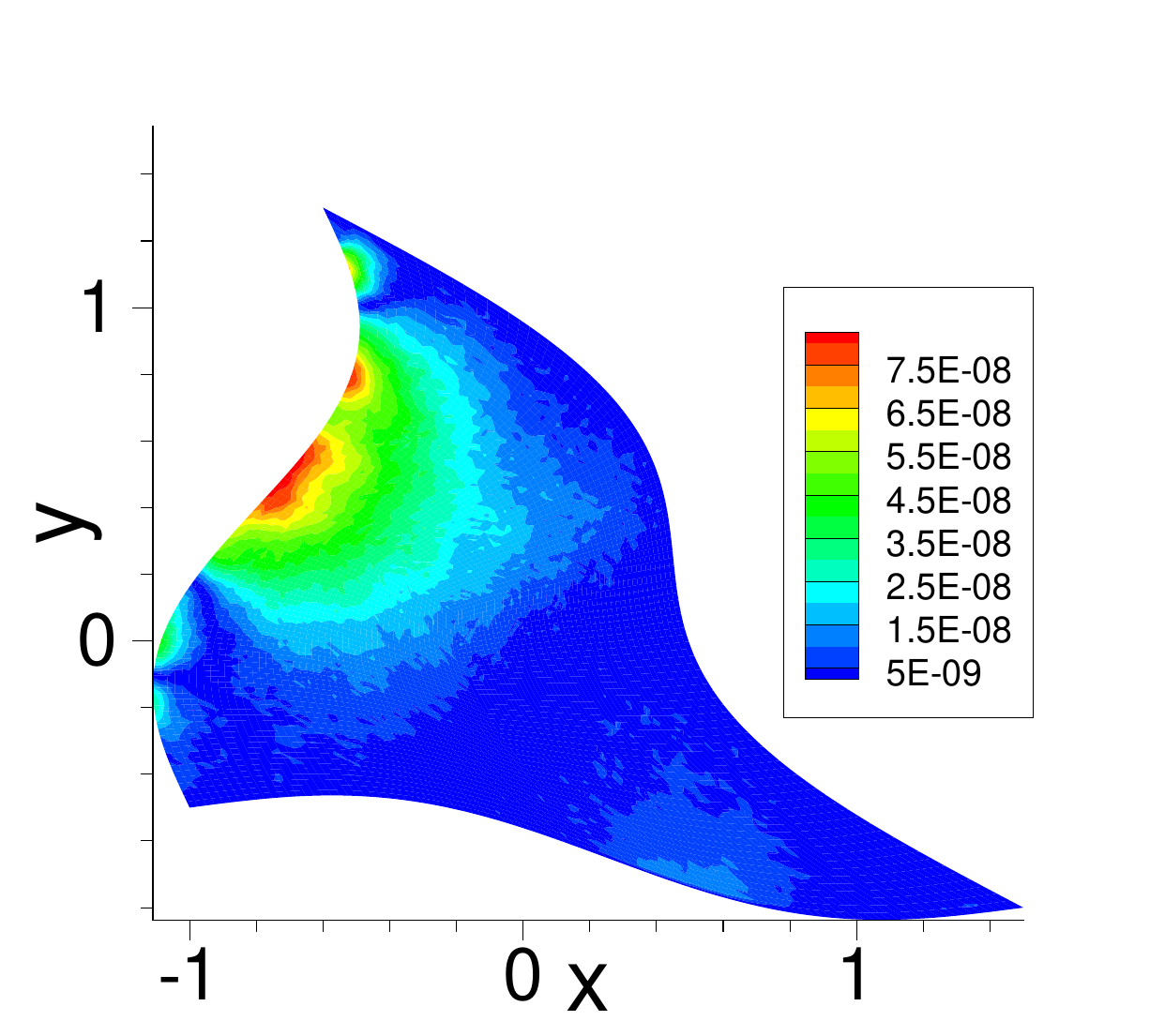}(g)
    \includegraphics[width=1.2in]{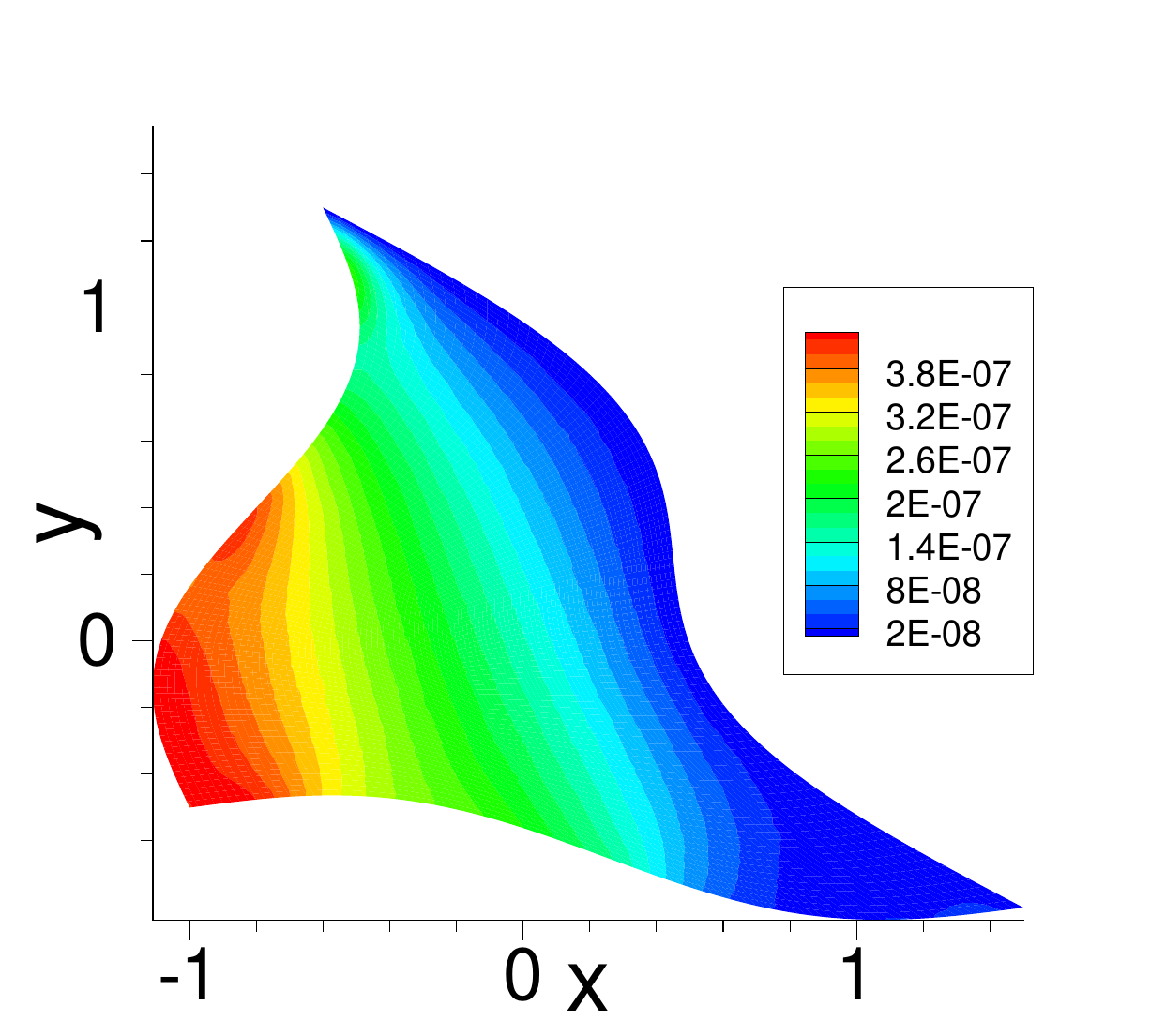}(h)
  }
  \caption{Variable-coefficient Poisson equation on domain \#1 (NBCs and RBCs):
    Distributions of the NN solutions (top row) and their point-wise absolute errors (bottom row)
    under four boundary configurations.
    (a,e) Config \#1:  NBC on $\overline{AC}$ and DBC on the other boundaries;
    (b,f) Config \#2: NBC on $\overline{AB}$ and $\overline{AC}$, and DBC on $\overline{BC}$;
    (c,g) Config \#3: RBC on $\overline{AC}$ with $\alpha_{AC}=1.0$, and DBC on the other boundaries;
    (d,h) Config \#4: RBC on $\overline{AB}$ and $\overline{AC}$ with
    $\alpha_{AB}=0.5$ and $\alpha_{AC}=1.0$, and DBC on $\overline{BC}$.
    Simulation parameters: (a,e) $R_m=3.3$,
    (b,f) $R_m=3.3$,
    (c,g) $R_m=3.4$,
    (d,h) $R_m=4.0$.
    $Q=100$ and $M=500$, together with the simple domain map, are used in all four configurations.
  }
  \label{fg_4}
\end{figure}

\begin{table}[tb]
  \centering\small
  \begin{tabular}{l|cccc}
    \hline
    & Config \#1 & Config \#2 & Config \#3 & Config \#4  \\ \hline
    max $u$-error (domain) & $9.113\times 10^{-8}$ & $8.459\times 10^{-7}$
    & $8.596\times 10^{-8}$ & $4.430\times 10^{-7}$ \\
    rms $u$-error (domain) & $1.462\times 10^{-8}$  & $1.873\times 10^{-7}$
    & $1.954\times 10^{-8}$ & $2.128\times 10^{-7}$ \\ \hline
    max bc-error ($\overline{AB}$) & $8.882\times 10^{-16}$ & $4.441\times 10^{-16}$
    & $8.882\times 10^{-16}$ & $3.747\times 10^{-16}$   \\
    rms bc-error ($\overline{AB}$) & $3.237\times 10^{-16}$ & $1.047\times 10^{-16}$
    & $3.237\times 10^{-16}$ & $1.476\times 10^{-16}$ \\ \hline
    max bc-error ($\overline{BC}$) & $0.0$ & $0.0$ & $0.0$ & $0.0$  \\
    rms bc-error ($\overline{BC}$) & $0.0$ & $0.0$ & $0.0$ & $0.0$ \\ \hline
    max bc-error ($\overline{AC}$) & $1.776\times 10^{-15}$ & $8.882\times 10^{-16}$
    & $1.776\times 10^{-15}$ & $1.776\times 10^{-15}$  \\
    rms bc-error ($\overline{AC}$) & $3.028\times 10^{-16}$ & $2.522\times 10^{-16}$
    & $4.909\times 10^{-16}$ & $4.544\times 10^{-16}$ \\
    \hline
  \end{tabular}
  \caption{Variable-coefficient Poisson equation on domain \#1 (NBCs and RBCs):
    Maximum/rms NN domain errors 
    and the maximum/rms boundary-condition errors (DBC, NBC, or RBC) on the three boundaries.
    Different configurations correspond to those in Figure \ref{fg_4}.
  }
  \label{tab_2}
\end{table}

\begin{figure}
  \centerline{
    \includegraphics[width=1.2in]{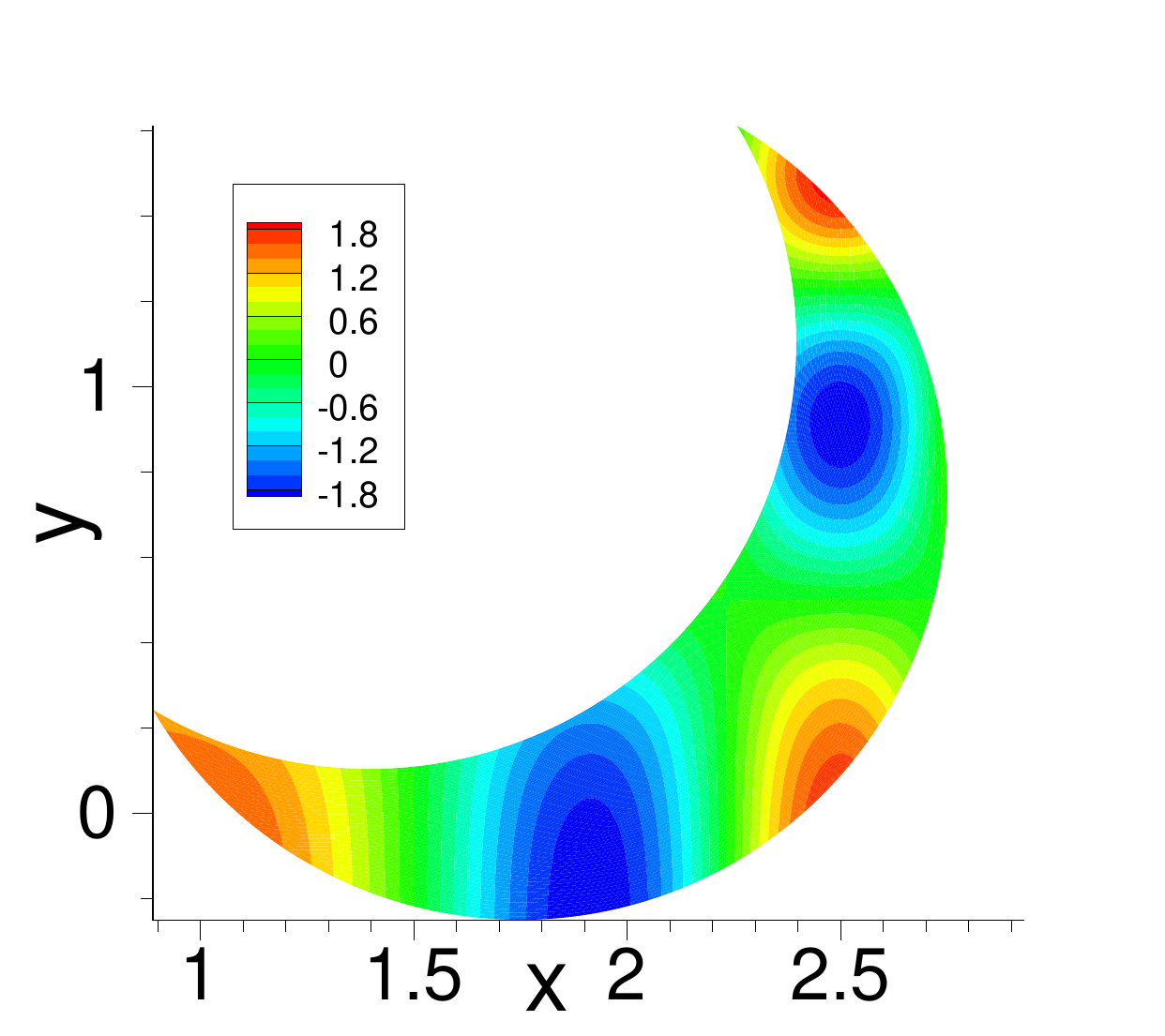}(a)
    \includegraphics[width=1.2in]{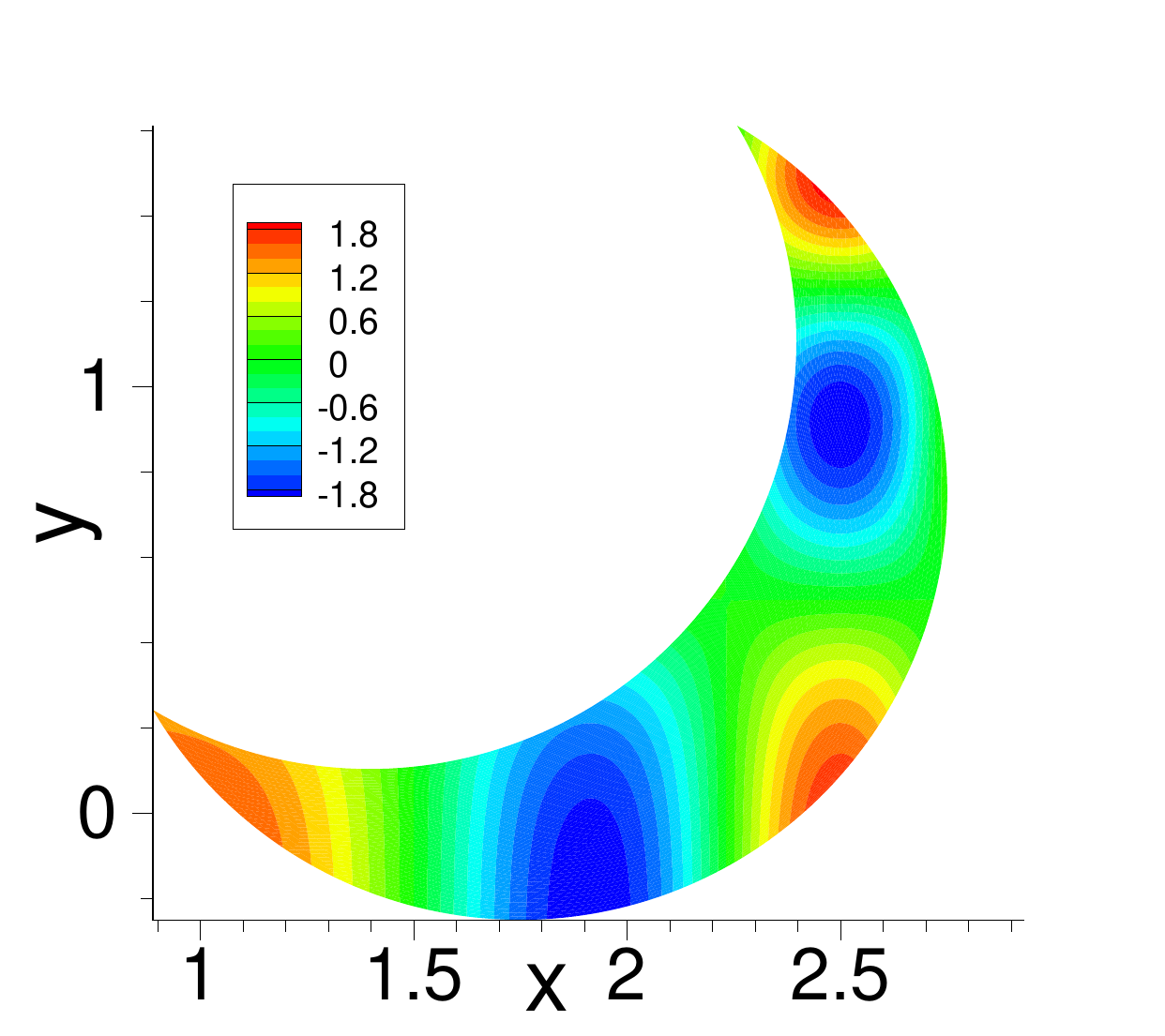}(b)
    \includegraphics[width=1.2in]{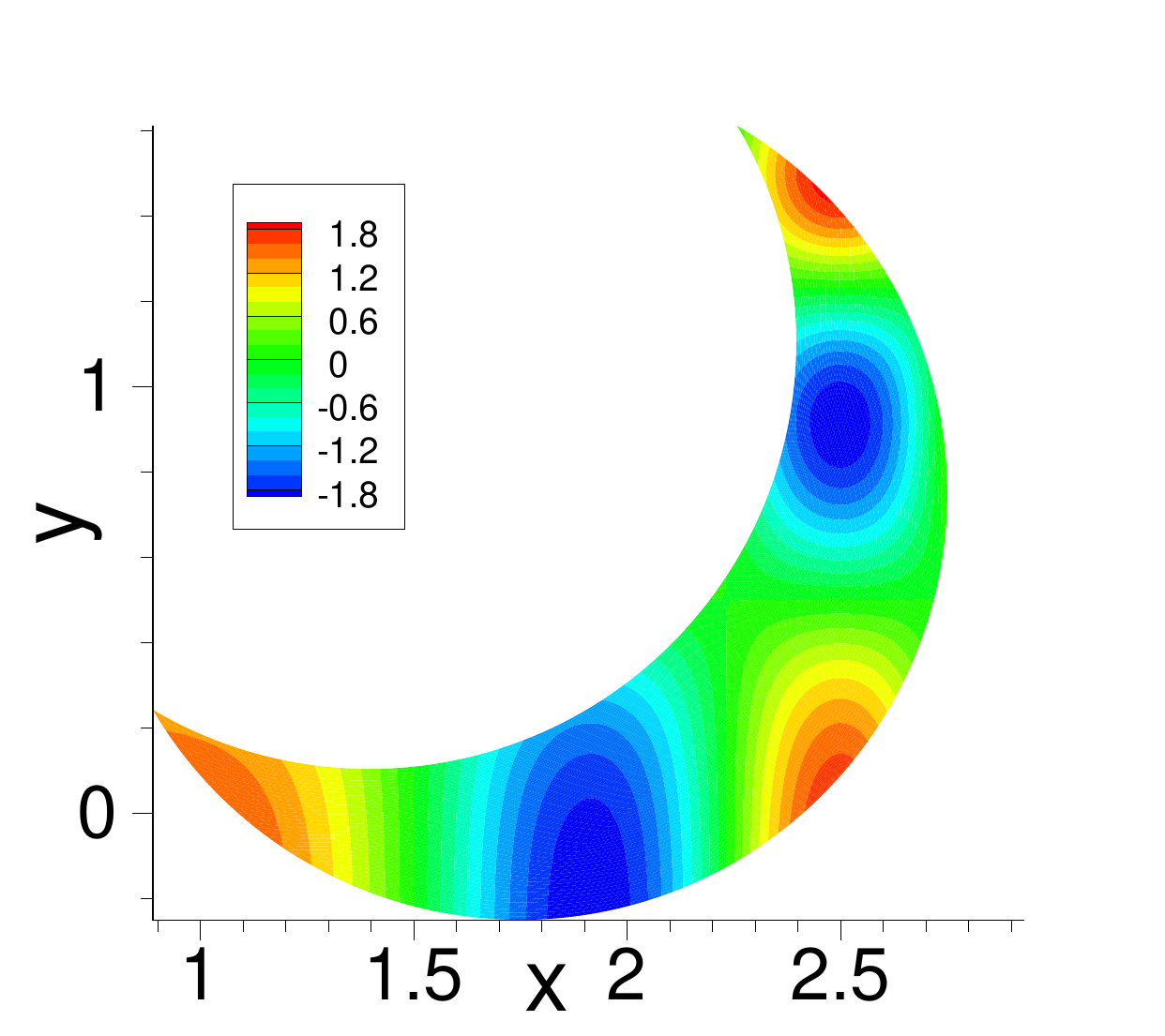}(c)
    \includegraphics[width=1.2in]{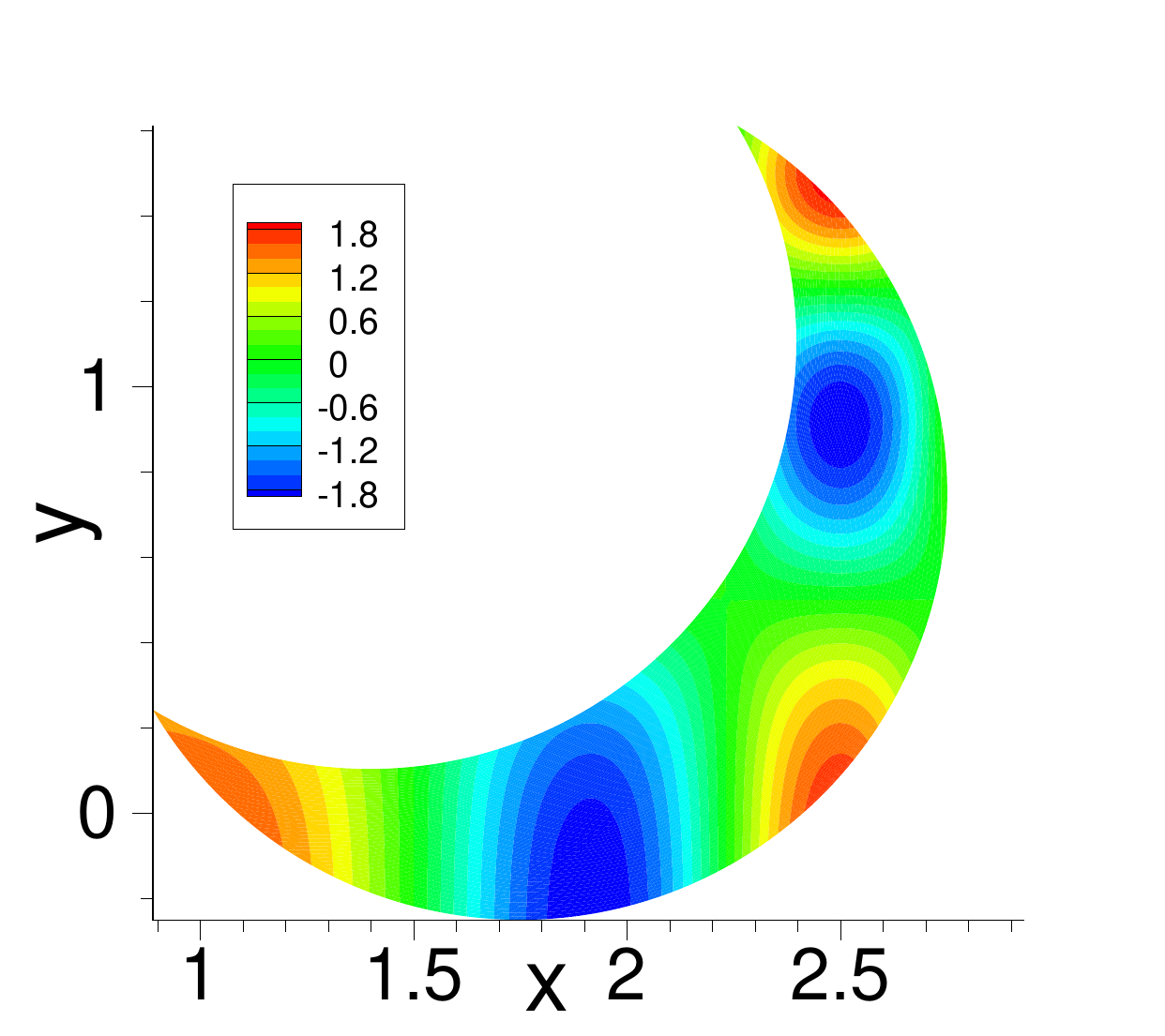}(d)
  }
  \centerline{
    \includegraphics[width=1.2in]{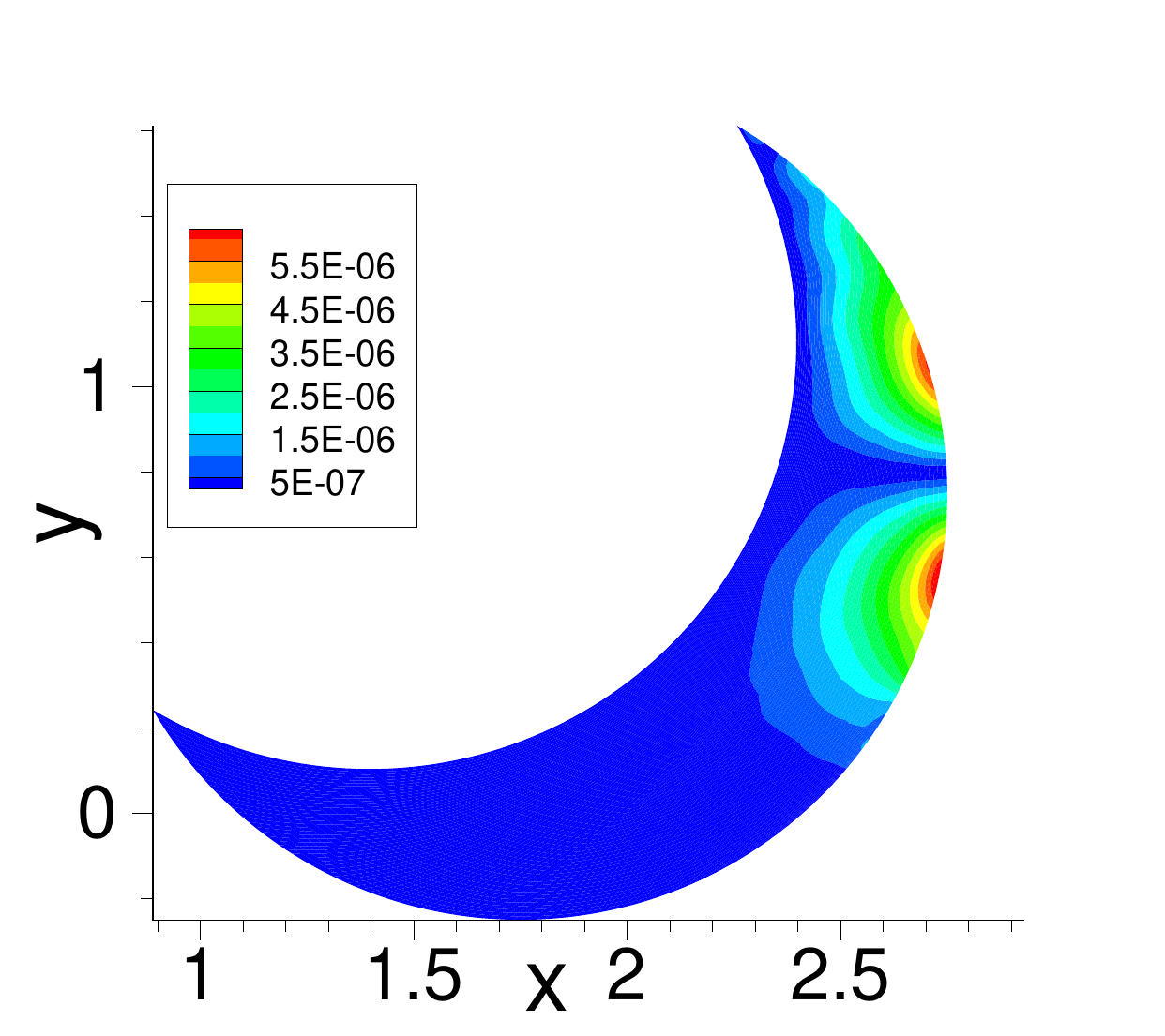}(e)
    \includegraphics[width=1.2in]{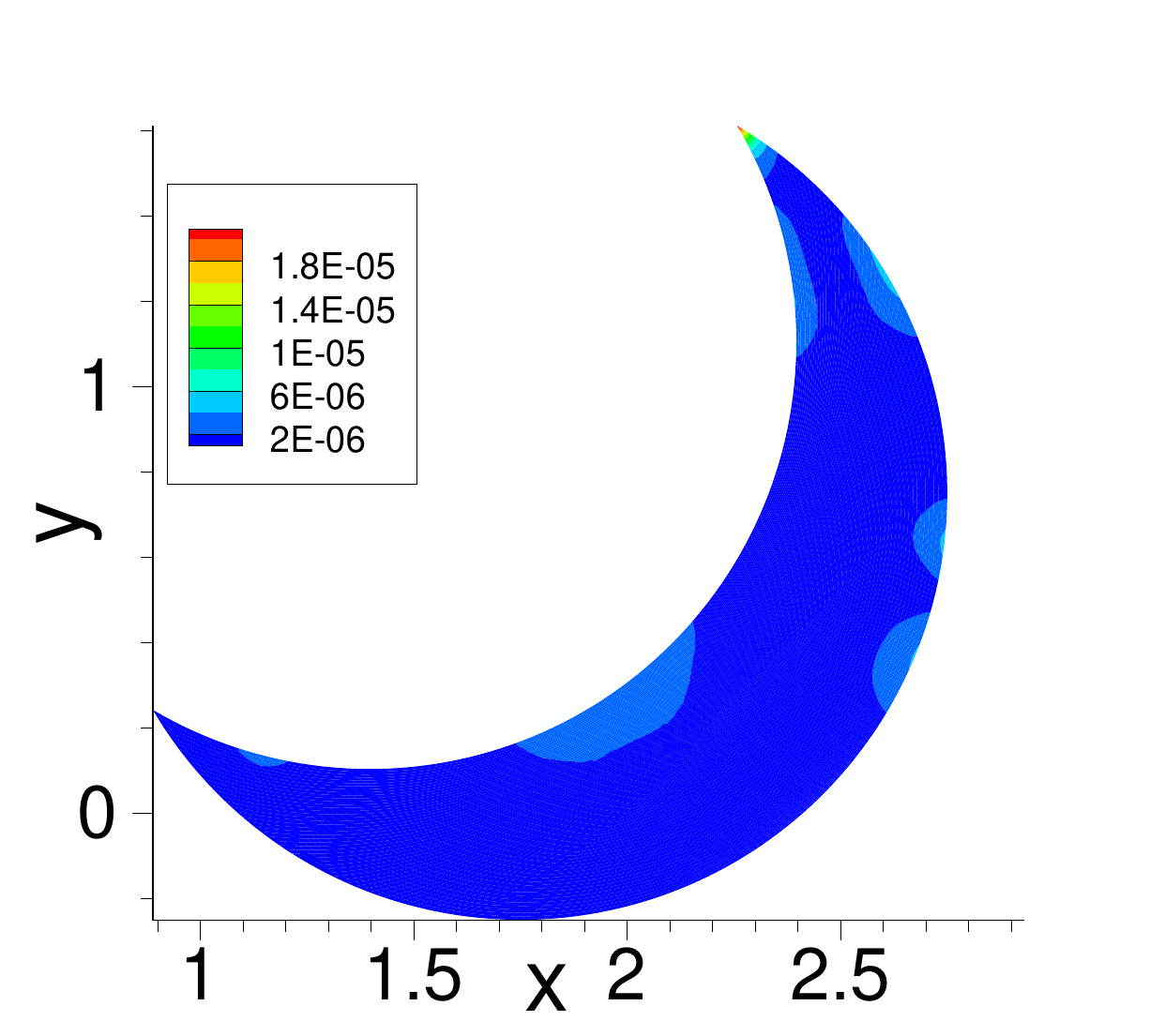}(f)
    \includegraphics[width=1.2in]{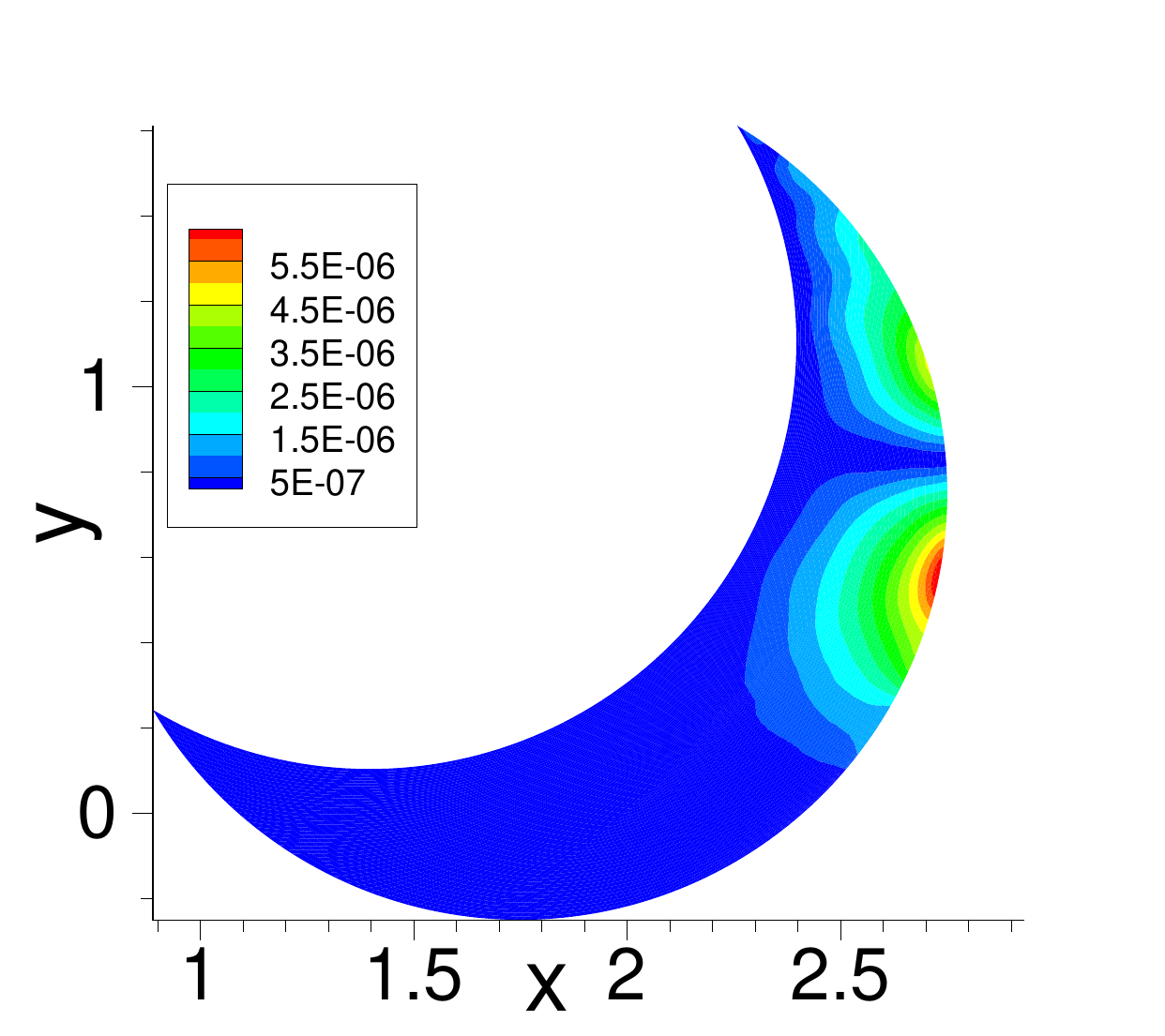}(g)
    \includegraphics[width=1.2in]{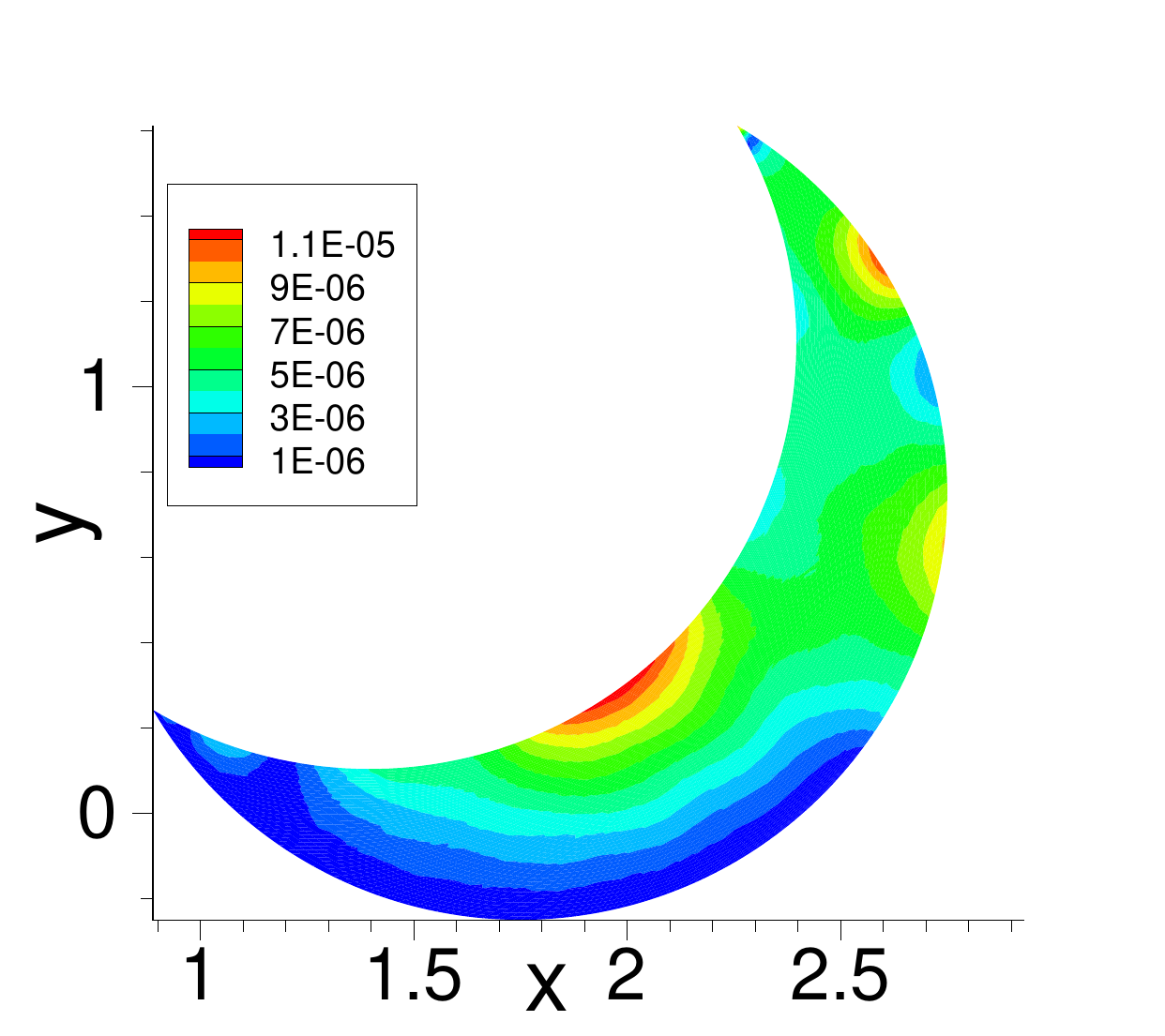}(h)
  }
  \caption{Variable-coefficient Poisson equation on domain \#3 (NBCs and RBCs):
    Distributions of the NN solutions (top row) and their point-wise absolute errors (bottom row)
    under four BC configurations.
    (a,e) Config \#1:  NBC on $\overline{AC}$, and DBC on the other boundaries;
    (b,f) Config \#2: NBC on $\overline{AB}$ and $\overline{AC}$, and DBC on $\overline{BC}$;
    (c,g) Config \#3: RBC on $\overline{AC}$ with $\alpha_{AC}=1.0$, and DBC on the other boundaries;
    (d,h) Config \#4: RBC on $\overline{AB}$ and $\overline{AC}$ with
    $\alpha_{AB}=0.5$ and $\alpha_{AC}=1.0$, and DBC on $\overline{BC}$.
    Simulation parameters: (a,e) $R_m=3.9$,
    (b,f) $R_m=3.3$,
    (c,g) $R_m=3.9$,
    (d,h) $R_m=3.2$.
    $Q=100$ and $M=500$, together with the Winslow map ($p_w=0$, $Q_w=50$),
    are used in all four configurations.
  }
  \label{fg_5}
\end{figure}

\begin{table}[tb]
  \centering\small
  \begin{tabular}{l|cccc}
    \hline
    & Config \#1 & Config \#2 & Config \#3 & Config \#4  \\ \hline
    max $u$-error (domain) & $6.542\times 10^{-6}$ & $2.139\times 10^{-5}$
    & $6.490\times 10^{-6}$ & $1.169\times 10^{-5}$  \\
    rms $u$-error (domain) & $1.286\times 10^{-6}$ & $1.317\times 10^{-6}$
    & $1.196\times 10^{-6}$ & $4.804\times 10^{-6}$ \\ \hline
    max bc-error ($\overline{AB}$) & $4.441\times 10^{-16}$ & $4.441\times 10^{-16}$
    & $4.441\times 10^{-16}$ & $8.882\times 10^{-16}$ \\
    rms bc-error ($\overline{AB}$) & $3.444\times 10^{-17}$ & $1.388\times 10^{-16}$
    & $3.444\times 10^{-17}$ & $2.126\times 10^{-16}$ \\ \hline
    max bc-error ($\overline{BC}$) & $0.0$ & $0.0$ & $0.0$ & $0.0$ \\
    rms bc-error ($\overline{BC}$) & $0.0$ & $0.0$ & $0.0$ & $0.0$ \\ \hline
    max bc-error ($\overline{AC}$) & $1.776\times 10^{-15}$ & $8.882\times 10^{-16}$
    & $1.776\times 10^{-15}$ & $1.776\times 10^{-15}$ \\
    rms bc-error ($\overline{AC}$) & $3.983\times 10^{-16}$ & $3.980\times 10^{-16}$
    & $6.444\times 10^{-16}$ & $6.748\times 10^{-16}$ \\
    \hline
  \end{tabular}
  \caption{Variable-coefficient Poisson equation on domain \#3 (NBCs and RBCs):
    Maximum/rms NN domain errors
    and maximum/rms boundary-condition errors (DBC, NBC, or RBC) on the three boundaries.
    Different configurations correspond to those in Figure \ref{fg_5}.
  }
  \label{tab_3}
\end{table}

The NN results under Neumann and Robin boundary conditions are
showcased in Figures \ref{fg_4} and \ref{fg_5} and in Tables \ref{tab_2} and \ref{tab_3}.
Figure \ref{fg_4} illustrates  the NN solutions (top row) and
their point-wise absolute errors (bottom row) on domain \#1 under four boundary
configurations: (i) NBC imposed on $\overline{AC}$,
combined with DBC on the other boundaries; (ii) NBC imposed on $\overline{AB}$
and $\overline{AC}$, combined with DBC on $\overline{BC}$; (iii) RBC imposed
on $\overline{AC}$ with $\alpha_{AC}=1.0$, combined with DBC on the other boundaries;
(iv) RBC imposed on $\overline{AB}$ and $\overline{AC}$ with $\alpha_{AB}=0.5$ and
$\alpha_{AC}=1.0$, combined with DBC on
$\overline{BC}$. Figure \ref{fg_5} shows the corresponding NN results on domain \#3
under similar configurations.
The simulation parameter values are provided in the captions of these figures,
including those parameters for computing the Winslow map with domain \#3.
We observe that the NN solutions are very accurate, with the maximum domain errors on the level
ranging from $10^{-7}$ to $10^{-8}$ for domain \#1 
and from $10^{-5}$ to $10^{-6}$ for domain \#3 with different configurations.

The current method also enforces the Neumann and Robin conditions on curved
domain boundaries to the machine accuracy. This point is demonstrated by
the numerical results in Tables \ref{tab_2} and \ref{tab_3}.
Table \ref{tab_2} lists the domain errors (max/rms) and the boundary-condition
errors (max/rms) on the three boundaries for domain \#1, corresponding to the four
different configurations in Figure \ref{fg_4}.
Table \ref{tab_3} lists the corresponding max/rms domain errors and boundary-condition errors
for domain \#3, under the four configurations in Figure \ref{fg_5}.
The maximum errors for NBC/RBC (and also DBC) range from $10^{-15}$ to $10^{-16}$ for
various configurations, underscoring the machine-epsilon level precision of our method
in enforcing these boundary conditions on curved boundaries.

\subsection{Nonlinear Helmholtz Equation}
\label{sec_32}

\begin{figure}
  \centerline{
    \includegraphics[width=1.1in]{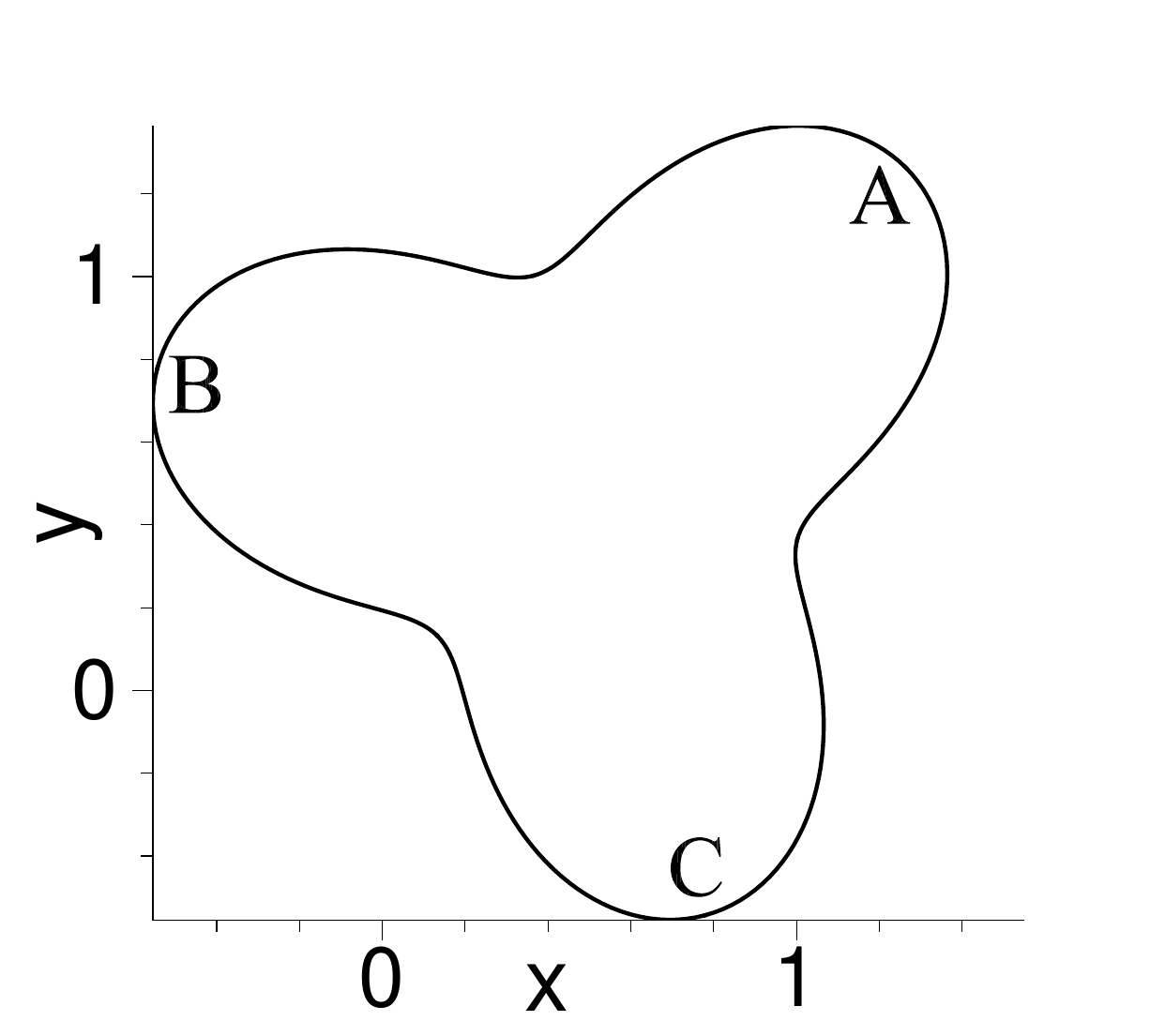}(a)
    \includegraphics[width=1.1in]{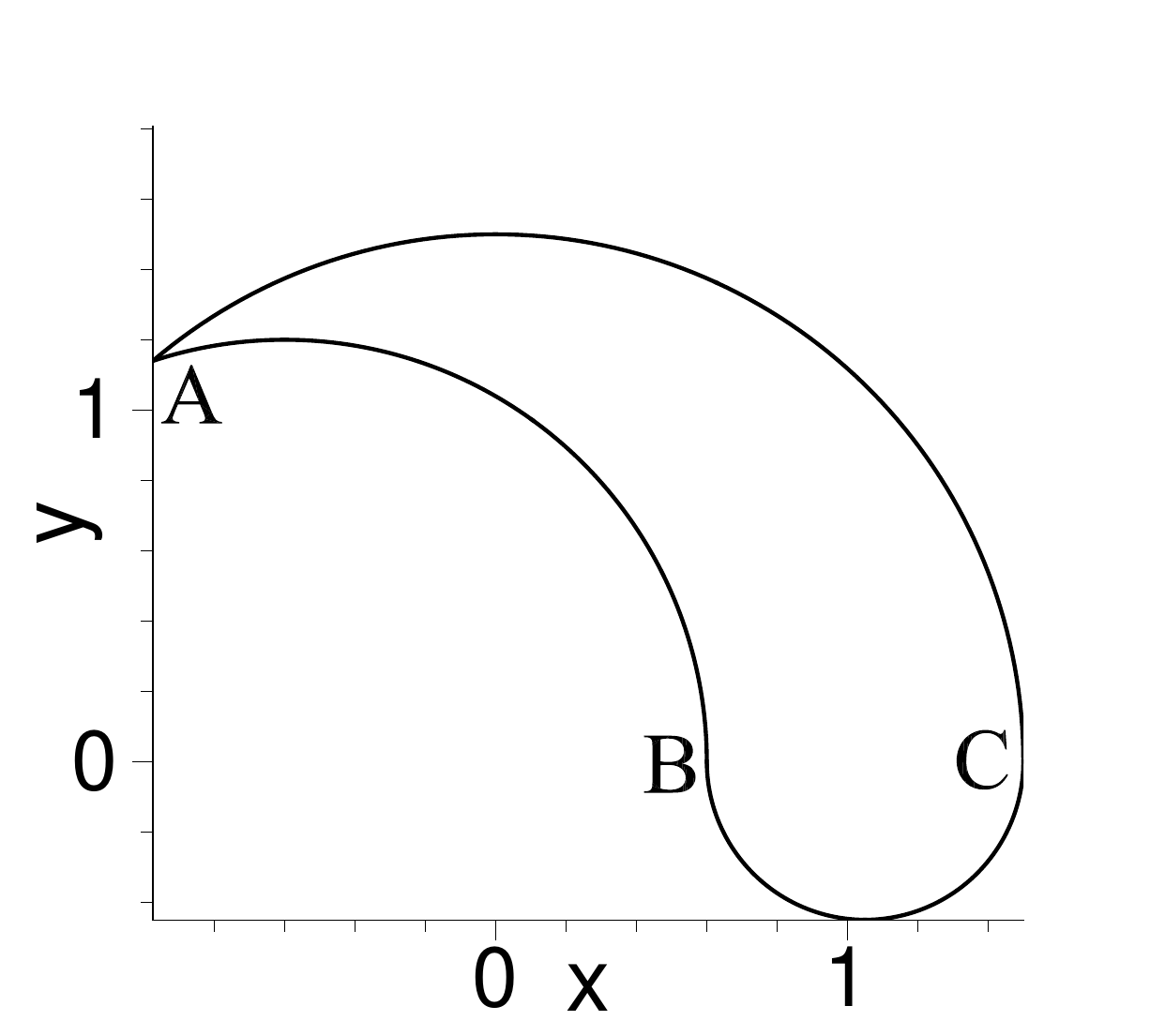}(b)
    \includegraphics[width=1.1in]{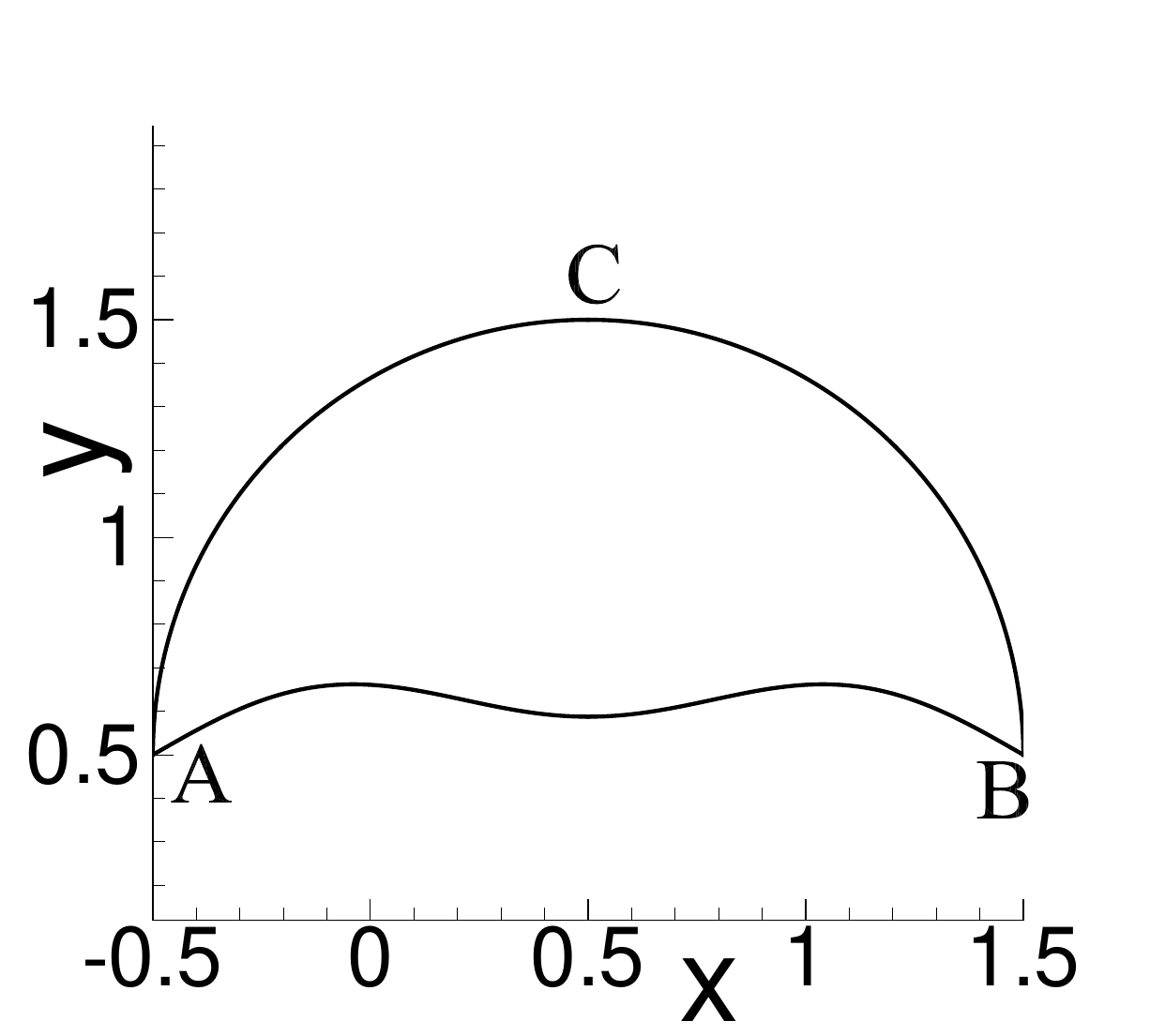}(c)
    \includegraphics[width=1.1in]{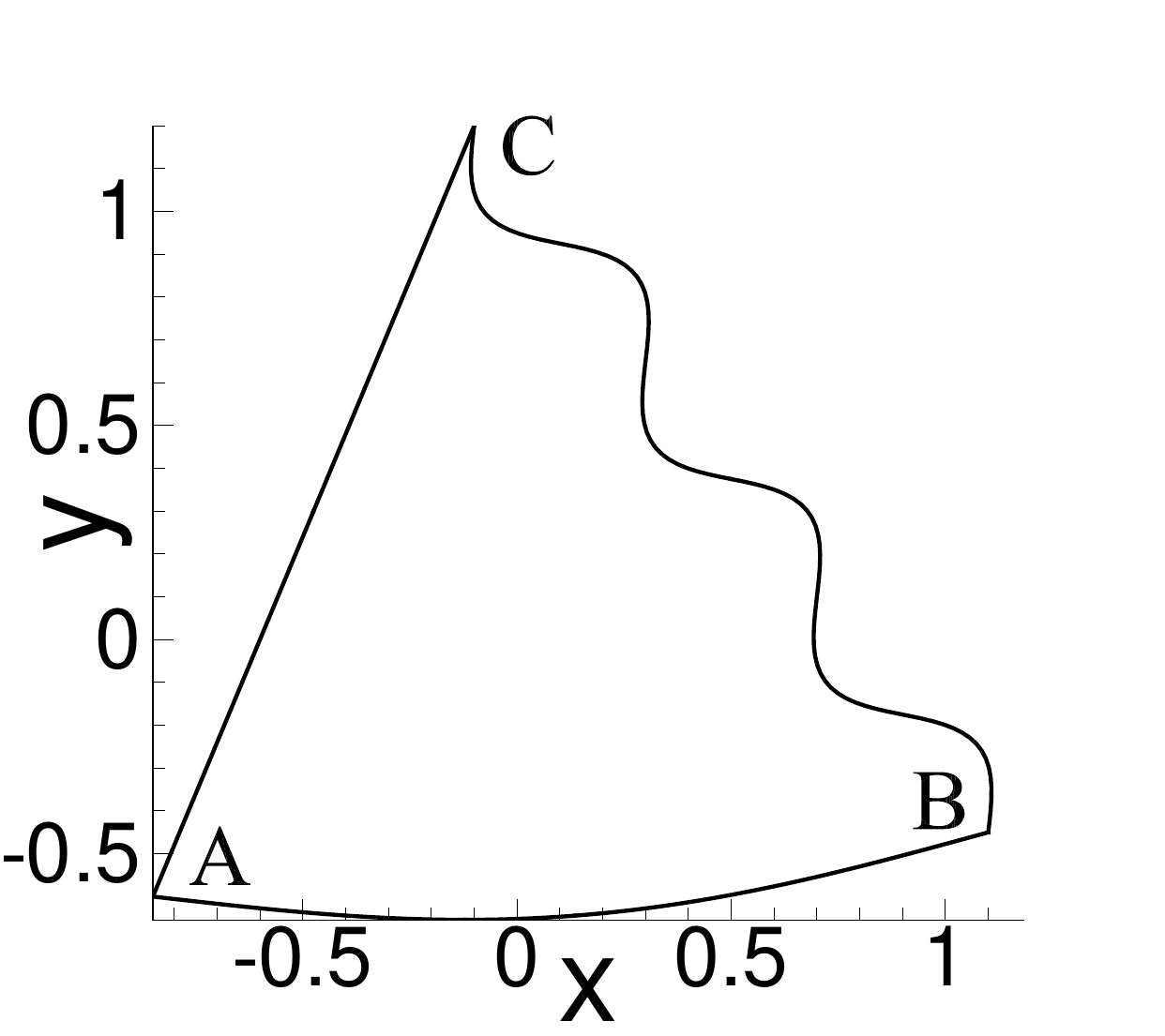}(d)
    \includegraphics[width=1.1in]{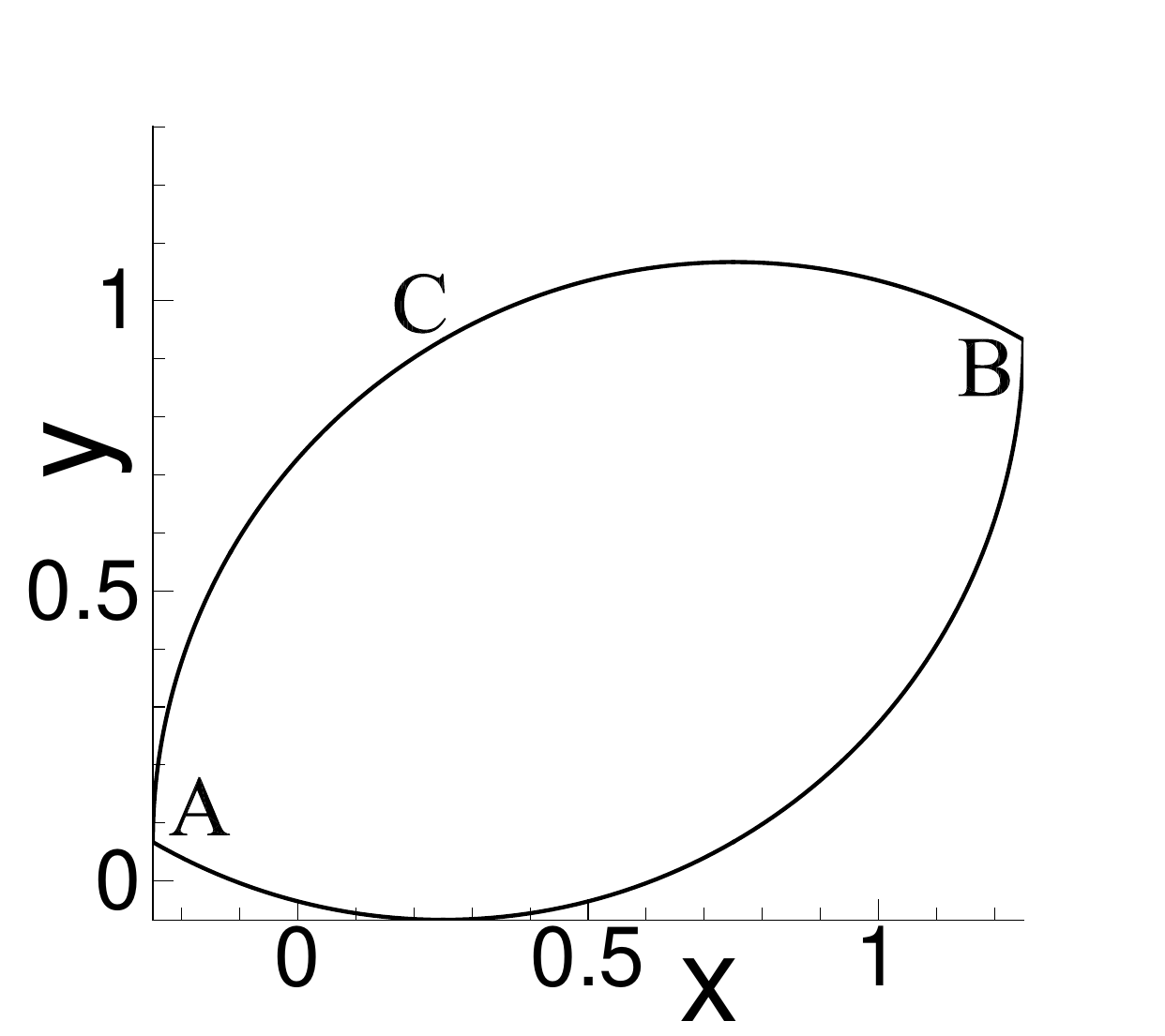}(e)
  }
  \centerline{
    \includegraphics[width=1.1in]{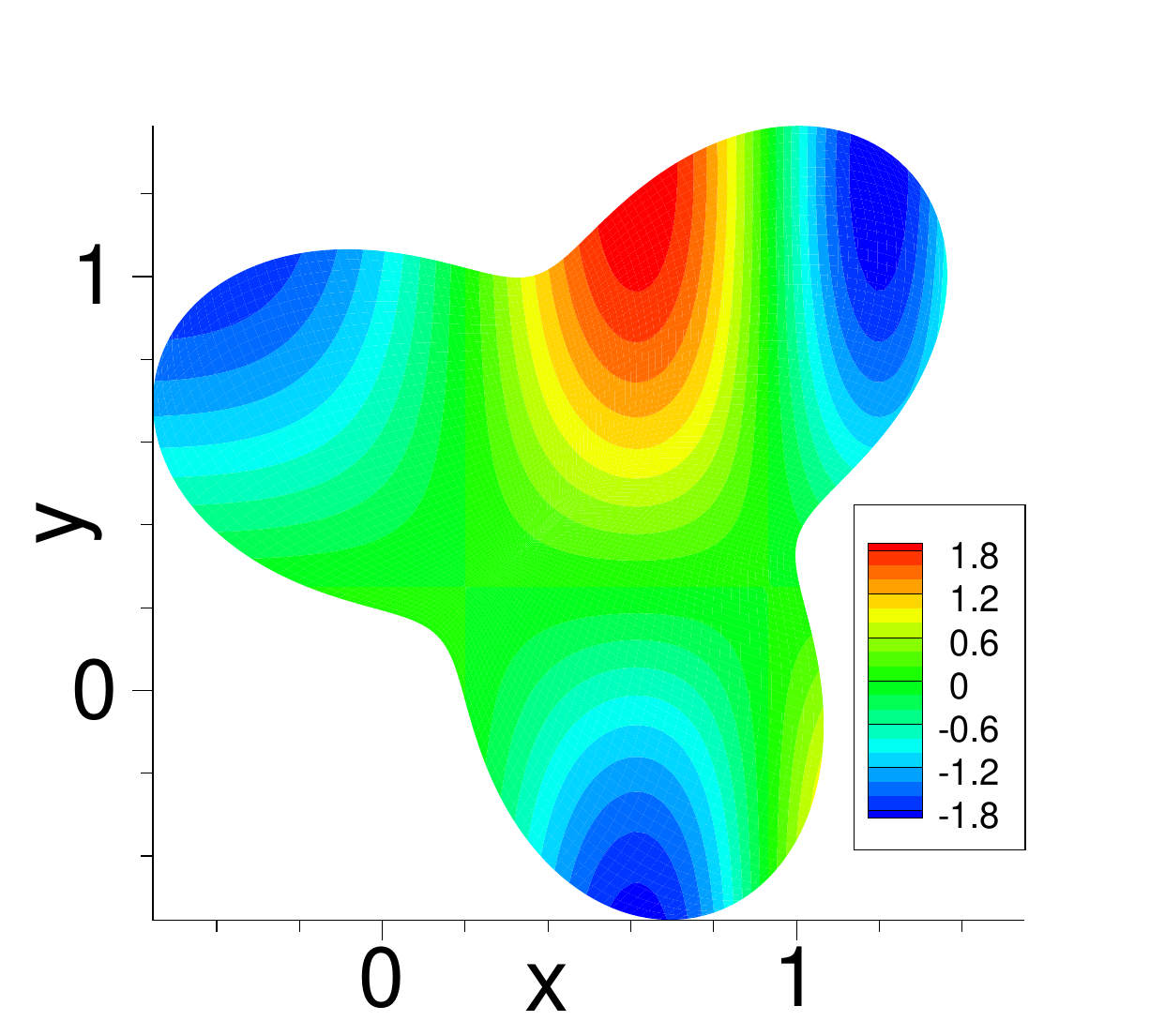}(f)
    \includegraphics[width=1.1in]{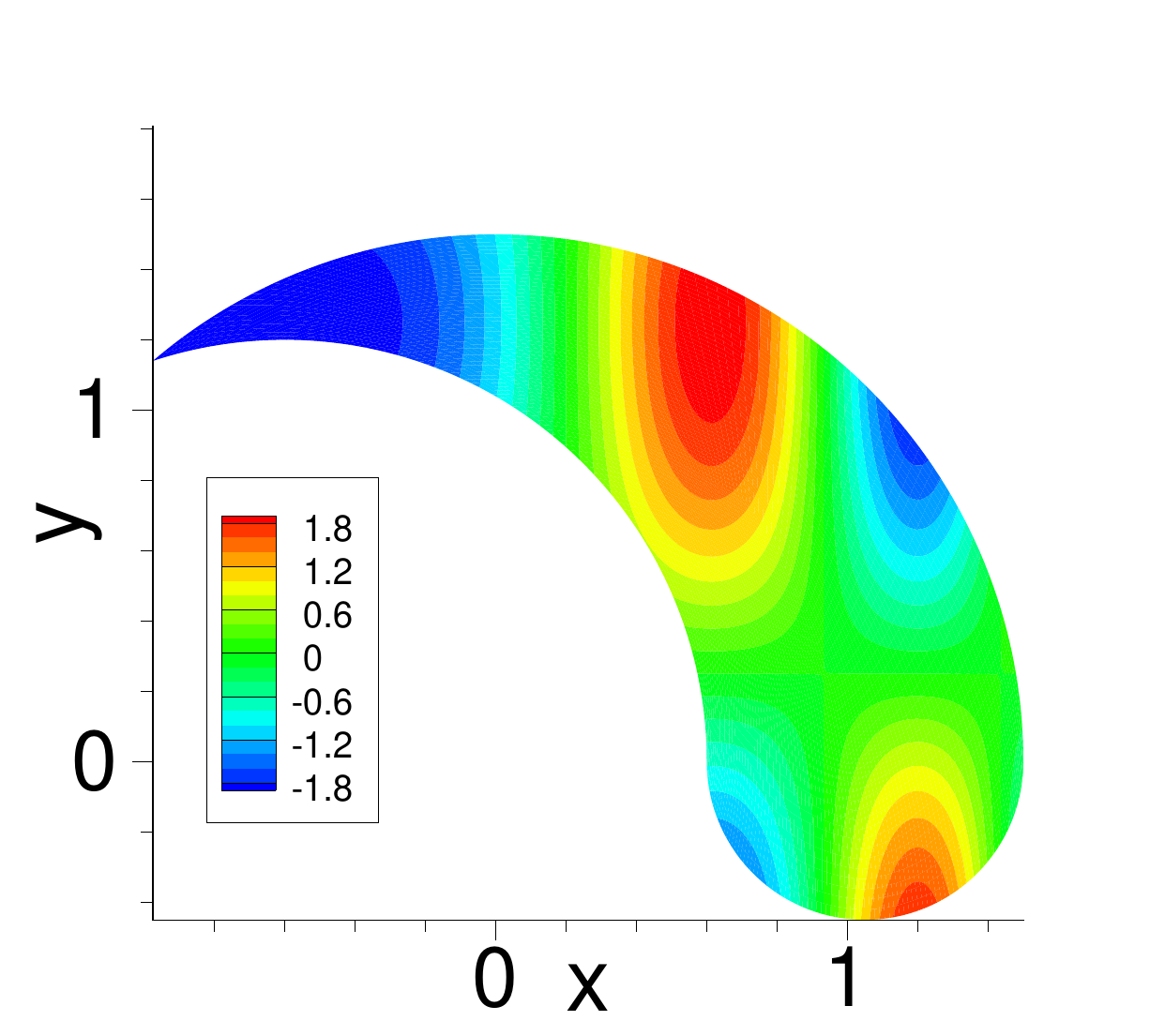}(g)
    \includegraphics[width=1.1in]{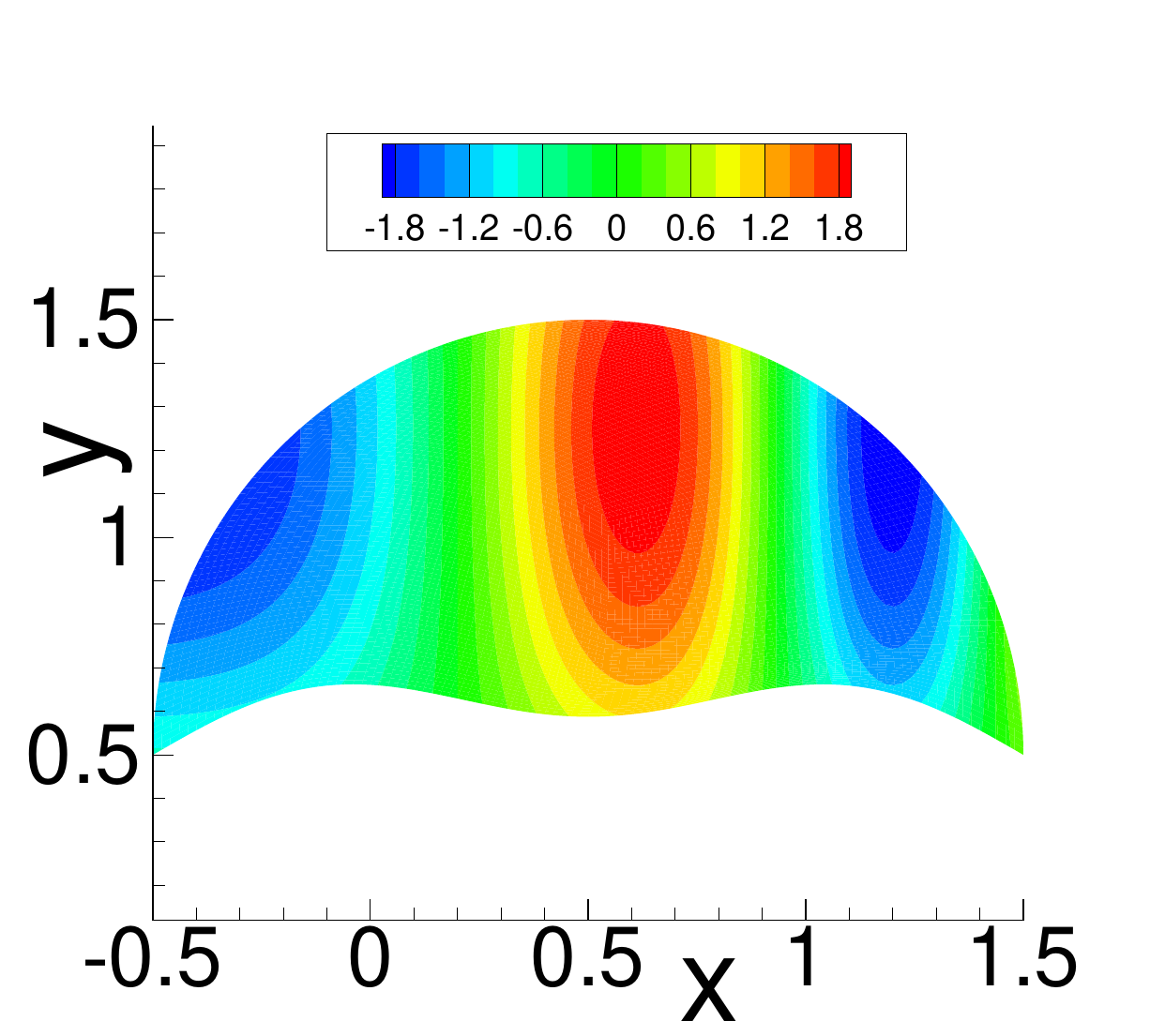}(h)
    \includegraphics[width=1.1in]{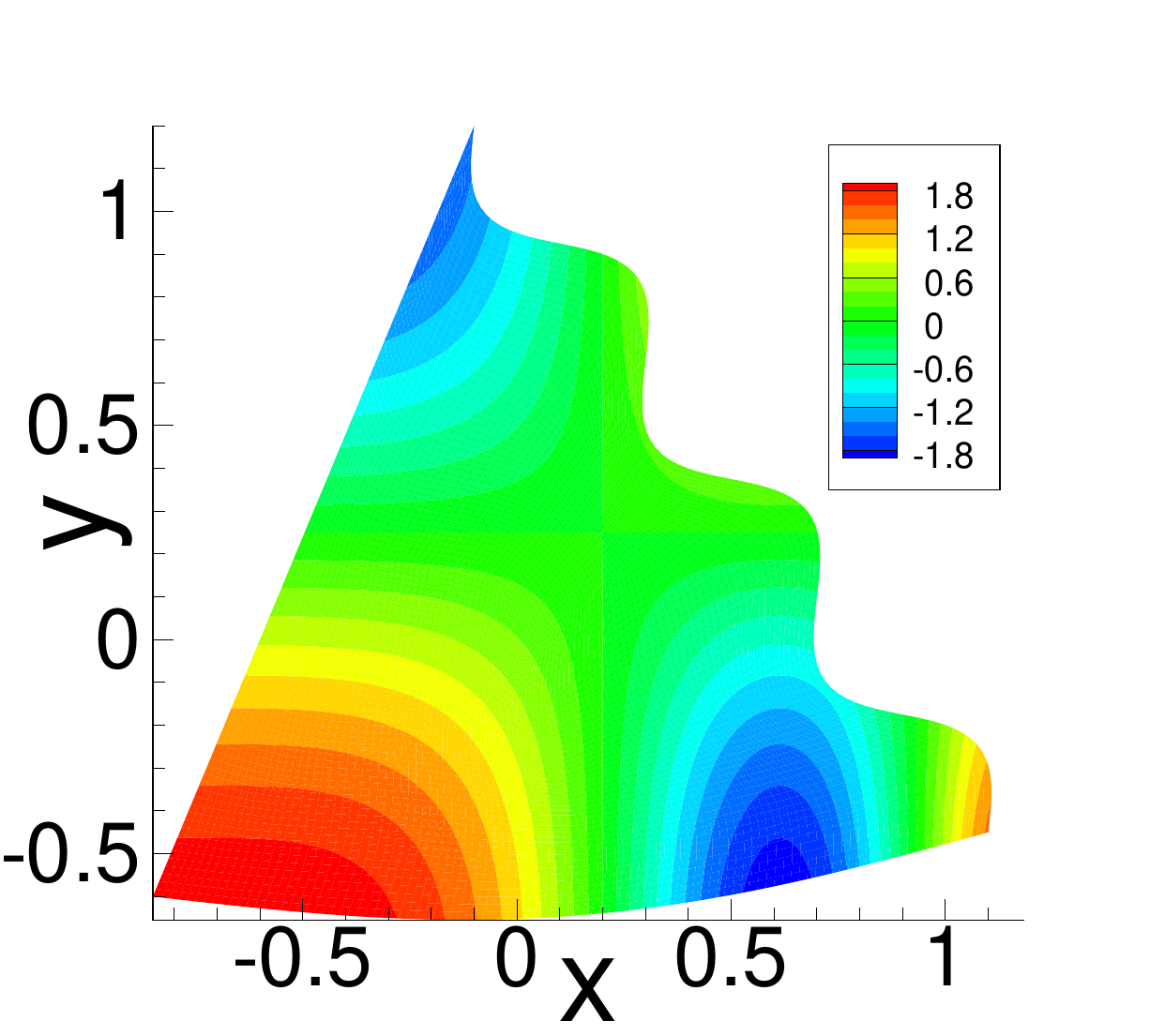}(i)
    \includegraphics[width=1.1in]{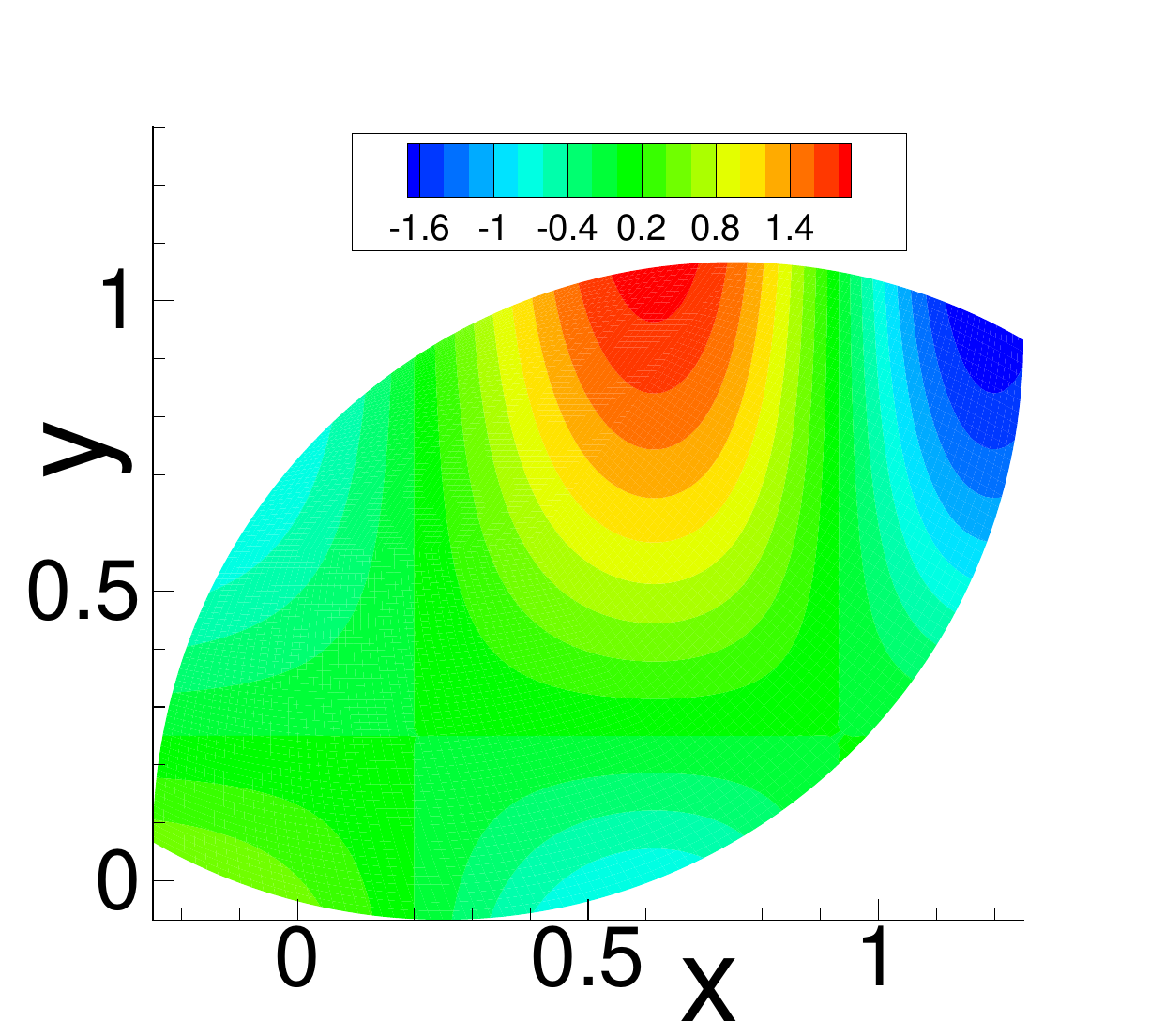}(j)
  }
  \caption{Nonlinear Helmholtz equation: Domain geometries (top row)
    and distributions of the exact solution (bottom row) for five domains:
    (a,f) domain \#1, (b,g) domain \#2, (c,h) domain \#3, (d,i) domain \#4,
    (e,j) domain \#5.
  }
  \label{fg_6}
\end{figure}

For the second test we employ 
the nonlinear Helmholtz equation and
a combination of Dirichlet/Neumann/Robin conditions on
the five domain geometries from Figure \ref{fg_6} (top row),
\begin{subequations}
\begin{align}
  & \frac{\partial^2 u}{\partial x^2} + \frac{\partial^2u}{\partial y^2} - 10u
  + \frac12\tanh(u) = f(x,y), \\
  & \left.\mathcal Bu\right|_{(x,y)\in\partial\Omega} = f_b(x,y),
\end{align}
\end{subequations}
where $u(x,y)$ is the unknown to be computed, $f$ and $f_b$ are prescribed source terms,
and $\mathcal B u$ denotes the boundary conditions as in the previous subsection.
The specific geometric parameters for these domains are provided in an Appendix
(Section \ref{sec_52}).
By choosing $f$ and $f_b$ appropriately, we have the following manufactured
exact solution to this problem
\begingroup\small\begin{equation}
  u(x,y) = 2\cos\left(\frac{\pi}{2}x^2+\frac{4\pi}{5} x + \frac{8\pi}{25} \right)
  \cos\left(\frac{\pi}{2}y+\frac{3\pi}{8} \right).
\end{equation}\endgroup
Distributions of this solution on the five domains are illustrated in
Figure \ref{fg_6} (plots in bottom row).

\begin{figure}
  \centerline{
    \includegraphics[width=1.1in]{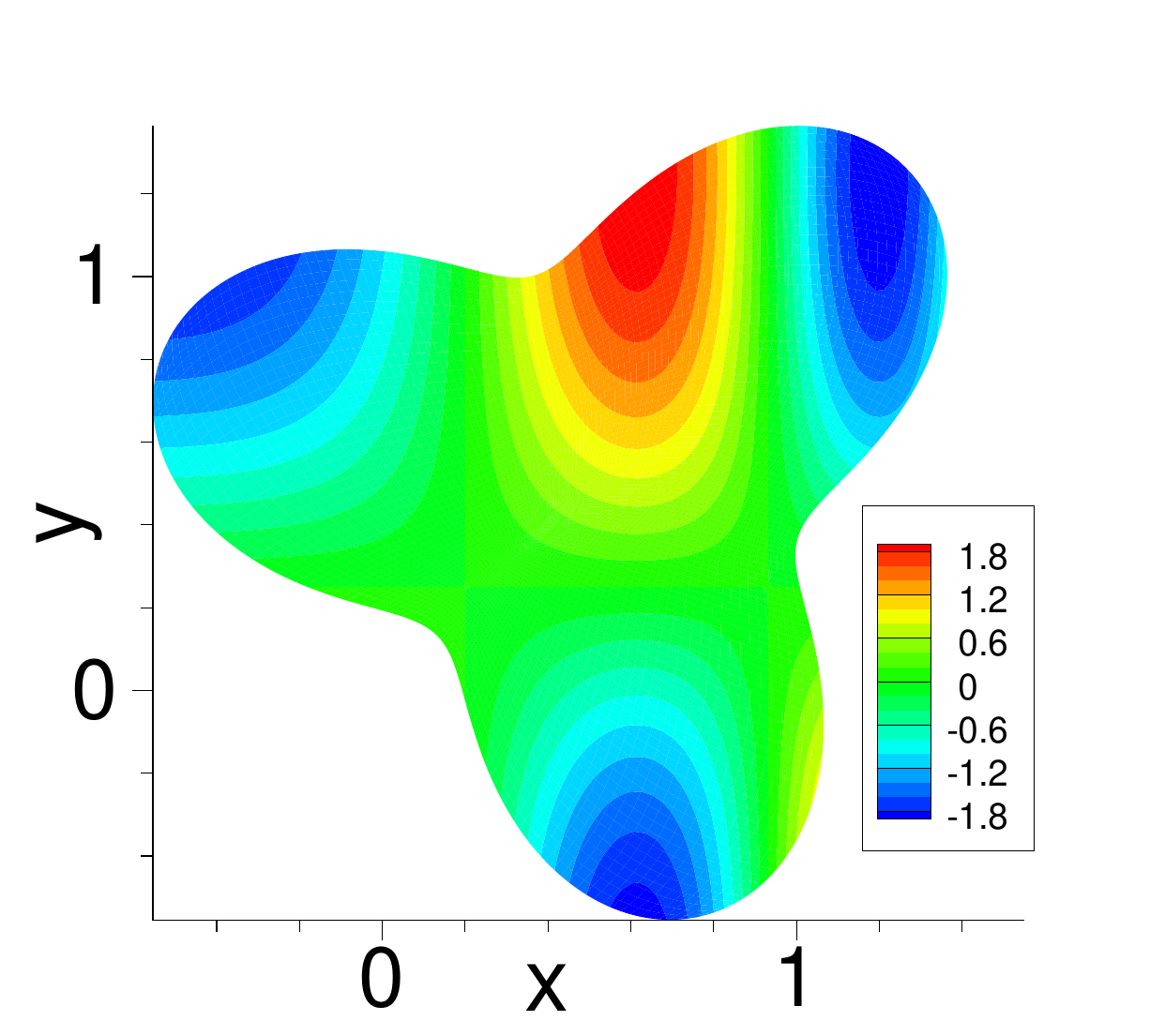}(a)
    \includegraphics[width=1.1in]{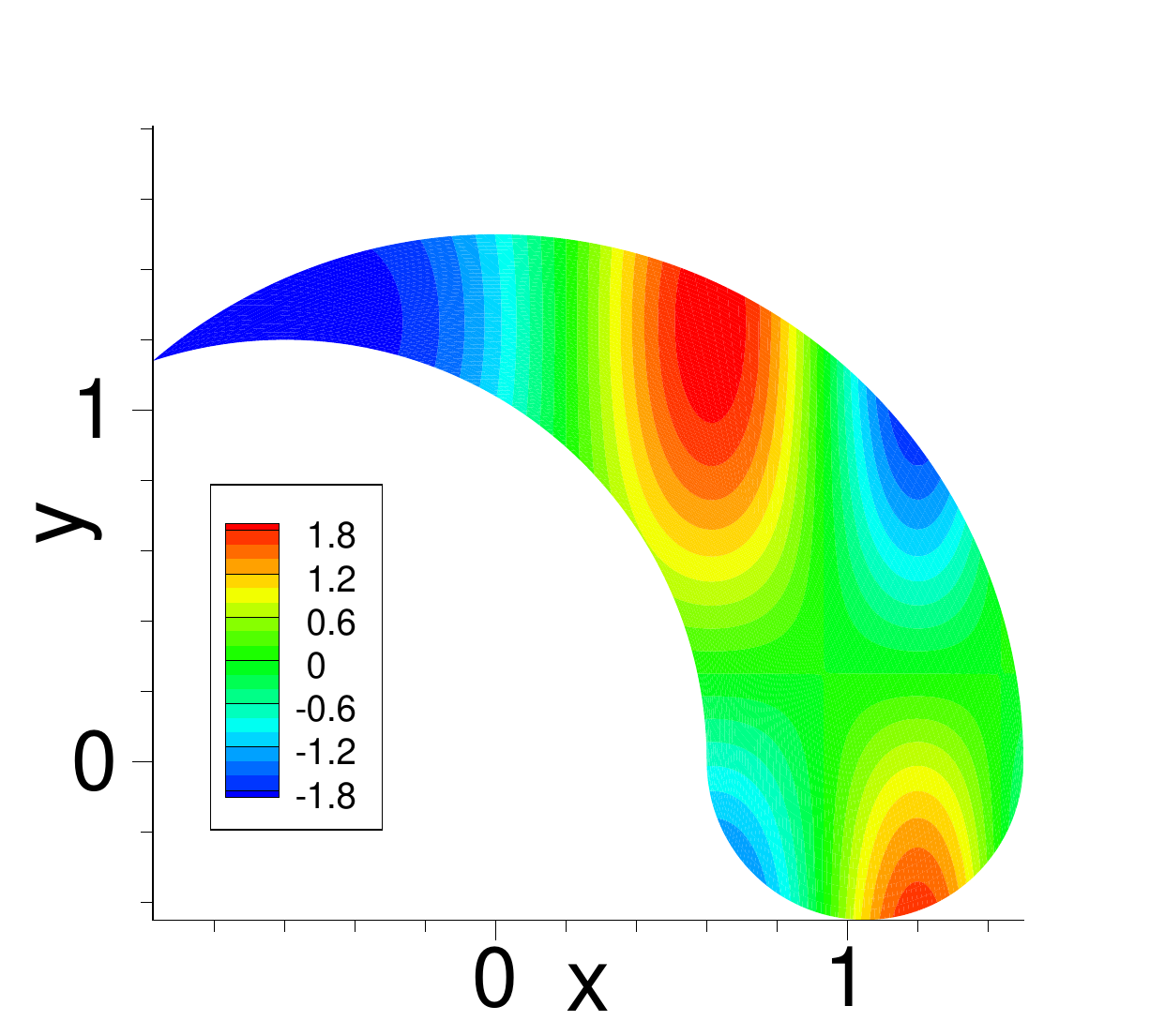}(b)
    \includegraphics[width=1.1in]{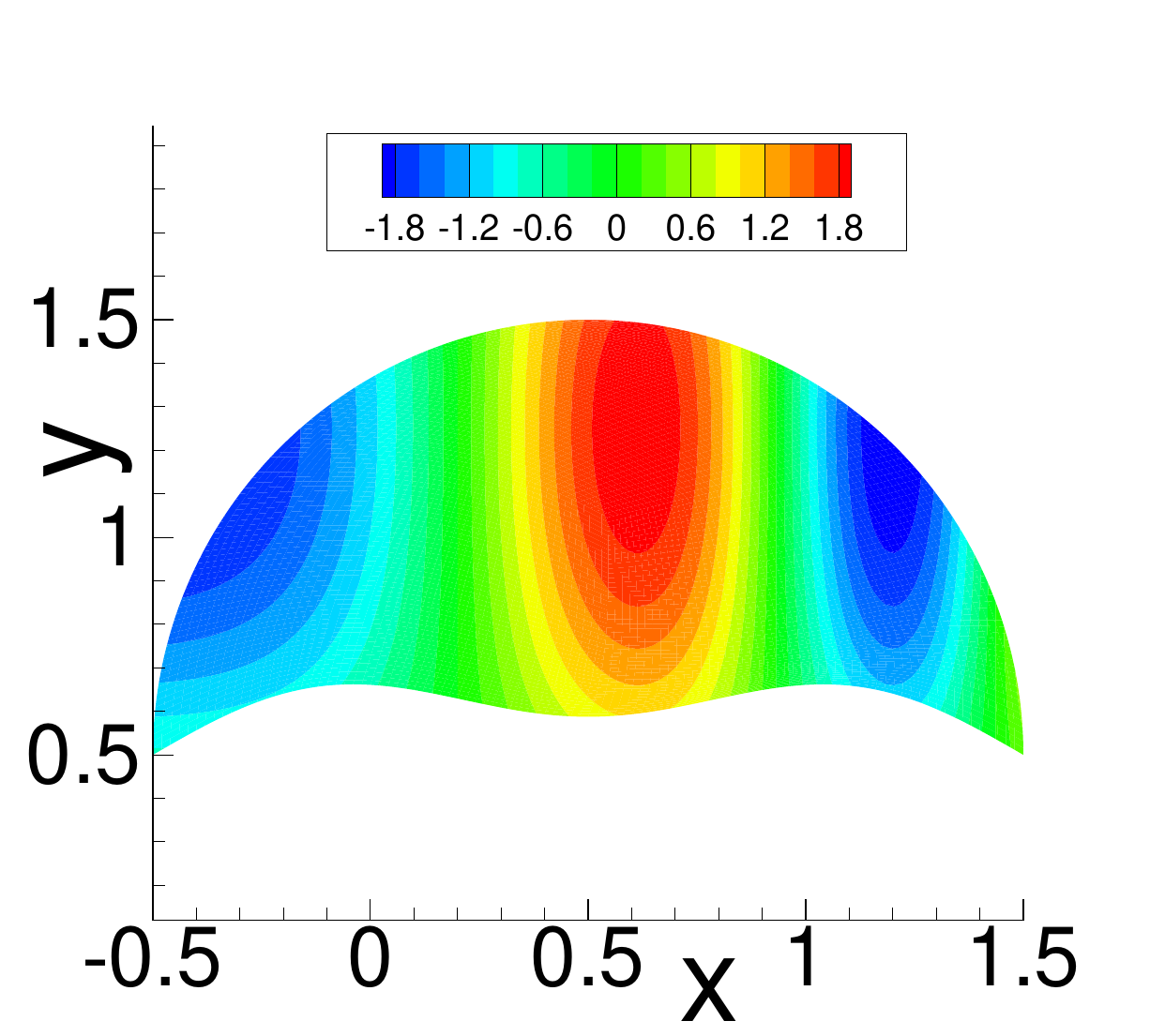}(c)
    \includegraphics[width=1.1in]{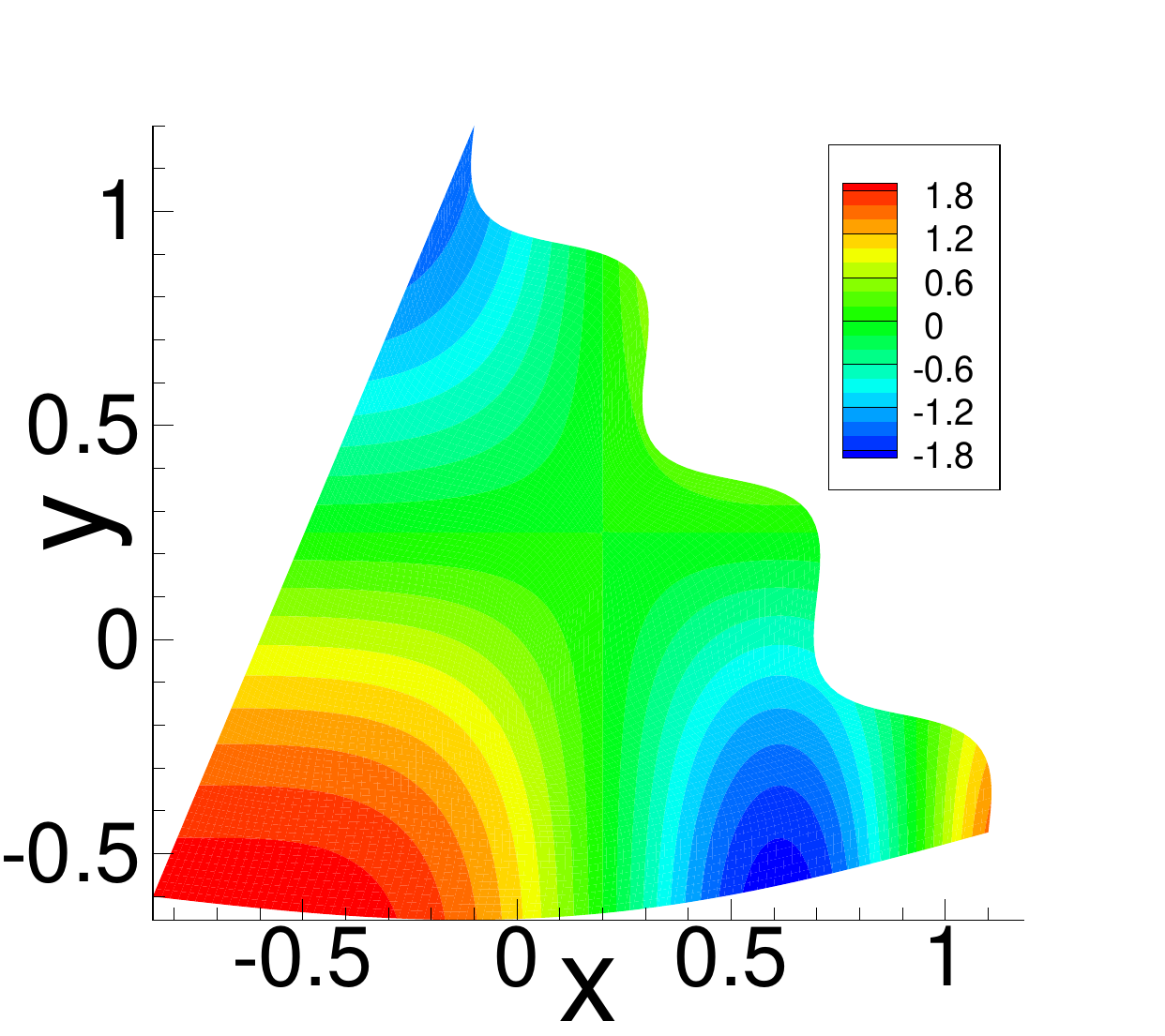}(d)
    \includegraphics[width=1.1in]{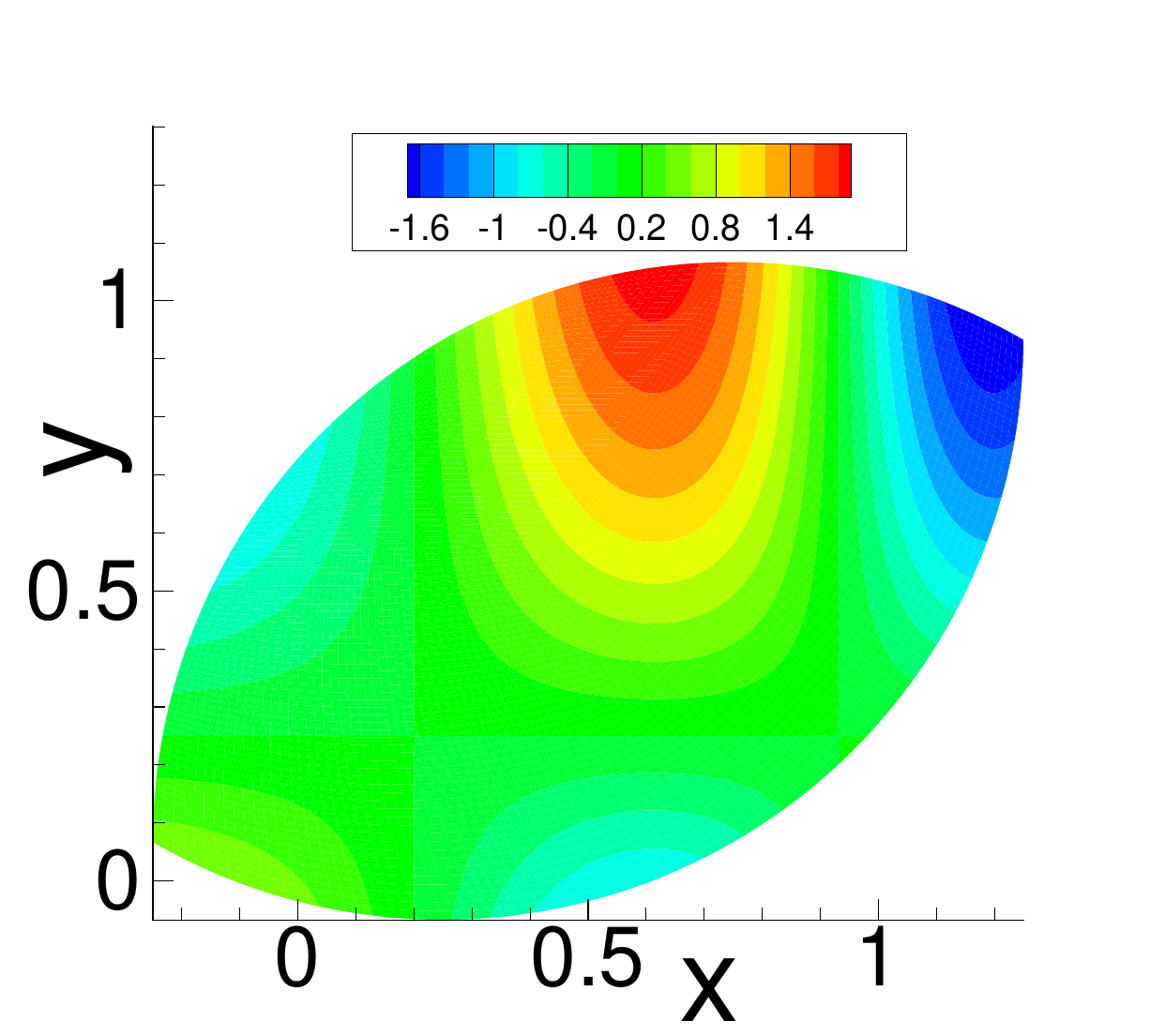}(e)
  }
  \centerline{
    \includegraphics[width=1.1in]{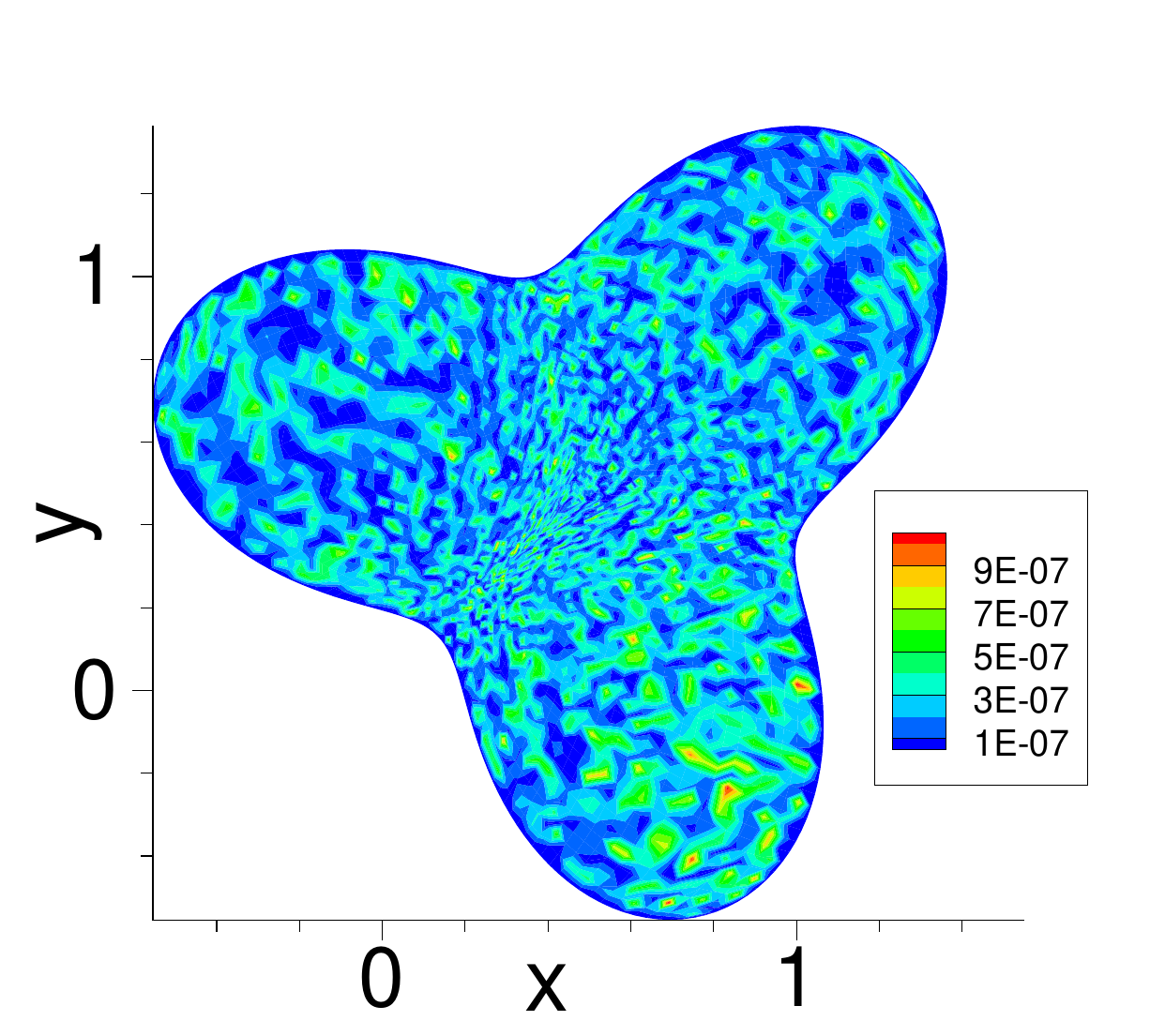}(f)
    \includegraphics[width=1.1in]{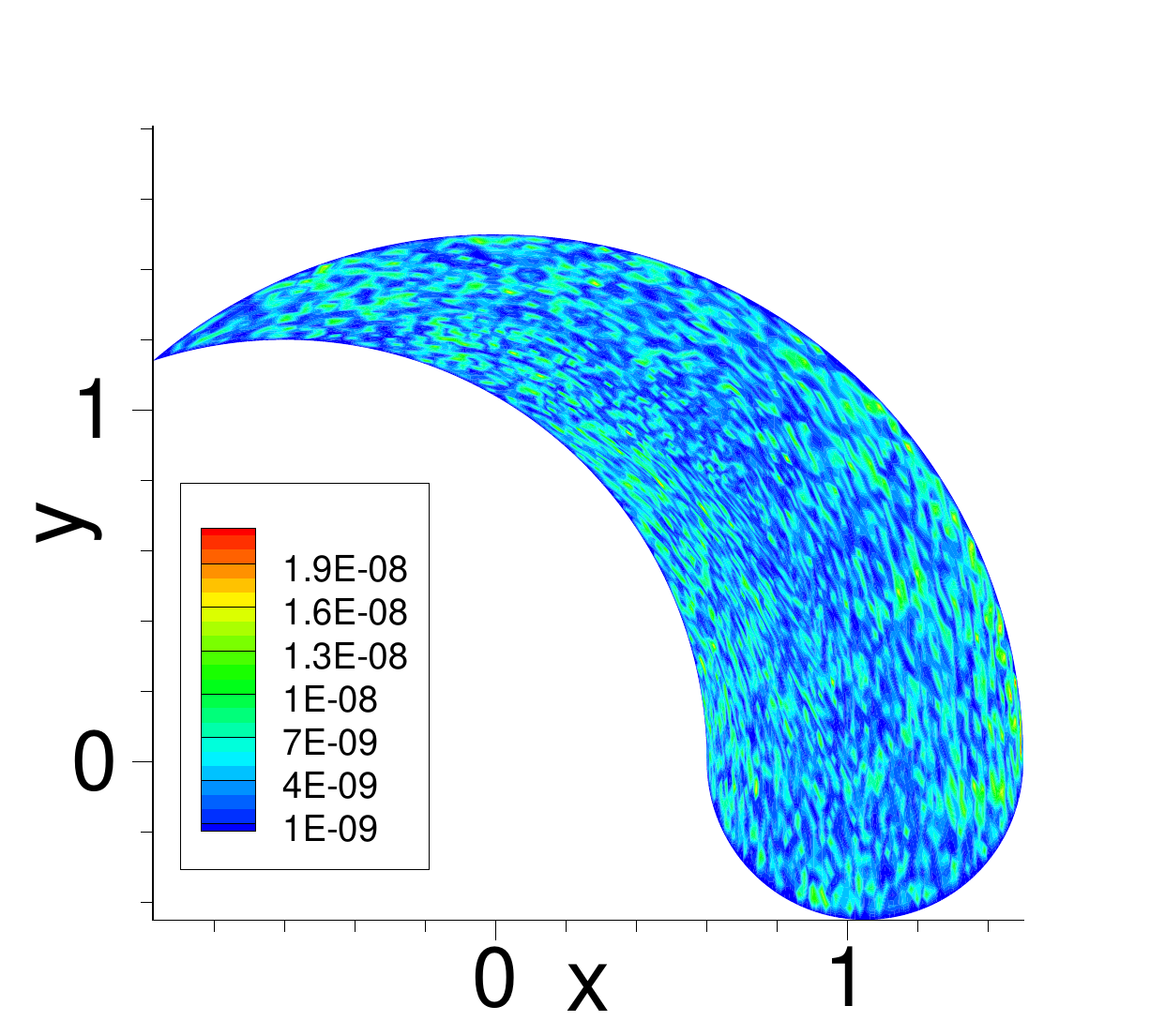}(g)
    \includegraphics[width=1.1in]{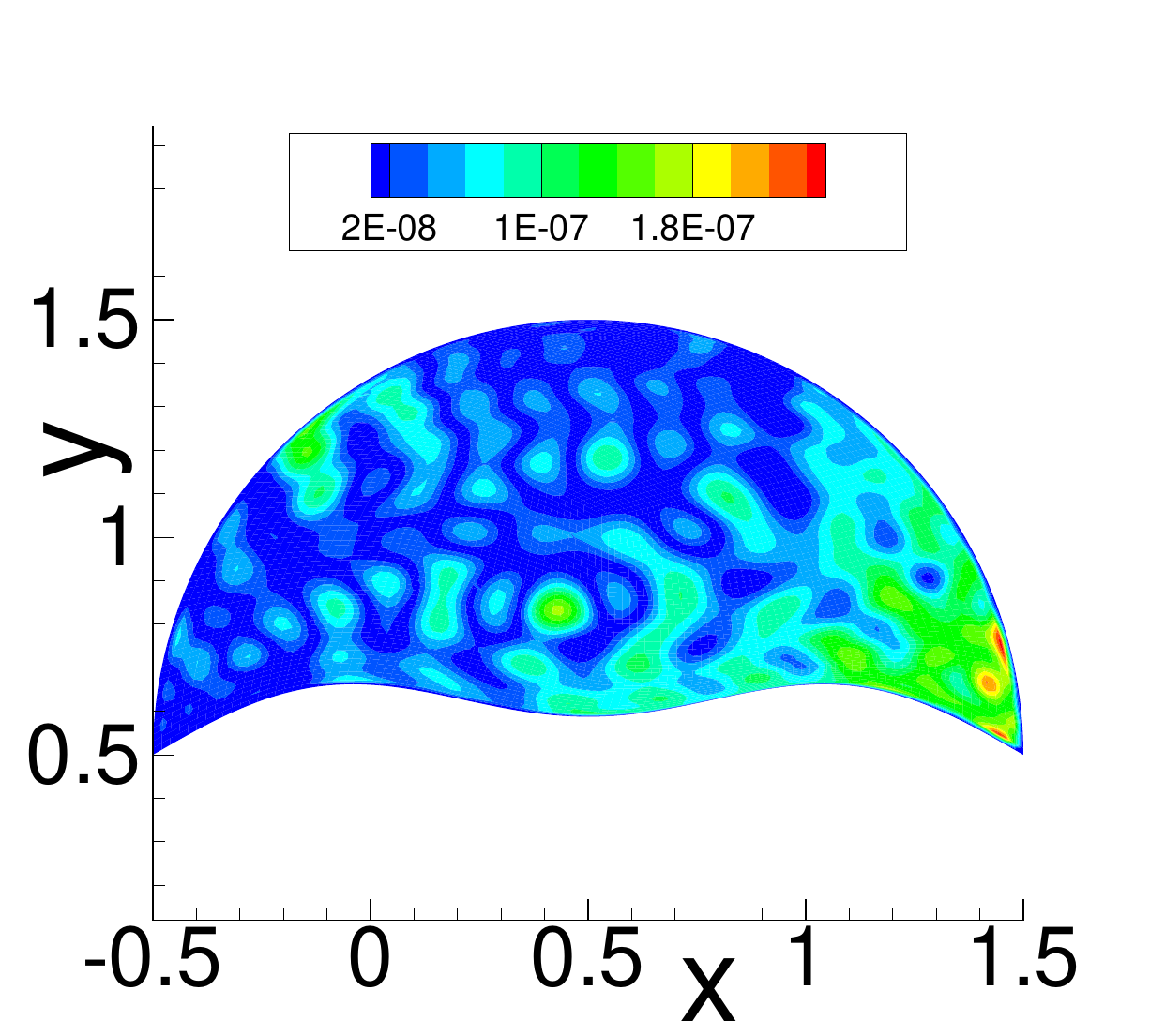}(h)
    \includegraphics[width=1.1in]{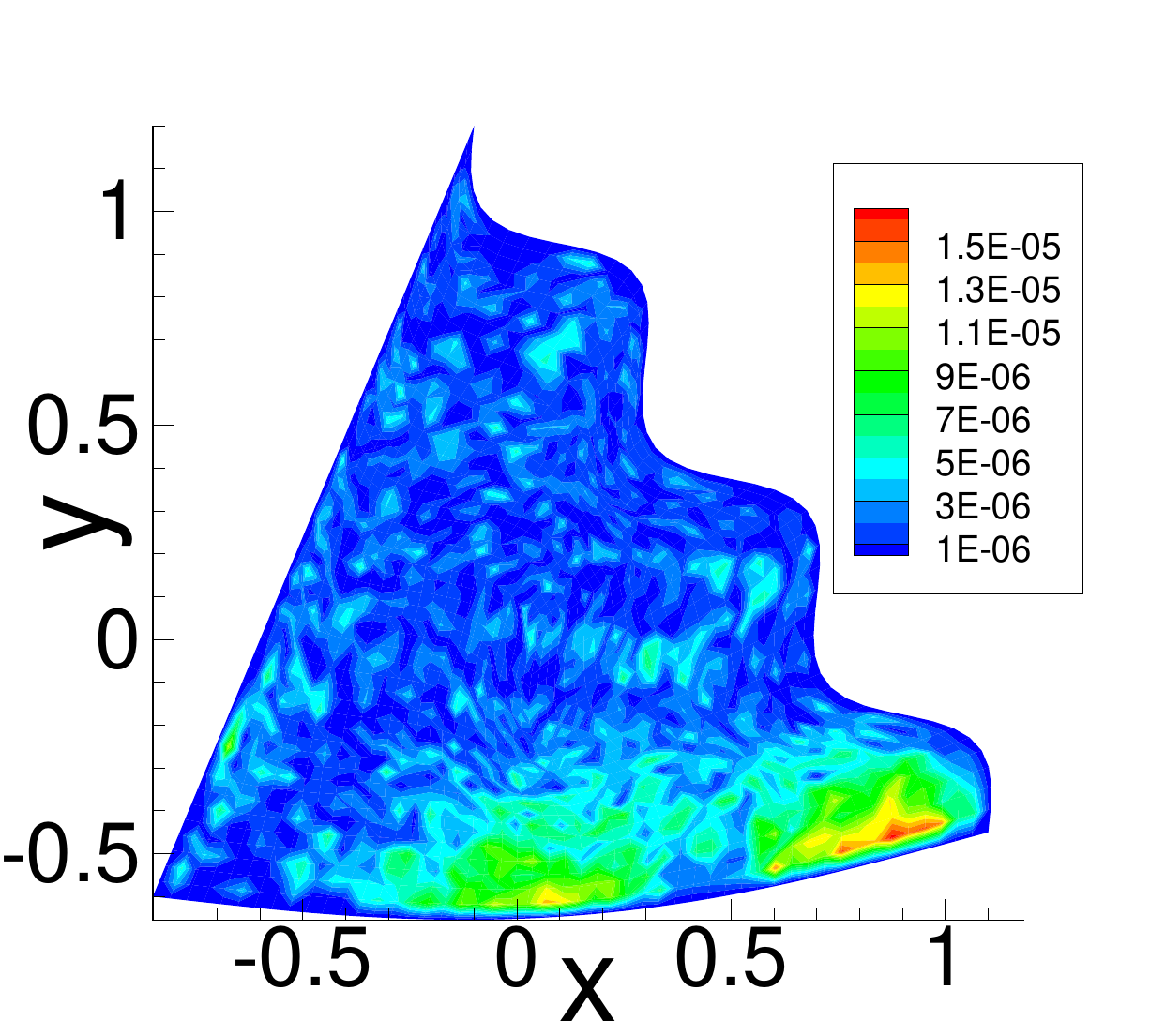}(i)
    \includegraphics[width=1.1in]{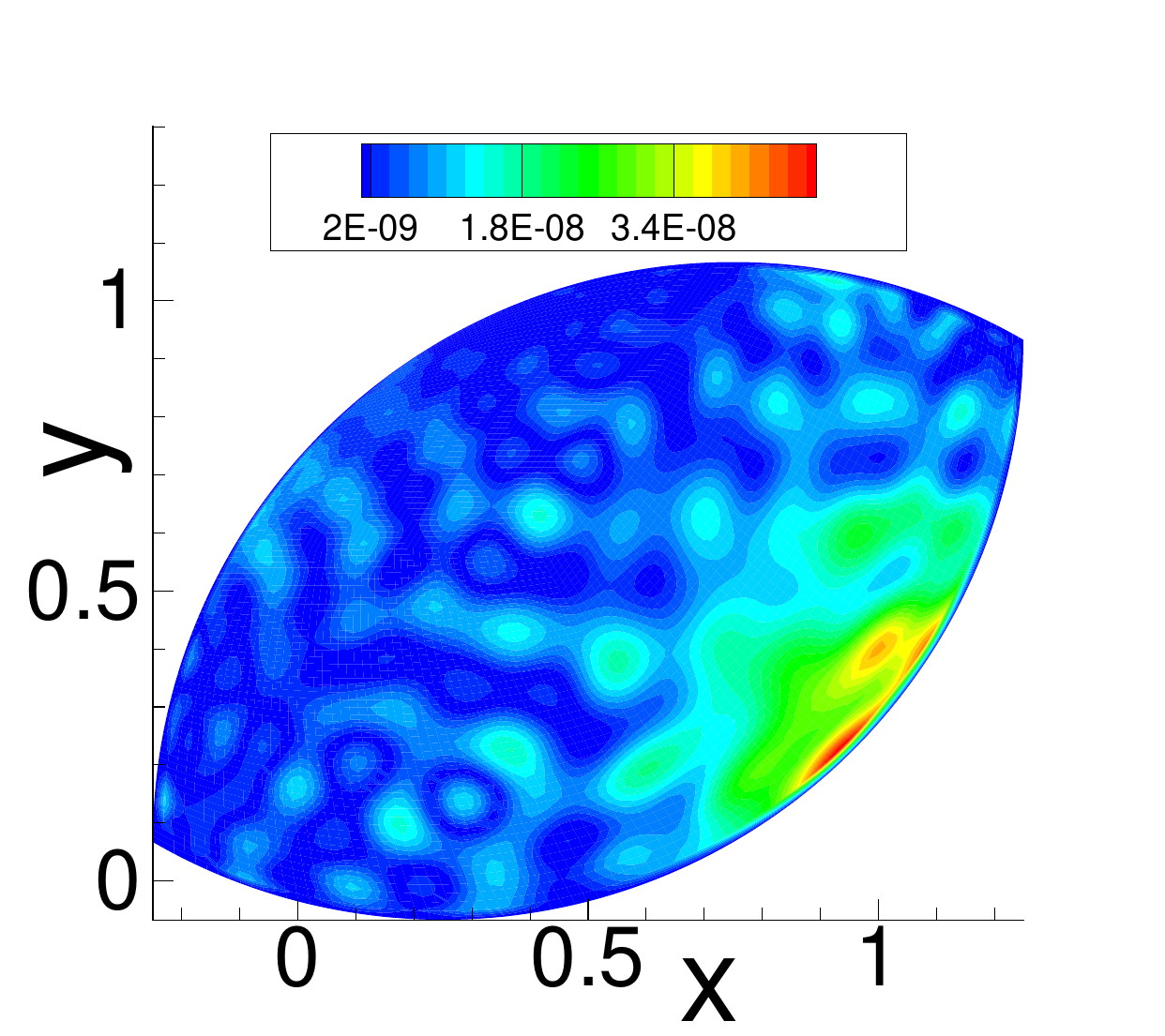}(j)
  }
  \caption{Nonlinear Helmholtz equation (DBC on all boundaries): Distributions of
    the NN solutions (top row) and their point-wise absolute errors (bottom row) on the five domains.
    Simulation parameters: (a,f) $R_m=3.9$, $Q=100$, $M=500$, Winslow map ($p_w=1$, $Q_w=100$);
    (b,g) $R_m=4.0$, $Q=100$, $M=500$, Winslow map ($p_w=1$, $Q_w=100$);
    (c,h) $R_m=4.0$, $Q=80$, $M=300$, simple map;
    (d,i) $R_m=4.8$, $Q=100$, $M=500$, simple map;
    (e,j) $R_m=4.0$, $Q=100$, $M=300$, simple map.
  }
  \label{fg_7}
\end{figure}

\begin{table}[tb]
  \centering\small
  \begin{tabular}{l|ccccc}
    \hline
    & Domain \#1 & Domain \#2 & Domain \#3 & Domain \#4 & Domain \#5 \\ \hline
    max $u$-error (domain) & $1.238\times 10^{-6}$ & $2.242\times 10^{-8}$ & $2.521\times 10^{-7}$
    & $1.887\times 10^{-5}$ & $4.867\times 10^{-8}$ \\
    rms $u$-error (domain) & $2.965\times 10^{-7}$ & $4.869\times 10^{-9}$ & $6.013\times 10^{-8}$
    & $3.586\times 10^{-6}$ & $9.449\times 10^{-9}$ \\ \hline
    max DBC-error ($\overline{AB}$) & $4.441\times 10^{-16}$ & $4.441\times 10^{-16}$ & $2.220\times 10^{-16}$ & $4.441\times 10^{-16}$ & $2.220\times 10^{-16}$ \\
    rms DBC-error ($\overline{AB}$) & $1.340\times 10^{-16}$ & $1.592\times 10^{-16}$ & $6.438\times 10^{-17}$ & $1.689\times 10^{-16}$ & $9.541\times 10^{-17}$ \\ \hline
    max DBC-error ($\overline{BC}$) & $5.130\times 10^{-16}$ & $6.245\times 10^{-16}$ & $3.192\times 10^{-16}$ & $5.551\times 10^{-16}$ & $4.441\times 10^{-16}$ \\
    rms DBC-error ($\overline{BC}$) & $1.372\times 10^{-16}$ & $1.277\times 10^{-16}$ & $1.010\times 10^{-16}$ & $1.420\times 10^{-16}$ & $1.252\times 10^{-16}$  \\ \hline
    max DBC-error ($\overline{AC}$) & $4.441\times 10^{-16}$ & $4.441\times 10^{-16}$ & $2.220\times 10^{-16}$ & $2.220\times 10^{-16}$ & $1.110\times 10^{-16}$ \\
    rms DBC-error ($\overline{AC}$) & $1.813\times 10^{-16}$ & $1.756\times 10^{-16}$ & $1.339\times 10^{-16}$ & $1.433\times 10^{-16}$ & $2.662\times 10^{-17}$ \\
    \hline
  \end{tabular}
  \caption{Nonlinear Helmholtz equation (DBC on all boundaries):
    Maximum/rms domain errors of the NN solution
    and the maximum/rms DBC errors on the three boundaries. Simulation parameters
    follow those of Figure \ref{fg_7}.
  }
  \label{tab_4}
\end{table}

The NN solution accuracy and the boundary-condition accuracy,
with Dirichlet conditions imposed on all boundaries, are
illustrated in Figure \ref{fg_7} and Table \ref{tab_4}.
Figure \ref{fg_7} plots contours of the NN solutions (top row)
and their point-wise absolute errors (bottom row) on the five domains.
The figure caption lists the NN and other simulation parameters, including
the Winslow parameters, when applicable.
On domains \#1 and \#2 the Winslow map \eqref{eq_12} has been used because the simple
map \eqref{eq_8} results in ``foldings'' and loses univalency,
while on domains \#3 to \#5 the simple map has been used in the simulations.
The NN method combined with exact BC enforcement
produces accurate  results. The level of the maximum domain errors
ranges from $10^{-5}$ for domain \#4 to $10^{-8}$ for domains \#2 and \#5. 

Table \ref{tab_4} lists the maximum/rms domain errors of the NN solution
and the maximum/rms boundary-condition (DBC) errors on different boundaries
for these five domains. The maximum DBC errors are on the order of $10^{-16}$
for all three boundaries on all five domains, indicating that our
method has enforced the Dirichlet conditions 
to the machine precision for this nonlinear problem.


\begin{figure}
  \centerline{
    \includegraphics[width=1.2in]{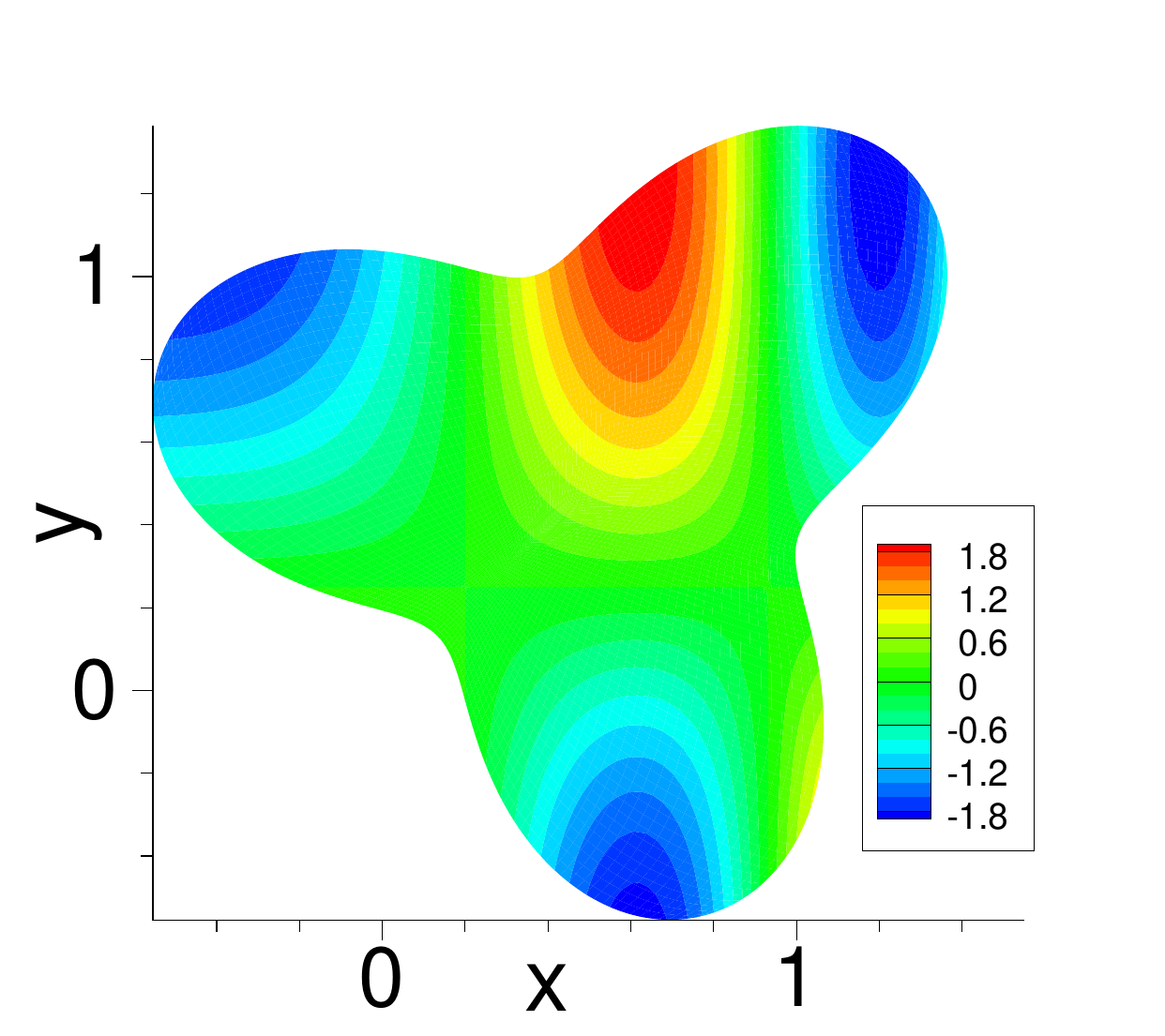}(a)
    \includegraphics[width=1.2in]{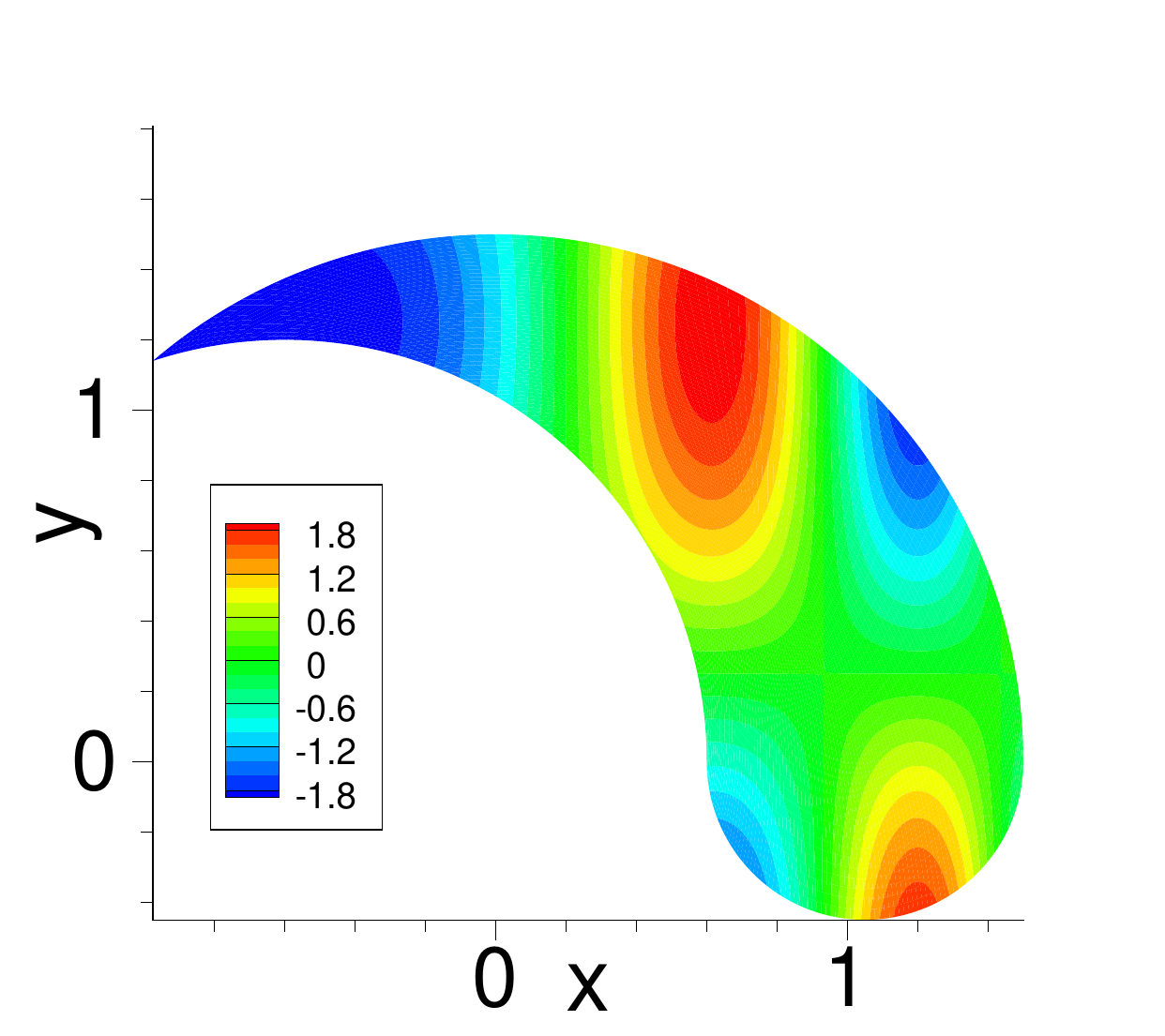}(b)
    \includegraphics[width=1.2in]{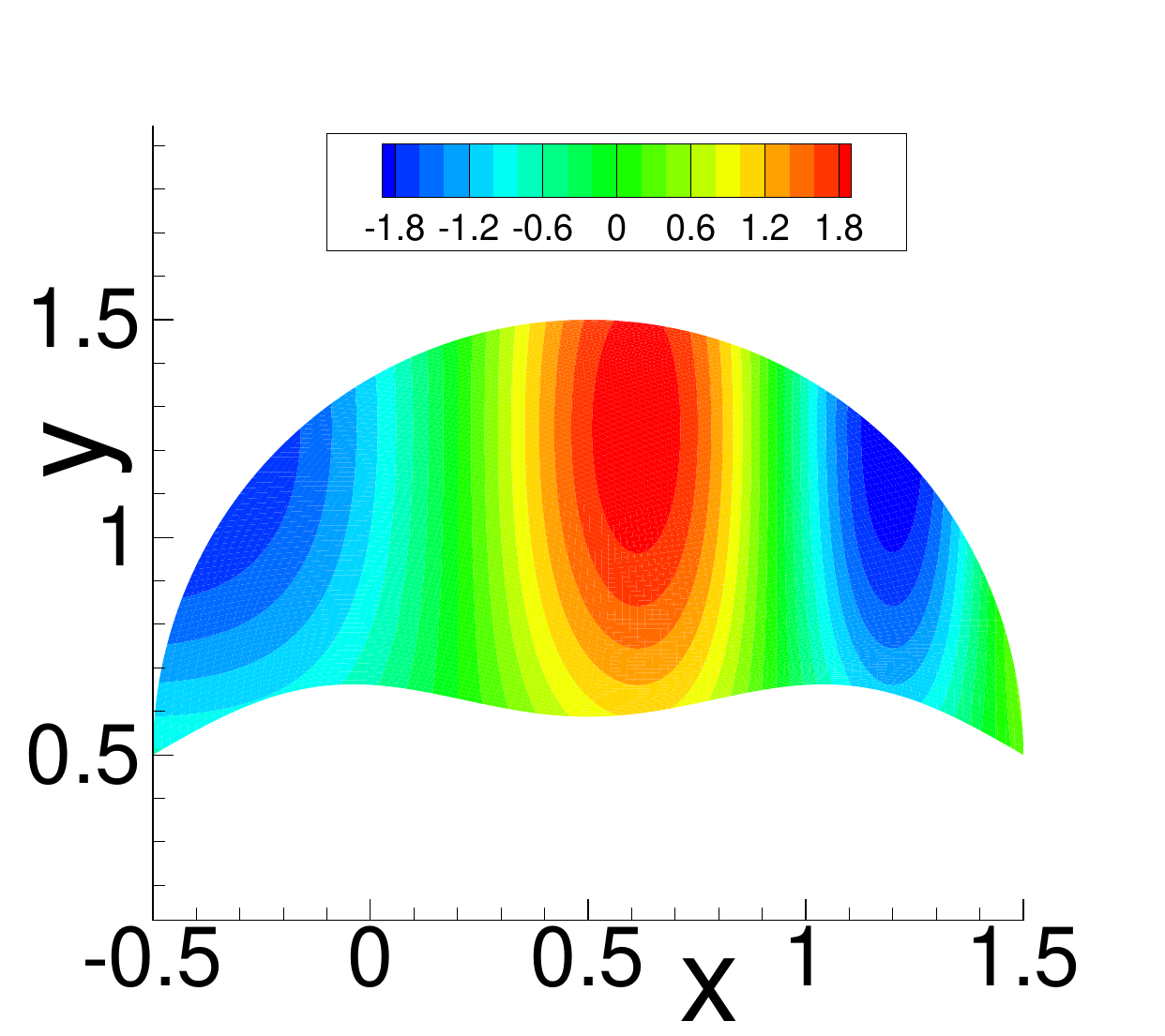}(c)
    \includegraphics[width=1.2in]{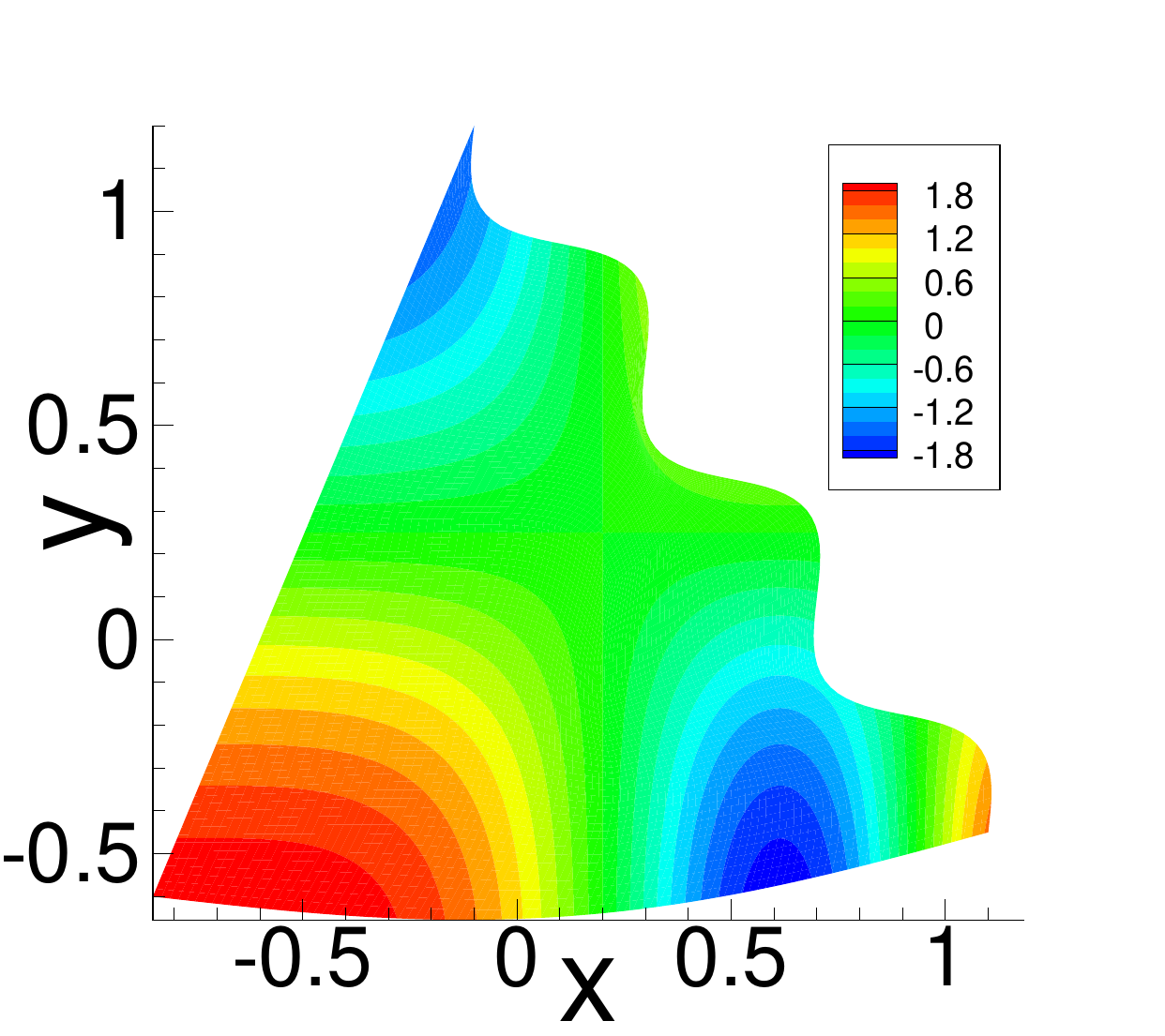}(d)
  }
  \centerline{
    \includegraphics[width=1.2in]{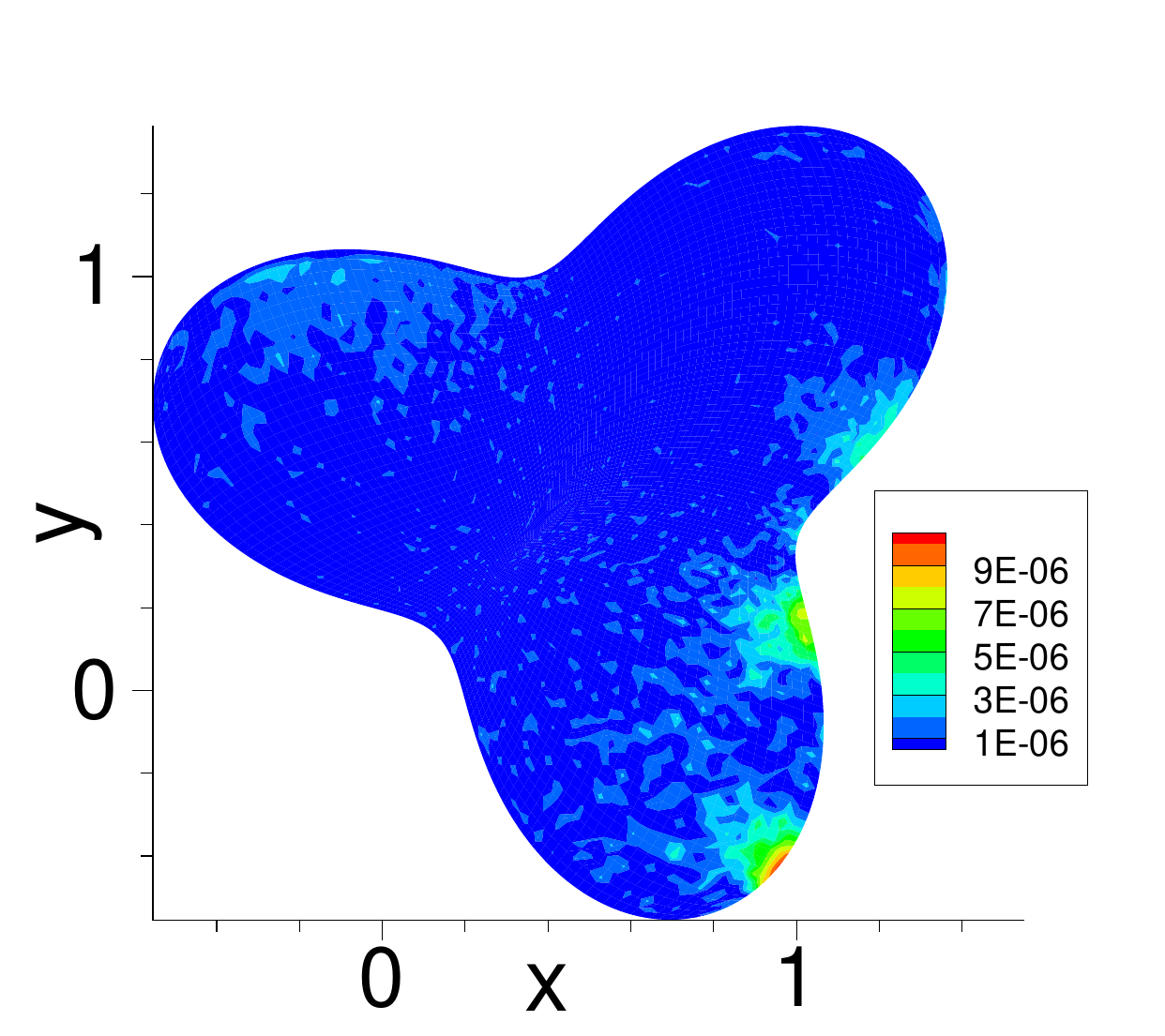}(e)
    \includegraphics[width=1.2in]{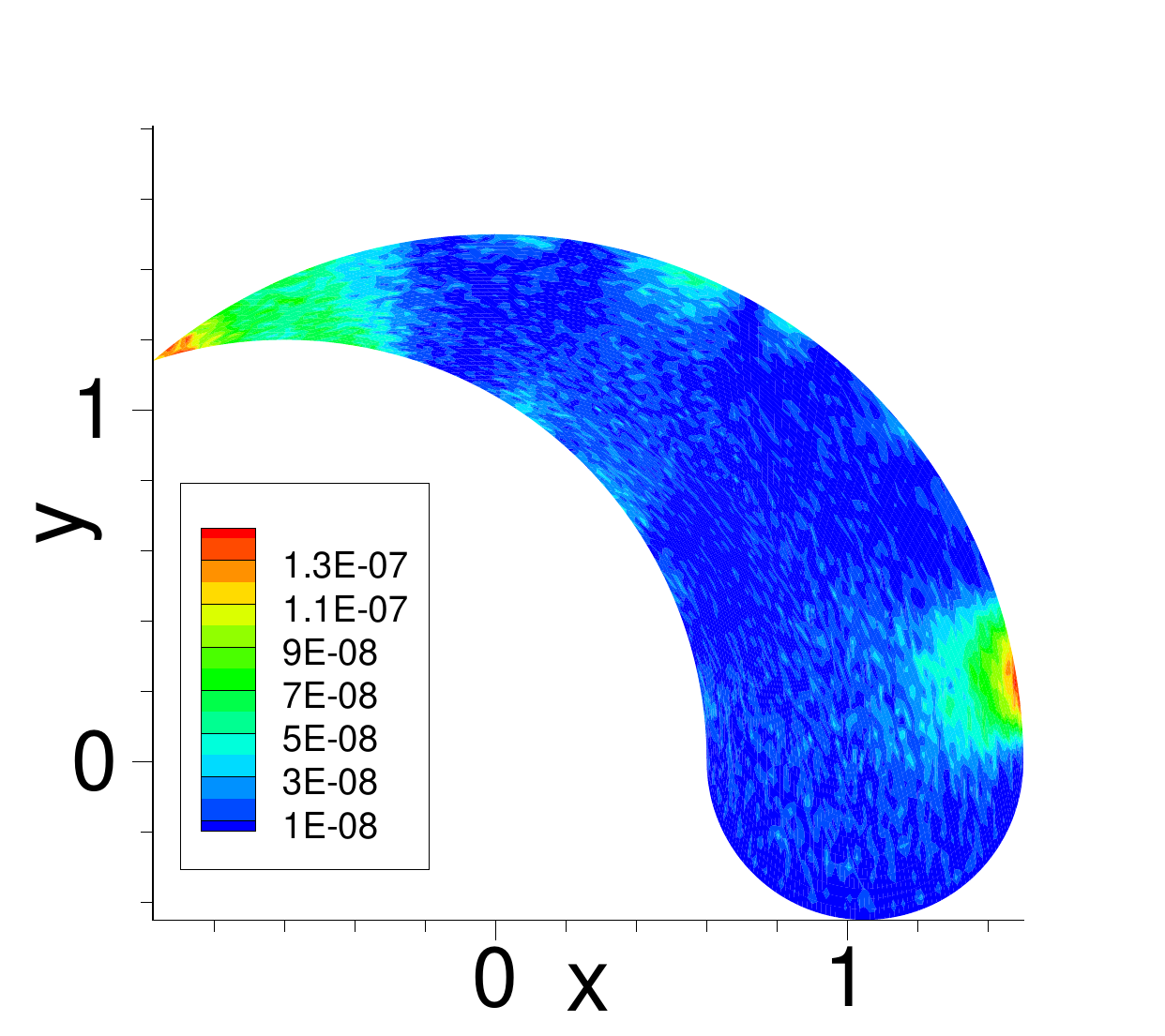}(f)
    \includegraphics[width=1.2in]{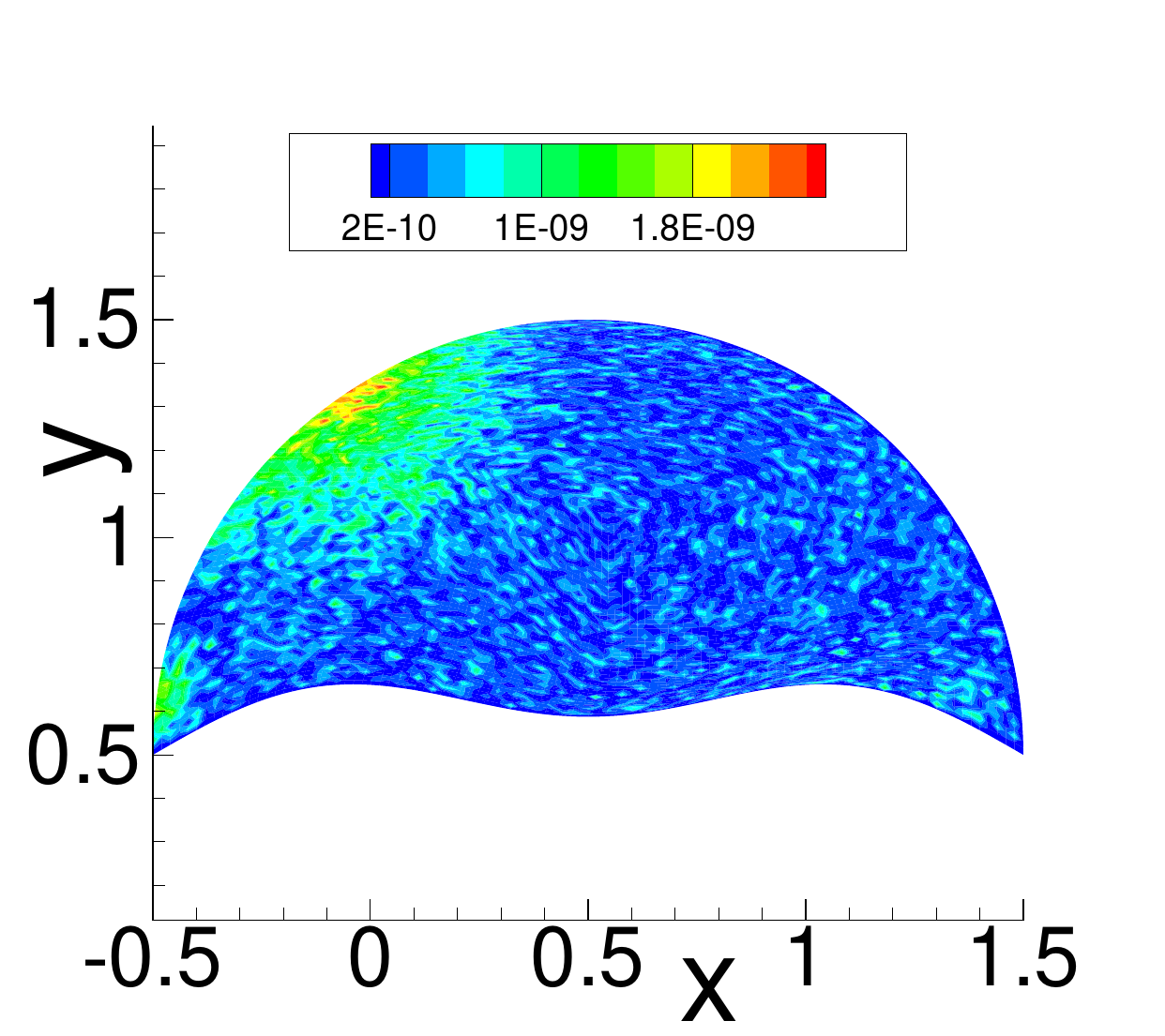}(g)
    \includegraphics[width=1.2in]{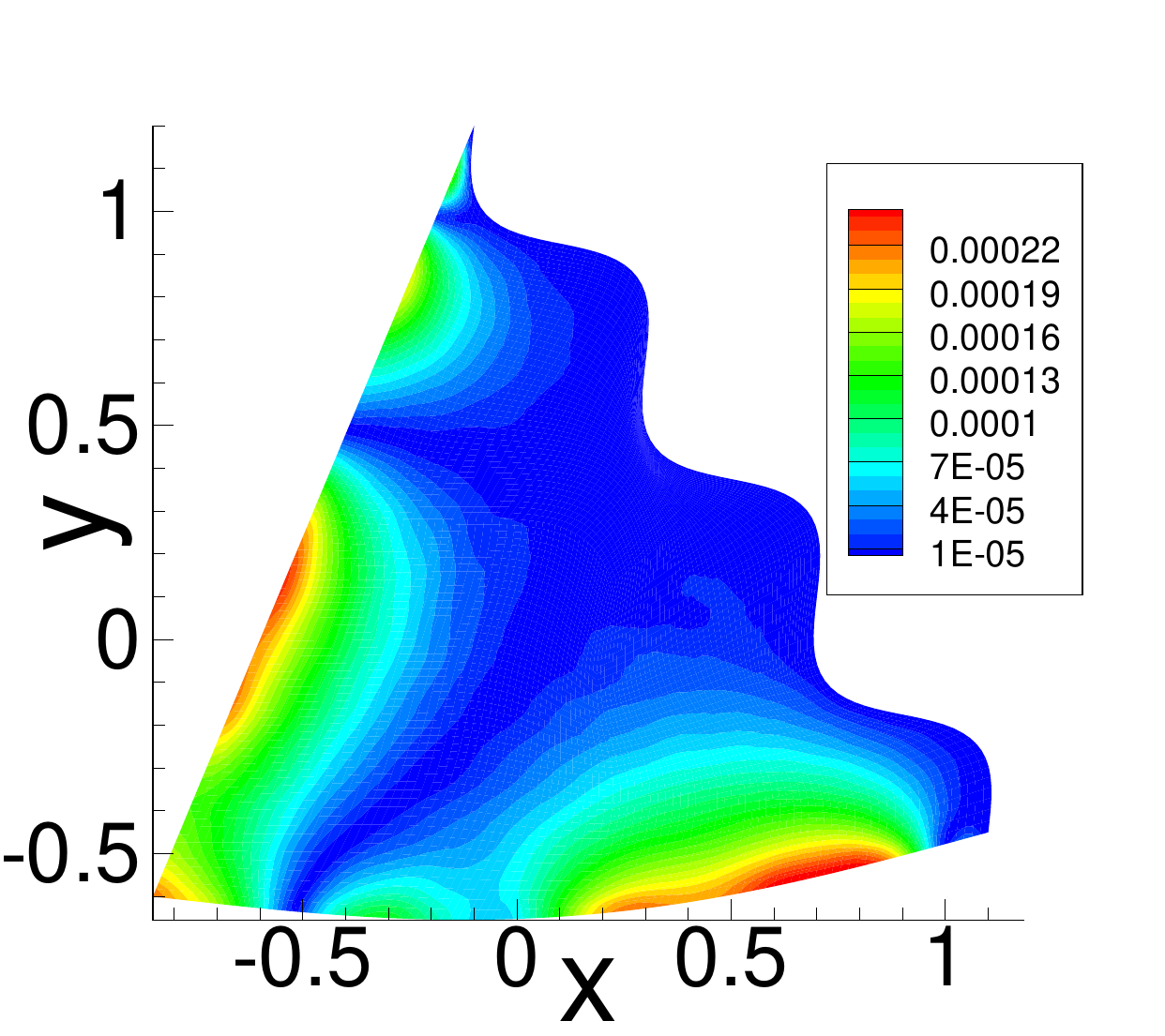}(h)
  }
  \caption{Nonlinear Helmholtz equation (NBCs/RBCs on boundaries):
    Distributions of the NN solutions (top row) and their point-wise absolute errors (bottom row)
    with domains \#1 to \#4 under four settings.
    (a,e) Case \#1:  NBC on $\overline{AC}$ and DBC on the other boundaries with domain \#1;
    (b,f) Case \#2: NBC on $\overline{AB}$ and $\overline{AC}$, and DBC on $\overline{BC}$, with
    domain \#2;
    (c,g) Case \#3: RBC on $\overline{AC}$ with $\alpha_{AC}=1.0$, and DBC on the other boundaries,
    with domain \#3;
    (d,h) Case \#4: RBC on $\overline{AB}$ and $\overline{AC}$ with
    $\alpha_{AB}=0.5$ and $\alpha_{AC}=1.0$, and DBC on $\overline{BC}$, with domain \#4;
    Simulation parameters: (a,e) $R_m=3.5$, $Q=100$, $M=500$, Winslow map ($p_w=1$, $Q_w=100$);
    (b,f) $R_m=3.7$, $Q=100$, $M=500$, Winslow map ($p_w=1$, $Q_w=100$);
    (c,g) $R_m=3.7$, $Q=100$, $M=500$, simple map;
    (d,h) $R_m=4.2$, $Q=100$, $M=500$, simple map.
  }
  \label{fg_8}
\end{figure}

\begin{table}[tb]
  \centering\small
  \begin{tabular}{l|cccc}
    \hline
    & Case \#1 & Case \#2 & Case \#3 & Case \#4  \\ \hline
    max $u$-error (domain) & $1.039\times 10^{-5}$ & $1.524\times 10^{-7}$ & $2.472\times 10^{-9}$ 
    & $2.517\times 10^{-4}$ \\
    rms $u$-error (domain) & $1.069\times 10^{-6}$ & $2.615\times 10^{-8}$ & $5.068\times 10^{-10}$ 
    & $7.118\times 10^{-5}$ \\ \hline
    max bc-error ($\overline{AB}$) & $8.882\times 10^{-16}$ & $8.882\times 10^{-16}$ & $2.220\times 10^{-16}$ 
    & $1.554\times 10^{-15}$  \\
    rms bc-error ($\overline{AB}$) & $3.361\times 10^{-16}$ & $2.688\times 10^{-16}$ & $8.345\times 10^{-17}$ 
    & $3.589\times 10^{-16}$  \\ \hline
    max bc-error ($\overline{BC}$) & $0.0$ & $0.0$ & $0.0$ & $0.0$  \\
    rms bc-error ($\overline{BC}$) & $0.0$ & $0.0$ & $0.0$ & $0.0$  \\ \hline
    max bc-error ($\overline{AC}$) & $1.776\times 10^{-15}$ & $4.441\times 10^{-16}$ & $1.665\times 10^{-15}$ 
    & $4.441\times 10^{-16}$  \\
    rms bc-error ($\overline{AC}$) & $4.872\times 10^{-16}$ & $1.624\times 10^{-16}$ & $4.393\times 10^{-16}$ 
    & $1.734\times 10^{-16}$  \\
    \hline
  \end{tabular}
  \caption{Nonlinear Helmholtz equation (NBCs/RBCs on boundaries):
    Maximum/rms domain errors of the NN solution
    and the maximum/rms boundary-condition errors (DBC, NBC, or RBC) on the three boundaries.
    Different cases correspond to those in Figure \ref{fg_8}.
  }
  \label{tab_5}
\end{table}


The effectiveness of the current method for enforcing Neumann and Robin
boundary conditions is exemplified by Figure \ref{fg_8} and Table \ref{tab_5} for
the nonlinear Helmholtz equation.
Figure \ref{fg_8} shows distributions of the NN solutions and their point-wise absolute errors
under four settings with four different domain geometries:
(i) Neumann condition imposed on $\overline{AC}$ and Dirichlet condition imposed
the other boundaries with domain \#1; (ii) Neumann condition imposed on $\overline{AB}$
and $\overline{AC}$ and Dirichlet condition imposed on $\overline{BC}$ with domain \#2;
(iii) Robin condition imposed on $\overline{AC}$ (with Robin coefficient $\alpha_{AC}=1.0$)
and Dirichlet condition imposed on the other boundaries with domain \#3;
(iv) Robin condition imposed on $\overline{AB}$ and $\overline{AC}$ (with coefficients
$\alpha_{AB}=1/2$ and $\alpha_{AC}=1.0$) and Dirichlet condition imposed on $\overline{BC}$
with domain \#4. The figure caption provides values for the simulation parameters
and Winslow parameters.
It is observed that the largest solution error occurs at the Neumann or Robin boundaries
in every case. The NN method captures the solution accurately, with the maximum domain
errors ranging from the order $10^{-4}$ for Case \#4 to $10^{-9}$ for case \#3.

Table \ref{tab_5} summarizes the maximum/rms domain errors of the NN solution and
the boundary-condition errors on the three boundaries for these four cases.
Evidently the current method has enforced the Neumann, Robin, and Dirichlet conditions
``exactly'', with the maximum numerical boundary-condition errors at
the level $10^{-16}\sim 10^{-15}$ for different cases.
These numerical results underscore the effectiveness and accuracy of our method for nonlinear
problems with complex domain geometries.

\subsection{Viscous Burgers Equation on Time-Evolving Domains}
\label{sec_33}

\begin{figure}
  \centerline{
    \includegraphics[width=1.1in]{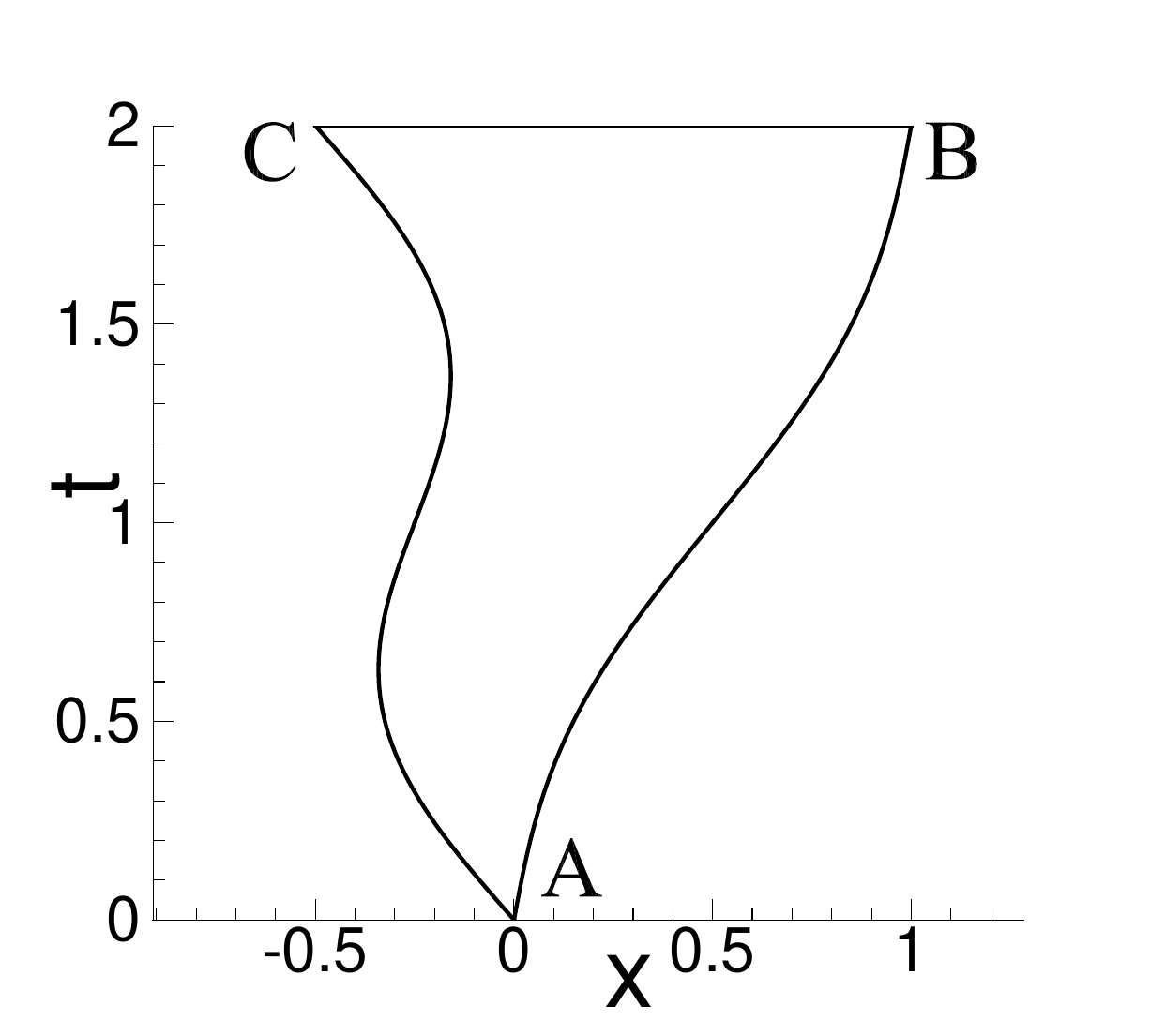}(a)
    \includegraphics[width=1.1in]{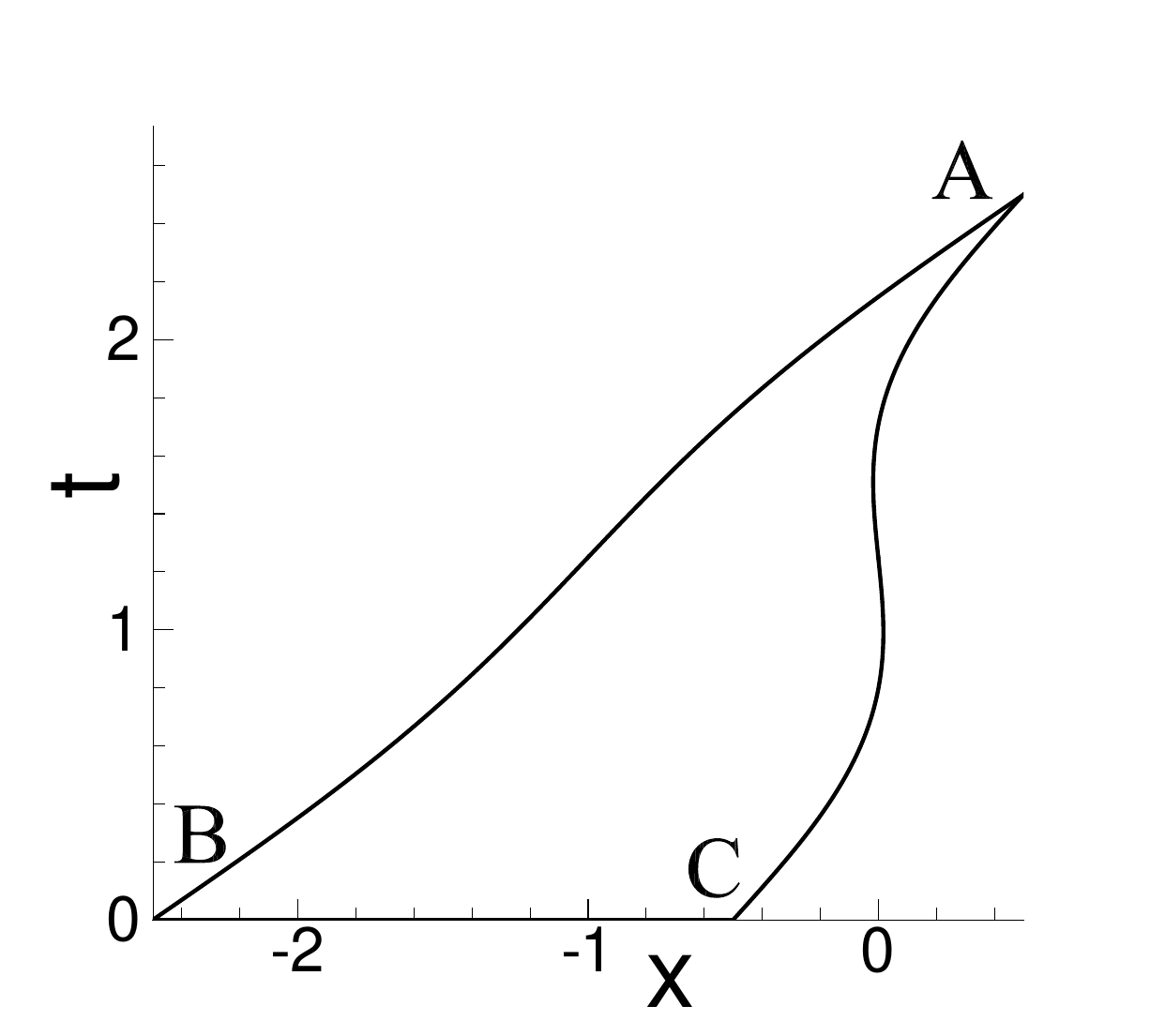}(b)
    \includegraphics[width=1.1in]{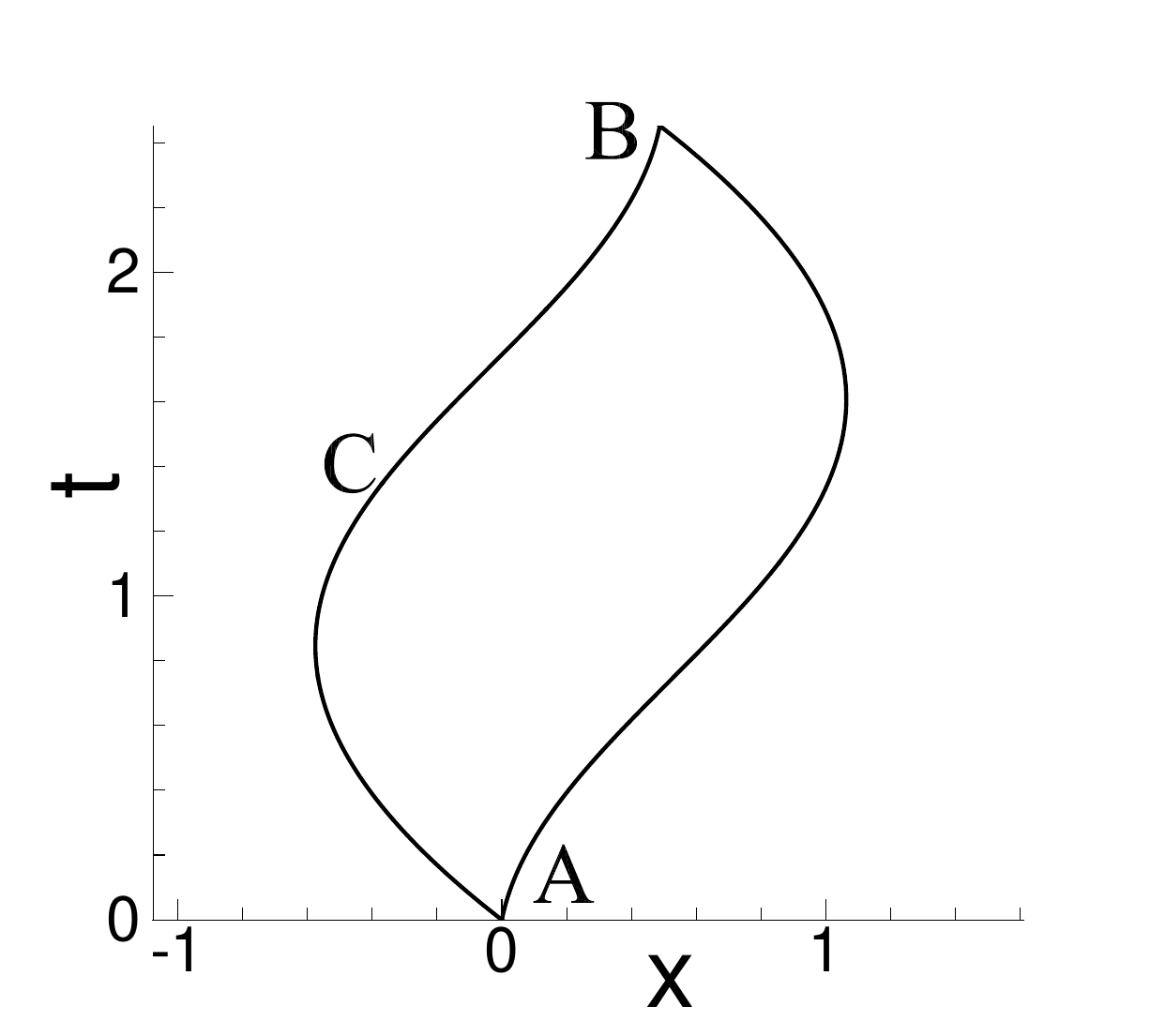}(c)
    \includegraphics[width=1.1in]{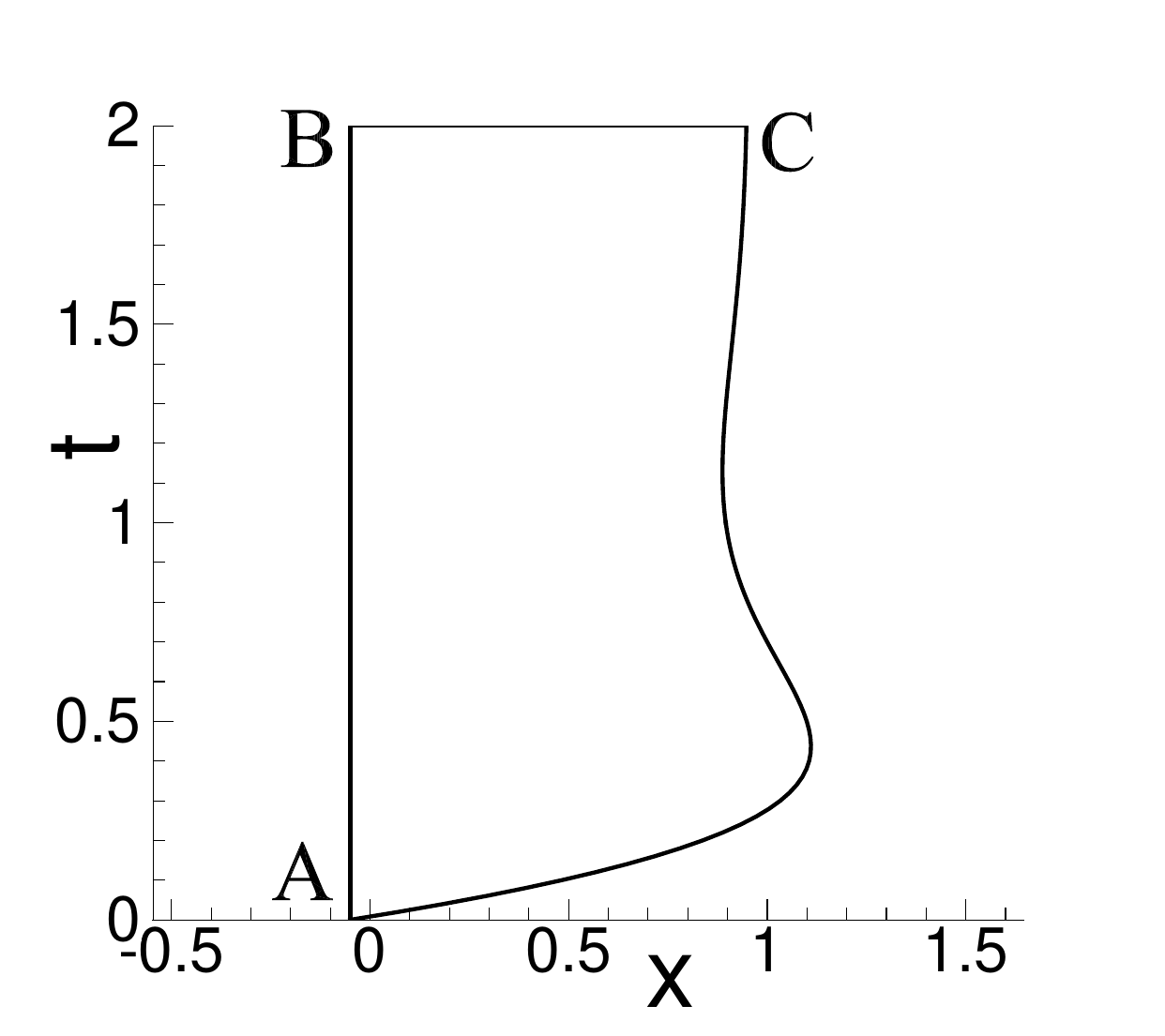}(d)
  }
  \centerline{
    \includegraphics[width=1.1in]{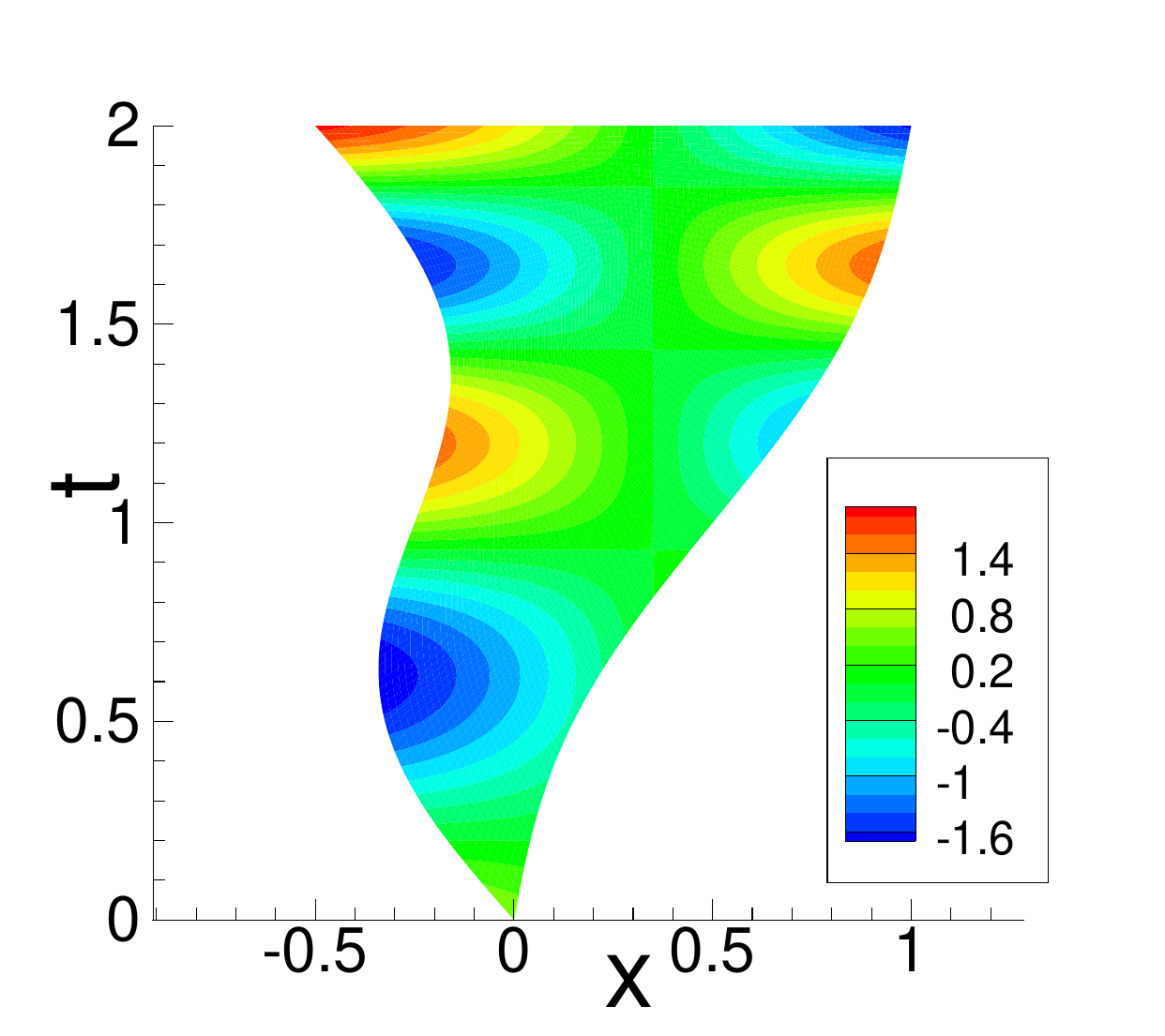}(e)
    \includegraphics[width=1.1in]{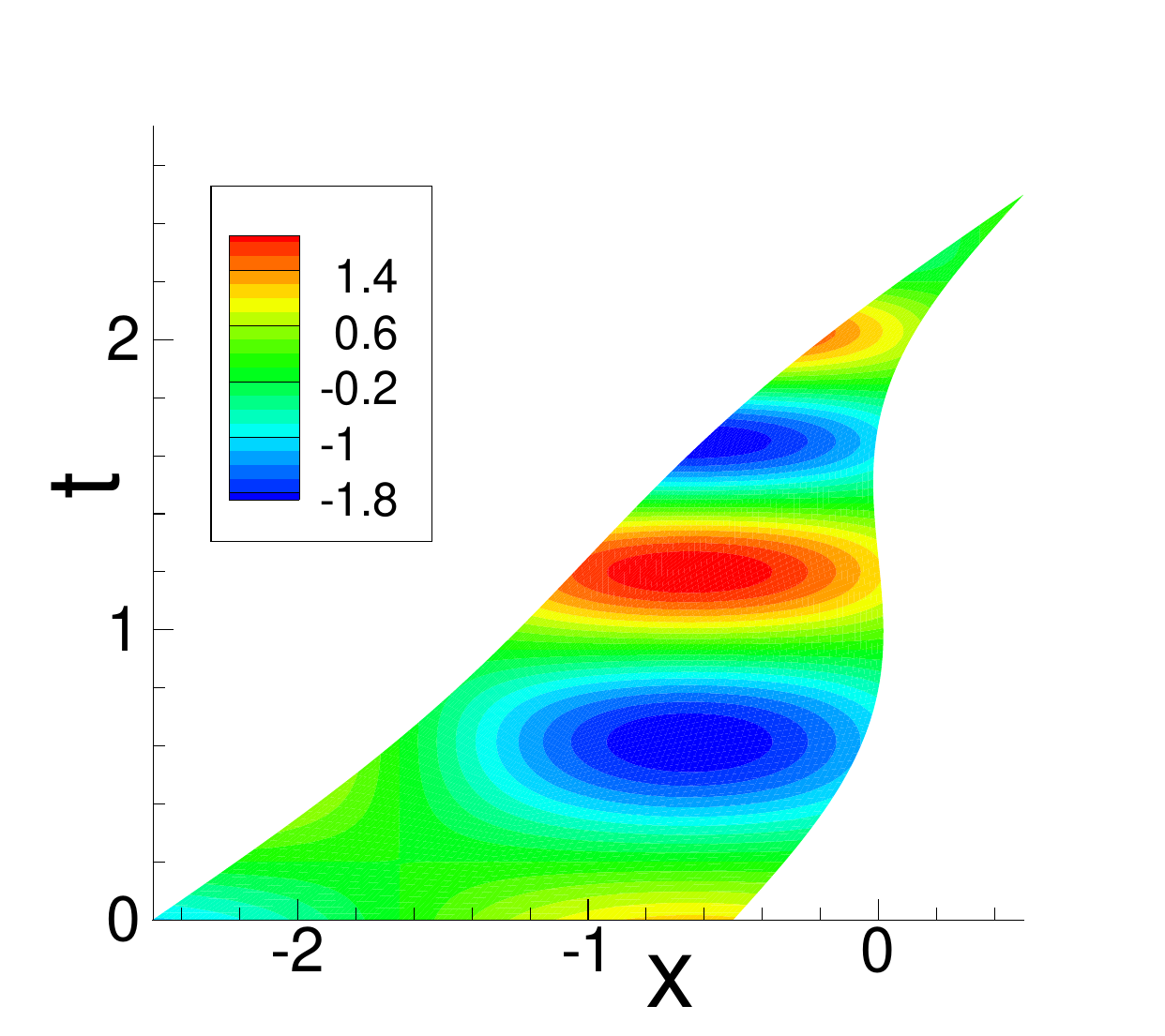}(f)
    \includegraphics[width=1.1in]{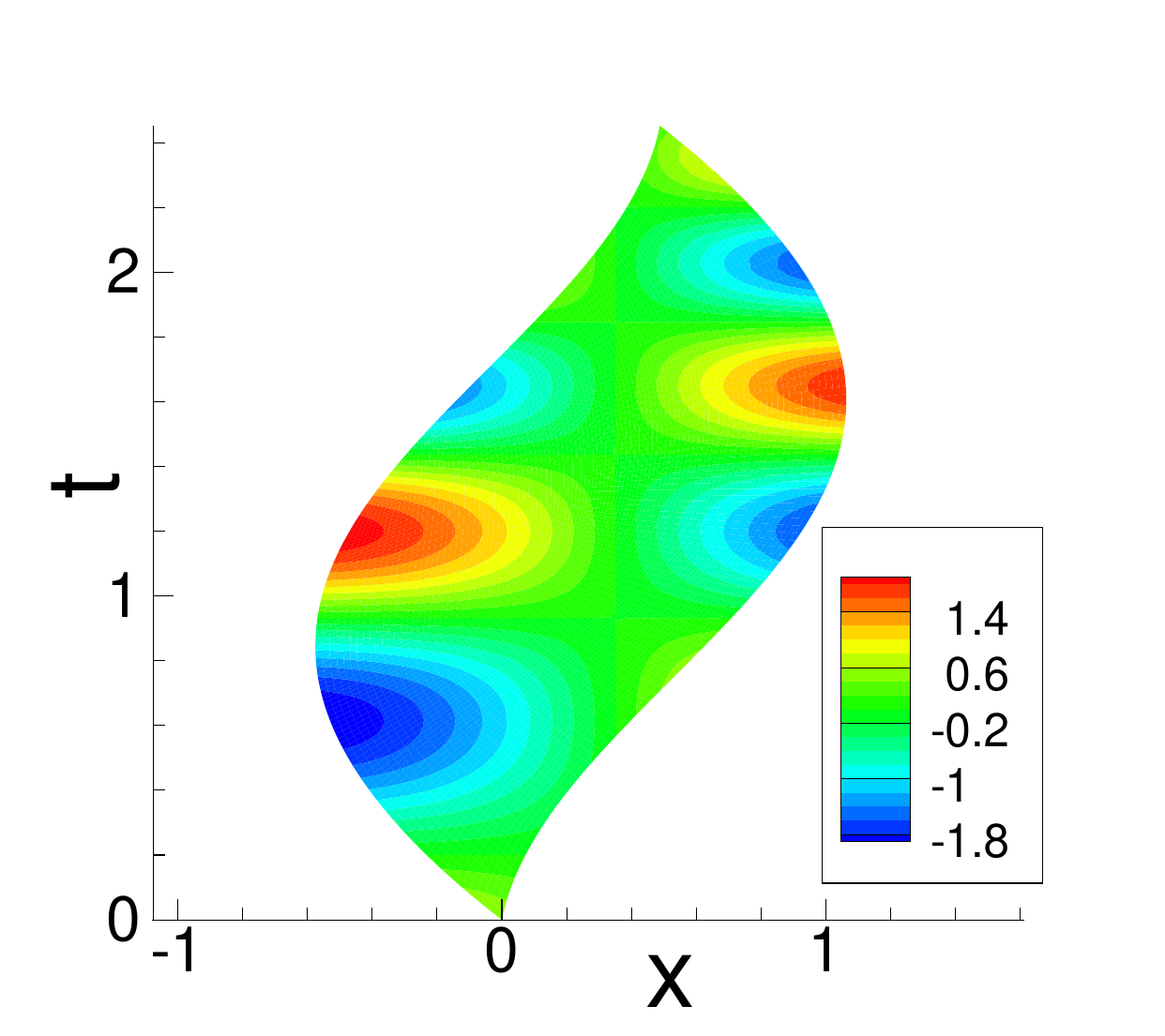}(g)
    \includegraphics[width=1.1in]{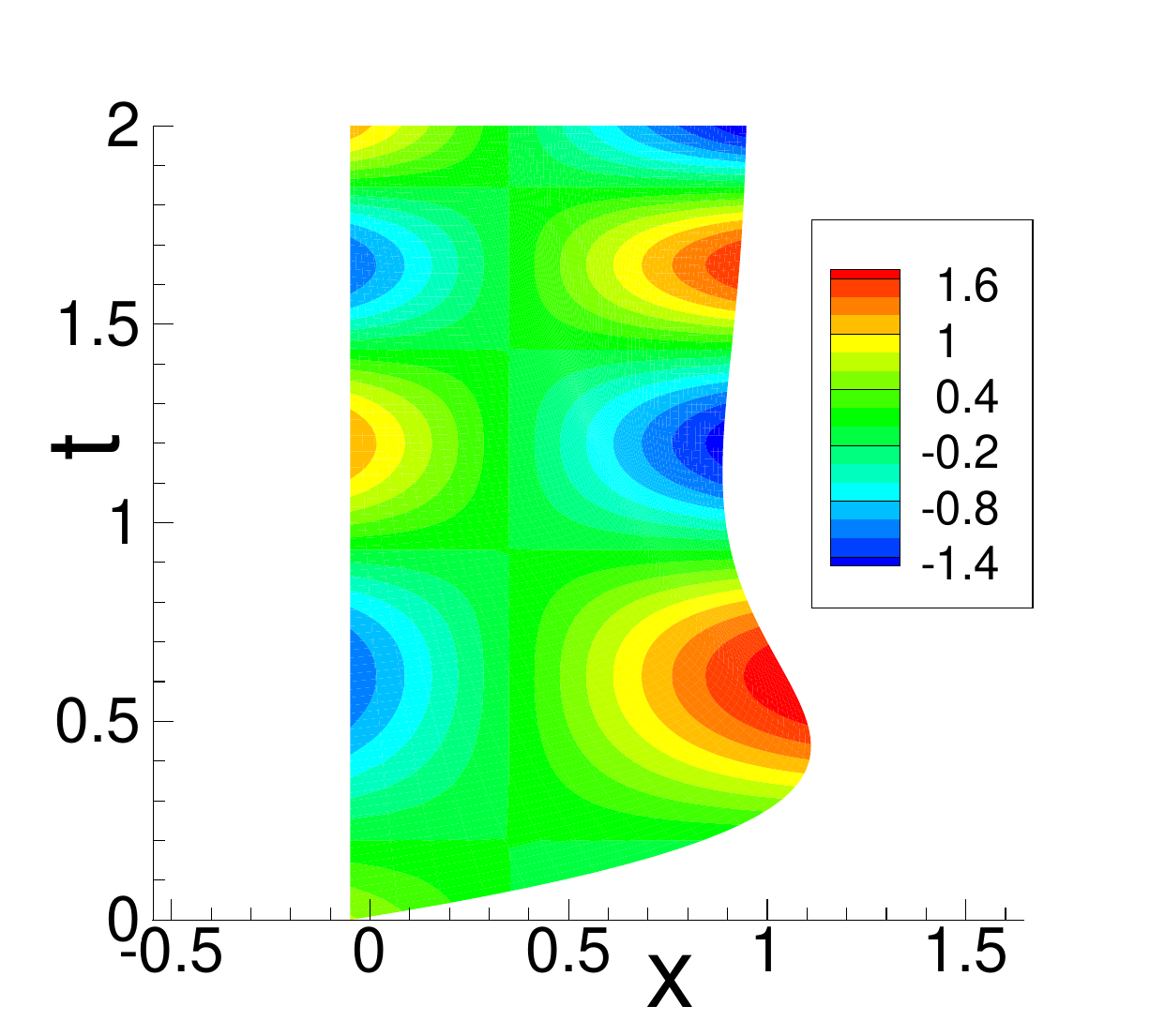}(h)
  }
  \caption{Burgers equation (time-evolving spatial domains):
    Domain geometries (top row)
    and distributions of the exact solution (bottom row) in the space-time plane:
    (a,e) Config \#1, DBCs on $\overline{AB}$ and $\overline{AC}$, no boundary condition
    on $\overline{BC}$; (b,f) Config \#2, DBCs on $\overline{AB}$, $\overline{BC}$ and
    $\overline{AC}$; 
    (c,g) Config \#3, DBCs on $\overline{AB}$, $\overline{BC}$ and
    $\overline{AC}$;  (d,h) Config \#4, NBC on $\overline{AB}$, DBC on $\overline{AC}$,
    and no boundary condition on $\overline{BC}$.
  }
  \label{fg_9}
\end{figure}

As another example we investigate a dynamic problem with
the viscous Burgers equation on one-dimensional (1D)  spatial domains
that evolve over time. Specifically, we consider the following IBVP,
\begin{subequations}
\begin{align}
  & \frac{\partial u}{\partial t} + u\frac{\partial u}{\partial x} = \nu\frac{\partial^2u}{\partial x^2}
  + f(x,t), \quad x\in D(t)=(a(t), b(t)), \ t\in(0,t_f), \\
  & u(x=a(t),t) = u_a(t), \quad t\in[0,t_f], \\
  & u(x=b(t),t) = u_b(t), \quad t\in[0,t_f], \\
  & u(x,t=0) = u_{in}(x), \quad x\in[a(0), b(0)]=[x_a,x_b],
\end{align}
\end{subequations}
where $u(x,t)$ is the field function to be computed, $\nu=0.01$ is the diffusivity,
$f(x,t)$ is a prescribed source term, and $a(t)$ and $b(t)$ are
prescribed functions that describe the evolution of the spatial domain $D(t)=[a(t),b(t)]$
for $t\in[0,t_f]$, with $t_f$ denoting the time horizon.
$u_a(t)$, $u_b(t)$ and $u_{in}(x)$ denote the prescribed boundary conditions or initial condition,
which must be compatible with one another at $t=0$ 
with $u_a(0)=u_{in}(x_a)$ and $u_b(0)=u_{in}(x_b)$, and also at $t=t_f$ if $D(t_f)$
becomes degenerate with $a(t_f)=b(t_f)$.
We choose $f$ and the initial/boundary conditions appropriately such that
this problem admits the solution,
\begingroup\small\begin{align}\label{eq_b126}
  & u(x,t) = 2\cos\left(\frac{\pi}{2}x + \frac{13\pi}{40} \right)
  \cos\left(\frac{\pi}{2}t^2 + \frac{4\pi}{5}t + \frac{8\pi}{25} \right).
\end{align}\endgroup


For this problem we focus on those cases in which the domain $D(t)$ becomes degenerate
at either $t=0$ or $t=t_f$ (or at both points), leading to $a(0)=b(0)$ or
$a(t_f)=b(t_f)$ (or both).
Physically, this corresponds to an evolving spatial domain $D(t)=[a(t),b(t)]$
that starts out as a point and expands over time or a spatial domain
that shrinks over time and eventually becomes a point at $t=t_f$.
We specifically consider the four configurations as illustrated by
Figure \ref{fg_9} (top row) in the space-time plane.
They are defined as follows.
\begin{itemize}
\item Config \#1: (see Figure \ref{fg_9}(a))
  \begin{align}\left\{\begin{array}{l}
     a(t) = x_A\left(1-\frac{t}{t_f} \right) + x_C\frac{t}{t_f}
    - \frac{1}{10}\sin\frac{2\pi t}{t_f}, \\
    b(t) = x_A\left(1-\frac{t}{t_f} \right) + x_B \frac{t}{t_f}
     - \frac15\sin\frac{2\pi t}{t_f}, \quad t\in[0,t_f],
  \end{array}\right.\end{align}
  with the following parameter values,
  \begin{align}
    & x_A=0, \quad x_B=1.0, \quad x_C=-0.5, \quad t_f=2.0.
  \end{align}
  This domain, $D(t)=[a(t),b(t)]$, is degenerate at $t=0$, leading to
  $x_a=x_b=a(0)=b(0)=x_A$, and it expands over time. The boundary distributions
  $u_a(t)$ and $u_b(t)$ are prescribed according to equation \eqref{eq_b126},
  and $u_{in}(x_A)=u_a(0)=u_b(0)$.

\item
  Config \#2: (see Figure \ref{fg_9}(b))
  \begin{align}
    \left\{\begin{array}{l}
    a(t) = X_1\frac{t}{t_f} + X_2\left(1-\frac{t}{t_f} \right)
    + 0.1\sin\frac{2\pi t}{t_f} - 0.5, \\
    b(t) = X_1\frac{t}{t_f} + X_3\left(1-\frac{t}{t_f} \right)
    + 0.2\sin\frac{2\pi t}{t_f} - 0.5, \quad t\in[0,t_f],
    \end{array}\right.
  \end{align}
  with the following parameter values,
  \begin{align}
    & X_1=1.0, \quad X_2=-2.0, \quad X_3=0.0, \quad t_f=2.5.
  \end{align}
  This domain shrinks over time and becomes degenerate at $t=t_f$.

\item Config \#3: (see Figure \ref{fg_9}(c))
  The boundary curves in the space-time plane are given by the following parametric
  forms,
  \begin{subequations}
    \begin{align}
      & \overline{AB}: \ \left\{ \begin{array}{l}
        x(\xi) = -\frac52\psi_1(\xi)\cos\frac{\pi}{16} + \frac52\phi_1(\xi)\sin\frac{\pi}{16}, \\
        t(\xi) = \frac52\psi_1(\xi)\sin\frac{\pi}{16} - \frac52\phi_1(\xi)\cos\frac{\pi}{16},
        \quad \xi\in[-1,1];
      \end{array}\right. \\
      & \overline{AC}: \ \left\{ \begin{array}{l}
        x(\xi) = -\frac52\psi_0\left(\phi_1(\xi)-1\right)\cos\frac{\pi}{16} + \frac54\phi_1(\xi)\sin\frac{\pi}{16}, \\
        t(\xi) = \frac52\psi_0\left(\phi_1(\xi)-1\right)\sin\frac{\pi}{16} + \frac54\phi_1(\xi)\cos\frac{\pi}{16},
        \quad \xi\in[-1,1];
      \end{array}\right. \\
      & \overline{BC}: \ \left\{ \begin{array}{l}
        x(\xi) = -\psi_0\left(\phi_0(\xi)\right)\cos\frac{\pi}{16}
        + \left[\frac52\phi_0(\xi)+\frac54\phi_1(\xi)\right]\sin\frac{\pi}{16}, \\
        t(\xi) = \psi_0\left(\phi_0(\xi)\right)\sin\frac{\pi}{16}
        + \left[\frac52\phi_0(\xi)+\frac54\phi_1(\xi)\right]\cos\frac{\pi}{16},
        \quad \xi\in[-1,1].
      \end{array}\right.
    \end{align}
  \end{subequations}
  In the above equations $\psi_0(\xi)$ and $\psi_1(\xi)$ are defined in \eqref{eq_29},
  and $\phi_0(\xi)$ and $\phi_1(\xi)$ are defined in \eqref{eq_5}.
  The spatial domain $D(t)$ for this configuration is degenerate at both $t=0$
  and $t=t_f=\frac52$.

\item Config \#4: (see Figure \ref{fg_9}(d))
  \begin{align}
    &\left\{\begin{array}{l}
       a(t) = -0.05, \\
       b(t) = -0.05 + 2e^{-3t}\sin(\pi t) + \left(1-e^{-3t} \right)\tanh(3t), \quad
       t\in[0,2].
       \end{array}\right.
  \end{align}
  The spatial domain $D(t)=[a(t),b(t)]$ 
  is degenerate at $t=0$. For this configuration we 
  consider a different set of boundary conditions.
  Here we impose a Neumann condition on the left boundary $x=a(t)=-0.05$ and the
  Dirichlet condition on the right boundary $x=b(t)$.
  
\end{itemize}


We employ a space-time approach to solve this problem, and treat the time
variable $t$ on the same footing as the space variable $x$. The problem is
then effectively reduced to a two-dimensional boundary-value problem
in the space-time plane.
Figures \ref{fg_9}(a,b,c,d) illustrate the space-time domain geometries
corresponding to these four configurations.
For Config \#1, Dirichlet conditions are imposed on $\overline{AB}$ and $\overline{AC}$,
with no boundary condition on $\overline{BC}$, which corresponds to the final time
instant $t=t_f$. For Configs \#2 and \#3, Dirichlet conditions
are imposed on all three boundaries $\overline{AB}$, $\overline{AC}$ and $\overline{BC}$.
For Config \#4, Neumann and Dirichlet conditions are imposed on
$\overline{AB}$ and $\overline{AC}$, respectively, with no boundary
condition on $\overline{BC}$.
Figures \ref{fg_9}(e,f,g,h) illustrate the distributions of the exact solution
for these four configurations in the space-time plane.

\begin{figure}
  \centerline{
    \includegraphics[width=1.1in]{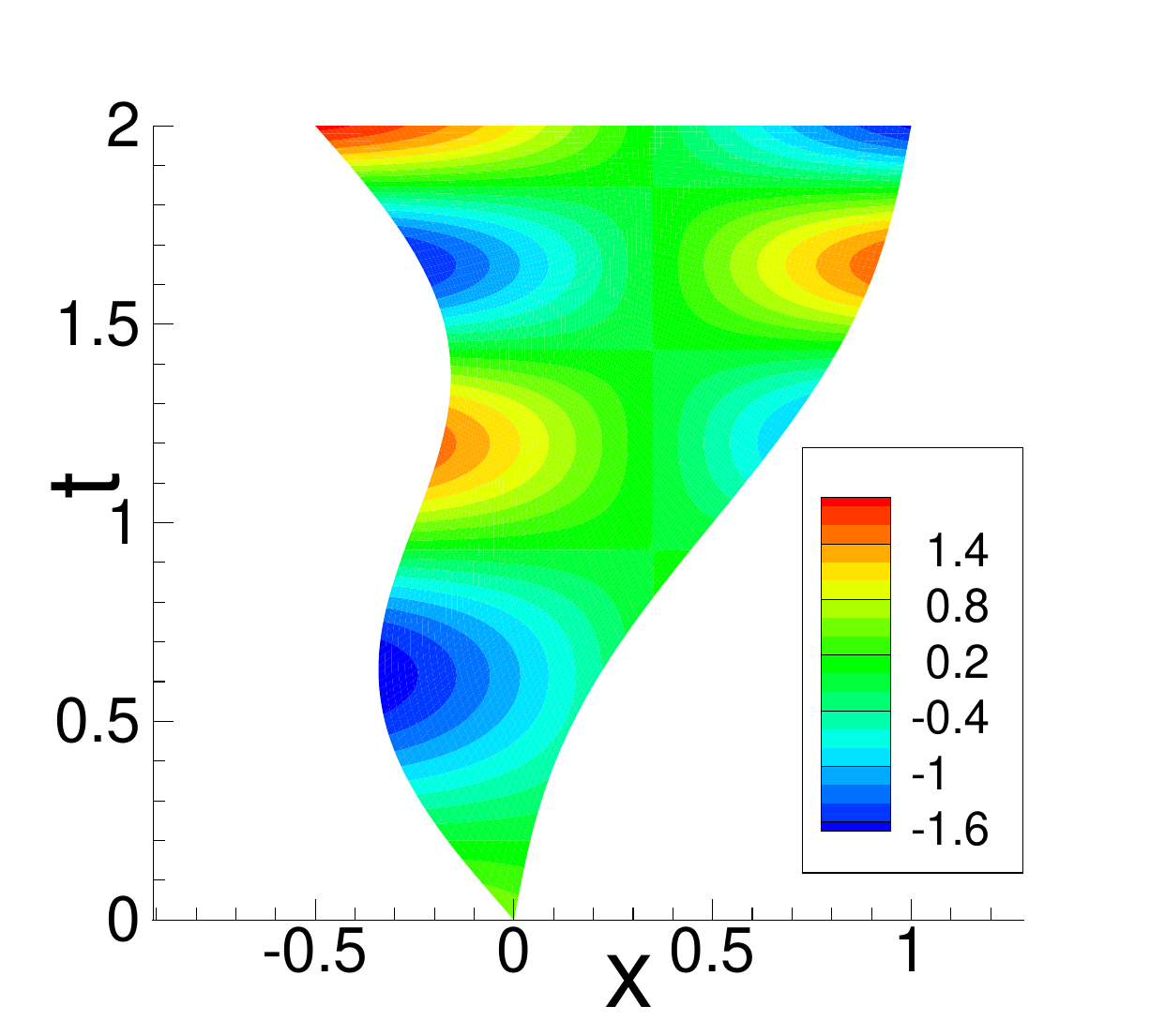}(a)
    \includegraphics[width=1.1in]{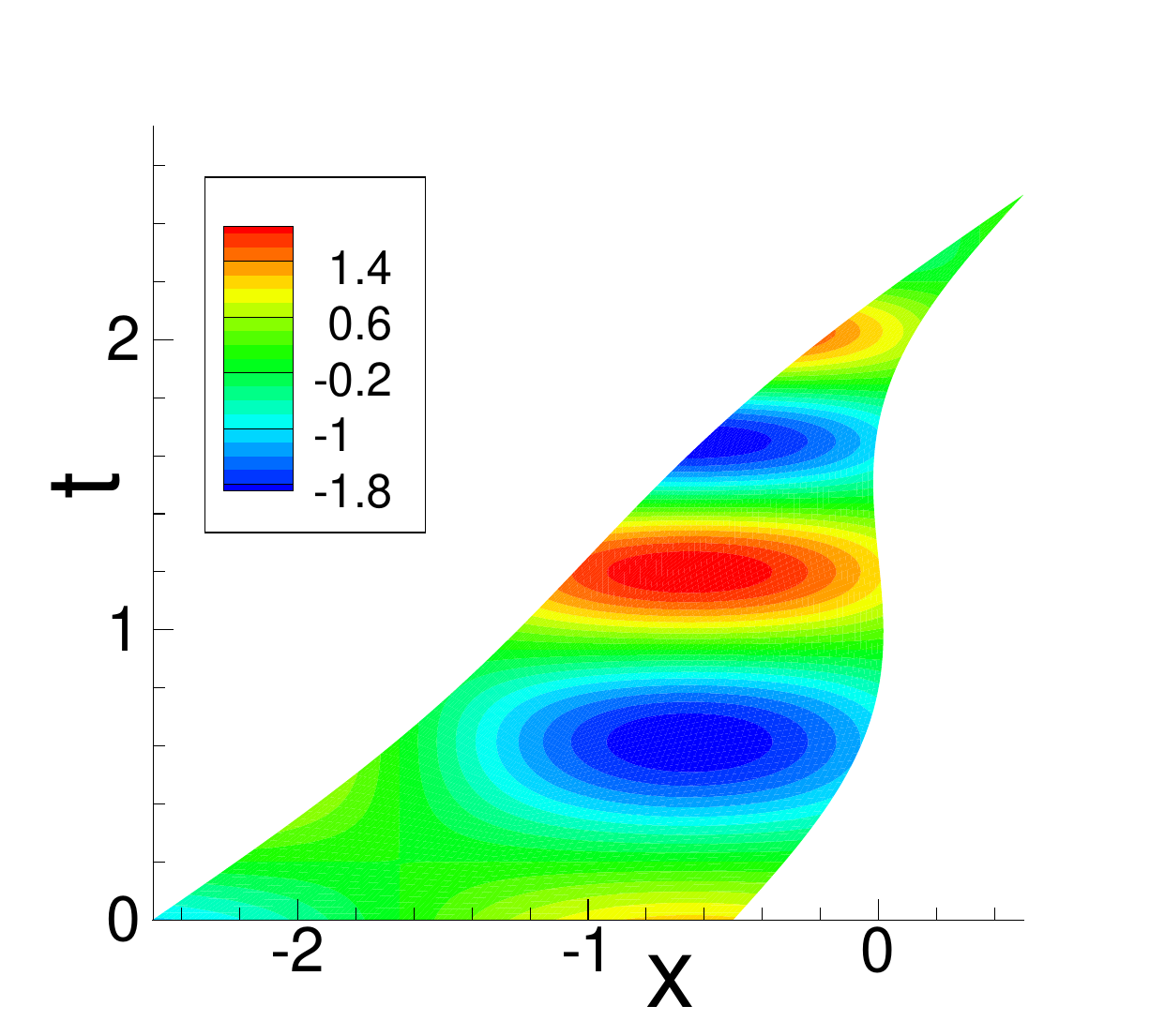}(b)
    \includegraphics[width=1.1in]{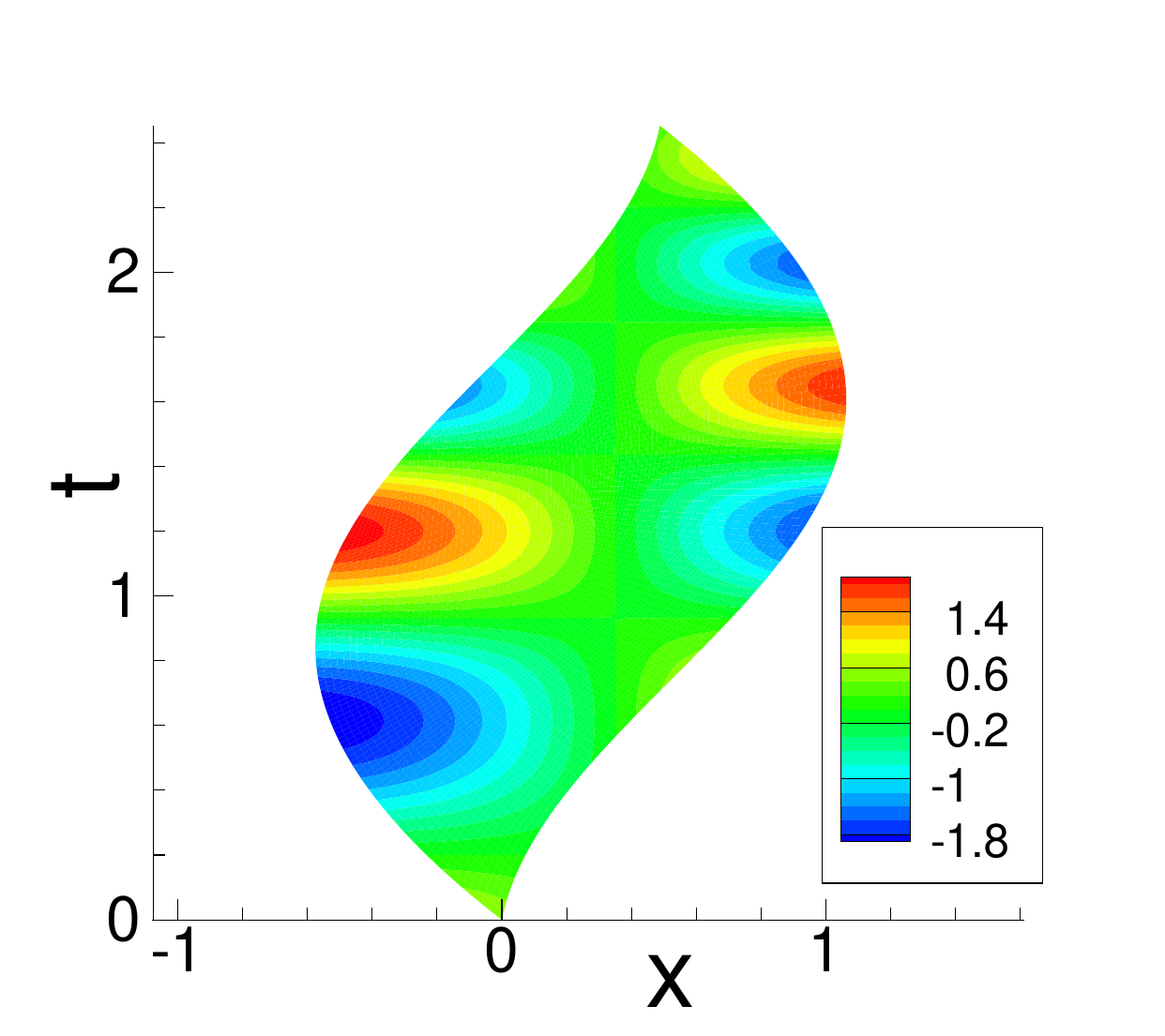}(c)
    \includegraphics[width=1.1in]{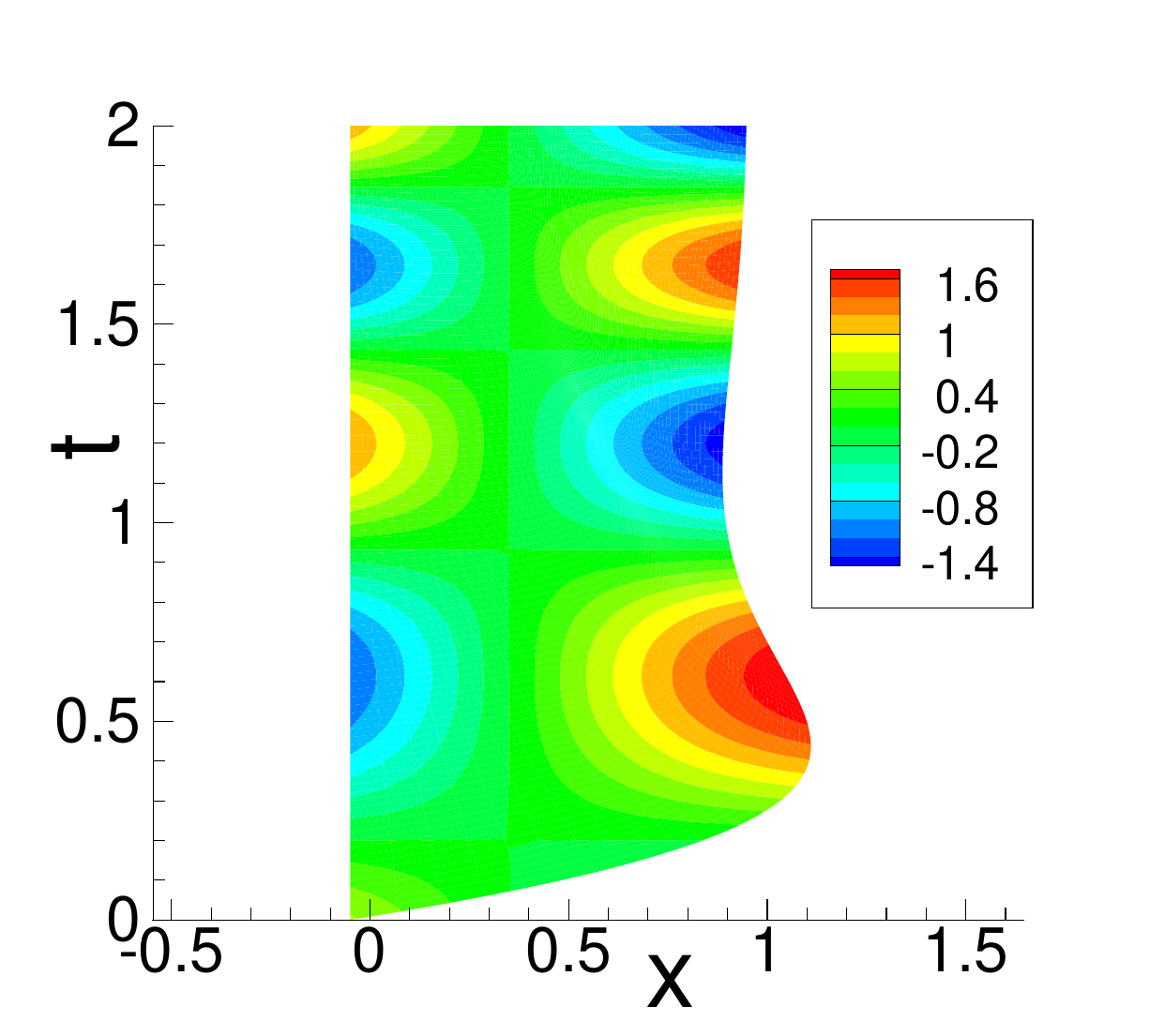}(d)
  }
  \centerline{
    \includegraphics[width=1.1in]{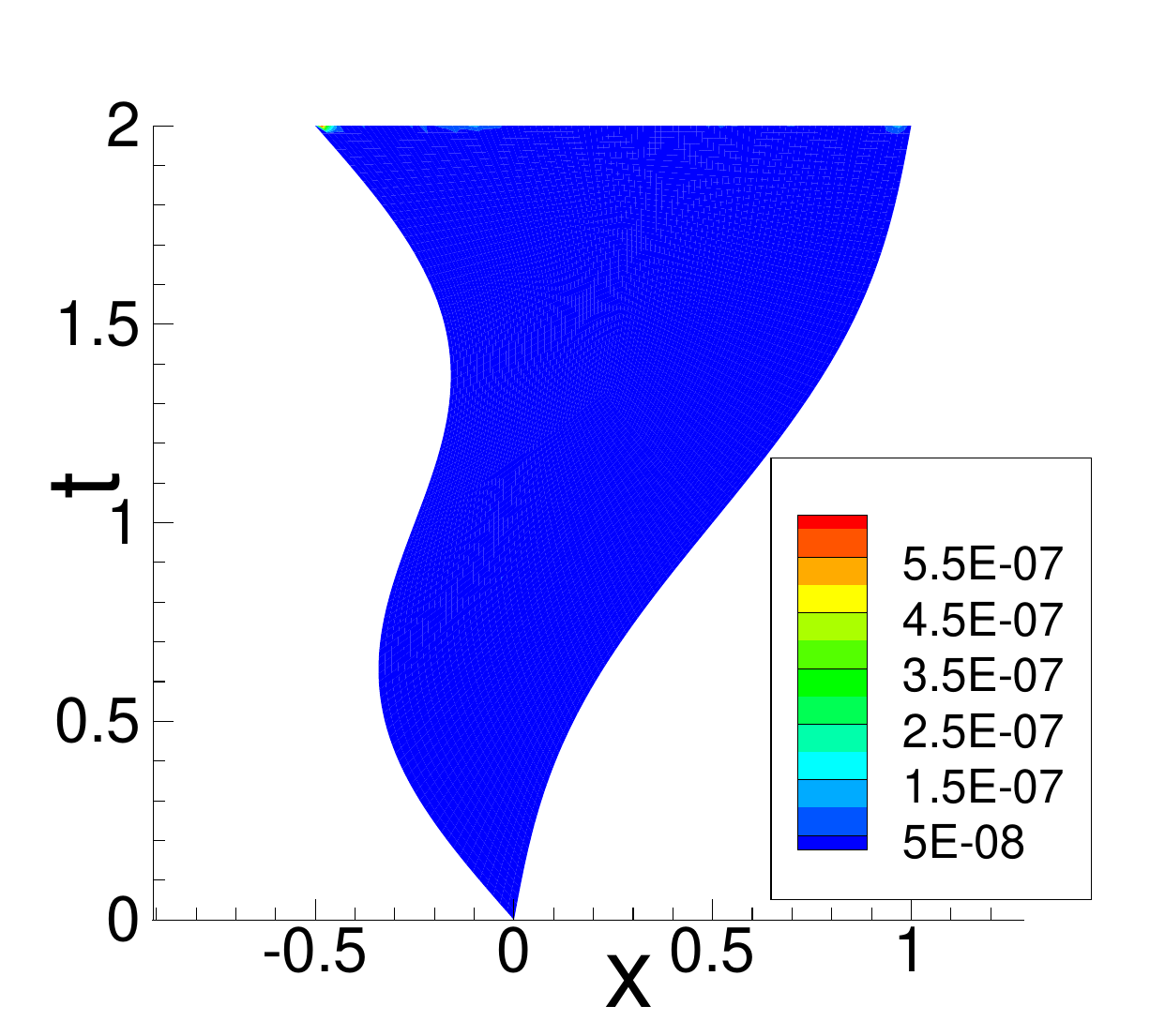}(e)
    \includegraphics[width=1.1in]{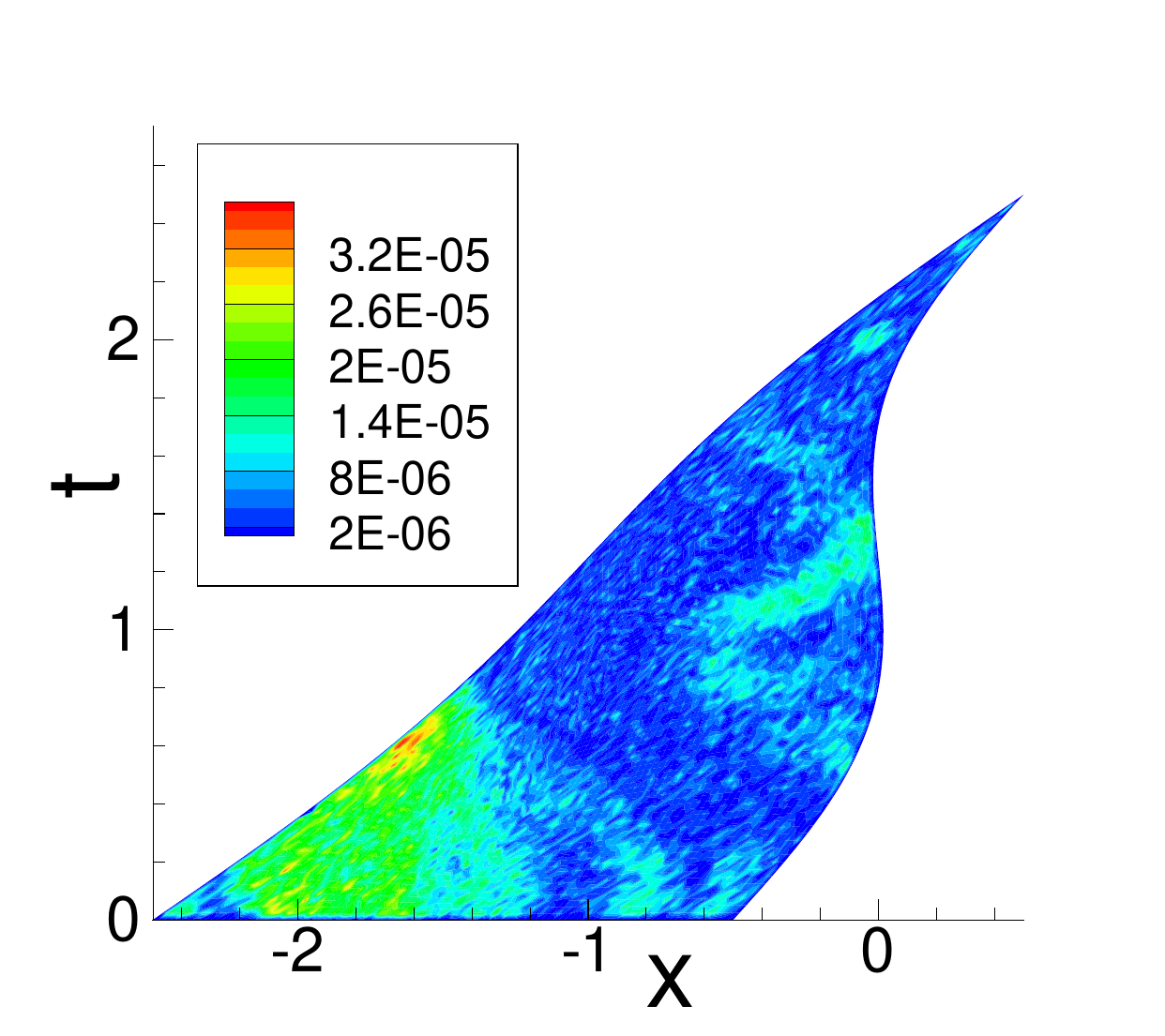}(f)
    \includegraphics[width=1.1in]{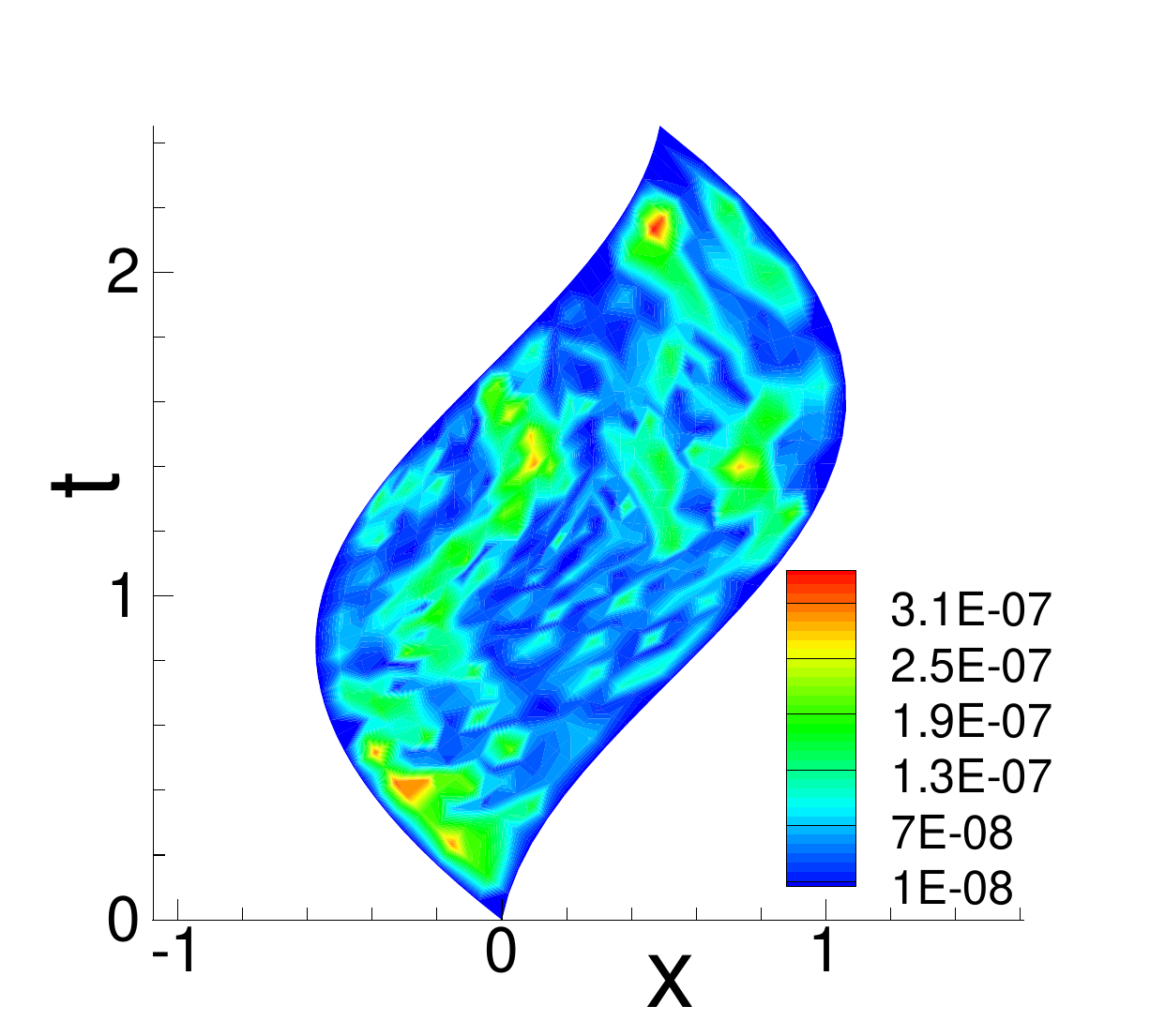}(g)
    \includegraphics[width=1.1in]{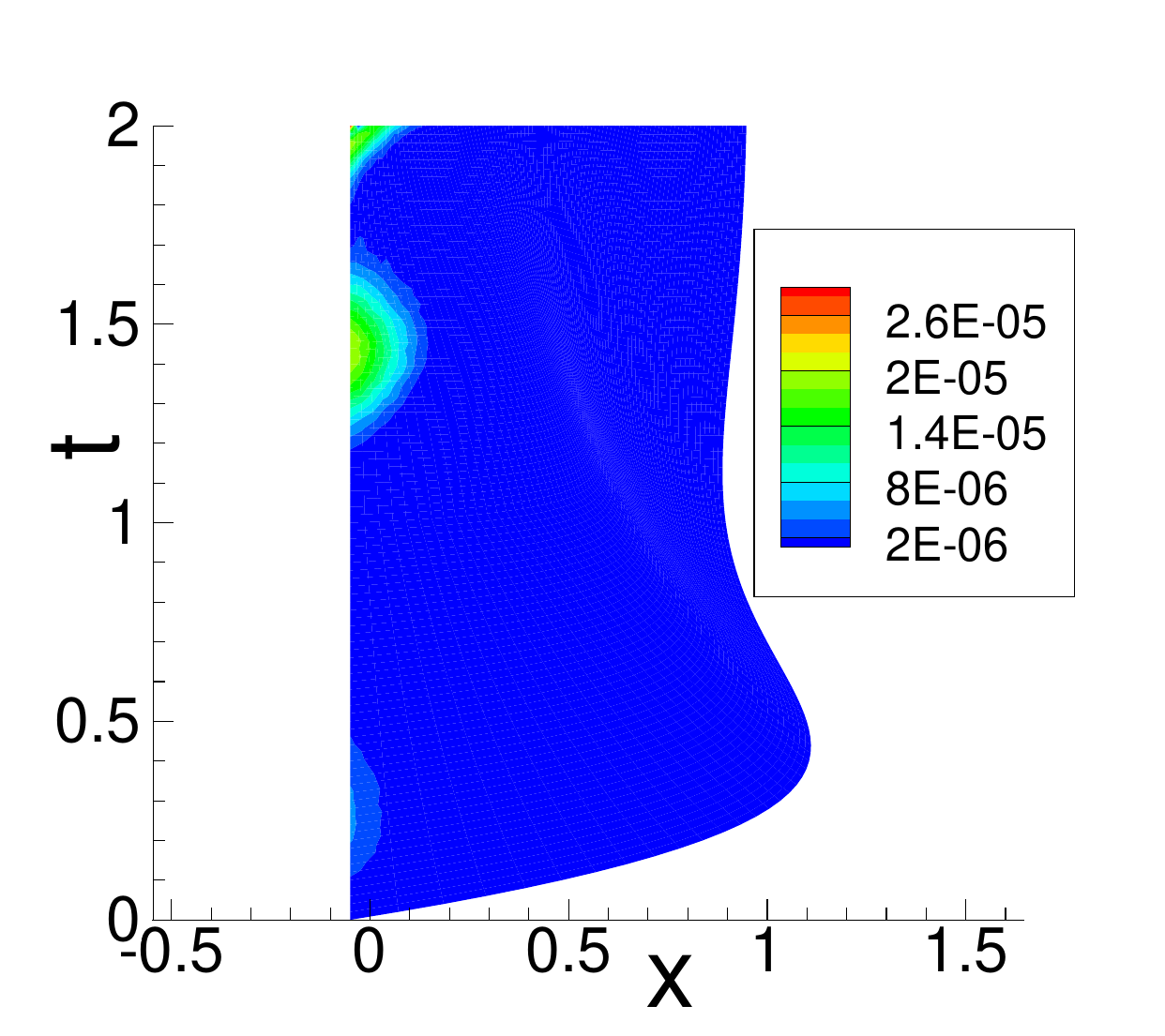}(h)
  }
  \caption{Burgers equation: Distributions of
    the NN solutions (top row) and their point-wise absolute errors (bottom row)
    in the space-time plane for the four configurations.
    Simulation parameters: (a,e): Config \#1, $R_m=4.0$, $Q=100$, $M=400$,
    Winslow map ($p_w=2$, $Q_w=200$);
    (b,f): Config \#2, $R_m=3.5$, $Q=100$, $M=400$, Winslow map ($p_w=1$, $Q_w=100$);
    (c,g) Config \#3, $R_m=4.2$, $Q=100$, $M=400$, Winslow map ($p_w=0$, $Q_w=100$);
    (d,h) Config \#4, $R_m=2.9$, $Q=100$, $M=500$, Winslow map ($p_w=0$, $Q_w=100$);
    Different configurations correspond to those as defined in Figure \ref{fg_9}.
  }
  \label{fg_10}
\end{figure}

\begin{table}[tb]
  \centering\small
  \begin{tabular}{l|cccc}
    \hline
    & Config \#1 & Config \#2 & Config \#3 & Config \#4  \\ \hline
    max $u$-error (domain) & $6.205\times 10^{-7}$ & $3.975\times 10^{-5}$ &
    $4.210\times 10^{-7}$ & $3.006\times 10^{-5}$  \\
    rms $u$-error (domain) & $1.898\times 10^{-8}$ & $9.017\times 10^{-6}$ &
    $1.049\times 10^{-7}$ & $2.100\times 10^{-6}$  \\ \hline
    max bc-error ($\overline{AB}$) & $2.220\times 10^{-16}$ & $2.220\times 10^{-16}$ &
    $4.441\times 10^{-16}$ & $1.776\times 10^{-15}$  \\
    rms bc-error ($\overline{AB}$) & $6.101\times 10^{-17}$ & $6.809\times 10^{-17}$ &
    $9.322\times 10^{-17}$ & $3.823\times 10^{-16}$  \\ \hline
    max bc-error ($\overline{BC}$) & -- & $3.331\times 10^{-16}$ & $4.441\times 10^{-16}$ & --  \\
    rms bc-error ($\overline{BC}$) & -- & $8.534\times 10^{-17}$ & $1.046\times 10^{-16}$ & --  \\ \hline
    max bc-error ($\overline{AC}$) & $2.220\times 10^{-16}$ & $2.220\times 10^{-16}$ &
    $4.441\times 10^{-16}$ & $2.220\times 10^{-16}$  \\
    rms bc-error ($\overline{AC}$) & $9.149\times 10^{-17}$ & $6.145\times 10^{-17}$ &
    $1.280\times 10^{-16}$ & $9.166\times 10^{-17}$  \\
    \hline
  \end{tabular}
  \caption{Burgers equation:
    Maximum/rms domain errors of the NN solution
    and the maximum/rms boundary-condition errors (DBC, NBC) for the four configurations.
    Different configurations correspond to those in Figures \ref{fg_9} and \ref{fg_10}.
    In Configs \#1 and \#4, note that no boundary condition is imposed on $\overline{BC}$.
  }
  \label{tab_6}
\end{table}

The characteristics of the NN results corresponding to these configurations
are shown in Figure \ref{fg_10}, with distributions of the obtained
NN solutions and their point-wise absolute errors in the space-time plane.
The simulation parameters are again included in the figure caption.
The NN method has captured the solution accurately, and the maximum domain
errors range from the level of $10^{-7}$ for Configs \#1 and \#3 to the
level of $10^{-5}$ for Configs \#2 and \#4.

Table \ref{tab_6} provides quantitative information about the NN solution
accuracy and boundary-condition accuracy, listing the maximum/rms
domain errors  and the maximum/rms boundary-condition
errors corresponding to the four configurations. 
One can observe that the current method has enforced the boundary conditions
to the machine precision for these time-evolving domains, resulting in
numerical boundary-condition errors on the order of $10^{-16}$.


\section{Concluding Remarks}
\label{sec_summary}

We have developed a method for the exact enforcement of Dirichlet, Neumann,
and Robin boundary conditions on general triangular domains with arbitrarily curved
boundaries. The method employs an exact mapping between a general curved triangular
domain and the standard triangle. For highly distorted physical domains, a Winslow-type
mapping is computed to maintain univalency while preserving the exactness of the
domain mapping. The method provides explicit, closed-form trial ansatz functions that
mathematically satisfy the prescribed Dirichlet, Neumann, and Robin boundary
conditions exactly.

The proposed method integrates several techniques for exact boundary-condition
enforcement on curved triangular domains, including the four-step procedure
of \cite{DongZ2026}, the TFC expression $(g-Pg+PF)$, and Barnhill-type transfinite
interpolation. The enforcement of Dirichlet conditions is conceptually analogous to
the construction of the mapping between the physical and standard triangular domains.
In contrast, the exact enforcement of Neumann and Robin conditions on curved triangular
boundaries presents significant technical challenges, particularly in satisfying the
induced corner compatibility constraints, and necessitates specialized algorithmic
constructions.
In addition to Barnhill-type transfinite interpolation, two other strategies are
essential to the formulations:
\begin{itemize}
\item
When either acute vertex of the standard triangle lies on a Neumann or Robin
boundary, all conditional compatibility constraints associated with that
vertex must be reformulated in unconditional forms.

\item
When the hypotenuse of the standard triangle corresponds to a Neumann or Robin
boundary, a separate TFC-style construction is required to define an auxiliary edge
function, together with the solution of an associated linear system for the gradient
components along the hypotenuse.
\end{itemize}

The formulations developed in this paper follow three well-defined stages:
(i) transforming the prescribed boundary conditions from the physical domain to
the standard triangle; (ii) analyzing the corner compatibility constraints induced
by the transformed conditions; and
(iii) constructing a general trial ansatz that satisfies the prescribed boundary
conditions and the associated corner compatibility constraints exactly.
The four-step procedure of \cite{DongZ2026}, Barnhill transfinite
interpolation, and the two specialized strategies  for handling the
hypotenuse and acute vertices collectively provide the essential tools for the
construction in the third stage.

Using these strategies, we have analyzed the corner compatibility
constraints and developed explicit
formulations for representative boundary-condition configurations. These include
(i) DBC on all triangle boundaries; (ii) NBC on one boundary, corresponding to
the hypotenuse or leg $\overline{AC}$, with DBC on the other boundaries;
(iii) NBC on two boundaries, corresponding to either both legs or the
hypotenuse and leg $\overline{AC}$, with DBC on the third boundary;
(iv) RBC on one boundary (the hypotenuse or leg $\overline{AC}$) with DBC
on the other boundaries; (v) RBC on two boundaries (both legs or the hypotenuse
and  leg $\overline{AC}$) with DBC on the third boundary.
All the other boundary-condition configurations on the triangular domain 
can be formulated analogously using the above configurations, either individually
or in combination.
For example, consider a triangular domain with Dirichlet, Neumann, and Robin
conditions prescribed on its three respective boundaries. By treating
the Neumann condition as a Robin condition with a zero Robin coefficient,
this case reduces to configuration (v) above. If Neumann conditions are
imposed on all three boundaries, the configuration can be treated as
a combination of two cases: (a) Neumann conditions on both legs of the standard
triangle, and (b) Neumann conditions on the hypotenuse and one leg.
The compatibility constraints at each vertex can then be derived from the
analyses in Sections \ref{sec_232} and \ref{sec_n_ACBC} and
the discussion in Remark \ref{rem_53}. The corresponding trial ansatz can be
constructed analogously using the four-step procedure, Barnhill transfinite
interpolation, and the other aforementioned strategies
for handling the hypotenuse and acute vertices.


The current method for exact BC enforcement has been combined
with ELM-type randomized neural networks for solving PDE problems.
We tested the combined NN method using several linear/nonlinear, stationary/dynamic
boundary-value or initial-boundary value problems on a variety of static or
evolving physical domains
with complex boundary geometries.
Extensive numerical experiments are conducted to assess the accuracy of
the computed solutions and boundary-condition enforcement. The results demonstrate
that the combined NN method produces highly accurate solution fields, with maximum
errors over the computational domains ranging from $10^{-5}$ to $10^{-9}$ across different
geometries and PDEs. Furthermore, the proposed method enforces Dirichlet, Neumann,
and Robin boundary conditions on curved domain boundaries exactly by construction,
with numerical boundary-condition errors consistently at the level of
machine precision across all the test problems.


The technique developed in this paper is based on triangular domains and therefore
inherits certain geometric limitations of this approach. Specifically, it applies to
simply connected two-dimensional domains with at most three boundary corners.
When applied to PDE problems, appropriate boundary regularity is assumed
according to the requirements of the underlying PDEs and boundary conditions.
Extending this and related approaches to domains with more general boundary geometries
is a subject for future investigation.

Despite these limitations, the proposed method provides an effective approach for
a broad class of problems involving complex boundary geometries. The
triangular-domain-based method developed here complements its quadrilateral-domain-based
counterpart in~\cite{DongZ2026}.
Together, these approaches may provide a foundation for developing more general and
powerful methods capable of handling increasingly complex domain geometries and
boundary-condition configurations.


\section{Appendix: Additional Formulations and Geometric Parameters}
\label{sec_append}

\subsection{Formulations with Hypotenuse Being a Neumann or Robin Boundary}

When the hypotenuse of  $\Omega_{st}$ corresponds to
a Neumann or Robin boundary, the formulation entails a separate TFC-style
construction of an auxiliary
edge function on the hypotenuse and an associated linear solve for $(V_{\xi},V_{\eta})$ there.
This  characteristic distinguishes the constructions for this set of configurations.

\subsubsection{Sole Neumann Condition Imposed on the Hypotenuse}
\label{sec_2312}

In this appendix we consider the configuration in which the sole Neumann boundary
corresponds to $\overline{BC}$, the hypotenuse of the
standard triangle $\Omega_{st}$.
The notations below follow those in the main text.

Consider the general triangular domain $\Omega$ in Figure~\ref{fg_1}(a)
and a scalar function $u(\mbs x)$ ($\mbs x\in\Omega$) satisfying,
\begingroup\small\begin{align}\label{eq_53}
    & u|_{\mbs x\in\overline{AB}} = u_{AB}(\mbs x),  \quad
     \mbs n\cdot\nabla u|_{\mbs x\in\overline{BC}} = u_{nBC}(\mbs x), \quad
     u|_{\mbs x\in\overline{AC}} = u_{AC}(\mbs x),
\end{align}\endgroup
where $u_{nBC}(\mbs x)$ denotes the prescribed Neumann data on $\overline{BC}$.
Using the map $\mbs x(\xi,\eta)$ from Section~\ref{sec_21},
we transform the Dirichlet conditions in~\eqref{eq_53}
into~\eqref{eq_16a}--\eqref{eq_16b}.
The Neumann condition is transformed into,
\begingroup\small\begin{align}\label{eq_54}
  R_{BC}(\eta) V_{\xi}(-\eta,\eta) + [1-R_{BC}(\eta)]V_{\eta}(-\eta,\eta)
  = T_{BC}(\eta), \quad \eta\in[-1,1],
\end{align}\endgroup
where 
\begingroup\small\begin{align}\label{eq_55}
  \left\{
    \begin{array}{l}
    R_{BC}(\eta) = K_{xBC}(\eta)W_{BC}(\eta)
    = \frac{[-\mbs x_{\xi}(-\eta,\eta)+\mbs x_{\eta}(-\eta,\eta)]
      \cdot\mbs x_{\eta}(-\eta,\eta)}{\|-\mbs x_{\xi}(-\eta,\eta)+\mbs x_{\eta}(-\eta,\eta) \|^2},\ 
    \mbs n_{BC}(\eta)=
    \begin{bmatrix}\frac{-y_{\xi}(-\eta,\eta)+y_{\eta}(-\eta,\eta)}
      {\|-\mbs x_{\xi}(-\eta,\eta)+ \mbs x_{\eta}(-\eta,\eta) \|}
      \\ -\frac{-x_{\xi}(-\eta,\eta)+x_{\eta}(-\eta,\eta)}
      {\|-\mbs x_{\xi}(-\eta,\eta)+ \mbs x_{\eta}(-\eta,\eta) \|}  \end{bmatrix},
    \\
    W_{BC}(\eta) = \frac{1}{K_{xBC}(\eta)+K_{yBC}(\eta) }
    = \frac{\text{det}\mbs J(-\eta,\eta)}{\|-\mbs x_{\xi}(-\eta,\eta)
      + \mbs x_{\eta}(-\eta,\eta) \|}, \quad
    T_{BC}(\eta) = u_{nBC}(\mbs x(-\eta,\eta))W_{BC}(\eta),
    \\
    \begin{bmatrix}K_{xBC}(\eta) \\ K_{yBC}(\eta) \end{bmatrix}
    = \mbs J^{-1}(-\eta,\eta)\mbs n_{BC}(\eta)
    = \frac{1}{\text{det}\mbs J(-\eta,\eta)}
    \begin{bmatrix} \frac{[-\mbs x_{\xi}(-\eta,\eta)+\mbs x_{\eta}(-\eta,\eta)]
      \cdot\mbs x_{\eta}(-\eta,\eta)}{\|-\mbs x_{\xi}(-\eta,\eta)+\mbs x_{\eta}(-\eta,\eta) \|}  \\
      - \frac{[-\mbs x_{\xi}(-\eta,\eta)+\mbs x_{\eta}(-\eta,\eta)]
      \cdot\mbs x_{\xi}(-\eta,\eta)}{\|-\mbs x_{\xi}(-\eta,\eta)+\mbs x_{\eta}(-\eta,\eta) \|}
    \end{bmatrix}.
    \end{array}\right.
  \end{align}\endgroup
Noting that $\mbs x_{BC}'(\eta)=-\mbs x_{\xi}(-\eta,\eta)+\mbs x_{\eta}(-\eta,\eta)$,
we have
\begingroup\small\begin{align}\label{eq_n112}
  & \overline{AB}\perp\overline{BC}\ \text{at}\ B \Longleftrightarrow
  R_{BC}(-1)=1; \qquad
  \overline{AC}\perp\overline{BC}\ \text{at}\ C \Longleftrightarrow
  R_{BC}(1)=0.
\end{align}\endgroup


The Neumann condition~\eqref{eq_54} and the Dirichlet
conditions~\eqref{eq_16a} and \eqref{eq_16b} must be compatible with one another.
The continuity of $V(\xi,\eta)$ at the three vertices leads to
\begin{subequations}\small
  \begin{align}
    & V(-1,-1) = F(-1,-1) = \lim_{\eta\rightarrow -1}F(-1,\eta)
    =\lim_{\xi\rightarrow -1}F(\xi,-1) = u_A, \label{eq_57a}  \\
    & V(1,-1) = F(1,-1) := \lim_{\xi\rightarrow 1} F(\xi,-1), \label{eq_57b}  \\
    & V(-1,1) = F(-1,1) := \lim_{\eta\rightarrow 1} F(-1,\eta).  \label{eq_57c}
  \end{align}
\end{subequations}
Differentiating equation~\eqref{eq_16a} w.r.t.~$\xi$ and evaluating it at $\xi=1$
leads to
\begingroup\small\begin{align}\label{eq_58}
  & V_{\xi}(1,-1) = F_{\xi}(1,-1) := \lim_{\xi\rightarrow 1}F_{\xi}(\xi,-1).
\end{align}\endgroup
Combine this equation with~\eqref{eq_54} evaluated at $\eta=-1$ and we have
\begin{subequations}\small
  \begin{align}
    & T_{BC}(-1) - F_{\xi}(1,-1)= 0, \qquad \text{if}\ \overline{AB}\perp\overline{BC}\
    \text{at}\ B; \label{eq_59a}  \\
    & V_{\eta}(1,-1) = \frac{T_{BC}(-1)-R_{BC}(-1)F_{\xi}(1,-1)}{1-R_{BC}(-1)}=:F_{\eta}^a(1,-1),
     \quad
    \text{if}\ \overline{AB}\not\perp\overline{BC}\ \text{at}\ B. \label{eq_59b}
  \end{align}
\end{subequations}
Differentiating equation~\eqref{eq_16b} w.r.t.~$\eta$ and evaluating it at $\eta=1$
leads to
\begingroup\small\begin{align}\label{eq_60}
  & V_{\eta}(-1,1) = F_{\eta}(-1,1) := \lim_{\eta\rightarrow 1}F_{\eta}(-1,\eta).
\end{align}\endgroup
Combine this equation with~\eqref{eq_54} evaluated at $\eta=1$ and we have
\begin{subequations}\label{eq_61}\small
  \begin{align}
    & T_{BC}(1) - F_{\eta}(-1,1) = 0, \qquad \text{if}\ \overline{AC}\perp\overline{BC}\
    \text{at}\ C; \label{eq_61a}  \\
    & V_{\xi}(-1,1) = \frac{T_{BC}(1)-[1-R_{BC}(1)]F_{\eta}(-1,1)}{R_{BC}(1)}=:F_{\xi}^a(-1,1),
    \quad
    \text{if}\ \overline{AC}\not\perp\overline{BC}\ \text{at}\ C. \label{eq_61b}
  \end{align}
\end{subequations}
Note that equations~\eqref{eq_59a} and \eqref{eq_61a} are constraints on
the prescribed Dirichlet/Neumann data, which are assumed to be
always satisfied.

Equations~\eqref{eq_59b} and~\eqref{eq_61b} are conditional constraints at vertices $B$ and $C$.
We transform them into equivalent unconditional forms,
\begin{subequations}\small
  \begin{align}
    & V_{\eta}(1,-1) = (1-\lambda_B)\Gamma_B + \lambda_B F_{\eta}^a(1,-1)
    =\left\{ \begin{array}{ll} \Gamma_B, & \overline{AB}\perp\overline{BC}\ \text{at}\ B, \\
      F_{\eta}^a(1,-1), & \text{otherwise}
    \end{array} \right\} =: F_{\eta}(1,-1), \label{eq_62a}
    \\
    & V_{\xi}(-1,1) = (1-\lambda_C)\Gamma_C + \lambda_C F_{\xi}^a(-1,1)
    =\left\{ \begin{array}{ll} \Gamma_C, & \overline{AC}\perp\overline{BC}\ \text{at}\ C, \\
      F_{\xi}^a(-1,1), & \text{otherwise}.
    \end{array} \right\} =: F_{\xi}(-1,1). \label{eq_62b}
  \end{align}
\end{subequations}
Here $\Gamma_B$ and $\Gamma_C$ are free scalar parameters,
and the constants $\lambda_B$ and $\lambda_C$ are defined in~\eqref{eq_33}.

The formulation for $u(\mbs x)$ satisfying
the conditions~\eqref{eq_53} then becomes:
Find the general form of $V(\xi,\eta)$ such that
it satisfies~\eqref{eq_16a}--\eqref{eq_16b},
\eqref{eq_54}, \eqref{eq_57b}--\eqref{eq_57c}, \eqref{eq_58}, \eqref{eq_60},
and \eqref{eq_62a}--\eqref{eq_62b}.

It is noted that  $V_{\xi}(-\eta,\eta)$ and $V_{\eta}(-\eta,\eta)$
are coupled together in the Neumann condition~\eqref{eq_54}.
Moreover, the conditions
\eqref{eq_57b}--\eqref{eq_57c}, \eqref{eq_58}, \eqref{eq_60},
and \eqref{eq_62a}--\eqref{eq_62b}
are point constraints on the vertices $B$ and $C$.
We would like to further  transform these conditions to facilitate
the construction for $V(\xi,\eta)$.

These conditions are linked together through the $V$ profile
along the hypotenuse $\overline{BC}$. Let
\begingroup\small\begin{equation}\label{eq_63}
  V(-\eta,\eta) = U_{BC}(\eta) =: F(-\eta,\eta), \quad \eta\in[-1,1],
\end{equation}\endgroup
where $U_{BC}(\eta)$ denotes the edge profile of $V(\xi,\eta)$ on $\overline{BC}$
that is to be determined. Differentiating this equation w.r.t.~$\eta$ leads to,
\begingroup\small\begin{equation}\label{eq_64}
  -V_{\xi}(-\eta,\eta) + V_{\eta}(-\eta,\eta) = U_{BC}'(\eta).
\end{equation}\endgroup
In light of~\eqref{eq_57b}--\eqref{eq_57c}, \eqref{eq_58}, \eqref{eq_60},
\eqref{eq_62a}--\eqref{eq_62b}, and~\eqref{eq_63}--\eqref{eq_64}, $U_{BC}(\eta)$ must satisfy
the following conditions,
\begingroup\small\begin{align}\label{eq_65}
  \left\{
  \begin{array}{ll}
  U_{BC}(-1)=F(1,-1), &
  U_{BC}'(-1)=-F_{\xi}(1,-1)+F_{\eta}(1,-1), \\
  U_{BC}(1)=F(-1,1), &
  U_{BC}'(1)=-F_{\xi}(-1,1)+F_{\eta}(-1,1),
  \end{array}\right.
\end{align}\endgroup
where $F_{\eta}(1,-1)$ and $F_{\xi}(-1,1)$ are defined in~\eqref{eq_62a}--\eqref{eq_62b}.

The general form of $U_{BC}(\eta)$ that satisfies the conditions in~\eqref{eq_65}
is given by the TFC expression,
\begingroup\small\begin{align}\label{eq_66}
  & U_{BC}(\eta) = h(-\eta,\eta) - \mathcal Qh(-\eta,\eta) + \mathcal QF(-\eta,\eta)
  =: U_{BC}^h(\eta),
\end{align}\endgroup
where $h(\xi,\eta)$ is a free (arbitrary) function that is sufficiently differentiable.
$\mathcal Q$ is an operator  defined by
\begin{subequations}\small
\begin{align}
  \mathcal Qh(-\eta,\eta) =&\ h(1,-1)\varphi_0(\eta) + h(-1,1)\varphi_1(\eta)
  + [-h_{\xi}(1,-1)+h_{\eta}(1,-1)]\psi_0(\eta) \notag \\
  &+ [-h_{\xi}(-1,1)+h_{\eta}(-1,1)]\psi_1(\eta), \\
  \mathcal QF(-\eta,\eta) =&\ F(1,-1)\varphi_0(\eta) + F(-1,1)\varphi_1(\eta)
  + [-F_{\xi}(1,-1)+F_{\eta}(1,-1)]\psi_0(\eta) \notag \\
  &+ [-F_{\xi}(-1,1)+F_{\eta}(-1,1)]\psi_1(\eta).
\end{align}
\end{subequations}
We will refer to $h(-\eta,\eta)$ as the edge
free function on $\overline{BC}$.

Combining equations~\eqref{eq_54} and \eqref{eq_64} gives rise to
\begingroup\small
\begin{align}\label{eq_68}
  & \begin{bmatrix}V_{\xi}(-\eta,\eta)\\ V_{\eta}(-\eta,\eta) \end{bmatrix}
  =\begin{bmatrix}T_{BC}(\eta) \\ T_{BC}(\eta) \end{bmatrix}
  +\begin{bmatrix}R_{BC}(\eta)-1 \\ R_{BC}(\eta) \end{bmatrix}U_{BC}'(\eta)
  =: \begin{bmatrix} F_{\xi}(-\eta,\eta)\\ F_{\eta}(-\eta,\eta)  \end{bmatrix},
\end{align}\endgroup
where $U_{BC}(\eta)$ is given by~\eqref{eq_66}.
The foregoing derivations establish the following result:
\begin{lemma}
  The set of constraints consisting of~\eqref{eq_63} and~\eqref{eq_68},
  with $U_{BC}(\eta)$ therein given by~\eqref{eq_66},
  is equivalent to
  the set of conditions consisting of equations~\eqref{eq_54}, 
  \eqref{eq_57b}--\eqref{eq_58}, \eqref{eq_60},
  and \eqref{eq_62a}--\eqref{eq_62b}, for any combination of
  $(\lambda_B,\lambda_C)=(0/1,0/1)$ therein.
\end{lemma}

The final transformed formulation problem is reduced to:
Find the general form of $V(\xi,\eta)$ such that it satisfies
the conditions~\eqref{eq_16a}, \eqref{eq_16b}, \eqref{eq_63}
and \eqref{eq_68}, in which $U_{BC}(\eta)$ is given by~\eqref{eq_66}.
Note that $U_{BC}(\eta)$ contains an edge free function $h(-\eta,\eta)$
and the free parameters $\Gamma_B$ and $\Gamma_C$.

The transfinite interpolation for these constraints is
\begingroup\small\begin{equation}\label{eq_69}
  \begin{split}
  PF(\xi,\eta) =&\ (P_1\oplus P_2)F(\xi,\eta)
  = P_1F(\xi,\eta)+P_2F(\xi,\eta)-P_1P_2F(\xi,\eta) \\
  =&\ A_F(\xi,\eta) + \alpha(\xi,\eta)\big\{
  F(\xi,-\xi) - A_F(\xi,-\xi) + (\xi+\eta)\left[F_{\eta}(\xi,-\xi)-F_{\eta}(-1,-\xi) \right]
  \big\} \\
  &\ +\beta(\xi,\eta)\big\{
  F(-\eta,\eta) - A_F(-\eta,\eta) + (\xi+\eta)\left[F_{\xi}(-\eta,\eta)-F_{\xi}(-\eta,-1) \right]
  \big\},
  \end{split}
\end{equation}\endgroup
where
\begin{subequations}\small
  \begin{align}
    & A_F(\xi,\eta) = P_2F(\xi,\eta) := F(-1,\eta) + F(\xi,-1) - F(-1,-1), \label{eq_70a} \\
    & P_1F(\xi,\eta) := \alpha(\xi,\eta)[F(\xi,-\xi) + (\xi+\eta)F_{\eta}(\xi,-\xi)]
    + \beta(\xi,\eta)[F(-\eta,\eta) + (\xi+\eta)F_{\xi}(-\eta,\eta)].
  \end{align}
\end{subequations}
Here $P_2F$ is obtained by the boolean sum of two 0-th order Taylor approximation
operators, $P_2F(\xi,\eta)=(T_{\xi}\oplus T_{\eta})F(\xi,\eta)$,
with $T_{\xi}F(\xi,\eta)=F(-1,\eta)$ and $T_{\eta}F(\xi,\eta)=F(\xi,-1)$.
$P_1F$ is obtained by blending two Taylor approximations of 1st order
along the $\xi$ and $\eta$ directions about $(-\eta,\eta)$ and $(\xi,-\xi)$
on $\overline{BC}$,
$T_{\xi}^{BC}F(\xi,\eta)=F(-\eta,\eta) + (\xi+\eta)F_{\xi}(-\eta,\eta)$
and $T_{\eta}^{BC}F(\xi,\eta)=F(\xi,-\xi) + (\eta+\xi)F_{\eta}(\xi,-\xi)$.
By verification we can establish the following result:
\begin{property}\label{lem_29}
  $PF(\xi,\eta)$ given by~\eqref{eq_69} satisfies,
  \begin{equation}\small
    \left\{\begin{split}
    & PF(-1,\eta) = F(-1,\eta), \ PF(\xi,-1) = F(\xi,-1), \
    PF(-\eta,\eta)=F(-\eta,\eta), \\
    & (PF)_{\xi}(-\eta,\eta)=F_{\xi}(-\eta,\eta),  \quad
    (PF)_{\eta}(-\eta,\eta)=F_{\eta}(-\eta,\eta).
    \end{split}\right.
  \end{equation}
\end{property}

The preliminary general form for $V(\xi,\eta)$ is
\begingroup\small\begin{align}\label{eq_72}
  & V(\xi,\eta) = g(\xi,\eta) - Pg(\xi,\eta) + PF(\xi,\eta),
\end{align}\endgroup
where $g(\xi,\eta)$ is a sufficiently-differentiable free function.
$Pg$ is given by, in light of~\eqref{eq_69},
\begin{equation}\label{eq_73}\small
  \begin{split}
    Pg(\xi,\eta) =&\ A_g(\xi,\eta) + \alpha(\xi,\eta)\big\{
  g(\xi,-\xi) - A_g(\xi,-\xi) + (\xi+\eta)\left[g_{\eta}(\xi,-\xi)-g_{\eta}(-1,-\xi) \right]
  \big\} \\
  &\ +\beta(\xi,\eta)\big\{
  g(-\eta,\eta) - A_g(-\eta,\eta) + (\xi+\eta)\left[g_{\xi}(-\eta,\eta)-g_{\xi}(-\eta,-1) \right]
  \big\},
  \end{split}
\end{equation}
where $A_g(\xi,\eta)=g(-1,\eta)+g(\xi,-1)-g(-1,-1)$.

Lastly, we update the placeholder terms in $PF$ by their compatible data.
Note that $\Gamma_B$, $\Gamma_C$ and $h(-\eta,\eta)$ in $U_{BC}(\eta)$ are free
parameters and free edge functions.
We choose
\begingroup\small\begin{equation}
  \Gamma_B = g_{\eta}(1,-1), \quad \Gamma_C=g_{\xi}(-1,1), \quad
  h(-\eta,\eta) = g(-\eta,\eta).
\end{equation}\endgroup
This eliminates $h(-\eta,\eta)$, $\Gamma_B$ and $\Gamma_C$ from
the final formulation, and leaves $g(\xi,\eta)$ as the sole unknown
to be determined.
The updated transfinite interpolation is,
\begingroup\small\begin{equation}\label{eq_76}
  \begin{split}
  PF^g(\xi,\eta) =&\ A_F(\xi,\eta) + \alpha(\xi,\eta)\big\{
  U_{BC}^g(-\xi) - A_F(\xi,-\xi) + (\xi+\eta)\left[F_{\eta}^{gg}(\xi,-\xi)-F_{\eta}(-1,-\xi) \right]
  \big\} \\
  &\ +\beta(\xi,\eta)\big\{
  U_{BC}^g(\eta) - A_F(-\eta,\eta) + (\xi+\eta)\left[F_{\xi}^{gg}(-\eta,\eta)-F_{\xi}(-\eta,-1) \right]
  \big\},
  \end{split}
\end{equation}\endgroup
where $A_F(\xi,\eta)$ is defined in~\eqref{eq_70a}. The other symbols
are given by
\begin{subequations}\label{eq_77}\small
  \begin{align}
    & U_{BC}^g(\eta) = g(-\eta,\eta) - \mathcal Qg(-\eta,\eta) + \mathcal QF^g(-\eta,\eta); \label{eq_77a} \\
    & \mathcal Q g(-\eta,\eta) = g(1,-1)\varphi_0(\eta) + g(-1,1)\varphi_1(\eta)
    + [-g_{\xi}(1,-1)+g_{\eta}(1,-1)]\psi_0(\eta) \notag \\
    &\qquad\qquad\quad\  + [-g_{\xi}(-1,1)+g_{\eta}(-1,1)]\psi_1(\eta); \label{eq_77b} \\
    & \mathcal Q F^g(-\eta,\eta) = F(1,-1)\varphi_0(\eta) + F(-1,1)\varphi_1(\eta)
    + [-F_{\xi}(1,-1)+F_{\eta}^g(1,-1)]\psi_0(\eta) \notag \\
    &\qquad\qquad\qquad  + [-F_{\xi}^g(-1,1)+F_{\eta}(-1,1)]\psi_1(\eta); \label{eq_77c} \\
    & F_{\eta}^g(1,-1) = (1-\lambda_B)g_{\eta}(1,-1) + \lambda_B F_{\eta}^a(1,-1); 
    \label{eq_77d} \\
    & F_{\xi}^g(-1,1) = (1-\lambda_C)g_{\xi}(-1,1) + \lambda_C F_{\xi}^a(-1,1);
    \label{eq_77e} \\
    & \begin{bmatrix} F_{\xi}^{gg}(-\eta,\eta)\\ F_{\eta}^{gg}(-\eta,\eta)  \end{bmatrix}
    = \begin{bmatrix}T_{BC}(\eta) \\ T_{BC}(\eta) \end{bmatrix}
    +\begin{bmatrix}R_{BC}(\eta)-1 \\ R_{BC}(\eta) \end{bmatrix}\frac{d}{d\eta}U_{BC}^{g}(\eta).
  \end{align}
\end{subequations}
\begin{property}\label{lem_211}
  The following relations hold, for any combination of $(\lambda_B,\lambda_C)=(0/1,0/1)$,
  \begin{align}\small
    \left\{\begin{array}{l}
    F_{\eta}^{gg}(-1,1) = F_{\eta}(-1,1), \quad
    F_{\xi}^{gg}(1,-1) = F_{\xi}(1,-1), \quad
    F_{\eta}^{gg}(-\eta,\eta) - F_{\xi}^{gg}(-\eta,\eta) = \frac{d}{d\eta}U_{BC}^{g}(\eta), \\
    R_{BC}(\eta)F_{\xi}^{gg}(-\eta,\eta) + [1-R_{BC}(\eta)]F_{\eta}^{gg}(-\eta,\eta)
    =T_{BC}(\eta).
  \end{array}\right.\end{align}
\end{property}

\begin{property}\label{lem_212}
  $PF^g(\xi,\eta)$ given by~\eqref{eq_76}, for any sufficiently differentiable
  $g(\xi,\eta)$ therein, satisfies:
  \begin{equation}\small
    \left\{\begin{split}
    & PF^g(-1,\eta) = F(-1,\eta), \quad PF^g(\xi,-1) = F(\xi,-1), \quad
    PF^g(-\eta,\eta)=U_{BC}^g(\eta),  \\
    & (PF^g)_{\xi}(-\eta,\eta)=F_{\xi}^{gg}(-\eta,\eta),\quad
    (PF^g)_{\eta}(-\eta,\eta)=F_{\eta}^{gg}(-\eta,\eta).
    \end{split}\right.
  \end{equation}
\end{property}
\begin{proof}By using Property~\ref{lem_211} and the relations in~\eqref{eq_32}
  and \eqref{eq_30} \end{proof}

The final form for $V(\xi,\eta)$ is then given by,
\begingroup\small\begin{equation}\label{eq_80}
  V(\xi,\eta) = g(\xi,\eta) - Pg(\xi,\eta) + PF^g(\xi,\eta),
\end{equation}\endgroup
where $g(\xi,\eta)$ is a free function, $Pg$ is defined in~\eqref{eq_73},
and $PF^g$ is given by~\eqref{eq_76}.
\begin{theorem}
  $V(\xi,\eta)$ given by~\eqref{eq_80}, for any sufficiently differentiable
  $g(\xi,\eta)$ therein, satisfies the Dirichlet conditions~\eqref{eq_16a} and
  \eqref{eq_16b} and the Neumann condition~\eqref{eq_54}.
\end{theorem}
\begin{proof} By noting $V_{\xi}(-\eta,\eta)=F_{\xi}^{gg}(-\eta,\eta)$
  and $V_{\eta}(-\eta,\eta)=F_{\eta}^{gg}(-\eta,\eta)$, and by
  using Properties \ref{lem_29}, \ref{lem_212}, and \ref{lem_211}.
\end{proof}

\begin{remark}
  The above formulation, with no change, applies to
  configurations where the boundary $\partial\Omega$ is smooth at vertices $B$ or $C$.
\end{remark}


\subsubsection{Two Neumann Conditions Imposed on the Hypotenuse and a Leg}
\label{sec_n_ACBC}

This appendix outlines the formulation for two Neumann conditions imposed
on the hypotenuse and a leg.
Without loss of generality, we assume that the two Neumann boundaries correspond
to $\overline{AC}$ and $\overline{BC}$.

Consider a scalar function $u(\mbs x)$ on the general triangular
domain $\Omega$ in Figure~\ref{fg_1}(a) that satisfies,
\begingroup\small\begin{equation}
  u|_{\mbs x\in\overline{AB}}=u_{AB}(\mbs x), \quad
  \mbs n\cdot\nabla u|_{\mbs x\in\overline{BC}} = u_{nBC}(\mbs x), \quad
  \mbs n\cdot\nabla u|_{\mbs x\in\overline{AC}} = u_{nAC}(\mbs x).
\end{equation}\endgroup
Leveraging the map $\mbs x(\xi,\eta)$ from Section \ref{sec_21}, these conditions are
transformed into \eqref{eq_16a}, \eqref{eq_54} and \eqref{eq_21},
respectively. Our goal is to construct the general form of $V(\xi,\eta)$
that satisfies these conditions.

An analysis of these conditions lead to the following compatibility constraints.
At vertex $A$ the conditions \eqref{eq_24a}--\eqref{eq_24b} must hold.
At vertex $B$ the conditions \eqref{eq_59a}--\eqref{eq_59b}, or equivalently
\eqref{eq_62a} (with $\Gamma_B$ being a free parameter), must be satisfied.
Vertex $C$ is where the two NBCs meet. A combination of \eqref{eq_21a}
and \eqref{eq_54} evaluated at $\eta=1$ gives rise to
\begingroup\small
\begin{align}\label{eq_n138}
  & \begin{bmatrix}V_{\xi}(-1,1) \\ V_{\eta}(-1,1) \end{bmatrix}
  = \begin{bmatrix}F_{\xi}^a(-1,1) \\ F_{\eta}^a(-1,1) \end{bmatrix}
  = \begin{bmatrix}
    \frac{[1-R_{BC}(1)]T_{AC}(1)-S_{AC}T_{BC}(1)}{1-R_{BC}(1)[1+S_{AC}(1)]} \\
    \frac{-R_{BC}(1)T_{AC}(1)+T_{BC}(1)}{1-R_{BC}(1)[1+S_{AC}(1)]}
  \end{bmatrix}
  =: \begin{bmatrix}F_{\xi}(-1,1) \\ F_{\eta}(-1,1) \end{bmatrix}.
\end{align}
\endgroup
Note that $1-R_{BC}(1)[1+S_{AC}(1)]\geqslant 0$ by the Cauchy-Schwarz inequality,
with equality if and only if $\overline{AC}\parallel \overline{BC}$ at $C$
(a smooth/degenerate vertex). We require that $C$ be a genuine corner, thus excluding
this degenerate case.
Differentiating \eqref{eq_21a} and \eqref{eq_54} w.r.t.~$\eta$ and evaluating them
at $\eta=1$ lead to,
\begingroup\small\begin{align}\label{eq_n139}
  & V_{\xi\eta}(-1,1) = -S_{AC}(1)V_{\eta\eta}(-1,1)
  + \big[T_{AC}'(1) - S_{AC}'(1)F_{\eta}^a(-1,1) \big] =: F_{\xi\eta}(-1,1),
\end{align}\endgroup
and 
\begin{subequations}\small
  \begin{align}
    & 0 = Q^a(-1,1) := \big[T_{BC}'(1)-R_{BC}'(1)F_{\xi}^a(-1,1) + R_{BC}'(1)F_{\eta}^a(-1,1) \big]
    \notag \\
    &\qquad\qquad\qquad\quad\ \
    - \big[2R_{BC}(1)-1 \big]\big[T_{AC}'(1)-S_{AC}'(1)F_{\eta}^a(-1,1) \big],
    \quad \text{if}\ \overline{AC}\perp\overline{BC}\ \text{at}\ C; \\
    & V_{\xi\xi}(-1,1) = S_1V_{\eta\eta} + S_2, \quad
    \text{if}\ \overline{AC}\not\perp\overline{BC}\ \text{at}\ C, \label{eq_n140b}
  \end{align}
\end{subequations}
where
\begingroup\small
\begin{equation}
  S_1 = 1+[1-2R_{BC}(1)]\frac{\|\mbs x_{BC}'(1)\|^2}{\|\mbs x_{AC}'(1) \|^2}, \quad
  S_2 = -\frac{Q^a(-1,1)}{R_{BC}(1)},
\end{equation}\endgroup
and we have used  \eqref{eq_27} and \eqref{eq_n112}.
We then transform the constraint \eqref{eq_n140b} into an unconditional
form
\begingroup\small\begin{align}\label{eq_n142}
  & V_{\xi\xi}(-1,1) = (1-\lambda_C)\Gamma_C^{\xi\xi} + \lambda_C\big[S_1V_{\eta\eta} + S_2\big]
  =: F_{\xi\xi}(-1,1),
\end{align}\endgroup
where $\Gamma_C^{\xi\xi}$ is a free parameter. Finally we note that there is no constraint
on $V(-1,1)$ or $V_{\eta\eta}(-1,1)$ and they are completely free. This is equivalent to,
\begingroup\small\begin{align}\label{eq_n143}
  & V(-1,1)=\Gamma_C^V=:F(-1,1), \quad
  V_{\eta\eta}(-1,1)= \Gamma_C^{\eta\eta} =: F_{\eta\eta}(-1,1),
\end{align}\endgroup
where $\Gamma_C^V$ and $\Gamma_{C}^{\eta\eta}$ are free parameters.

At this point we proceed with two strategies:
(i) Transform the point constraints \eqref{eq_24b}, \eqref{eq_62a}, 
\eqref{eq_n138}, \eqref{eq_n139}, \eqref{eq_n142}--\eqref{eq_n143} into
conditions about the edge profiles $V(-1,\eta)$ and $V(-\eta,\eta)$.
(ii) Reformulate the NBC \eqref{eq_54} on the hypotenuse
into constraints about $V_{\xi}(-\eta,\eta)$ and $V_{\eta}(-\eta,\eta)$
in a way similar to in Section \ref{sec_2312}.

We transform the point constraints into
\begingroup\small
\begin{align}\label{eq_n144}
  V(-1,\eta) = F(-1,\eta) :=&\ F(-1,-1)\rho_0(\eta) + F(-1,1)\rho_1(\eta)
  + \lambda_A F_{\eta}(-1,-1)\upsilon_0(\eta) 
  + F_{\eta}(-1,1)\upsilon_1(\eta) \notag \\
  &+ F_{\eta\eta}(-1,1)\omega_1(\eta),
\end{align}\endgroup
on $\overline{AC}$, and
\begingroup\small\begin{align}\label{eq_n145}
  V(-\eta,\eta) = U_{BC}(\eta) =: F(-\eta,\eta),
\end{align}\endgroup
on $\overline{BC}$. Here $U_{BC}(\eta)$ is given by
\begingroup\small\begin{align}\label{eq_n146}
  U_{BC}(\eta) =&\ h(-\eta,\eta) - \mathcal Qh(-\eta,\eta) + \mathcal QF(-\eta,\eta),
\end{align}\endgroup
where $h(-\eta,\eta)$ is an edge free function and
\begin{subequations}\label{eq_147}
  \small
  \begin{align}
   \mathcal Qh(-\eta,\eta) =&\ h(1,-1)\rho_0(\eta) + h(-1,1)\rho_1(\eta)
    + [-h_{\xi}(1,-1)+h_{\eta}(1,-1)]\upsilon_0(\eta) \notag \\
    & + [-h_{\xi}(-1,1) + h_{\eta}(-1,1)]\upsilon_1(\eta)
    + [h_{\xi\xi}(-1,1) - 2h_{\xi\eta}(-1,1) + h_{\eta\eta}(-1,1)]\omega_1(\eta), \label{eq_n147a} \\
   \mathcal QF(-\eta,\eta) =&\ F(1,-1)\rho_0(\eta) + F(-1,1)\rho_1(\eta)
    + [-F_{\xi}(1,-1)+F_{\eta}(1,-1)]\upsilon_0(\eta) \notag \\
    & + [-F_{\xi}(-1,1) + F_{\eta}(-1,1)]\upsilon_1(\eta)
    + [F_{\xi\xi}(-1,1) - 2F_{\xi\eta}(-1,1) + F_{\eta\eta}(-1,1)]\omega_1(\eta). \label{eq_n147b}
  \end{align}
\end{subequations}
Note that we have also used \eqref{eq_64},
which follows from \eqref{eq_n145}.

The Neumann condition \eqref{eq_54}, together with \eqref{eq_n145},
can be reformulated into the form \eqref{eq_68} about $V_{\xi}(-\eta,\eta)$
and $V_{\eta}(-\eta,\eta)$, where $U_{BC}(\eta)$ is now defined by
\eqref{eq_n146}--\eqref{eq_n147b} and contains the free parameters
$\Gamma_B$, $\Gamma_C^V$, $\Gamma_C^{\xi\xi}$, and $\Gamma_C^{\eta\eta}$ through
the $\mathcal QF(-\eta,\eta)$ terms.

Therefore the formulation problem becomes the following: Find the general form of
$V(\xi,\eta)$ such that the conditions \eqref{eq_16a}, \eqref{eq_n144},
\eqref{eq_21b}, \eqref{eq_n145}, \eqref{eq_68}, \eqref{eq_n139},
and \eqref{eq_n142}--\eqref{eq_n143} are satisfied.

The transfinite interpolation satisfying these constraints is given by,
\begingroup\small
\begin{align}\label{eq_148}
  PF(\xi,\eta) =&\ (P_1\oplus P_2)F(\xi,\eta)
  =  P_1F(\xi,\eta) + P_2F(\xi,\eta) - P_1P_2F(\xi,\eta) \notag \\
  =&\ A_F(\xi,\eta) + \kappa(\xi,\eta)\big[
    F(\xi,-\xi) + (\xi+\eta)F_{\eta}(\xi,-\xi) - A_F(\xi,-\xi) 
    - (\xi+\eta)A_{F,\eta}(\xi,-\xi)
    \big] \notag \\
  &+ \delta(\xi,\eta)\big[F(-\eta,\eta) + (\xi+\eta)F_{\xi}(-\eta,\eta) - A_F(-\eta,\eta)
    - (\xi+\eta)A_{F,\xi}(-\eta,\eta) \big].
\end{align}\endgroup
Here $\kappa(\xi,\eta)$ and $\delta(\xi,\eta)$ are the $C^2$ Barnhill polynomials.
The operators $P_1$ and $P_2$ are defined as the blending (based on Barnhill polynomials)
or the boolean sum of Taylor approximation operators, and are given by,
\begin{subequations}\label{eq_149}
  \small
  \begin{align}
    P_1F(\xi,\eta) &= \kappa(\xi,\eta)[F(\xi,-\xi) + (\xi+\eta)F_{\eta}(\xi,-\xi)]
    + \delta(\xi,\eta)[F(-\eta,\eta) + (\xi+\eta)F_{\xi}(-\eta,\eta)];
    \\
    A_F(\xi,\eta) &= P_2F(\xi,\eta) := (T_{\xi}\oplus T_{\eta})F(\xi,\eta)
    = T_{\xi}F(\xi,\eta) + T_{\eta}F(\xi,\eta) - T_{\xi}T_{\eta}F(\xi,\eta) \notag \\
    &\qquad\qquad\quad\ \
    = F(\xi,-1) + F(-1,\eta) - F(-1,-1) + (\xi+1)[F_{\xi}(-1,\eta)-F_{\xi}(-1,-1)]; \label{eq_149b} \\
    T_{\xi}F(\xi,\eta) &= F(-1,\eta) + (\xi+1)F_{\xi}(-1,\eta); \qquad
    T_{\eta}F(\xi,\eta) = F(\xi,-1).
  \end{align}
\end{subequations}
Note that the term $F(-1,\eta)$ in the $A_F(\xi,\eta)$ expression \eqref{eq_149b}
is defined by \eqref{eq_n144}.
The following properties about $PF$ can be established by verification.
\begin{property}\label{lem_57}
  $PF(\xi,\eta)$ given by \eqref{eq_148} satisfies
  \begin{align}\small
    \left\{
    \begin{array}{l}
      PF(\xi,-1) = F(\xi,-1), \quad PF(-\eta,\eta)=F(-\eta,\eta), \\
      PF(-1,\eta)=F(-1,\eta)=F(-1,-1)\rho_0(\eta) + F(-1,1)\rho_1(\eta)
      + \lambda_AF_{\eta}(-1,-1)\upsilon_0(\eta) + F_{\eta}(-1,1)\upsilon_1(\eta)  \\
      \qquad\qquad\qquad\qquad\qquad\quad
      + F_{\eta\eta}(-1,1)\omega_1(\eta), \\
      PF_{\xi}(-1,\eta) = F_{\xi}(-1,\eta), \quad
      PF_{\xi}(-\eta,\eta) = F_{\xi}(-\eta,\eta), \quad
      PF_{\eta}(-\eta,\eta) = F_{\eta}(-\eta,\eta), \\
      PF_{\xi\xi}(-1,1) = F_{\xi\xi}(-1,1), \quad PF_{\eta\eta}(-1,1)=F_{\eta\eta}(-1,1), \quad
      PF_{\xi\eta}(-1,1) = F_{\xi\eta}(-1,1).
    \end{array}\right.
  \end{align}
\end{property}

This gives rise to the preliminary form for $V(\xi,\eta)$,
\begingroup\small\begin{align}\label{eq_n151}
  & V(\xi,\eta) = g(\xi,\eta) - Pg(\xi,\eta) + PF(\xi,\eta),
\end{align}\endgroup
where $g(\xi,\eta)$ is a free function, and $Pg$ is defined
according to \eqref{eq_148} and specifically given by
\begin{subequations}\label{eq_n152}
  \small
\begin{align}
  Pg(\xi,\eta) =&\ A_g(\xi,\eta) + \kappa(\xi,\eta)\big[
    g(\xi,-\xi) + (\xi+\eta)g_{\eta}(\xi,-\xi) - A_g(\xi,-\xi) 
    - (\xi+\eta)A_{g,\eta}(\xi,-\xi)
    \big] \notag \\
  &+ \delta(\xi,\eta)\big[g(-\eta,\eta) + (\xi+\eta)g_{\xi}(-\eta,\eta) - A_g(-\eta,\eta)
    - (\xi+\eta)A_{g,\xi}(-\eta,\eta) \big], \\
  A_g(\xi,\eta) =&\ g(\xi,-1) - g(-1,-1)
  + (\xi+1)\big[g_{\xi}(-1,\eta) - g_{\xi}(-1,-1) \big] + \big[g(-1,-1)\rho_0(\eta) 
  \notag \\
  &+ g(-1,1)\rho_1(\eta)
    + \lambda_Ag_{\eta}(-1,-1)\upsilon_0(\eta) + g_{\eta}(-1,1)\upsilon_1(\eta) 
    + g_{\eta\eta}(-1,1)\omega_1(\eta) \big].
\end{align}
\end{subequations}

Lastly, we update the placeholder terms in $PF(\xi,\eta)$ by their compatible data.
We choose the free parameters and the edge free function $h(-\eta,\eta)$ involved in $PF$
as follows,
\begingroup\small
\begin{equation}\label{eq_n153}
  \Gamma_B = g_{\eta}(1,-1), \quad \Gamma_C^V=g(-1,1), \quad
  \Gamma_C^{\eta\eta}=g_{\eta\eta}(-1,1), \quad \Gamma_C^{\xi\xi} = g_{\xi\xi}(-1,1),
  \quad h(-\eta,\eta) = g(-\eta,\eta).
\end{equation}\endgroup
The updated transfinite interpolation is then given by
\begingroup\small
\begin{align}\label{eq_n154}
  PF^g(\xi,\eta) =&\ A_F^g(\xi,\eta) + \kappa(\xi,\eta)\big[
    U_{BC}^g(-\xi) + (\xi+\eta)F_{\eta}^{gg}(\xi,-\xi) - A_F^g(\xi,-\xi) 
    - (\xi+\eta)A^g_{F,\eta}(\xi,-\xi)
    \big] \notag \\
  &+ \delta(\xi,\eta)\big[U_{BC}^g(\eta) + (\xi+\eta)F_{\xi}^{gg}(-\eta,\eta) - A_F^g(-\eta,\eta)
    - (\xi+\eta)A^g_{F,\xi}(-\eta,\eta) \big]. 
\end{align}\endgroup
The symbols in this equation are defined by
\begin{subequations}
  \small
  \begin{align}
    A_F^g(\xi,\eta) =&\  F(\xi,-1) + F^g(-1,\eta) - F(-1,-1)
    + (\xi+1)\big[F_{\xi}^g(-1,\eta) - F_{\xi}(-1,-1) \big],  \\
    F^g(-1,\eta) =&\ F(-1,-1)\rho_0(\eta) + g(-1,1)\rho_1(\eta)
    + \lambda_AF_{\eta}^a(-1,-1)\upsilon_0(\eta) + F_{\eta}^a(-1,1)\upsilon_1(\eta) \notag \\
    &+ g_{\eta\eta}(-1,1)\omega_1(\eta),  \\
    U_{BC}^g(\eta) =&\ g(-\eta,\eta) - \mathcal Qg(-\eta,\eta) + \mathcal QF^g(-\eta,\eta), \\
    \mathcal Qg(-\eta,\eta) =&\ g(1,-1)\rho_0(\eta) + g(-1,1)\rho_1(\eta)
    + \big[-g_{\xi}(1,-1) + g_{\eta}(1,-1) \big]\upsilon_0(\eta) \notag \\
    &+ \big[-g_{\xi}(-1,1) + g_{\eta}(-1,1) \big]\upsilon_1(\eta)
    + \big[g_{\xi\xi}(-1,1) - 2g_{\xi\eta}(-1,1) + g_{\eta\eta}(-1,1) \big]\omega_1(\eta), \\
    \mathcal QF^g(-\eta,\eta) =&\ F(1,-1)\rho_0(\eta) + g(-1,1)\rho_1(\eta)
    + \big[-F_{\xi}(1,-1) + F^g_{\eta}(1,-1) \big]\upsilon_0(\eta) \notag \\
    &+ \big[-F^a_{\xi}(-1,1) + F^a_{\eta}(-1,1) \big]\upsilon_1(\eta)
    + \big[F^g_{\xi\xi}(-1,1) - 2F^g_{\xi\eta}(-1,1) + g_{\eta\eta}(-1,1) \big]\omega_1(\eta), \\
    F_{\eta}^g(1,-1) =&\ (1-\lambda_B)g_{\eta}(1,-1) + \lambda_B F_{\eta}^a(1,-1), \\
    F_{\xi\xi}^g(-1,1) =&\ (1-\lambda_C)g_{\xi\xi}(-1,1) + \lambda_C[S_1g_{\eta\eta}(-1,1)+S_2], \\
    F_{\xi\eta}^g(-1,1) =&\ -S_{AC}(1)g_{\eta\eta}(-1,1) +
    \big[T_{AC}'(1) - S_{AC}'(1)F_{\eta}^a(-1,1) \big], \\
    F_{\xi}^{gg}(-\eta,\eta) =&\ T_{BC}(\eta) + [R_{BC}(\eta)-1]\frac{d}{d\eta} U_{BC}^{g}(\eta), \\
    F_{\eta}^{gg}(\xi,-\xi) = &\ T_{BC}(-\xi) +
    R_{BC}(-\xi) \left.\frac{d}{d\eta}U_{BC}^g(\eta)\right|_{\eta=-\xi}, \\
    F_{\xi}^g(-1,\eta) =&\ T_{AC}(\eta) - S_{AC}(\eta) V_{\eta}^g(-1,\eta), \\
    V_{\eta}^g(-1,\eta) =&\ g_{\eta}(-1,\eta) - [g(-1,-1)-F(-1,-1)]\rho_0'(\eta)
    -\lambda_A[g_{\eta}(-1,-1) - F_{\eta}^a(-1,-1)]\upsilon_0'(\eta) \notag \\
    &- [g_{\eta}(-1,1) - F_{\eta}^a(-1,1)]\upsilon_1'(\eta).
  \end{align}
\end{subequations}
The following properties about $PF^g(\xi,\eta)$ can be established by verification.
\begin{property}\label{lem_58}
  $PF^g(\xi,\eta)$ given by \eqref{eq_n154} satisfies, for any
  sufficiently differentiable $g(\xi,\eta)$ therein,
  \begingroup\small\begin{align}
    \left\{\begin{array}{l}
    PF^g(\xi,-1) = F(\xi,-1), \quad PF^g(-1,\eta)=F^g(-1,\eta), \quad
    PF^g(-\eta,\eta) = U^g_{BC}(\eta), \\
    (PF^g)_{\xi}(-1,\eta)=F_{\xi}^g(-1,\eta), \quad
    (PF^g)_{\xi}(-\eta,\eta) = F_{\xi}^{gg}(-\eta,\eta), \quad
    (PF^g)_{\eta}(-\eta,\eta) = F_{\eta}^{gg}(-\eta,\eta), \\
    (PF^g)_{\xi\xi}(-1,1) = F_{\xi\xi}^g(-1,1), \quad
    (PF^g)_{\xi\eta}(-1,1) = F_{\xi\eta}^g(-1,1), \quad
    (PF^g)_{\eta\eta}(-1,1) = g_{\eta\eta}(-1,1).
    \end{array}\right.
  \end{align}\endgroup
\end{property}

The final form for $V(\xi,\eta)$ is given by
\begingroup\small\begin{align}\label{eq_n157}
  V(\xi,\eta) = g(\xi,\eta) - Pg(\xi,\eta) + PF^g(\xi,\eta),
\end{align}\endgroup
where $g(\xi,\eta)$ denotes a free function,
$Pg$ is defined in \eqref{eq_n152}, and $PF^g$ is given by \eqref{eq_n154}.
The following result holds.
\begin{theorem}
  $V(\xi,\eta)$ given by \eqref{eq_n157} satisfies the Dirichlet
  condition \eqref{eq_16a} and the Neumann conditions \eqref{eq_21a}
  and \eqref{eq_54}, for any $g(\xi,\eta)$ therein that is
  sufficiently differentiable.
\end{theorem}
\begin{proof} By using Properties \ref{lem_57} and \ref{lem_58}, and for
  the Neumann conditions by further noting that
  \begingroup\small\begin{align}
    \left\{\begin{array}{l}
    V_{\xi}(-1,\eta) = F_{\xi}^g(-1,\eta), \quad V_{\xi}(-\eta,\eta) = F_{\xi}^{gg}(-\eta,\eta),
    \quad V_{\eta}(-\eta,\eta) = F_{\eta}^{gg}(-\eta,\eta), \\
    V_{\eta}(-1,\eta) = g_{\eta}(-1,\eta) - \big[g(-1,-1)-F(-1,-1) \big]\rho_0'(\eta) 
    -\lambda_A\big[g_{\eta}(-1,-1)-F_{\eta}^a(-1,-1) \big]\upsilon_0'(\eta) \\
    \qquad\qquad\quad\
    -\big[g_{\eta}(-1,1)-F_{\eta}^a(-1,1) \big]\upsilon_1'(\eta) = V_{\eta}^g(-1,\eta).
    \end{array}\right.
  \end{align}\endgroup
\end{proof}

\begin{remark}
  The use of $C^2$ Barnhill polynomials in the construction of
  the $P_1$ operator is critical.
  They are required to ensure the point constraints
  involving the second derivatives (equations \eqref{eq_n139}, \eqref{eq_n142} and
  \eqref{eq_n143}) at the double-Neumann vertex $C$.
\end{remark}

\begin{remark}\label{rem_53}
  If the Neumann conditions are imposed on $\overline{BC}$ and
  $\overline{AB}$ (instead of $\overline{AC}$),
  a reciprocal formulation can be constructed analogously.
\end{remark}

\subsubsection{Sole Robin Condition Imposed on the Hypotenuse}
\label{sec_r_BC}

For this configuration
we aim to construct the general form of $u(\mbs x)$ that satisfies,
\begingroup\small\begin{equation}
  u|_{\mbs x\in\overline{AB}} = u_{AB}(\mbs x),  \quad
     u|_{\mbs x\in\overline{AC}} = u_{AC}(\mbs x),  \quad
     \mbs n\cdot\nabla u|_{\mbs x\in\overline{BC}} + \alpha_{BC}u|_{\mbs x\in\overline{BC}} = u_{rBC}(\mbs x),
\end{equation}\endgroup
where $\alpha_{BC}$ and $u_{rBC}$ are the Robin coefficient and Robin
data on $\overline{BC}$, respectively.
The transformed Robin condition is, under the map from Section \ref{sec_21},
\begingroup\small\begin{align}\label{eq_117}
  R_{BC}(\eta) V_{\xi}(-\eta,\eta) + [1-R_{BC}(\eta)]V_{\eta}(-\eta,\eta)
  + \alpha_{BC}W_{BC}(\eta)V(-\eta,\eta)
  = T_{rBC}(\eta), 
\end{align}\endgroup
where $R_{BC}(\eta)$ and $W_{BC}(\eta)$ are defined in \eqref{eq_55},
and $T_{rBC}(\eta)=u_{rBC}(\mbs x(-\eta,\eta))W_{BC}(\eta)$.
The Dirichlet conditions are transformed into \eqref{eq_16a} and \eqref{eq_16b},
respectively.


By an analysis similar to in Section \ref{sec_2312}, we attain the compatibility
constraints at vertex $B$,
\begingroup\small\begin{subequations}\label{eq_118}
  \begin{align}
    & T_{rBC}(-1) - R_{BC}(-1)F_{\xi}(1,-1) - \alpha_{BC}W_{BC}(-1)F(1,-1) = 0, \quad
    \text{if}\ \overline{AB}\perp\overline{BC}\ \text{at}\ B, \label{eq_n161a} \\
    & V_{\eta}(1,-1) = \frac{T_{BC}(-1) - R_{BC}(-1)F_{\xi}(1,-1) - \alpha_{BC}W_{BC}(-1)F(1,-1)}{1-R_{BC}(-1)} =: F_{\eta}^a(1,-1), \
    \notag \\
    &\qquad\qquad\qquad
    \text{if}\ \overline{AB}\not\perp\overline{BC}\ \text{at}\ B. \label{eq_118d}
  \end{align}
\end{subequations}\endgroup
At vertex $C$ we have the following constraints,
\begingroup\small\begin{subequations}\label{eq_119}
  \begin{align}
    & T_{rBC}(1) - [1-R_{BC}(1)]F_{\eta}(-1,1) - \alpha_{BC}W_{BC}(1)F(-1,1) = 0, \quad
    \text{if}\ \overline{AC}\perp\overline{BC}\ \text{at}\ C, \\
    & V_{\xi}(-1,1) = \frac{T_{rBC}(1) - [1-R_{BC}(1)]F_{\eta}(-1,1) - \alpha_{BC}W_{BC}(1)F(-1,1)}{R_{BC}(1)}
    =: F_{\xi}^a(-1,1), \notag \\
    &\qquad\qquad\qquad
    \text{if}\ \overline{AC}\not\perp\overline{BC}\ \text{at}\ C. \label{eq_119d}
  \end{align}
\end{subequations}\endgroup

The construction for $V(\xi,\eta)$ mirrors that of Section \ref{sec_2312},
with $T_{BC}(\eta)$ therein replaced by $T_{rBC}(\eta)$,
and with $F_{\eta}^a(1,-1)$ and $F_{\xi}^a(-1,1)$
now defined by \eqref{eq_118d} and \eqref{eq_119d}. The final form
for $V(\xi,\eta)$ is given by
\begin{align}\label{eq_120}
  V(\xi,\eta) = g(\xi,\eta) - Pg(\xi,\eta) + PF^g(\xi,\eta),
\end{align}
where $g$ is a free function and $Pg$ is defined by \eqref{eq_73}.
$PF^g$ is given by \eqref{eq_76}, 
in which $U_{BC}^g(\eta)$ is defined by \eqref{eq_77a}--\eqref{eq_77e} and 
$F_{\xi}^{gg}(-\eta,\eta)$ and $F_{\eta}^{gg}(\xi,-\xi)$ are given by
\begingroup\small
\begin{align}\label{eq_121}
  & \begin{bmatrix} F_{\xi}^{gg}(-\eta,\eta)\\ F_{\eta}^{gg}(-\eta,\eta)  \end{bmatrix}
  = \begin{bmatrix}T_{rBC}(\eta) - \alpha_{BC}W_{BC}(\eta)U_{BC}^g(\eta) \\
  T_{rBC}(\eta) - \alpha_{BC}W_{BC}(\eta)U_{BC}^g(\eta) \end{bmatrix}
    +\begin{bmatrix}R_{BC}(\eta)-1 \\ R_{BC}(\eta) \end{bmatrix}\frac{d}{d\eta}U_{BC}^{g}(\eta).
\end{align}
\endgroup

\begin{theorem}
  $V(\xi,\eta)$ given by \eqref{eq_120} satisfies the Dirichlet conditions
  \eqref{eq_16a} and \eqref{eq_16b} and the Robin condition \eqref{eq_117},
  for any $g(\xi,\eta)$ therein that is sufficiently differentiable.
\end{theorem}
\begin{proof} By verification and by noting that
  \begingroup\small\begin{equation}
    \left\{\begin{array}{l}
    F_{\eta}^{gg}(-1,1) = F_{\eta}(-1,1), \quad F_{\xi}^{gg}(1,-1)=F_{\xi}(1,-1), \quad
    -F_{\xi}^{gg}(-\eta,\eta) + F_{\eta}^{gg}(-\eta,\eta) = \frac{d}{d\eta}U_{BC}^g(\eta); \\
    PF^g(-1,\eta) = F(-1,\eta), \quad PF^g(\xi,-1)=F(\xi,-1), \quad
    PF^g(-\eta,\eta) = U_{BC}^g(\eta), \\
    (PF^g)_{\xi}(-\eta,\eta) = F_{\xi}^{gg}(-\eta,\eta), \quad
    (PF^g)_{\eta}(-\eta,\eta) = F_{\eta}^{gg}(-\eta,\eta).
    \end{array}\right.
  \end{equation}\endgroup
\end{proof}

\subsubsection{Two Robin Conditions Imposed on the Hypotenuse and a Leg}
\label{sec_r_ACBC}

If two Robin boundaries are present, with one of them being
the hypotenuse, we focus on the configuration in which these boundaries
correspond to $\overline{AC}$ and $\overline{BC}$, without loss of generality.
Our objective is to formulate the general form of a scalar function $u(\mbs x)$
that satisfies,
\begingroup\small
\begin{equation}
  u|_{\mbs x\in\overline{AB}} = u_{AB}(\mbs x),  \quad
     \mbs n\cdot\nabla u|_{\mbs x\in\overline{BC}}
     + \alpha_{BC}u|_{\mbs x\in\overline{BC}} = u_{rBC}(\mbs x), \quad
  \mbs n\cdot\nabla u|_{\mbs x\in\overline{AC}}
  + \alpha_{AC} u|_{\mbs x\in\overline{AC}} = u_{rAC}(\mbs x). 
\end{equation}\endgroup
Using the map from Section \ref{sec_21}, these conditions
are transformed into \eqref{eq_16a}, \eqref{eq_117}, and \eqref{eq_111}
about  $V(\xi,\eta)$.

These boundary conditions must be compatible.
At vertex $A$, the compatibility constraints consist of \eqref{eq_112a}--\eqref{eq_112b}.
At vertex $B$, equations \eqref{eq_n161a}--\eqref{eq_118d}
constitute the compatibility constraints, and equation~\eqref{eq_62a} is
the equivalent unconditional form,
in which $F_{\eta}^a(1,-1)$ is now defined by \eqref{eq_118d}
and $\Gamma_B$ is a free parameter.
An analysis analogous to in Section \ref{sec_n_ACBC} leads to
the following constraints at vertex $C$,
\begingroup\small
\begin{align}
  &\begin{bmatrix}V_{\xi}(-1,1)\\ V_{\eta}(-1,1)  \end{bmatrix}
  =\begin{bmatrix}F_{\xi}^a(-1,1)\\ F_{\eta}^a(-1,1) \end{bmatrix}
  - \begin{bmatrix}F_{\xi}^b(-1,1)\\ F_{\eta}^b(-1,1) \end{bmatrix}V(-1,1)
  =: \begin{bmatrix}F_{\xi}(-1,1)\\ F_{\eta}(-1,1)  \end{bmatrix}, \label{eq_n167} \\
  & V_{\xi\eta}(-1,1) = -S_{AC}(1)V_{\eta\eta}(-1,1) + Q_{AC}^a(1) + Q_{AC}^b(1)V(-1,1)
  =: F_{\xi\eta}(-1,1), \label{eq_n168}
\end{align}\endgroup
and
\begin{subequations}\label{eq_n169}
  \small
  \begin{align}
    & V_{\xi\xi}(-1,1) = S_1V_{\eta\eta}(-1,1) + S_2 + S_3V(-1,1), \quad
    \text{if}\ \overline{AC}\not\perp\overline{BC}\ \text{at}\ C, \label{eq_n169a} \\
    & V(-1,1) = -\frac{Q^a(-1,1)}{Q^b(-1,1)} =: F^a(-1,1), 
    \qquad\qquad\
    \text{if}\ \overline{AC}\perp\overline{BC}\ \text{at}\ C\ \text{and}\
    Q^b(-1,1)\neq 0, \label{eq_n169b} \\
    & Q^a(-1,1)=0,
    \qquad\qquad\qquad\qquad\qquad\qquad\qquad\ \;
    \text{if}\ \overline{AC}\perp\overline{BC}\ \text{at}\ C\ \text{and}\
    Q^b(-1,1)= 0. \label{eq_n169c}
  \end{align}
\end{subequations}
The symbols in the above equations are defined by,
\begingroup\small\begin{align}
  \left\{\begin{array}{l}
  F_{\xi}^a(-1,1) = \frac{[1-R_{BC}(1)]T_{rAC}(1) - S_{AC}(1)T_{rBC}(1)}
  {1-R_{BC}(1)[1+S_{AC}(1)]}, \quad
  F_{\eta}^a(-1,1) = \frac{-R_{BC}(1)T_{rAC}(1) + T_{rBC}(1)}
  {1-R_{BC}(1)[1+S_{AC}(1)]}, \\
  F_{\xi}^b(-1,1) = \frac{[1-R_{BC}(1)]\alpha_{AC}W_{AC}(1) - S_{AC}(1)\alpha_{BC}W_{BC}(1)}
  {1-R_{BC}(1)[1+S_{AC}(1)]}, \
  F_{\eta}^b(-1,1) = \frac{-R_{BC}(1)\alpha_{AC}W_{AC}(1) + \alpha_{BC}W_{BC}(1)}
  {1-R_{BC}(1)[1+S_{AC}(1)]},
  \\
  Q_{AC}^a(1) = T_{rAC}'(1) - [S_{AC}'(1)+\alpha_{AC}W_{AC}(1)]F_{\eta}^a(-1,1), \\
  Q_{AC}^b(1) = [S_{AC}'(1) + \alpha_{AC}W_{AC}(1)]F_{\eta}^b(-1,1)
  - \alpha_{AC}W_{AC}'(1), \\
  Q^a(-1,1) = T_{rBC}'(1) - [R_{BC}'(1)-\alpha_{BC}W_{BC}(1)][F_{\xi}^a(-1,1)-F_{\eta}^a(-1,1)]
  -[2R_{BC}(1)-1]Q_{AC}^a(1), \\
  Q^b(-1,1) = [R_{BC}'(1)-\alpha_{BC}W_{BC}(1)][F_{\xi}^b(-1,1)-F_{\eta}^b(-1,1)]
  - [2R_{BC}(1)-1]Q_{AC}^b(1) \\
  \qquad\qquad\quad\ \
  - \alpha_{BC}W_{BC}'(1), \\
  S_1 = 1 + [1-2R_{BC}(1)]\frac{\|\mbs x_{BC}'(1) \|^2}{\|\mbs x_{AC}'(1) \|^2}, \quad
  S_2 = -\frac{Q^a(-1,1)}{R_{BC}(1)}, \quad
  S_3 = -\frac{Q^b(-1,1)}{R_{BC}(1)}.
  \end{array}\right.
\end{align}\endgroup


To flag whether $Q^b(-1,1)=0$ when $\overline{AC}$ and $\overline{BC}$ are orthogonal
at $C$, we introduce a constant
\begingroup\small\begin{align}
  & \gamma_C = \text{sign}(|Q^b(-1,1)|) = \left\{\begin{array}{ll}
  0, & \text{if}\ Q^b(-1,1)=0, \\ 1, & \text{otherwise}.
  \end{array}\right.
\end{align}\endgroup
We next transform the conditional constraints in \eqref{eq_n169a} and \eqref{eq_n169b}
into equivalent unconditional forms,
\begingroup\small\begin{align}
  V_{\xi\xi}(-1,1) &= (1-\lambda_C)\Gamma_C^{\xi\xi}
  + \lambda_C\big[S_1V_{\eta\eta}(-1,1) + S_2 + S_3V(-1,1) \big] =: F_{\xi\xi}(-1,1),
  \label{eq_n172}\\
  V(-1,1) &= \big[1-(1-\lambda_C)\gamma_C \big]\Gamma_C^V
  + (1-\lambda_C)\gamma_CF^a(-1,1) =: F(-1,1), \label{eq_n173}
\end{align}\endgroup
where $\Gamma_C^{\xi\xi}$ and $\Gamma_C^V$ are free (undetermined) parameters.
No compatibility constraint exists on $V_{\eta\eta}(-1,1)$,
which is equivalent to the condition
\begingroup\small\begin{align}
  & V_{\eta\eta}(-1,1) = \Gamma_C^{\eta\eta} =: F_{\eta\eta}(-1,1), \label{eq_n174}
\end{align}\endgroup
with $\Gamma_C^{\eta\eta}$ being a free parameter.

Equations \eqref{eq_112b}, \eqref{eq_62a}, \eqref{eq_n167}, \eqref{eq_n168},
\eqref{eq_n172}, \eqref{eq_n173}, and \eqref{eq_n174}, together with
\eqref{eq_23a} and \eqref{eq_57b},
constitute the point constraints on the vertices of the triangle.
It should be noted that all the point constraints on vertices $B$ and $C$ are
in unconditional forms and
involve several free parameters $\Gamma_B$, $\Gamma_C^V$, $\Gamma_C^{\xi\xi}$,
and $\Gamma_C^{\eta\eta}$.

Using the same strategy as in Section \ref{sec_n_ACBC},
we transform the point constraints into conditions on
the edge profiles along $\overline{AC}$ and $\overline{BC}$.
Specifically, on $\overline{AC}$ we have the condition \eqref{eq_n144},
in which $F(-1,1)$, $F_{\eta}(-1,-1)$, $F_{\eta}(-1,1)$ and $F_{\eta\eta}(-1,1)$ are now
defined in \eqref{eq_n173}, \eqref{eq_112b}, \eqref{eq_n167}
and \eqref{eq_n174}, respectively.
On $\overline{BC}$ we have the condition \eqref{eq_n145},
in which $U_{BC}(\eta)$ is defined by \eqref{eq_n146}--\eqref{eq_n147b}
and $F_{\eta}(1,-1)$, $F_{\xi}(-1,1)$, $F_{\xi\xi}(-1,1)$, $F_{\xi\eta}(-1,1)$
are now defined by \eqref{eq_62a}, \eqref{eq_n167}, \eqref{eq_n172},
and \eqref{eq_n168}, respectively.
Again following the same strategy as in Section \ref{sec_n_ACBC}, we transform
the Robin condition \eqref{eq_117} together with \eqref{eq_n145} into,
\begin{subequations}\label{eq_n175}
  \small
\begin{align}
  & V_{\xi}(-\eta,\eta) = T_{rBC}(\eta) - \alpha_{BC}W_{BC}(\eta)U_{BC}(\eta)
  + [R_{BC}(\eta)-1]U_{BC}'(\eta) =: F_{\xi}(-\eta,\eta), \label{eq_n175a} \\
  & V_{\eta}(-\eta,\eta) = T_{rBC}(\eta) - \alpha_{BC}W_{BC}(\eta)U_{BC}(\eta)
  + R_{BC}(\eta)U_{BC}'(\eta) \qquad\
  =: F_{\eta}(-\eta,\eta). \label{eq_n175b}
\end{align}
\end{subequations}

Our goal is to find the general form of $V(\xi,\eta)$ that satisfies
\eqref{eq_16a}, \eqref{eq_n144}, \eqref{eq_111b}, \eqref{eq_n145},
\eqref{eq_n175a}, \eqref{eq_n175b}, together with \eqref{eq_n172},
\eqref{eq_n168}, and \eqref{eq_n174}. 

These constraints have the same forms to those in Section
\ref{sec_n_ACBC}. Therefore, the transfinite interpolation $PF(\xi,\eta)$
and the preliminary $V(\xi,\eta)$ are the same
as those from Section \ref{sec_n_ACBC}, respectively given by
\eqref{eq_148} and \eqref{eq_n151}, in
which $A_{F}(\xi,\eta)$, $Pg(\xi,\eta)$ and $A_g(\xi,\eta)$ are
defined in \eqref{eq_149b} and \eqref{eq_n152}.
It should be noted that $PF(\xi,\eta)$ contains the terms
$F(-1,1)$, $F_{\xi}(-1,1)$, $F_{\eta}(-1,1)$, $F_{\xi\xi}(-1,1)$,
$F_{\xi\eta}(-1,1)$, $F_{\eta\eta}(-1,1)$, $F_{\xi}(-\eta,\eta)$ and
$F_{\eta}(-\eta,\eta)$, which are now defined
by \eqref{eq_n173}, \eqref{eq_n167}, \eqref{eq_n172}, \eqref{eq_n168},
\eqref{eq_n174}, \eqref{eq_n175a} and \eqref{eq_n175b},
respectively, and contain several
free parameters and the edge free function $h(-\eta,\eta)$ (through $U_{BC}(\eta)$).

Substitution of the placeholder terms in $PF$ by their compatible data
closes the loop and results in the updated transfinite interpolant.
We first eliminate the free parameters and the edge free function
by relating them to the free function $g(\xi,\eta)$ based on
\eqref{eq_n153}, which leaves $g(\xi,\eta)$ as the only unknown
in the final formulation.
The updated transfinite interpolant is given by,
\begingroup\small\begin{align}\label{eq_n176}
  PF^g(\xi,\eta) =&\ A_F^g(\xi,\eta) + \kappa(\xi,\eta)\big[
    U_{BC}^g(-\xi) + (\xi+\eta)F_{\eta}^{gg}(\xi,-\xi) - A_F^g(\xi,-\xi) 
    - (\xi+\eta)A^g_{F,\eta}(\xi,-\xi)
    \big] \notag \\
  &+ \delta(\xi,\eta)\big[U_{BC}^g(\eta) + (\xi+\eta)F_{\xi}^{gg}(-\eta,\eta) - A_F^g(-\eta,\eta)
    - (\xi+\eta)A^g_{F,\xi}(-\eta,\eta) \big]. 
\end{align}\endgroup
where 
\begin{subequations}
  \small
  \begin{align}
    A_F^g(\xi,\eta) =&\  F(\xi,-1) + F^g(-1,\eta) - F(-1,-1)
    + (\xi+1)\big[F_{\xi}^g(-1,\eta) - F_{\xi}(-1,-1) \big],  \\
    F^g(-1,\eta) =&\ F(-1,-1)\rho_0(\eta) + F^{ag}(-1,1)\rho_1(\eta)
    + \lambda_AF_{\eta}^a(-1,-1)\upsilon_0(\eta) + F_{\eta}^{ag}(-1,1)\upsilon_1(\eta) \notag \\
    &+ g_{\eta\eta}(-1,1)\omega_1(\eta),  \\
    U_{BC}^g(\eta) =&\ g(-\eta,\eta) - \mathcal Qg(-\eta,\eta) + \mathcal QF^g(-\eta,\eta), \\
    \mathcal Qg(-\eta,\eta) =&\ g(1,-1)\rho_0(\eta) + g(-1,1)\rho_1(\eta)
    + \big[-g_{\xi}(1,-1) + g_{\eta}(1,-1) \big]\upsilon_0(\eta) \notag \\
    &+ \big[-g_{\xi}(-1,1) + g_{\eta}(-1,1) \big]\upsilon_1(\eta)
    + \big[g_{\xi\xi}(-1,1) - 2g_{\xi\eta}(-1,1) + g_{\eta\eta}(-1,1) \big]\omega_1(\eta), \\
    \mathcal QF^g(-\eta,\eta) =&\ F(1,-1)\rho_0(\eta) + F^{ag}(-1,1)\rho_1(\eta)
    + \big[-F_{\xi}(1,-1) + F^g_{\eta}(1,-1) \big]\upsilon_0(\eta) \notag \\
    &+ \big[-F^{ag}_{\xi}(-1,1) + F^{ag}_{\eta}(-1,1) \big]\upsilon_1(\eta)
    + \big[F^g_{\xi\xi}(-1,1) - 2F^g_{\xi\eta}(-1,1) + g_{\eta\eta}(-1,1) \big]\omega_1(\eta), \\
    F_{\eta}^g(1,-1) =&\ (1-\lambda_B)g_{\eta}(1,-1) + \lambda_B F_{\eta}^a(1,-1), \\
    F^{ag}(-1,1) =&\ \big[1-(1-\lambda_C)\gamma_C \big]g(-1,1)
    + (1-\lambda_C)\gamma_C F^a(-1,1), \\
    F_{\xi}^{ag}(-1,1) =&\ F_{\xi}^a(-1,1) - F_{\xi}^b(-1,1)F^{ag}(-1,1), \\
    F_{\eta}^{ag}(-1,1) =&\ F_{\eta}^a(-1,1) - F_{\eta}^b(-1,1)F^{ag}(-1,1), \\
    F_{\xi\xi}^g(-1,1) =&\ (1-\lambda_C)g_{\xi\xi}(-1,1)
    + \lambda_C[S_1g_{\eta\eta}(-1,1)+S_2 + S_3F^{ag}(-1,1)], \\
    F_{\xi\eta}^g(-1,1) =&\ -S_{AC}(1)g_{\eta\eta}(-1,1) +
    \big[Q_{AC}^a(1) + Q_{AC}^b(1)F^{ag}(-1,1) \big], \\
    F_{\xi}^{gg}(-\eta,\eta) =&\ T_{rBC}(\eta) - \alpha_{BC}W_{BC}(\eta)U_{BC}^g(\eta)
    + [R_{BC}(\eta)-1]\frac{d}{d\eta} U_{BC}^{g}(\eta), \\
    F_{\eta}^{gg}(\xi,-\xi) = &\ T_{rBC}(-\xi) - \alpha_{BC}W_{BC}(-\xi)U_{BC}^g(-\xi)  +
    R_{BC}(-\xi) \left.\frac{d}{d\eta}U_{BC}^g(\eta)\right|_{\eta=-\xi}, \\
    F_{\xi}^g(-1,\eta) =&\ T_{rAC}(\eta) - S_{AC}(\eta) V_{\eta}^g(-1,\eta)
    - \alpha_{AC}W_{AC}(\eta)V^g(-1,\eta), \\
    V^g(-1,\eta) =&\ g(-1,\eta) - [g(-1,-1)-F(-1,-1)]\rho_0(\eta)
    - [g(-1,1) - F^{ag}(-1,1)]\rho_1(\eta) \notag \\
    &- \lambda_A[g_{\eta}(-1,-1)-F_{\eta}^a(-1,-1)]\upsilon_0(\eta)
    - [g_{\eta}(-1,1) - F_{\eta}^{ag}(-1,1)]\upsilon_1(\eta), \\
    V_{\eta}^g(-1,\eta) =&\ g_{\eta}(-1,\eta) - [g(-1,-1)-F(-1,-1)]\rho_0'(\eta)
    - [g(-1,1) - F^{ag}(-1,1)]\rho_1'(\eta) \notag \\
    &- \lambda_A[g_{\eta}(-1,-1)-F_{\eta}^a(-1,-1)]\upsilon_0'(\eta)
    - [g_{\eta}(-1,1) - F_{\eta}^{ag}(-1,1)]\upsilon_1'(\eta).
  \end{align}
\end{subequations}
\begin{property}\label{lem_511}
  For any combination of $(\lambda_A,\lambda_B,\lambda_C,\gamma_C)=(0/1,0/1,0/1,0/1)$,
  the following relations hold,
  \begingroup\small\begin{align}
    \left\{\begin{array}{l}
    \lim_{\eta\rightarrow -1}F_{\xi}^g(-1,\eta)=: F_{\xi}^g(-1,-1) = F_{\xi}(-1,-1), \\
    F_{\xi}^{gg}(1,-1) = F_{\xi}(1,-1), \quad
    F_{\eta}^{gg}(-1,1) = F_{\eta}^{ag}(-1,1), \quad
    F_{\xi,\eta}^g(-1,1) = F_{\xi\eta}^g(-1,1), \\
    -F_{\xi}^g(-1,1) + F_{\eta}^{gg}(-1,1) = \frac{dU_{BC}^g}{d\eta}(1), \\
    \left.\frac{d}{d\xi}F_{\eta}^{gg}(\xi,-\xi) \right|_{\xi=-1}
    - A_{F,\eta\xi}^g(-1,1) + A_{F,\eta\eta}^g(-1,1) = 0. 
    \end{array}\right.
  \end{align}\endgroup
\end{property}
\begin{proof} By verification. \end{proof}

\begin{property}\label{lem_512}
  $PF^g(\xi,\eta)$ given by \eqref{eq_n176} satisfies, for any
  sufficiently differentiable $g(\xi,\eta)$
  therein,
  \begingroup\small\begin{align}
    \left\{\begin{array}{l}
    PF^g(\xi,-1) = F(\xi,-1), \quad PF^g(-1,\eta)=F^g(-1,\eta), \quad
    PF^g(-\eta,\eta) = U_{BC}^g(\eta) \\
    (PF^g)_{\xi}(-1,\eta)=F_{\xi}^g(-1,\eta), \quad
    (PF^g)_{\xi}(-\eta,\eta) = F_{\xi}^{gg}(-\eta,\eta), \quad
    (PF^g)_{\eta}(-\eta,\eta) = F_{\eta}^{gg}(-\eta,\eta), \\
    (PF^g)_{\xi\xi}(-1,1) = F_{\xi\xi}^g(-1,1), \quad
    (PF^g)_{\xi\eta}(-1,1) = F_{\xi\eta}^g(-1,1), \quad
    (PF^g)_{\eta\eta}(-1,1) = g_{\eta\eta}(-1,1).
    \end{array}\right.
  \end{align}\endgroup
\end{property}
\begin{proof} By verification, and by using Property \ref{lem_511} and
the relations in \eqref{eq_n10} and \eqref{eq_93}. \end{proof}

The final form for $V(\xi,\eta)$ is given by
\begingroup\small\begin{align}\label{eq_180}
  & V(\xi,\eta) = g(\xi,\eta) - Pg(\xi,\eta) + PF^g(\xi,\eta),
\end{align}\endgroup
where $g$ is a free function,
$Pg$ is defined in \eqref{eq_n152}, and $PF^g$ is given by
\eqref{eq_n176}. By using Properties \ref{lem_512} and \ref{lem_57} we
can establish the following result.
\begin{theorem}
  $V(\xi,\eta)$ given by \eqref{eq_180} satisfies the Dirichlet
  condition \eqref{eq_16a} and the Robin conditions \eqref{eq_117}
  and \eqref{eq_111a}, for any $g(\xi,\eta)$ therein that is sufficiently
  differentiable.
\end{theorem}

\begin{remark}
  If the two Robin conditions are imposed on $\overline{BC}$ and
  $\overline{AB}$ (instead of $\overline{AC}$),
  a reciprocal formulation for $V(\xi,\eta)$ can be constructed analogously by
  following the same strategy.
\end{remark}

\subsection{Geometric Domain Parameters}
\label{sec_52}

\noindent\underline{\bf Section \ref{sec_31} Variable-coefficient Poisson equation:}\\[5pt]
\underline{Domain \#1:}
\vspace{-7pt}
\begin{align*}\small
  \begin{array}{l}
    \mbs x_A=(-1,-0.5), \quad \mbs x_B=(1.5,-0.8), \quad \mbs x_C=(-0.6,1.3); \\
    \mbs x_{AB}(\xi) = \mbs x_A\phi_0(\xi) + \mbs x_B\phi_1(\xi) - (0,0.1\sin(\pi\xi)),
    \quad \xi\in[-1,1]; \\
    \mbs x_{AC}(\eta) = \mbs x_A\phi_0(\eta) + \mbs x_C\phi_1(\eta)
    + (0.2\sin(\pi \eta),0), \quad \eta\in[-1,1]; \\
    \mbs x_{BC}(\eta) = \mbs x_B\phi_0(\eta) + \mbs x_C\phi_1(\eta)
    + (0.3\sin(\pi \eta),0), \quad \eta\in[-1,1].
  \end{array}
\end{align*}

\noindent\underline{Domain \#2:}
\vspace{-7pt}
\begin{align*}\small
  \begin{array}{l}
    \mbs x_A=(-1.5,0.5), \quad \mbs x_B=(0.5,0.5), \quad \mbs x_C=(-1,3); \\
    \mbs x_{AB}(\xi) = (\cos(\pi\phi_0(\xi)), -\sin(\pi\phi_0(\xi))) + (-0.5,0.5),
    \quad \xi\in[-1,1]; \\
    \mbs x_{AC}(\eta) = (-1+0.5\sin(\pi\eta),\ 2.5\phi_1(\eta)+0.5 ),
    \quad \eta\in[-1,1]; \\
    \mbs x_{BC}(\eta) = (-0.5\sin(\pi\eta),\ 2.5\phi_1(\eta)+0.5),
    \quad \eta\in[-1,1].
  \end{array}
\end{align*}

\noindent\underline{Domain \#3:}
\vspace{-7pt}
\begin{align*}\small
  \begin{array}{l}
    \mbs x_A=(\cos(\theta_{12}-\pi/3),\ \sin(\theta_{12}-\pi/3)) + \mbs x_{o_1},
    \quad \mbs x_B=(\cos(\theta_{12}+\pi/3),\ \sin(\theta_{12}+\pi/3)) + \mbs x_{o_1}, \\
    \mbs x_C=(\cos(\theta_{12}+\pi),\ \sin(\theta_{12}+\pi)) + \mbs x_{o_1};
    \quad \mbs x_{o_1}=(1.75,0.5), \ \ \theta_{12}=3\pi/4, \ \
    \mbs x_{o_2}=(\cos\theta_{12},\sin\theta_{12})+\mbs x_{o_1}; \\
    \mbs x_{AB}(\xi) = (\cos(t),\sin(t)) + \mbs x_{o_2},
    \quad t=(4\pi/3)\phi_0(\xi) + (2\pi/3)\phi_1(\xi) + \theta_{12},
    \quad \xi\in[-1,1]; \\
    \mbs x_{AC}(\eta) = (\cos(t),\sin(t)) + \mbs x_{o_1}, \quad
    t = (5\pi/3)\phi_0(\eta) + \pi\phi_1(\eta) + \theta_{12},
    \quad \eta\in[-1,1]; \\
    \mbs x_{BC}(\eta) = (\cos(t),\sin(t)) + \mbs x_{o_1}, \quad
    t = (\pi/3)\phi_0(\eta)+\pi\phi_1(\eta)+\theta_{12},
    \quad \eta\in[-1,1].
  \end{array}
\end{align*}

\noindent\underline{Domain \#4:}
\vspace{-7pt}
\begin{align*}\small
  \begin{array}{l}
    \mbs x_A=\mbs x_{AB}(-1), \quad \mbs x_B=\mbs x_{AB}(1), \quad \mbs x_C=\mbs x_{AC}(1);
     \\
     \mbs x_{AB}(\xi) = ((1+0.7\cos(2t))\cos(t),\ (1+0.7\cos(2t))\sin(t) ) + (0,0.25),
     \ t = (3\pi/4)\phi_1(\xi),
    \ \xi\in[-1,1]; \\
    \mbs x_{AC}(\eta) = ((1+0.7\cos(2t))\cos(t),\ (1+0.7\cos(2t))\sin(t)) + (0,0.25),
    \ t = 2\pi\phi_0(\eta) + (5\pi/4)\phi_1(\eta), 
    \ \eta\in[-1,1]; \\
    \mbs x_{BC}(\eta) = ((1+0.7\cos(2t))\cos(t),\ (1+0.7\cos(2t))\sin(t)) + (0,0.25),
    \ t = (3\pi/4)\phi_0(\eta) + (5\pi/4)\phi_1(\eta),
    \ \eta\in[-1,1].
  \end{array}
\end{align*}

\noindent\underline{Domain \#5:}
\vspace{-7pt}
\begin{align*}\small
  \begin{array}{l}
    \mbs x_A=\mbs x_{AB}(-1), \quad \mbs x_B=\mbs x_{AB}(1), \quad \mbs x_C=\mbs x_{AC}(1);
     \\
     \mbs x_{AB}(\xi) = ((0.8+0.16\cos(6t))\cos(t),\ (0.8+0.16\cos(6t))\sin(t) ) + (-0.25,0.25),
     \ t = (2\pi/3)\phi_1(\xi),
    \ \xi\in[-1,1]; \\
    \mbs x_{AC}(\eta) = ((0.8+0.16\cos(6t))\cos(t),\ (0.8+0.16\cos(6t))\sin(t)) + (-0.25,0.25),
    \ t = 2\pi\phi_0(\eta) + (4\pi/3)\phi_1(\eta), \\
    \quad \eta\in[-1,1]; \\
    \mbs x_{BC}(\eta) = ((0.8+0.16\cos(6t))\cos(t),\ (0.8+0.16\cos(6t))\sin(t)) + (-0.25,0.25),
    \ t = (2\pi/3)\phi_0(\eta) + (4\pi/3)\phi_1(\eta), \\
    \quad \eta\in[-1,1].
  \end{array}
\end{align*}

\noindent\underline{\bf Section \ref{sec_32} Nonlinear Helmholtz equation:}\\[5pt]
\underline{Domain \#1:}
\vspace{-7pt}
\begin{align*}\small
  \begin{array}{l}
    \mbs x_A=\mbs x_{AB}(-1), \quad \mbs x_B=\mbs x_{AB}(1), \quad \mbs x_C=\mbs x_{AC}(1); \quad
    (X(t),Y(t)) = [0.8+0.28\cos(3t)](\cos t, \sin t); \quad \mbs x_o=(0.5,0.5);
     \\
     \mbs x_{AB}(\xi) = \frac{\sqrt{2}}{2}(X(t)-Y(t), X(t)+Y(t)) + \mbs x_o,
     \quad
     \ t = (2\pi/3)\phi_1(\xi),
    \quad  \xi\in[-1,1]; \\
    \mbs x_{AC}(\eta) = \frac{\sqrt{2}}{2}(X(t)-Y(t), X(t)+Y(t)) + \mbs x_o,
    \quad t = 2\pi\phi_0(\eta) + (4\pi/3)\phi_1(\eta), 
    \quad \eta\in[-1,1]; \\
    \mbs x_{BC}(\eta) = \frac{\sqrt{2}}{2}(X(t)-Y(t), X(t)+Y(t)) + \mbs x_o,
    \quad t = (2\pi/3)\phi_0(\eta) + (4\pi/3)\phi_1(\eta),
    \quad \eta\in[-1,1].
  \end{array}
\end{align*}

\noindent\underline{Domain \#2:}
\vspace{-7pt}
\begin{align*}\small
  \begin{array}{l}
    \mbs x_A=\mbs x_{AB}(-1), \quad \mbs x_B=\mbs x_{AB}(1), \quad \mbs x_C=\mbs x_{AC}(1);
    \quad \mbs x_{o_1}=(0,0), \quad \mbs x_{o_2}=(-0.6,0), \quad \mbs x_{o_3}=(1.05,0);
    \\
    R_1=1.5, \ R_2=1.2,\ R_3=0.45,\ h=0.6; \ \
    \theta_2 = \cos^{-1}\left(\frac{R_2^2 + h^2 - R_1^2}{2 R_2 h}\right), \ \ 
    \theta_A = \cos^{-1}\left(\frac{R_1^2+R_2^2 - h^2}{2R_1 R_2}\right), \ \
    \theta_1 = \theta_2 + \theta_A;
    \\
    \mbs x_{AB}(\xi) = R_2(\cos t, \sin t) + \mbs x_{o_2},
    \quad t=\theta_2\phi_0(\xi),
    \quad \xi\in[-1,1]; \\
    \mbs x_{AC}(\eta) = R_1(\cos t, \sin t) + \mbs x_{o_1}, \quad t=\theta_1\phi_0(\eta),
    \quad \eta\in[-1,1]; \\
    \mbs x_{BC}(\eta) = R_3(\cos t, \sin t) + \mbs x_{o_3}, \quad t = -\pi\phi_0(\eta),
    \quad \eta\in[-1,1].
  \end{array}
\end{align*}

\noindent\underline{Domain \#3:}
\vspace{-7pt}
\begin{align*}\small
  \begin{array}{l}
    \mbs x_A=(-0.5,0.5), \quad \mbs x_B=(1.5,0.5), \quad \mbs x_C=(0.5,1.5); \\
    \mbs x_{AB}(\xi) = (-0.5\phi_0(\xi)+1.5\phi_1(\xi),\
    0.5+[2/(\pi\sqrt{3})-0.21]\sin[\pi\phi_1(\xi)] + 0.07\sin[3\pi\phi_1(\xi)] ),
    \quad \xi\in[-1,1]; \\
    \mbs x_{AC}(\eta) = (\cos t, \sin t) + (0.5,0.5),
    \quad t = \pi\phi_0(\eta)+(\pi/2)\phi_1(\eta),
    \quad \eta\in[-1,1]; \\
    \mbs x_{BC}(\eta) = (\cos t, \sin t) + (0.5,0.5),
    \quad t = (\pi/2)\phi_1(\eta),
    \quad \eta\in[-1,1].
  \end{array}
\end{align*}

\noindent\underline{Domain \#4:}
\vspace{-7pt}
\begin{align*}\small
  \begin{array}{l}
    \mbs x_A=(-0.85,-0.6), \quad \mbs x_B=(1.1,-0.45), \quad \mbs x_C=(-0.1,1.2); \\
    \bm\tau^{AB}=(\tau_x^{AB},\tau_y^{AB})=\frac{\mbs x_B-\mbs x_A}{\|\mbs x_B-\mbs x_A \|},
    \ \mbs n^{AB}=(\tau_y^{AB},-\tau_x^{AB}),
    \ \bm\tau^{BC}=(\tau_x^{BC},\tau_y^{BC})=\frac{\mbs x_C-\mbs x_B}{\|\mbs x_C-\mbs x_B \|},
    \ \mbs n^{BC}=(\tau_y^{BC},-\tau_x^{BC});
    \\
    \mbs x_{AB}(\xi) = \mbs x_A \phi_0(\xi) + \mbs x_B\phi_1(\xi)
    + 0.12\sin[\pi\phi_1(\xi)]\mbs n^{AB},
    \quad \xi\in[-1,1]; \\
    \mbs x_{AC}(\eta) = \mbs x_A\phi_0(\eta) + \mbs x_C\phi_1(\eta),
    \quad \eta\in[-1,1]; \\
    \mbs x_{BC}(\eta) = \mbs x_B\phi_0(\eta) + \mbs x_C\phi_1(\eta)
    + 0.1\sin[6\pi\phi_1(\eta)]\mbs n^{BC},
    \quad \eta\in[-1,1].
  \end{array}
\end{align*}

\noindent\underline{Domain \#5:}
\vspace{-7pt}
\begin{align*}\small
  \begin{array}{l}
    \mbs x_A=\mbs x_{AB}(-1), \quad \mbs x_B=\mbs x_{AB}(1), \quad \mbs x_C=\mbs x_{AC}(1);
    \quad \mbs x_o = (0.75,-\sqrt{3}/4+0.5),
    \\
    \mbs x_{AB}(\xi) = (\cos t, \sin t) + \mbs x_C, \quad t = -(2\pi/3)\phi_0(\xi),
    \quad \xi\in[-1,1]; \\
    \mbs x_{AC}(\eta) = (\cos t, \sin t) + \mbs x_o,
    \quad t = \pi\phi_0(\eta)+(2\pi/3)\phi_1(\eta),
    \quad \eta\in[-1,1]; \\
    \mbs x_{BC}(\eta) = (\cos t, \sin t) + \mbs x_o,
    \quad t = (\pi/3)\phi_0(\eta) + (2\pi/3)\phi_1(\eta),
    \quad \eta\in[-1,1].
  \end{array}
\end{align*}

\section*{Declaration of Generative AI Use}

Claude Code (using the Claude Sonnet 5.5 model) was used to generate an initial version
of the computer code implementing the numerical method presented in this paper.
The authors subsequently reviewed, revised, and validated the AI-generated code to
produce the final implementation. All numerical experiments and simulations reported
in this paper were designed and conducted by the authors, who take full responsibility
for the implementation and the reported results.

\bibliographystyle{plain}
\bibliography{reference,bc,pinn_bc,elm,elm1,elm2,dnn,dnn1,dnn2,dnn3,dnn4,ref1,ml,bc_ref}

\end{document}